\documentclass[11pt, reqno]{amsart}
\usepackage{amssymb}
\usepackage{bbm}
\usepackage{amsthm}
\usepackage{nccmath}
\usepackage[bookmarks=false]{hyperref}
\usepackage{etoolbox}
\usepackage{array}
\usepackage{xparse}
\usepackage{enumitem}
\usepackage{tikz}
\usepackage{mathtools}
\usepackage[font=scriptsize]{caption}

\calclayout

\newtheorem{thm}{Theorem}[section]
\newtheorem{theorem}[thm]{Theorem}
\newtheorem{lem}[thm]{Lemma}
\newtheorem{lemma}[thm]{Lemma}
\newtheorem{prop}[thm]{Proposition}
\newtheorem{cor}[thm]{Corollary}

\theoremstyle{definition}
\newtheorem{definition}[thm]{Definition}

\theoremstyle{remark}

\newtheorem{rmk}[thm]{Remark}

\newcommand{\thmref}[1]{Theorem~\ref{#1}}
\newcommand{\corref}[1]{Corollary~\ref{#1}}
\newcommand{\secref}[1]{\S\ref{#1}}
\newcommand{\lemref}[1]{Lemma~\ref{#1}} 
\newcommand{\propref}[1]{Proposition~\ref{#1}}
\newcommand{\defref}[1]{Definition~\ref{#1}}

\newcommand*{\figref}[1]{Figure~\ref{#1}}

\newcommand*{\dis}{\displaystyle}

\newcommand*{\qq}{\qquad}

\newcommand*{\tx}[1]{\text{#1}}

\newcommand*{\lamb}{\lambda}

\newcommand*{\ep}{\epsilon}
\newcommand*{\kap}{\kappa}

\newcommand*{\suchthat}{\, \middle| \,}

\newcommand*{\n}{\!\!}

\newcommand*{\Pminus}{P_{-}}
\newcommand*{\Pminusbar}{\myoverline{-3}{0}{P}_{-}}

\newcommand*{\myoverline}[3]{\mkern -#1mu\overline{\mkern#1mu#3\mkern#2mu}\mkern -#2mu}	
\newcommand*{\vboldbar}{\myoverline{0}{0}{\vbold}}
\newcommand*{\Zbar}{\myoverline{-3}{0}{\Z}}
\newcommand*{\zbar}{\myoverline{-2}{0}{\z}}

\newcommand*{\Dapbar}{\myoverline{-3}{-1}{D}_\ap}
\newcommand*{\wbar}{\myoverline{0}{-0.5}{\omega}}

\newcommand*{\Ubar}{\myoverline{0}{0}{\U}}

\newcommand*{\fbar}{\myoverline{0}{0}{f}}

\newcommand*{\Dspbar}{\myoverline{0}{0}{\Dsp}}

\newcommand*{\half}{\frac{1}{2}}
\newcommand*{\onebytwo}{\frac{1}{2}}
\newcommand*{\threebytwo}{\frac{3}{2}}

\newcommand*{\onebythree}{\frac{1}{3}}
\newcommand*{\twobythree}{\frac{2}{3}}

\newcommand*{\onebyfour}{\frac{1}{4}}

\newcommand*{\Rsp}{\mathbb{R}}
\newcommand*{\Csp}{\mathbb{C}}
\newcommand*{\Nsp}{\mathbb{N}}
\newcommand*{\Zsp}{\mathbb{Z}}
\newcommand*{\Dsp}{\mathbb{D}}
\newcommand*{\Tsp}{\mathbb{T}}

\newcommand*{\Scalsp}{\mathcal{S}}

\newcommand*{\Sone}{\mathbb{S}^1}

\newcommand*{\Lone}{L^1}
\newcommand*{\Ltwo}{L^2}

\newcommand*{\Linfty}{L^{\infty}}
\newcommand*{\Hhalf}{\dot{H}^\half}

\newcommand*{\al}{\alpha}
\newcommand*{\ap}{{\alpha'}}

\newcommand*{\bp}{{\beta'}}

\newcommand*{\xp}{{x'}}
\newcommand*{\yp}{{y'}}
\newcommand*{\zp}{{z'}}

\DeclareMathOperator*{\Tr}{Tr}
\newcommand*{\diff}{\mathop{}\! d}
\newcommand*{\difff}{\mathop{}\!\! d}
\newcommand*{\compose}[1]{\circ{#1}}

\newcommand*{\conv}{*}
\newcommand*{\Hil}{\mathbb{H}}

\newcommand*{\Hiltil}{\widetilde{\mathbb{H}}}
\newcommand*{\Pa}{\mathbb{P}_A}
\newcommand*{\Ph}{\mathbb{P}_H}

\newcommand*{\Id}{\mathbb{I}}
\newcommand*{\Imag}{\tx{Im}}
\newcommand*{\Real}{\tx{Re}}

\newcommand*{\Av}{Av}
\newcommand*{\Avg}{\Av}

\newcommand*{\sgn}{sgn}

\newcommand*{\grad}{\nabla}
\newcommand*{\Dt}{D_t}

\newcommand*{\pt}{\partial_t}
\newcommand*{\ps}{\partial_s}
\newcommand*{\px}{\partial_x}
\newcommand*{\py}{\partial_y}
\newcommand*{\pz}{\partial_z}
\newcommand*{\pxp}{\partial_\xp}
\newcommand*{\pyp}{\partial_\yp}
\newcommand*{\pzp}{\partial_\zp}
\newcommand*{\pap}{\partial_\ap}
\newcommand*{\papabs}{\abs{\pap}}
\newcommand*{\pbp}{\partial_\bp}
\newcommand*{\pal}{\partial_\al}
\newcommand*{\palabs}{\abs{\pal}}

\newcommand*{\Dal}{D_{\al}}
\newcommand*{\Dap}{D_{\ap}}
\newcommand*{\Dapep}{\Dap^\ep}

\newcommand*{\Dapabs}{\abs{D_{\ap}}}
\newcommand*{\Dapfrac}{\frac{1}{\Zap}\pap}
\newcommand*{\Dapbarfrac}{\frac{1}{\Zapbar}\pap}

\newcommand*{\Dapabsfrac}{\frac{1}{\Zapabs}\pap}

\newcommand*{\E}{E} 
\newcommand*{\Eone}{E_1}
\newcommand*{\Etwo}{E_2}
\newcommand*{\Ethree}{E_3}
\newcommand*{\Ea}{E_a}

\newcommand*{\Ebdd}{\widetilde{E}} 
\newcommand*{\Eonebdd}{\widetilde{E}_1}
\newcommand*{\Etwobdd}{\widetilde{E}_2}
\newcommand*{\Ethreebdd}{\widetilde{E}_3}
\newcommand*{\Eabdd}{\widetilde{E}_a}

\newcommand*{\Ecal}{\mathcal{E}}
\newcommand*{\Ecaltil}{\widetilde{\mathcal{E}}}
\newcommand*{\Ecalone}{\mathcal{E}_1}
\newcommand*{\Ecaltilone}{\widetilde{\mathcal{E}}_1}
\newcommand*{\Ecaltwo}{\mathcal{E}_2}
\newcommand*{\Ecaltiltwo}{\widetilde{\mathcal{E}}_2}
\newcommand*{\Ecalthree}{\mathcal{E}_3}
\newcommand*{\Ecaltilthree}{\widetilde{\mathcal{E}}_3}

\newcommand*{\Ecalsigma}{\mathcal{E}_{\sigma}}

\newcommand*{\Fcal}{\mathcal{F}}
\newcommand*{\Fcaltil}{\widetilde{\mathcal{F}}}

\newcommand*{\vbold}{\mathbf{v}}

\newcommand*{\Pfrak}{\mathfrak{P}}

\newcommand*{\A}{\mathcal{A}}

\newcommand*{\Azero}{A_0}
\newcommand*{\Ag}{A_g}

\newcommand*{\bvar}{b}

\newcommand*{\bap}{\bvar_\ap}
\newcommand*{\bvarap}{\bap}

\newcommand*{\avar}{a}

\newcommand*{\h}{h}
\newcommand*{\hep}{\h^\ep}
\newcommand*{\hvart}{\h_t}
\newcommand*{\hal}{\h_\al}

\newcommand*{\hinv}{\h^{-1}}
\newcommand*{\htil}{\widetilde{\h}}

\newcommand*{\U}{U}
\newcommand*{\Uep}{\U^\ep}

\newcommand*{\Util}{\widetilde{U}}
\newcommand*{\Uzp}{\U_\zp}

\newcommand*{\Ut}{\U_t}

\newcommand*{\g}{g}

\newcommand*{\f}{f}

\newcommand*{\w}{\omega}

\newcommand*{\Psiep}{\Psi^\ep}

\newcommand*{\Psizp}{\Psi_{\zp}}

\newcommand*{\onePsizp}{\frac{1}{\Psizp}}

\newcommand*{\Psit}{\Psi_{t}}

\newcommand*{\Qzero}{Q_0}

\newcommand*{\Qzerobdd}{\widetilde{Q}_0}

\newcommand*{\Jzero}{J_0}
\newcommand*{\Jone}{J_1}

\newcommand*{\Jzerobdd}{\widetilde{J}_0}
\newcommand*{\Jonebdd}{\widetilde{J}_1}

\newcommand*{\Rzero}{R_0}
\newcommand*{\Rone}{R_1}

\newcommand*{\Rzerobdd}{\widetilde{R}_0}
\newcommand*{\Ronebdd}{\widetilde{R}_1}

\newcommand*{\z}{z}

\newcommand*{\zal}{\z_\al}

\newcommand*{\zalbar}{\zbar_\al}

\newcommand*{\zalabs}{\abs{\zal}}

\newcommand*{\zt}{\z_t}

\newcommand*{\ztbar}{\zbar_t}
\newcommand*{\ztal}{\z_{t\al}}

\newcommand*{\ztalbar}{\zbar_{t\al}}

\newcommand*{\ztt}{\z_{tt}}

\newcommand*{\zttbar}{\zbar_{tt}}
\newcommand*{\zttal}{\z_{tt\al}}

\newcommand*{\ztttbar}{\zbar_{ttt}}

\newcommand*{\Z}{Z}
\newcommand*{\Zep}{\Z^\ep}
\newcommand*{\Zap}{\Z_{,\ap}}

\newcommand*{\Zapep}{\Zap^\ep}

\newcommand*{\Zapbar}{\Zbar_{,\ap}}

\newcommand*{\Zapabs}{\abs{\Zap}}

\newcommand*{\Zt}{\Z_t}
\newcommand*{\Ztep}{\Zt^\ep}
\newcommand*{\Ztbar}{\Zbar_t}

\newcommand*{\Ztbarep}{\Zbar_t^\ep}
\newcommand*{\Ztepbar}{\Ztbarep}
\newcommand*{\Ztap}{\Z_{t,\ap}}
\newcommand*{\Ztapep}{\Ztap^\ep}
\newcommand*{\Ztapbar}{\Zbar_{t,\ap}}
\newcommand*{\Ztapepbar}{\Ztapbar^\ep}

\newcommand*{\Ztbarbp}{\Zbar_{t,\bp}}
\newcommand*{\Ztbpbar}{\Zbar_{t,\bp}}

\newcommand*{\Ztt}{\Z_{tt}}
\newcommand*{\Zttep}{\Zep_{tt}}
\newcommand*{\Zttbar}{\Zbar_{tt}}

\newcommand*{\Zttap}{\Z_{tt,\ap}}
\newcommand*{\Zttapbar}{\Zbar_{tt,\ap}}
\newcommand*{\Zttbarap}{\Zttapbar}
\newcommand*{\Zttbpbar}{\Zbar_{tt,\bp}}

\newcommand*{\Ztttbar}{\Zbar_{ttt}}

\newcommand*{\nobrac}[1]{ #1 }

\DeclarePairedDelimiter{\oldbrac}{\lparen}{\rparen}			
\NewDocumentCommand{\brac}{ s o m }{						
	\IfBooleanT{#1}{
  		\IfValueT{#2}{\oldbrac[#2]{#3}}
		\IfValueF{#2}{\oldbrac{#3}} 
	}
	\IfBooleanF{#1}{
  		\IfValueT{#2}{\PackageError{mypackage}{Incorrect use of brac. Insert star}{}}
		\IfValueF{#2}{\oldbrac*{#3}} 
	}		
}

\DeclarePairedDelimiter\oldcbrac{\lbrace}{\rbrace}				
\NewDocumentCommand{\cbrac}{ s o m }{					
	\IfBooleanT{#1}{
  		\IfValueT{#2}{\oldcbrac[#2]{#3}}
		\IfValueF{#2}{\oldcbrac{#3}} 
	}
	\IfBooleanF{#1}{
  		\IfValueT{#2}{\PackageError{mypackage}{Incorrect use of cbrac. Insert star}{}}
		\IfValueF{#2}{\oldcbrac*{#3}} 
	}		
}

\DeclarePairedDelimiter\oldsqbrac{\lbrack}{\rbrack}				
\NewDocumentCommand{\sqbrac}{ s o m }{					
	\IfBooleanT{#1}{
  		\IfValueT{#2}{\oldsqbrac[#2]{#3}}
		\IfValueF{#2}{\oldsqbrac{#3}} 
	}
	\IfBooleanF{#1}{
  		\IfValueT{#2}{\PackageError{mypackage}{Incorrect use of sqbrac. Insert star}{}}
		\IfValueF{#2}{\oldsqbrac*{#3}} 
	}		
}

\DeclarePairedDelimiter{\oldabs}{\lvert}{\rvert}
\NewDocumentCommand{\abs}{ s o m }{						
	\IfBooleanT{#1}{
  		\IfValueT{#2}{\oldabs[#2]{#3}}
		\IfValueF{#2}{\oldabs{#3}} 
	}
	\IfBooleanF{#1}{
  		\IfValueT{#2}{\PackageError{mypackage}{Incorrect use of abs. Insert star}{}}
		\IfValueF{#2}{\oldabs*{#3}} 
	}		
}

\DeclarePairedDelimiterX{\oldnorm}[1]{\lVert}{\rVert}{#1}
\NewDocumentCommand{\norm}{ s o o m }{					
	\IfValueT{#2} {
		\IfBooleanT{#1}{
  			\IfValueT{#3}{\oldnorm[#2]{#4}_{#3}}
			\IfValueF{#3}{\oldnorm{#4}_{#2}} 
		}
		\IfBooleanF{#1}{
  			\IfValueT{#3}{\PackageError{mypackage}{Incorrect use of norm. Insert star}{}}
			\IfValueF{#3}{\oldnorm*{#4}_{#2}} 
		}
	}
	\IfValueF{#2} {
		\IfBooleanT{#1}{\oldnorm{#4}}	
		\IfBooleanF{#1}{\oldnorm*{#4}}		
	}	
}

\makeatletter
\def\black@#1{%
    \noalign{%
        \ifdim#1>\displaywidth
            \dimen@\prevdepth
            \nointerlineskip
            \vskip-\ht\strutbox@
            \vskip-\dp\strutbox@
            \vbox{\noindent\hbox to \displaywidth{\hbox to#1{\strut@\hfill}}}%
            \prevdepth\dimen@
        \fi
    }%
}
\makeatother

\makeatletter
\renewcommand{\tocsection}[3]{%
  \indentlabel{\@ifnotempty{#2}{\bfseries\ignorespaces#1 #2\quad}}\bfseries#3}
\renewcommand{\tocsubsection}[3]{%
  \indentlabel{\@ifnotempty{#2}{\ignorespaces#1 #2\quad}}#3}

\newcommand\@dotsep{4.5}
\def\@tocline#1#2#3#4#5#6#7{\relax
  \ifnum #1>\c@tocdepth 
  \else
    \par \addpenalty\@secpenalty\addvspace{#2}%
    \begingroup \hyphenpenalty\@M
    \@ifempty{#4}{%
      \@tempdima\csname r@tocindent\number#1\endcsname\relax
    }{%
      \@tempdima#4\relax
    }%
    \parindent\z@ \leftskip#3\relax \advance\leftskip\@tempdima\relax
    \rightskip\@pnumwidth plus1em \parfillskip-\@pnumwidth
    #5\leavevmode\hskip-\@tempdima{#6}\nobreak
    \leaders\hbox{$\m@th\mkern \@dotsep mu\hbox{.}\mkern \@dotsep mu$}\hfill
    \nobreak
    \hbox to\@pnumwidth{\@tocpagenum{\ifnum#1=1\bfseries\fi#7}}\par
    \nobreak
    \endgroup
  \fi}
\AtBeginDocument{%
\expandafter\renewcommand\csname r@tocindent0\endcsname{0pt}
}
\def\l@subsection{\@tocline{2}{0pt}{2.5pc}{5pc}{}}
\makeatother

\makeatletter
 \def\@testdef #1#2#3{%
   \def\reserved@a{#3}\expandafter \ifx \csname #1@#2\endcsname
  \reserved@a  \else
 \typeout{^^Jlabel #2 changed:^^J%
 \meaning\reserved@a^^J%
 \expandafter\meaning\csname #1@#2\endcsname^^J}%
 \@tempswatrue \fi}
\makeatother

\makeatletter
\newcommand*{\rom}[1]{\expandafter\@slowromancap\romannumeral #1@}
\makeatother

\makeatletter
\patchcmd{\@sect}{\@addpunct.}{}{}{}
\patchcmd{\subsection}{-.5em}{1em}{}{}
\makeatother

\AtBeginDocument{%
   \def\MR#1{}
}

\begin{document}

\title[2D Water waves]{Uniform in gravity estimates for 2D water waves}
\author[Siddhant Agrawal]{Siddhant Agrawal}
\address{\parbox{\linewidth}{Department of Mathematics\\
University of Colorado Boulder\\
2300 Colorado Avenue\\
Boulder, CO 80309-0395\\
USA\smallskip}}
\email{Siddhant.Agrawal@colorado.edu}

\begin{abstract}
We consider the 2D gravity water waves equation on an infinite domain. We prove a local wellposedness result which allows interfaces with corners and cusps as initial data and which is such that the time of existence of solutions is uniform even as the gravity parameter $g \to 0$. For $g>0$, we prove an improved blow up criterion for these singular solutions and we also prove an existence result for $g = 0$. Moreover the energy estimate used to prove this result is scaling invariant.  

As an application of this energy estimate, we then consider the water wave equation with no gravity where the fluid domain is homeomorphic to the disc. We prove a local wellposedness result which allows for interfaces with angled crests and cusps as initial data and then by a rigidity argument, we show that there exists initial interfaces with angled crests for which the energy blows up in finite time, thereby proving the optimality of this local wellposedness result. For smooth initial data, this local wellposedness result gives a longer time of existence as compared to previous results when the initial velocity is small and we also improve upon the blow up criterion. 
\end{abstract}


\maketitle
\tableofcontents

\section{Introduction}

We are concerned with the motion of a fluid in dimension two with a free boundary. In this work we will identify 2D vectors with complex numbers. The fluid is assumed to be inviscid, incompressible and irrotational. The fluid domain $\Omega(t) \subset \Csp$ and the air region are separated by an interface $\partial\Omega(t)$. The air and the fluid are assumed to have constant densities of 0 and 1 respectively. The effect of surface tension is ignored. We will consider two different models:

In the first model we assume that the interface $\partial\Omega(t)$ is homeomorphic to $\Rsp$ and tends to real line at infinity. The fluid is below the air region and there is no bottom. The fluid is subject to a uniform gravitational field $-gi$ acting in the downward direction (here $g \geq 0$).
The motion of the fluid is then governed by the Euler equation
\begin{align}\label{eq:Euler}
\begin{split}
& u_t + (u \cdot \grad)u = -gi -\grad P  \qq\text{ on } \Omega (t) 	\\
& \tx{div } u = 0, \quad \tx{curl } u = 0 \qq\text{ on } \Omega(t) \\
& P = 0 \qq\text{ on } \partial\Omega (t) \\
& (1, u) \tx{ is tangent to the free surface } (t, \partial\Omega(t))
\end{split}
\end{align}
along with the decay conditions $ u \to 0, \grad P \to -gi$  as $ \abs{(x,y)} \to \infty$. We will consider this problem for all values of the gravity parameter $g \geq 0$. 

For the second model we will consider the problem when the domain $\Omega(t)$ is bounded and is homeomorphic to the unit disc $\Dsp$. In this case we assume that there is no gravity. With these assumptions, the equations become
\begin{align}\label{eq:Euler2}
\begin{split}
& u_t + (u \cdot \grad) u = -\grad P  \qq\text{ on } \Omega (t) 	\\
& \tx{div } u = 0, \quad \tx{curl } u = 0 \qq\text{ on } \Omega(t)  \\
& P = 0 \qq\text{ on } \partial\Omega (t) \\
& (1, u) \tx{ is tangent to the free surface } (t, \partial\Omega(t)) 
\end{split}
\end{align}

The earliest results on local well-posedness for the Cauchy problem are for small data in 2D and were obtained by Nalimov \cite{Na74}, Yoshihara \cite{Yo82,Yo83} and Craig \cite{Cr85}. In the case of zero surface tension and gravity $g = 1$, Wu \cite{Wu97,Wu99} obtained the proof of local well-posedness for arbitrary data in Sobolev spaces. Continuous dependence of solutions on the initial data was proved in \cite{HuIfTa16, Ng16}.  These results has been extended in various directions, see the works \cite{ChLi00, Li05, La05, CoSh07, ZhZh08, ShZe08, ShZe11, CaCoFeGaGo13, AlBuZu14, KuTu14, Po14, AlBuZu16, HuIfTa16, WaZh17, HaIfTa17, AlBuZu18, Ai19, Ai20, Gi21, WaZhZhZh21, AlBuZu22, AiIfTa24, IfPiTaTa25}.

An important quantity related to the well-posedness of the problem in the zero surface tension case is the Taylor sign condition proposed by Taylor in \cite{Ta50}. The condition is that there exists a constant $c>0$ such that 
\begin{align*}
-\frac{\partial P}{\partial n} \geq c >0 \quad \tx{ on } \partial\Omega(t)
\end{align*}
where $n$ is the outward unit normal. In \cite{Eb87}, Ebin gave an example of initial data with non-zero vorticity not satisfying the Taylor sign condition, for which the problem is ill posed. In \cite{BeHoLo93} Beale, Hou and Lowengrub proved that the linearized problem around a solution is wellposed if the Taylor sign condition is satisfied. In \cite{Wu97} Wu proved that for the infinite bottom case and gravity $g = 1$, the Taylor sign condition is satisfied if the interface is $C^{1,\alpha}$ for $\alpha>0$. This was later shown to be true for flat bottoms and with perturbations to flat bottom by Lannes \cite{La05}.  See also \cite{HuIfTa16, Su23}.

In \cite{CaCoFeGaGo13} Castro, C\'ordoba, Fefferman, Gancedo and 
  G\'omez-Serrano proved the existence of splash singularities for the water wave equation (see also \cite{CoSh14}). In order to study solutions with non $C^1$ interfaces, Kinsey and Wu \cite{KiWu18} proved an a priori estimate for angled crested water waves in the case of zero surface tension. Using this, Wu \cite{Wu19} proved a local wellposedness result which allows the initial data to have interfaces with corners and cusps. In \cite{Ag20} we proved that the singular solutions constructed in \cite{Wu19} are rigid and in particular the angle of the corner does not change with time. In a recent work \cite{CoEnGr23} for a zero gravity model, C\'{o}rdoba, Enciso and Grubic construct solutions with interfaces which have corners and cusps, with the property that the angle of the corner changes with time. Note that at the corners and cusps in \cite{Wu19, CoEnGr23}, the Taylor sign condition is not satisfied and one has $-\frac{\partial P}{\partial n} = 0$ at those singularities. For the related question concerning the contact angle problem for water waves, see \cite{Po19, MiWa20, MiWa21, MiWa24}. 

In this paper we explore the question of local wellposedness when the gravity parameter $g \to 0$. For example consider the following question: For some fixed initial data, does one have a uniform time of existence for the solutions as $g \to 0$? All the above results give a time of existence $T \to 0$ as $g \to 0$. A related question of taking the zero surface tension limit $\sigma \to 0$ when $g = 1$ was first studied by Ambrose and Masmoudi \cite{AmMa05, AmMa09}. As we explain below, taking $g \to 0$ when $\sigma = 0$ is significantly more difficult due to issues with the Taylor sign condition. 

To understand our main motivation for studying this problem, we first need to recall the scaling symmetries of the water wave equation and our previous work \cite{Ag21} on surface tension. Let $\Z(\cdot,t)$ and $\Zt(\cdot,t)$ be the interface and the velocity on the boundary in conformal coordinates. If $(\Z(\ap,t), \Zt(\ap,t)) $ is a solution to the water wave equation in the time interval $[0,T]$  with gravity $g \geq 0$  and surface tension $\sigma \geq 0$, then for any $\lamb>0$ and $s \in \Rsp$, the functions $((\Z_\lamb)(\ap,t), (\Zt)_{\lamb}(\ap,t))$ given by 
\begin{align}\label{eq:scaling}
(\Z_\lamb)(\ap,t) = \lamb^{-1}\Z(\lamb\ap, \lamb^s t), \quad  (\Zt)_{\lamb}(\ap,t) = \lamb^{s-1}\Zt(\lamb\ap, \lamb^s t) 
\end{align}
is also a solution to the water wave equation  in the time interval $[0, \lamb^{-s} T]$ with gravity $g_\lamb = \lamb^{2s - 1}g$ and surface tension $ \sigma_\lamb = \lamb^{2s-3}\sigma$. 

Note that if $g > 0$ and $\sigma = 0$, then $s = \half$ keeps $g_\lamb = g$ and $\sigma_\lamb = 0$. If $g = 0$ and $\sigma > 0$, then $s = \threebytwo$ keeps $g_\lamb = 0$ and $\sigma_\lamb = \sigma$. If $g= \sigma = 0$, then we see that for all $s \in \Rsp$ we have that $g_\lamb = \sigma_\lamb = 0$. In this paper when we talk about scaling transformations, we will consider all transformations of the type \eqref{eq:scaling} and not restrict ourselves to the case where $s = \half$ or $s = \threebytwo$, as we are interested in the entire family of solutions for $g \geq 0$.

The work of Kinsey and Wu \cite{KiWu18} and Wu \cite{Wu19} on angled crested water waves was restricted to the case of surface tension $\sigma = 0$ and in \cite{Ag21} we studied the problem of angled crested like water waves with $\sigma \geq 0$. If the interface is smooth but is close to being angled crested, then its curvature is very large and hence surface tension becomes quite dominant, and hence this problem is very relevant physically. In \cite{Ag21} we showed that the natural extension of the energy of \cite{KiWu18} to the case of non-zero surface tension does not allow initial interfaces with corners (or indeed any singularities). We also proved an energy estimate for an energy called $\Ecalsigma(t)$, which works for all $\sigma \geq 0$ and which reduces to a lower order version of the energy of \cite{KiWu18} for $\sigma = 0$. Although this energy $\Ecalsigma(t)$ does not allow initial data with corner interfaces when $\sigma > 0$, it does allow smooth interfaces with very large curvature. More precisely we proved a local wellposedness result based on this energy, and showed that there exists a large class of smooth initial data with $\Linfty$ norm of curvature $\norm[\infty]{\kap} \sim \sigma^{-\frac{1}{3} + \delta}$ for any fixed $\delta > 0$ arbitrarily small, for which the time of existence is $O(1)$ even as $\sigma \to 0$. Note that for $\sigma$ small, this gives a time of existence much larger as compared to previous existence results, as those would yield a time of existence $T \lesssim \frac{1}{\norm[\infty]{\kap}} \sim \sigma^{\frac{1}{3} - \delta} $ which is very small if $\sigma$ is small.   

This growth rate of $\sigma^{-\onebythree}$ for the $\Linfty$ norm of the curvature in the work \cite{Ag21} was heuristically explained by the following scaling argument: As one is interested in the zero surface tension limit of solutions in \cite{Ag21}, one is comparing solutions with different values of surface tension on the same time interval and so in \eqref{eq:scaling} one should put $s = 0$. If one also ignores gravity i.e. puts $g = 0$, then this yields from  \eqref{eq:scaling} $\Z_\lamb(\ap,t) = \lamb^{-1}\Z(\lamb\ap, t)$ and $\sigma_\lamb = \lamb^{-3}\sigma$ and $g_\lamb = 0$. In this regime we see that the curvature $\kap_\lamb (\al,t) = \lamb\kap(\lamb\al,t)$ which yields $\sigma_\lamb\kap_\lamb^3 (\al,t) = \sigma\kap^3(\lamb\al,t)$. Hence $\norm*[\infty]{\sigma^{\onebythree}\kap}$ is invariant under this scaling and so the $\Linfty$ norm of the curvature grows like $\sigma^{-\onebythree}$ as $\sigma \to 0$.

One of our main motivation is to place this heuristic scaling argument on a more rigorous footing. Note that there is huge gap between what is proven in \cite{Ag21} and this heuristic. First of all the model considered in \cite{Ag21} had $g =1$ and moreover the energy estimate for $\Ecal_\sigma(t)$ was not scaling invariant under \eqref{eq:scaling} with $s = 0$, which is required to make sense of the scaling argument. To see one of the main difficulties in constructing a modified energy $\widetilde{\Ecal}_\sigma(t)$ and proving an energy estimate which works for $g = 0$ and is scaling invariant under \eqref{eq:scaling} with $s = 0$, observe that one consequence of such an energy estimate would be that if one lets $\sigma \to 0$, then for $\sigma = 0$ one would get an energy estimate for $g = \sigma = 0$, which is scaling invariant under \eqref{eq:scaling} with $s = 0$ and allows initial interfaces with corners and cusps similar to \cite{KiWu18} and \cite{Wu19}. This reduced problem itself is wide open as even for nice smooth decaying initial data, no existence results are known for $g = \sigma = 0$ on the infinite domain and moreover the only scaling invariant energy estimates in water waves are the works \cite{HuIfTa16} and \cite{AiIfTa24} which are for $g = 1, \sigma = 0$ and are invariant under \eqref{eq:scaling} for $s = \frac{1}{2}$. 

In this paper we tackle this reduced problem and solve it completely. We actually prove a much stronger result as our energy estimate is invariant under \eqref{eq:scaling} for all values of $s \in \Rsp$ and not just $s = 0$ and works for all $g \geq 0$. As a consequence of this new energy estimate, we then also prove a blow up result for the model \eqref{eq:Euler2} with sharp bounds on the blow up time.

We now give a heuristic argument to illustrate one of the main difficulties of this problem. Consider the following simplified model of the water wave equation
\begin{align*}
\Dt^2 f + \brac{-\frac{\partial P}{\partial n}}\palabs f = l.o.t
\end{align*}
Here $\Dt$ is the material derivate, $\palabs = \sqrt{-\Delta}$,  l.o.t are lower order terms and $f$ is some variable such as the velocity on the interface. To prove energy estimates for this equation, one can naturally take the following energy
\begin{align*}
E(t) =  \int_\Rsp \abs{\Dt f}^2 \diff \al + \int_\Rsp \brac{-\frac{\partial P}{\partial n}} \abs{\palabs^\half f}^2 \diff \al
\end{align*}
and take the time derivative to get 
\begin{align}\label{eq:approxtimeenergy}
\frac{\diff}{\diff t} E(t) = \int_\Rsp \Dt\brac{-\frac{\partial P}{\partial n}} \abs{\palabs^\half f}^2 \diff \al + errors
\end{align}
Now if $g >0$ and fixed, and the interface is $C^{1,\gamma}$ for some $\gamma > 0$, then as shown in \cite{Wu97}, we have a positive lower bound on $-\frac{\partial P}{\partial n}$ and hence the energy $E(t)$ controls the $\Hhalf(\Rsp)$ norm of $f$. Then one shows that we can control $\norm[\infty]{ \Dt\brac{-\frac{\partial P}{\partial n}}}$ from the energy and hence the first term on the right hand side of \eqref{eq:approxtimeenergy} is controlled. 

Now as we show later in the introduction, $-\frac{\partial P}{\partial n} \to g$ at infinity. So as $g \to 0$, the energy $E(t)$ no longer controls the $\Hhalf(\Rsp)$ norm of $f$ uniformly. In the extreme case of $g = 0$, we will see that if the interface is $C^{1, \gamma}$ and the velocity on the boundary is in $H^1(\Rsp)$ and is not identically zero, then $-\frac{\partial P}{\partial n} > 0$ everywhere but  $-\frac{\partial P}{\partial n} \to 0$ at infinity. Therefore in the case of $g = 0$, the energy $E(t)$ does not control the $\Hhalf(\Rsp)$ norm of $f$, but a weighted  $\Hhalf(\Rsp)$ norm of $f$ with the weight decaying to zero at infinity. Hence to control the first term of \eqref{eq:approxtimeenergy}, it is no longer enough to simply get a bound on $\norm[\infty]{ \Dt\brac{-\frac{\partial P}{\partial n}}}$ by the energy. 

This is one of the fundamental issues one faces when trying to prove an energy estimate which works uniformly as $g \to 0$. Similar issues come up in several places in the energy estimate. Note that this issue occurs even for nice smooth initial data and we are interested in studying interfaces with corners and cusps, and so we need to deal with both these issues simultaneously. We overcome this difficulty by proving new estimates for the Taylor sign condition term and then use these to prove a new energy estimate which works for all values of gravity $g \geq 0$. We explain this later in the introduction in greater detail. First let us now informally state our first main result which is for the model \eqref{eq:Euler}. We construct an energy $\Ecal(t)$ and prove a local wellposedness result based on this energy and also prove a blow up criterion based on a quantity called $B(t)$. The result is as follows:
\begin{thm}\label{thm:intromain}
(Informal) Let $g \geq 0$ and consider an initial data with $\Ecal(0)<\infty$. Then there exists a time $T \gtrsim \frac{1}{\Ecal(0)^\half}$ such that there exists a solution to \eqref{eq:Euler} with gravity $g$ in $[0,T]$. 

If $g>0$, then the solution is unique and in this case if $T^*>0$ is the maximal time of existence, then either $T^* = \infty$ or $\limsup_{t \to T^*} B(t) = \infty$. 
\end{thm} 
The precise statement of this theorem is given in \thmref{thm:existencemain}. The quantity $B(t)$ can be seen from \thmref{thm:existencemain} and the energy $\Ecal(t)$ is defined just before \thmref{thm:existencemain}. We now mention some of the key features of this result:
\begin{enumerate}
\item The energy $\Ecal(t)$ is finite if the initial data belongs to regular enough Sobolev spaces on $\Rsp$. Moreover the energy is finite for interfaces with a finite number of acute angled corners or cusps. Therefore the above theorem proves local wellposedness for interfaces with angled crests and cusps similar to \cite{Wu19} and the energy $\Ecal(t)$ is a modified and improved version of the energy used in \cite{Wu19}. 
\item The energy $\Ecal(t)$ is an increasing function of gravity. Hence for a fixed initial data and gravity $g$ satisfying $0\leq g \leq g_0$, we have $\Ecal(0) \leq C$ where $C$ depends only on $g_0$ and the initial data. Therefore, from the lower bound on the time of existence $T \gtrsim \frac{1}{\Ecal(0)^\half}$, we see that the time of existence of solutions is uniform as $g \to 0$. 

This is the first result which proves uniform time of existence of solutions as $g\to 0$ in the non-periodic setting, as  all previous results gave a time of existence $T\to 0$ as $g \to 0$. The above theorem also gives the first local existence result for $g = \sigma = 0$ in the non-periodic setting. For the periodic case, see the comments made after \thmref{thm:introbdd} below. 

\item The lower bound of the time of existence $T \gtrsim \frac{1}{\Ecal(0)^\half}$ is scaling invariant with respect to all the scaling transformations \eqref{eq:scaling}. To understand this, fix $g \geq 0$ and consider an initial data $(\Z,\Zt)(0)$ with $\Ecal(0) < \infty$. Then for this initial data and gravity $g$, the above theorem gives a time of existence of $T \gtrsim \frac{1}{\Ecal(0)^\half}$. Now let $\lamb > 0$ and $s \in \Rsp$ and consider the initial data $(\Z,\Zt)_\lamb(0)$ given by \eqref{eq:scaling} and solve \eqref{eq:Euler} with this initial data and gravity $g_\lamb = \lamb^{2s - 1}g$. Now the energy is such that $\Ecal_\lamb(0) = \lamb^{2s}\Ecal(0)$ and hence the above theorem gives a time of existence of $T_\lamb \gtrsim \lamb^{-s}\frac{1}{\Ecal(0)^\half}$ which is consistent as that from the scaling \eqref{eq:scaling} namely $T_\lamb = \lamb^{-s} T$.

The proof of this lower bound on the time of existence, relies on a new a priori estimate \thmref{thm:aprioriE} which is scaling invariant with respect to all the scaling transformations of \eqref{eq:scaling}. To the best of our knowledge, this is the first such a priori estimate as previously the scaling invariant estimates of \cite{HuIfTa16} and \cite{AiIfTa24} were restricted to $g= 1$ and $s = \frac{1}{2}$. 

\item From the precise version of this result namely \thmref{thm:existencemain}, we see that $B(t)$ scales like $\norm[\infty]{\grad u}(t)$ and the new blow up criterion in the above theorem is a significant improvement over the blow up criterion for singular water waves in \cite{Wu19}. 
\end{enumerate}
In the above theorem we do not prove uniqueness of solutions for the $g = 0$ case. Although we believe that uniqueness holds in this case, we do not put in any effort to prove this as that is not the main focus of this paper. 

For model \eqref{eq:Euler2}, we prove a local wellposedness result similar to \thmref{thm:intromain} by modifying its proof and this result improves several aspects of previous known results. The energy is modified to $\Ecaltil(t)$ and the quantity in the blow up criterion is modified to $ \widetilde{B}(t)$. Moreover we then prove a blow up result for this model with angled crested initial interfaces, which proves the optimality of this local wellposedness result. The informal statement of the theorem is as follows:

\begin{figure}[h]
\centering
 \begin{tikzpicture}[scale=0.35]
\draw [thick] (-2,0) to [out=15,in=180] (6,3.5) to [out=0,in=180-15] (10,0) ;
\draw [thick] (-2,0) to [out=-15,in=180] (6,-3.5) to [out=0,in=-180+15] (10,0) ;
\draw [thick, ->] (-0.5,0.0) -- (1,0);
\draw [thick, ->] (9,0.0) -- (7.5,0);
\end{tikzpicture}
\caption{An example of blow up of the energy $\Ecaltil(t)$}\label{fig:blowup}
\end{figure}

\begin{thm}\label{thm:introbdd}
(Informal) Consider an initial data with $\Ecaltil(0)<\infty$. Then there exists a time $T \gtrsim \frac{1}{\Ecaltil(0)^\half}$ such that there exists a unique solution to \eqref{eq:Euler2}  in $[0,T]$. If $T^*>0$ is the maximal time of existence, then either $T^* = \infty$ or $\limsup_{t \to T^*} \widetilde{B}(t) = \infty$. 

Moreover there exist initial data with $\Ecaltil(0)<\infty$ for which the maximal time of existence $T^*$ is finite with $T^* \approx \frac{1}{\Ecaltil(0)^\half}$ and $\limsup_{t \to T^*} \widetilde{B}(t) = \infty$. In particular $\limsup_{t \to T^*} \Ecaltil(t) = \infty$.
\end{thm}
The precise statement of the first part of this result is given in \thmref{thm:existencemainbdd} and the precise blow up result is \thmref{thm:blowup}. The quantity $\widetilde{B}(t)$ can be seen from \thmref{thm:existencemainbdd} and the energy $\Ecaltil(t)$ is defined just before \thmref{thm:existencemainbdd}. We mention some of the key features of this result:

\begin{enumerate}
\item Similar to \thmref{thm:intromain}, the above theorem proves local wellposedness  for initial data which belongs to regular enough Sobolev spaces on $\Tsp$ and also for initial interfaces having a finite number of acute angled corners or cusps. Note that the above model is for $g = 0$ and we do have uniqueness in the above theorem.
\item The above theorem gives significantly larger time of existence of solutions if the initial velocity is small, even for smooth initial data. To understand this, consider a smooth initial domain and initial velocity $\ep v_0$ with $v_0$ fixed and $\ep > 0$ small. As the pressure is quadratic in the velocity, this implies that the lower bound on $-\frac{\partial P}{\partial n}$ scales like $\ep^2$. Now the time of existence, in all previous local wellposedness results for model \eqref{eq:Euler2}, depends \emph{quantitatively} on this lower bound and thereby gives a time of existence $T \to 0$ as $\ep \to 0$.

In contrast, the time of existence in the above theorem \emph{does not} depend on the lower bound on $-\frac{\partial P}{\partial n}$. For such initial data, the energy $\Ecaltil(0)$ scales like $\ep^2$ and hence the time of existence that we get from the above theorem is $T \gtrsim \frac{1}{\ep} \to \infty$ as $\ep \to 0$. This is what one would expect naturally, because if $\ep = 0$ then the trivial solution is the clear global in time solution. We also remark that the energy $\Ecaltil(0) = 0 $ if and only if the initial velocity is a constant function.

A similar argument applies as well to the model \eqref{eq:Euler} with periodic boundary conditions. Although we do not prove results in this case, this situation is essentially identical to the model \eqref{eq:Euler2} at least with respect to local wellposedness issues. For model  \eqref{eq:Euler} with periodic boundary conditions, for any given initial data, standard local wellposedness arguments already give uniform in time of existence as gravity $g \to 0$ (this is because of compactness). However the time of existence one obtains from such arguments for $g = 0$ is very small if the initial velocity is small, in exactly the same way as described above. The analogous theorem to \thmref{thm:introbdd} for model \eqref{eq:Euler} with periodic boundary conditions will remove this issue in the same way as mentioned above. 

\item The blow up criterion in the above theorem is an improvement over previous blow up criterions for the model \eqref{eq:Euler2}, because in contrast to previous results the blow up criterion above does not impose the validity of the Taylor sign condition (see \thmref{thm:existencemainbdd} for the precise definition of $\widetilde{B}(t)$). Moreover the blow up criterion does not depend on the chord arc condition or whether the interface is self intersecting or not. 
\item The blow up in the above result is a new type of singularity formation, as the blow up of $\widetilde{B}(t)$ implies that either some weighted norm of the gradient of the velocity blows up or a weighted norm of the second derivative of the interface in conformal coordinates blows up. This blow up result is not related to the splash/splat singularity formation in \cite{CaCoFeGaGo13} in any sense. Note that for the above theorem, the interface is parametrized in conformal coordinates $Z(\cdot,t):\Tsp \to \Csp$, and whether the mapping $Z(\cdot,t)$ is injective or not plays no role in the analysis and does not affect the quantity $ \widetilde{B}(t)$ in the blow up criterion. In particular the solution can be continued even if there is a splash/splat singularity. Of course if that happens, then the physical relevance of the solution is lost past the splash/splat singularity. 

\item This blow up result is proved by utilizing a pinch off scenario (see \figref{fig:blowup}). The proof is via a contradiction argument and in particular we do not known whether a pinch off actually happens or not, though we do know that $ \widetilde{B}(t)$ and the energy $\Ecaltil(t)$ blows up. As the proof is via contradiction, we do not know the exact behavior of the solution at the blow up time nor do we know the exact time of blow up. However we give very tight lower and upper bounds on the time of blow up, which in particular proves the optimality of the lower bound on the time of existence $T \gtrsim \frac{1}{\Ecaltil(0)^\half}$ in the local wellposedness result. In particular the blow up example shows that this bound cannot be improved to $T \gtrsim \frac{1}{\Ecaltil(0)^{\half + \delta}}$ for any $\delta > 0$, even for initial data with small $\Ecaltil(0)$. 
\end{enumerate}

The proof of the blow up crucially uses the new blow up criterion proved in the above theorem. The proof of the blow up also strongly uses the fact that for interfaces with corners/cusps and with $\Ecaltil(0) < \infty$, the singularities are rigid similar to \cite{Ag20} (see \thmref{thm:mainangle}). In particular this blow up result cannot be easily modified to prove blow up for smooth solutions. 

Let us now explain the construction of the energy and some key ideas in the proof. 

\medskip
\textbf{The energy, heuristics and key estimates:}
\smallskip

Consider the model \eqref{eq:Euler} and we parametrize the interface $\partial \Omega(t)$ in conformal coordinates $Z(\cdot,t):\Rsp \to \partial \Omega(t) \subset \Csp$ and use $\ap \in \Rsp$ as the conformal parameter. The velocity on the boundary in conformal coordinates is denoted by $\Zt$ and so $\Zt(\ap,t) = u(Z(\ap,t),t)$. Note that here we are using the notation of Wu \cite{Wu19} and in particular $\partial_t Z \neq \Zt$ but $\Dt Z = \Zt$ where $\Dt$ is the material derivative.

Now if $\ps$ is the arc length derivative on $\partial \Omega(t)$, then in conformal coordinates this transforms to the weighted derivative $\frac{1}{\Zapabs}\pap$ on $\Rsp$. The Dirichlet to Neumann map $\mathcal{G}$ on $\partial\Omega(t)$ transforms to $\frac{1}{\Zapabs}\papabs$ on $\Rsp$, where $\papabs = \sqrt{-\Delta}$. The arc length measure $\diff \sigma$  on $\partial\Omega(t)$ transforms to the measure $\Zapabs \diff \ap$ on $\Rsp$. Finally as shown by Wu \cite{Wu97} the quantity $-\frac{\partial P}{\partial n}$ on $\partial \Omega(t)$ transforms to $\frac{\Ag}{\Zapabs}$ on $\Rsp$ where 
\begin{align}\label{eq:Agintro}
\Ag(\ap,t) = g + \frac{1}{2\pi}\int_\Rsp \abs{\frac{\Zt(\ap,t) - \Zt(\bp,t)}{\ap-\bp}}^2 \diff\bp
\end{align}
For $\Zt(\cdot,t) \in H^1(\Rsp)$, we see that $0\leq g \leq \Ag \leq g + C\norm[2]{\Ztap}^2$, for some universal constant $C>0$ (see \propref{prop:Hardy}). Now if the interface is $C^{1,\gamma}$ for some $\gamma>0$ and decays to the flat interface at infinity appropriately, then there exists constants $c_1, c_2 > 0$ such that $0< c_1 \leq \Zapabs(\ap) \leq c_2 < \infty$ for all $\ap \in \Rsp$ (see Theorem 3.5 in \cite{Po92}). Hence if $g>0$ and the interface is $C^{1,\gamma}$, then there exists $\tilde{c}_1, \tilde{c}_2 >0$ such that $0 <  \tilde{c}_1 \leq \frac{\Ag}{\Zapabs}(\ap) \leq  \tilde{c}_2 < \infty $ for all $\ap \in \Rsp$ and so
\begin{align*}
0 < \tilde{c}_1 \leq   -\frac{\partial P}{\partial n}(x) \leq \tilde{c}_2 < \infty \qq \tx{ for all } x\in \partial \Omega(t)
\end{align*}
Therefore the Taylor sign condition is satisfied, and this is the same argument as in \cite{Wu97}. 

If the interface has a corner of angle $\nu\pi$ with $0<\nu<1$ at $\Z(0,t)$, then near $\ap = 0$ we have
\begin{align}\label{eq:Zasymptotics}
\Z(\ap,t) \sim (\ap)^\nu  \qq \frac{1}{\Zap}(\ap,t) \sim (\ap)^{1 - \nu} \qq  \pap\frac{1}{\Zap}(\ap,t) \sim (\ap)^{- \nu}
\end{align} 
In particular we see that $\frac{1}{\Zapabs} = 0$ at $\ap = 0$, and therefore $-\frac{\partial P}{\partial n} = 0$ at the corner and so the Taylor sign condition is not satisfied. The water waves equation on the boundary $\partial \Omega(t)$ can be written as (see \cite{Wu16, ChLi00})
\begin{align}\label{eq:WWeqgeometric}
\Dt^2 f + \brac{-\frac{\partial P}{\partial n}}\mathcal{G} f = l.o.t
\end{align}
where $\Dt$ is the material derivative, $\mathcal{G}$ is the Dirichlet to Neumann map and $f$ is some variable such as the velocity on the boundary. In conformal coordinates this equation transform into the following equation on $\Rsp$ (here we still use $\Dt$ and $f$ as the material derivative and the main variable respectively)
\begin{align}\label{eq:WWintro}
\Dt^2 f + \frac{\Ag}{\Zapabs^2}\papabs f = l.o.t
\end{align}
So the issue with the interface having a corner is that $\frac{1}{\Zapabs} = 0$ at those points, which causes issues with the energy estimate for the above equation. This issue was overcome by Kinsey and Wu \cite{KiWu18} by using weighted energy estimates, with the weights being powers of $\frac{1}{\Zap}$, and a slightly modified energy estimate was also proved in \cite{Wu19}. The highest order energy used in \cite{Wu19} is 
\begin{align}\label{eq:Wuintro}
E(t) = \int_\Rsp \frac{\Zapabs^2}{\Ag} \abs{\Dt f}^2 \diff \ap + \int_\Rsp \abs{\papabs^\half f}^2 \diff \ap
\end{align}
with $f = \frac{1}{\Zap^2}\pap\brac{ \frac{\Ztapbar}{\Zap} } $. 

In this paper, we are interested in studying the problem when $g \to 0$. From the formula \eqref{eq:Agintro} we see that if $g = 0$, then from \lemref{lem:decayatinfinity} we have  $\Azero(\ap,t) \to 0$ as $\abs{\ap} \to \infty$. 
Note that as $\Zap \to 1$ as $\abs{\ap} \to \infty$, and as $-\frac{\partial P}{\partial n} = \frac{\Ag}{\Zapabs}$, we see that for $g = 0$ we also have $-\frac{\partial P}{\partial n} \to 0$ as $\abs{\ap} \to \infty$. This causes the standard energy estimates to fail for the water wave equation \eqref{eq:WWintro}, even for smooth initial data, and we are interested in studying angled crested interfaces. So in particular we need to deal with two issues simultaneously, namely that  $\Azero(\ap,t) \to 0$ as $\abs{\ap} \to \infty$ and that $\frac{1}{\Zapabs} = 0$ at the corners/cusps. This will require us to very carefully select the energy and deal with the two weights $\Azero$ and $\frac{1}{\Zapabs}$ in a very careful and deliberate manner. 

First let us try to understand how previous energies such as the ones in \cite{Wu19} and \cite{Ag21} do not work. We immediately see that the above energy \eqref{eq:Wuintro} used in \cite{Wu19} is very problematic when $g \to 0$, because of the $\Ag$ factor in the denominator in the first term. Even if one imposes extra decay conditions on the initial data to make this energy finite at $t = 0$, the proof of the energy estimate in \cite{Wu19} fundamentally uses the fact that $g = 1$ in several places. As an example, one of the most important terms to control in the energy estimate for water waves with angled crests is the $\Linfty$ norm of the gradient of the velocity, which in this context is the term $ \norm[\infty]{\frac{\Ztapbar}{\Zap}} $. To control this term from the energy, in \cite{Wu19} one interpolates between two terms, one of which is $\norm[2]{\frac{1}{\Zap}\pap\brac{\frac{\Ztapbar}{\Zap}} }$. This term is added to the energy as a lower order term. However to then control the time derivative of this term, it is estimated by assuming $g = 1$ and then controlled it by the first term of \eqref{eq:Wuintro} (see (4.30) and (4.32) in \cite{Wu19}). This same issue is also present in the energy of \cite{KiWu18}. 

In \cite{Ag21} we worked on water waves with angled crested like solutions with surface tension and proved an energy estimate which works for all values of surface tension $\sigma \geq 0$. When $\sigma = 0$, the highest order energy in \cite{Ag21} is
\begin{align}\label{eq:Etsurfacetension}
E(t) = \int_\Rsp \abs{\papabs^\half (\Dt f)}^2 \diff \ap + \int_\Rsp \Ag \abs{\frac{1}{\Zapabs}\pap f}^2 \diff \ap
\end{align}
where $f = \frac{\Ztapbar}{\Zapbar}$. Now this energy is finite as $g \to 0$ without imposing any artificial decay conditions on the initial data, however $g = 1$ is again fundamentally used in \cite{Ag21} in several places to prove the energy estimate. Once again this can be seen by noting the proof of the estimate for $\norm[\infty]{\frac{\Ztapbar}{\Zap}}$ (see estimate 1 and 4 in section 4.1 in \cite{Ag21}). 
This term is again controlled by interpolating between two terms, one of which is $\norm[2]{\frac{1}{\Zap}\pap\brac{\frac{\Ztapbar}{\Zap}}}$. This term is controlled by using $g = 1$ and then using the second term of \eqref{eq:Etsurfacetension}. 

To avoid these issues, in this paper we use energy of the form
\begin{align}\label{eq:El}
E_l(t) = \int_\Rsp \abs{\Dt f}^2 \diff \ap + \int_\Rsp \abs{\papabs^\half\brac{\frac{\sqrt{\Ag}}{\Zap} f}}^2 \diff \ap
\end{align}
for $f = \pap\frac{1}{\Zap}$. To understand the fundamental difficulty in controlling the time derivative of the above energy, first note that we want to use \eqref{eq:WWintro} as the main equation.  We have
\begin{align*}
& \half \frac{\diff }{\diff t} \brac{ \norm[2]{\Dt f}^2 + \norm[\Hhalf]{\frac{\sqrt{\Ag}}{\Zap}f }^2 } \\
& \approx \Real\cbrac{\int_\Rsp (\Dt^2 f)(\Dt \fbar)\diff \ap + \int_\Rsp  \cbrac{\papabs\brac*[\bigg]{\frac{\sqrt{\Ag}}{\Zap} f}} \Dt\brac*[\bigg]{\frac{\sqrt{\Ag}}{\Zapbar} \bar{f}} \difff\ap }
\end{align*}
For the second term we use
\begin{align*}
 \Dt\brac*[\bigg]{\frac{\sqrt{\Ag}}{\Zapbar} \bar{f}}  = \cbrac{\frac{\Dt\Ag}{2\Ag} + \Zapbar\Dt\frac{1}{\Zapbar} }\frac{\sqrt{\Ag}}{\Zapbar}\bar{f} + \frac{\sqrt{\Ag}}{\Zapbar}\Dt\bar{f}
\end{align*}
Now the first term in the above expression is an error which needs to be controlled in $\Hhalf$. To control this, we in particular need a bound for $\norm[\infty]{\frac{\Dt\Ag}{\Ag}} $. This is very problematic as if $g = 0$, then $\Ag \to 0$ at infinity which makes the denominator go to zero at infinity, and hence it would seem that it is not possible to control this term.

Surprisingly we show that this term can indeed be controlled. We overcome the difficulty by using the special structure of $\Ag$ and prove the following estimate (see estimate 15 in \secref{sec:quantEa} for the proof)
\begin{align}\label{eq:DtAgintro}
\norm[\infty]{\frac{\Dt\Ag}{\Ag}} \lesssim \norm[\infty]{\frac{\Ztapbar}{\Zap}} + \norm[2]{\Ztapbar}\norm[2]{\pap\frac{1}{\Zap}}
\end{align}
Note that the constant appearing in the above estimate is universal and does not depend the interface, the velocity etc. In this paper, whenever we use the $\lesssim$ symbol, the constants in the inequalities are universal.

As $- \frac{\partial P}{\partial n} = \frac{\Ag}{\Zapabs}$, from this estimate we also immediately get the estimate
\begin{align*}
\norm[\infty]{\frac{\Dt\brac{-\frac{\partial P}{\partial n}}}{\brac{-\frac{\partial P}{\partial n}}}} \lesssim \norm[\infty]{\frac{\Ztapbar}{\Zap}} + \norm[2]{\Ztapbar}\norm[2]{\pap\frac{1}{\Zap}}
\end{align*}

To the best of our knowledge, both of these estimates are new even in the special case of the interface being flat i.e. $\Zap \equiv 1$ (in which case the second term on the right hand side in the above estimate is zero. Note that these estimates are nontrivial even in this special case). Both of these estimates are fully scaling invariant with respect to all the scaling transformations \eqref{eq:scaling}. Moreover the only term on the right hand side which controls the interface namely $\norm[2]{\pap\frac{1}{\Zap}}$, allows corners of angle less than $\pi/2$ and cusps (see the heuristics \eqref{eq:Zasymptotics}). In addition to the estimate \eqref{eq:DtAgintro}, we also need to prove other estimates with a similar difficulty (see the introduction of \secref{sec:quantEa} for more details). These estimates are the crux of the proof of the energy estimate in this paper. Previously several improved bounds for the material derivative of the Taylor sign condition term have been obtained, see for example \cite{HuIfTa16, AiIfTa24}. Unfortunately those estimates are not applicable to the problem we are studying in this paper, as those estimates do not provide a uniform bound for $\norm[\infty]{\frac{\Dt\brac{-\frac{\partial P}{\partial n}}}{\brac{-\frac{\partial P}{\partial n}}}}$ as $g \to 0$, which is fundamental to obtaining a uniform time of existence as $g \to 0$. 

To describe the main energy estimate in this paper, let us first define
\begin{align}\label{eq:Bt}
B(t) = \nobrac{\norm[\Linfty\cap\Hhalf]{\frac{\Ztapbar}{\Zap}} + \brac{\norm[2]{\Ztapbar} + \sqrt{g}}\norm[2]{\pap\frac{1}{\Zap}} }
\end{align}
This $B(t)$ is the same as the blow up criterion mentioned in \thmref{thm:intromain}. We prove the following estimate for the energy $E_l(t)$ in \eqref{eq:Elestimatemainproof} (recall that $E_l$ was defined above in \eqref{eq:El})
\begin{align}\label{eq:Elestimate}
\frac{\diff }{\diff t} E_l  & \lesssim B^2(t)\norm[2]{\pap\frac{1}{\Zap}} E_l^\half(t)  + B(t)E_l(t)
\end{align}
We then define our main energy $\Ea(t)$ (see \secref{sec:resultunbdd})
\begin{align}\label{eq:Eaintrodef}
\begin{aligned}
\Eone(t) & = \brac{\norm[2]{\Ztapbar}^2 + g}\norm[2]{\pap\frac{1}{\Zap}}^2\\
\Etwo(t) & =  \brac{\norm[2]{\Ztapbar}^2 + g}E_l(t) \\
\Ea(t) & = \brac{\Eone(t)^2 + \Etwo(t)}^\half
\end{aligned}
\end{align}
We also prove that $B(t) \lesssim \Ea(t)^{1/2}$ which is proved in \eqref{eq:controllBt}. Using the above estimate for $E_l(t)$ and some lower order estimates,  directly leads us to our main energy estimate
\begin{align}\label{eq:mainenergyestintro}
\frac{d\Ea(t)}{dt} \lesssim B(t) \Ea(t) \lesssim \Ea(t)^{3/2}
\end{align}
This energy estimate works uniformly for all values of gravity $g \geq 0$. We defined $\Eone(t), \Etwo(t)$ and $\Ea(t)$ in the above fashion as we not only want to have uniform in gravity estimates, but also a scaling invariant energy estimate. The energies $\Eone(t), \Etwo(t)$ and $\Ea(t)$ scale like $\norm[\infty]{\grad u}^2(t)$,  $\norm[\infty]{\grad u}^4(t)$ and $\norm[\infty]{\grad u}^2(t)$ respectively. The main energy estimate \eqref{eq:mainenergyestintro} is fully scaling invariant with respect all the scaling transformations \eqref{eq:scaling}. This energy also allows interfaces with corners of angles less than $\pi/2$ and cusps, for example $E_1(t)$ contains $\norm[2]{\pap\frac{1}{\Zap}}$ which allows corners from the heuristic \eqref{eq:Zasymptotics}. Therefore this energy estimate satisfies all the properties we were looking for to solve the reduced problem mentioned at the start of the introduction. 

In this paper we also prove an energy estimate for an energy called $E(t)$ which is a weighted half order higher in regularity as compared to $\Ea(t)$ (see  \secref{sec:resultunbdd}). This energy is used to prove the uniqueness of solutions by adapting the uniqueness proof in \cite{Wu19}, and also to prove the new blow up criterion. The energy estimate for $E(t)$ also satisfies all the nice features of the energy estimate for $\Ea(t)$.  

Although we do not prove even higher order energy estimates in this paper, one can indeed prove such estimates to arbitrarily higher levels of regularity. Note that as we are working with angled crested interfaces, one has to work with weighted spacial derivatives and not standard spacial derivatives to construct higher order energies. To do this, one can simply use the energy \eqref{eq:El} but with 
\begin{align*}
f = \brac{\frac{\Ag}{\Zapabs^2}\pap}^{\hspace{-1mm} n} \pap\frac{1}{\Zap}
\end{align*}
for any $n \geq 0$. When $n = 0$, then we get the energy \eqref{eq:El} used in this paper. Having constructed these energies, one can then prove estimates similar to \eqref{eq:Elestimate}. Then analogous to \eqref{eq:Eaintrodef} and \eqref{eq:mainenergyestintro} one can then construct and prove scaling invariant estimates, similar to the one for $\Ea(t)$. 

Note that this new energy does indeed allow interfaces with angled crests with angle less than $\pi/2$. To see this, first note that $\Ag > 0$ at each corner/cusp (as $\Ag > 0$ everywhere and only goes to zero at infinity when $g = 0$). Now if there is a corner of angle $\nu\pi$ at $\Z(0,t)$, then similar to \eqref{eq:Zasymptotics} we have near $\ap = 0$
\begin{align*}
\brac{\frac{1}{\Zap^2}\pap}^{\hspace{-1mm} n} \pap\frac{1}{\Zap}(\ap,t) & \sim (\ap)^{- \nu + n(1 -2\nu)}  \in \Ltwo_{loc}(\Rsp) \\
\frac{1}{\Zap}\brac{\frac{1}{\Zap^2}\pap}^{\hspace{-1mm} n} \pap\frac{1}{\Zap}(\ap,t) & \sim (\ap)^{(n+ 1)(1 -2\nu)}  \in \Linfty_{loc}(\Rsp)
\end{align*}
if $0\leq \nu < \half$. 

We remark that the weighted derivative $ \nobrac{\frac{\Ag}{\Zapabs^2}\pap}$ is the same as the weighted derivative $\brac{-\frac{\partial P}{\partial n}} \partial_s$ where $\partial_s$ is the arc length derivative on the interface. One could also instead work with the weighted derivative $ \nobrac{\frac{\Ag}{\Zapabs^2}\papabs}$ which is then the same as the weighted derivative $\brac{-\frac{\partial P}{\partial n}} \mathcal{G}$ (compare \eqref{eq:WWeqgeometric} and \eqref{eq:WWintro}). Note from \eqref{eq:WWeqgeometric} that $\Dt^2 \sim \brac{-\frac{\partial P}{\partial n}} \mathcal{G}$, so therefore one could work directly work with the material derivative $\Dt$ to construct higher order energies instead of working with these weighted spacial derivatives. This is exactly what is used in this paper to construct the energy $E_3(t)$ in comparison to $E_2(t)$ in \eqref{eq:EoneEtwoEthree}, where we note that essentially $\Dt^2 \pap\frac{1}{\Zap} \approx \Dt\Dap^2\Ztbar$ (modulo lower order corrections). Working with the material derivative $\Dt$ is also much better than working with any weighted spacial derivatives if there is non-zero surface tension, and this fact was crucially used in \cite{Ag21} to construct the energies (see section 3.2 in \cite{Ag21} for a discussion on this issue)

The paper is organized as follows: In  \secref{sec:notation} we introduce the notation and write down the system of equations in conformal coordinates. In \secref{sec:mainresults} we explain our main results in detail. In \secref{sec:unboundedsection} we prove the results for the unbounded domain case stated in \secref{sec:resultunbdd} and in \secref{sec:bounded} we prove the results for the bounded domain case stated in \secref{sec:resultbdd}. Finally in \secref{sec:appendix} we collect some of the identities and estimates used throughout the paper.

\section{Notation, preliminaries and equations of motion}\label{sec:notation}

In this paper we write $a \lesssim b$ if there exists a universal constant $C>0$ so that $a\leq Cb$. We write $a \lesssim_M b$ if there exits a constant $C(M)$ depending only on $M$ so that $a \leq C(M) b$. Similar definitions for $a \lesssim_{M_1, M_2} b$ etc. For singular integrals, all integrals will be understood in the principle value sense and we will suppress writing p.v. in front of the integrals. In this paper we will follow the notation used in \cite{Ag21} which is itself based on the notation used in \cite{Wu97, Wu09, KiWu18, Wu19}.-

\subsection{Notation for the unbounded domain case}\label{sec:Notationunbounded}

In this section we recall the notation used in \cite{Ag21}.  The Fourier transform is defined as
\[
\hat{f}(\xi) = \frac{1}{\sqrt{2\pi}}\int_\Rsp e^{-ix\xi}f(x) \diff x
\] 
We will denote by $\Scalsp(\Rsp)$ the Schwartz space of rapidly decreasing functions and $\Scalsp'(\Rsp)$ is the space of tempered distributions. A Fourier multiplier with symbol $a(\xi)$ is the operator $T_a$ defined formally by the relation $\dis \widehat{T_a{f}} = a(\xi)\hat{f}(\xi)$. For $s \geq 0$ the operators $\papabs^s $ and $\langle\pap\rangle^s$ are defined as the Fourier multipliers with symbols $\abs{\xi}^s$ and $(1 + \abs{\xi}^2)^{\frac{s}{2}}$ respectively. 
The Sobolev space $H^s(\Rsp)$ for $s\geq 0$  is the space of functions with  $\norm[H^s]{f} = \norm*[\Ltwo(\diff x)]{\langle\pap\rangle^s f} < \infty$. The homogenous Sobolev space $\Hhalf(\Rsp)$ is the space of functions modulo constants with  $\norm[\Hhalf]{f} = \norm*[\Ltwo(\diff x)]{\papabs^\half f} < \infty$.  

From now on compositions of functions will always be in the spatial variables. We write $f = f(\cdot,t), g = g(\cdot,t), f \compose g(\cdot,t) :=  f(g(\cdot,t),t)$. Define the operator $U_g$ as given by $U_g f = f\compose g$. Observe that $U_f U_g = U_{g\compose f}$.  Let $[A,B] := AB - BA$ be the commutator of the operators $A$ and $B$. If $A$ is an operator and $f$ is a function, then $(A + f)$ will represent the addition of the operators A and the multiplication operator $T_f$ where $T_f (g) = fg$.  We denote the convolution of $f$ and $g$ by $f \conv g$. We will denote the spacial coordinates in $\Omega(t) $ with $z = x+iy$, whereas $\zp = \xp + i\yp$ will denote the coordinates in the lower half plane $\Pminus = \cbrac{(x,y) \in \Rsp^2 \suchthat y<0}$. As we will frequently work with holomorphic functions, we will use the holomorphic derivatives $\pz = \half(\px-i\py)$ and $\pzp = \half(\pxp-i\pyp)$. In this paper all norms will be taken in the spacial coordinates unless otherwise specified. For example for a function $f:\Rsp\times [0,T] \to \Csp$ we write $\norm[2]{f} = \norm[2]{f(\cdot,t)} = \norm[\Ltwo(\Rsp,\diff \ap)]{f(\cdot,t)}$. Also for a function $f:\Pminus \to \Csp$ we write $\sup_{\yp<0}\norm[\Ltwo(\Rsp,\diff\xp)]{f} = \sup_{\yp<0}\norm[\Ltwo(\Rsp,\diff\xp)]{f(\cdot,\yp)}$. The Poisson kernel is given by
\begin{align}\label{eq:Poissonkernel}
K_\ep(x) = \frac{\ep}{\pi(\ep^2 + x^2)} \qquad \tx{ for } \ep>0
\end{align}

Let the interface be parametrized in Lagrangian coordinates by $\z(\cdot,t): \Rsp \to \Sigma(t)$ satisfying $\z_{\al}(\al,t) \neq 0 $ for all  $\al \in \Rsp$. Hence $\zt(\al,t) = u(\z(\al,t),t)$ is the velocity of the fluid on the interface and $\ztt(\al,t) = (u_t + (u\cdot\grad)u)(\z(\al,t),t)$ is the acceleration. 

Let $\Psi(\cdot,t): \Pminus \to  \Omega(t)$ be conformal maps satisfying $\lim_{\z\to \infty} \Psi_\z(\z,t) =1$  and that $\lim_{\z\to \infty} \Psi_t(\z,t) =0$.  With this, the only ambiguity left in the definition of $\Psi$ is that of the choice of translation of the conformal map at $t=0$, which does not play any role in the analysis. Let $\Phi(\cdot,t):\Omega(t) \to \Pminus $ be the inverse of the map $\Psi(\cdot,t)$ and define $\h(\cdot,t):\Rsp \to \Rsp$ as
\begin{align}\label{def:hunbddorig}
\h(\al,t) = \Phi(\z(\al,t),t)
\end{align}
hence $\h(\cdot,t)$ is a homeomorphism. As we use both Lagrangian and conformal parameterizations, we will denote the Lagrangian parameter by $\al$ and the conformal parameter by $\ap$. Let $\hinv(\cdot,t)$ be its spacial inverse i.e. $\h(\hinv(\ap,t),t) = \ap$. From now on, we will fix our Lagrangian parametrization at $t=0$ by imposing $h(\al,0)= \al \quad  \tx { for all } \al \in \Rsp$. Hence the Lagrangian parametrization is the same as conformal parametrization at $t=0$. Define the variables
\[
\begin{array}{l l l}
 \Z(\ap,t) = \z\compose \hinv (\ap,t)  & \Zap(\ap,t) = \pap \Z(\ap,t) &  \quad  \tx{ Hence } \quad  \brac*[]{\dfrac{\zal}{\hal}} \compose \hinv = \Zap \\
  \Zt(\ap,t) = \zt\compose \hinv (\ap,t)  & \Ztap(\ap,t) = \pap \Zt(\ap,t) &  \quad   \tx{ Hence } \quad  \brac*[]{\dfrac{\ztal}{\hal}} \compose \hinv = \Ztap \\
 \Ztt(\ap,t) = \ztt\compose \hinv (\ap,t)  & \Zttap(\ap,t) = \pap \Ztt(\ap,t) &  \quad   \tx{ Hence } \quad  \brac*[]{\dfrac{\zttal}{\hal}} \compose \hinv = \Zttap \\
\end{array}
\]
Hence $\Z(\ap,t), \Zt(\ap,t)$ and $\Ztt(\ap,t)$ are the parameterizations of the boundary, the velocity and the acceleration in conformal coordinates and in particular $\Z(\cdot,t)$ is the boundary value of the conformal map $\Psi(\cdot,t)$. Note that as $\Z(\ap,t) = \z(\hinv(\ap,t),t)$ we see that  $\pt \Z \neq \Zt$. Similarly $\pt \Zt \neq \Ztt$.  The substitute for the time derivative is the material derivative. Define:
\begingroup
\allowdisplaybreaks
\begin{fleqn}
\begin{align}\label{eq:mainoperators}
\begin{split}
& \Dt =  \tx{material derivative} = \pt + \bvar\pap  \qquad \tx{ where } \bvar = \hvart \compose \hinv \\
& \Dap = \Dapfrac \qquad \Dapbar = \Dapbarfrac \qquad \Dapabs = \Dapabsfrac \\
& \Hil  =   \text{Hilbert transform  = Fourier multiplier with symbol } -sgn(\xi) \\
& \qquad  \Hil f(\alpha ' ) =  \frac{1}{i\pi} p.v. \int \frac{1}{\alpha ' - \beta'}f(\beta') \diff\beta' \\
& \Ph = \text{Holomorphic projection} = \frac{\Id + \Hil}{2} \\
& \Pa = \text{Antiholomorphic projection} = \frac{\Id - \Hil}{2} \\ 
& \papabs   = \ i\Hil \partial_{\alpha'} = \sqrt{-\Delta} = \text{ Fourier multiplier with symbol } |\xi| \ \\
& \papabs^{1/2}  = \text{ Fourier multiplier with symbol } \abs{\xi}^{1/2} \\
& \w =  \frac{\Zap}{\Zapabs} \\
\end{split}
\end{align}
\end{fleqn}
\endgroup
Now we have $\Dt \Z = \Zt$ and $\Dt \Zt = \Ztt$ and more generally $\Dt (f(\cdot,t)\compose \hinv) = (\pt f(\cdot,t)) \compose \hinv$ or equivalently $\pt (F(\cdot,t)\compose \h) = (\Dt F(\cdot,t)) \compose \h$. This means that $\Dt = U_h^{-1}\pt U_h$ i.e. $\Dt$ is the material derivative in conformal coordinates. 

Define $\U: \Pminusbar\times [0,T) \to \Csp $ and $\Pfrak: \Pminusbar \times [0,T) \to \Rsp$ as
\begin{align}\label{eq:UPfrak}
\U = \vboldbar \compose \Psi \qquad \Pfrak = P \compose \Psi
\end{align}
and observe that $\U(\cdot,t)$ is a holomorphic function on $\Pminus$. Also note that its boundary value is given by $ \Ztbar(\ap,t) = \U(\ap,t)$ for all $\ap\in \Rsp$. Hence we can write the Euler equations \eqref{eq:Euler} as equations on $\Pminus$ 
\begin{align}\label{eq:EulerRiem}
\begin{aligned}
& \Ut - \Psit \frac{\Uzp}{\Psizp} + \Ubar \frac{\Uzp}{\Psizp} = -\frac{1}{\Psizp}(\pxp - i\pyp)\Pfrak + ig & \tx{ on } \Pminus \\
& \U(\cdot,t) \tx{ is holomorphic}  & \tx{ on } \Pminus \\
& \Pfrak = 0  &  \tx{ on } \partial \Pminus \\
& \tx{Trace of } \onePsizp(\Ubar - \Psit) \tx{ is real valued}  & \tx{ on } \partial \Pminus
\end{aligned}
\end{align}
along with the condition that $\Psi(\cdot,t)$ is conformal and the decay conditions $\U \to 0$, $(\pxp - i\pyp) \Pfrak \to ig$, $\Psizp \to 1$ and $ \Psi_t \to 0$ as $\zp \to \infty$. The condition that the particles on the boundary stay on the boundary is equivalent to saying that the trace of $\onePsizp(\Ubar - \Psit)$ is real valued and in particular $\brac{\onePsizp(\Ubar - \Psit)}\Big\vert_{\partial \Pminus} = b$ where $\Dt = \pt + \bvar\pap$ is the material derivative on the boundary $\partial \Pminus$. Also it should be noted that the process of obtaining \eqref{eq:EulerRiem} from \eqref{eq:Euler} is reversible so long as the interface $\partial\Omega(t) = \cbrac{\Z(\ap,t) \suchthat \ap\in \Rsp}$ is non-self intersecting. See \cite{Ag20} for more details.

The Hilbert transform defined above in \eqref{eq:mainoperators} satisfies the following property.
\begin{lem}[\cite{Ti86}]\label{lem:Hilunbounded}
Let $1<p<\infty$ and let $F:\Pminus \to \Csp$ be a holomorphic function in the lower half plane with $F(z) \to 0$ as $z\to \infty$. Then the following are equivalent
\begin{enumerate}
\item $\dis \sup_{y<0} \norm[p]{F(\cdot + iy)} < \infty$
\item $F(z)$ has a boundary value $f$, non-tangentially almost everywhere with $f \in L^p$ and $\Hil(f) = f$. 
\end{enumerate}
\end{lem}
In particular this says if $\U$ decays appropriately at infinity, then the boundary value of $\U$ namely $\Ztbar$ will satisfy $\Hil \Ztbar = \Ztbar$. Similarly as $\Psi_z \to 1$ as $\z \to \infty$, we have that $\Hil\brac{\frac{1}{\Zap} - 1} = \frac{1}{\Zap} - 1$. Also note that as product of holomorphic functions is holomorphic, we can use the above lemma to conclude that several other functions on the boundary also satisfy similar identities e.g. $\Hil(\frac{\Ztbar}{\Zap}) = \frac{\Ztbar}{\Zap}$, $\Hil(\Dap \Ztbar) = \Dap\Ztbar $ etc.

For $n \geq 1$ and for functions $f_1,\cdots,f_n, g:\Rsp \to \Csp$ we define the function $\sqbrac{f_1,\cdots,f_n ; g}:\Rsp \to \Csp$ as 
\begin{align}\label{eq:fonefn}
\sqbrac{f_1, \cdots, f_n;  g}(\ap) = \frac{1}{i\pi} \int \brac{\frac{f_1(\ap) - f_1(\bp)}{\ap - \bp}}\cdots \brac{\frac{f_n(\ap) - f_n(\bp)}{\ap-\bp}} g(\bp) \diff \bp
\end{align}
Hence with this notation we see that
\begin{align*}
\sqbrac{f,\Hil}g = \frac{1}{i\pi}\int \frac{f(\ap) - f(\bp)}{\ap - \bp} g(\bp) \diff \bp = \sqbrac{f;g}
\end{align*}

\subsection{The system of equations on the boundary for the unbounded domain case}\label{sec:systembdryunbdd}

To solve the system \eqref{eq:Euler}, in \cite{Ag21} we obtained a  system for the variables $(\Zap,\Zt)$ which we then solve. The system is as follows:
\begin{align}\label{eq:systemone}
\begin{split}
\bvar & = \Real(\Id - \Hil)\brac{\frac{\Zt}{\Zap}} \\
\Ag & = \g - \Imag \sqbrac{\Zt,\Hil}\Ztapbar \\
(\pt + \bvar\pap)\Zap &= \Ztap - \bap\Zap  \\
(\pt + \bvar\pap)\Ztbar & = i\g -i\frac{\Ag}{\Zap} 
\end{split}
\end{align}
along with the condition that their harmonic extensions, namely $\Psizp(\cdot + iy) = K_{-y}\conv \Zap$ and $\U(\cdot + iy) = K_{-y}\conv \Ztbar$ for all $y<0$, \footnote{here $K_{-y}$ is the Poisson kernel \eqref{eq:Poissonkernel}} are holomorphic functions on $\Pminus$ and satisfy \footnote{We observe that for such a $\Psizp$  we can uniquely define $\log(\Psizp) : \Pminus \to \Csp$ such that $\log(\Psizp)$ is a continuous function with $\Psizp = \exp\cbrac{\log(\Psizp)}$ and $(\log(\Psizp))(\zp) \to 0$ as $\zp \to \infty$. }
\begin{align*}
\lim_{c \to \infty} \sup_{\abs{\zp}\geq c}\cbrac{\abs{\Psizp(\zp) - 1} + \abs{\U(\zp)}}  = 0 
\qquad \tx{ and } \quad \Psizp(\zp) \neq 0 \quad \tx{ for all } \zp \in \Pminus 
\end{align*}
After solving the above system one can obtain $\Z(\cdot,t)$ by the formula 
\begin{align*}
\Z(\ap,t) = \Z(\ap,0) +  \int_0^t \cbrac{\Zt(\ap,s) - \bvar(\ap,s)\Zap(\ap,s)} \diff s
\end{align*}
and hence $(\pt + \bvar\pap)\Z = \Dt\Z = \Zt$. Hence one can view the system being in variables $(\Z,\Zt)$ instead of the variables $(\Zap,\Zt)$. Note that once one has a solution to the above system, we can recover a solution to the system  \eqref{eq:EulerRiem} by letting $\Psi$ and $\U$ be defined as above and defining $\Pfrak$ as the unique solution to the equation
\begin{align}\label{eq:DeltaPfrak}
\Delta \Pfrak = -2\abs{\U_\zp}^2 \quad \tx{ on } \Pminus, \qq \Pfrak = 0 \quad \tx{ on } \partial \Pminus
\end{align}
along with the condition $(\pxp - i\pyp) \Pfrak \to ig$ as $\zp \to \infty$. See \cite{Wu19} for the details. 

We observe that the above system allows self intersecting interfaces. However if the interface is self-intersecting then it becomes nonphysical and so its relation to the Euler equation  \eqref{eq:Euler} is lost. See \cite{Ag21} for more details.

Now to get the function $\h(\al,t)$ (recall the definition \eqref{def:hunbddorig}), we solve the ODE
\begin{align}\label{eq:h}
\begin{split}
\frac{\diff h}{\diff t} &= \bvar(\h, t) \\
h(\al,0) &= \al
\end{split}
\end{align}
Observe that as long as $\sup_{[0,T]} \norm[\infty]{\bvarap}(t) <\infty$ we can solve this ODE uniquely and for any $t\in [0,T]$ we have that $h(\cdot,t)$ is a homeomorphism. Hence it makes sense to talk about the functions $z = \Z\compose \h, \zt = \Zt \compose \h$ which are Lagrangian parameterizations of the interface and the velocity on the boundary. We also note that the last  equation in \eqref{eq:systemone} can be written as 
\begin{align}\label{form:Zttbar}
\Zttbar -i\g = -i \frac{\Ag}{\Zap} 
\end{align}
We note here that from the calculations in \cite{Wu97}, we see that the gradient of the pressure on the boundary in conformal coordinates is (here $\hat{n}$ is the outward unit normal)
\begin{align}\label{eq:gradientpressure}
-\frac{\partial P}{\partial \hat{n}}\compose \hinv = \frac{\Ag}{\Zapabs}
\end{align}

\subsection{Notation for the bounded domain case}

Let the unit disc be $\Dsp = \cbrac{ (x,y) \in \Rsp^2 \suchthat x^2 + y^2 < 1}$ and let $\Sone = \partial \Dsp$. In the following we will identify functions $f: \Sone \to \Csp$ with their pullbacks $\tilde{f}: \Rsp \to \Csp$, where $\tilde{f}(\ap) = f(e^{i\ap})$. We will frequently abuse notation and for a function $f:\Sone \to \Csp$ we will usually write $f(\ap)$ instead of $f(e^{i\ap})$. In particular if there exists a function $F: \Dsp \to \Csp$ whose boundary value is $f: \Sone \to \Csp$, then by abuse of notation we will also say that the boundary value of $F$ is $\tilde{f}$. We will denote the boundary value of $F$ by $\Tr(F)$. If $g:\Rsp \to \Csp$ is a $2\pi$ periodic function, then in this paper whenever we use the $L^p$ or Sobolev norms of $g$, what we mean is that we are computing the norms by looking at $g$ as a function on $\Sone$ and not as a function on $\Rsp$. 

We define the Fourier transform for a function $f: \Sone \to \Csp$ as
\begin{align*}
\hat{f}(n) & = \frac{1}{2\pi}\int_0^{2\pi} f(\ap) e^{-in\ap} \diff \ap 
\end{align*}
and the inverse Fourier transform as
\begin{align*}
 f(\ap) & = \sum_{n = -\infty}^{\infty} \hat{f}(n) e^{in\ap}
\end{align*}
The $L^p$ norm of $f$ is defined as 
\begin{align*}
\norm[p]{f} = \brac{\int_0^{2\pi} \abs*{f(\ap)}^p  \diff \ap}^{\frac{1}{p}}
\end{align*}
and the Sobolev norms of $f$ are defined in the same way as in \secref{sec:Notationunbounded}. Hence in particular we observe that
\begin{align*}
\norm[\Hhalf]{f}^2 = \norm*[\big][2]{\papabs^\half f}^2 = \int_0^{2\pi} \fbar \papabs f \diff \ap
\end{align*}

Similar to \secref{sec:Notationunbounded}, compositions of functions are again only taken in spacial coordinates and we maintain the notation for the operator $U_g$ and convolution of functions.  We keep the notation $z = x+ iy$ for $z \in \Omega(t)$ and $\zp = \xp + i\yp$ for $\zp \in \Dsp$. We have the same definitions for $\pz, \pzp$ and we also suppress the time variables when writing norms. The Poisson kernel for the disc is given by
\begin{align*}
P_{r}(\theta) = \frac{1 - r^2}{1 -2r\cos(\theta) + r^2} \qq \tx{ for } 0\leq r < 1
\end{align*}

Let the interface be parametrized in Lagrangian coordinates by a $2\pi$ periodic counter-clockwise  parametrization $\z(\cdot,t): \Rsp \to \partial\Omega(t)$ satisfying $\z_{\al}(\al,t) \neq 0 $ for all  $\al \in \Rsp$. Hence we see that $\zt(\al,t) = u(\z(\al,t),t)$ is the velocity of the fluid on the interface and $\ztt(\al,t) = (u_t + (u\cdot\grad)u)(\z(\al,t),t)$ is the acceleration. 

We fix a point $x_0 \in \Omega(0)$ and let $X(x_0,t)$ be the Lagrangian trajectory of this particle. Hence $\pt X(x_0,t) = u(X(x_0,t),t)$. Let $\Psi(\cdot,t): \Dsp \to  \Omega(t)$ be conformal maps satisfying $ \Psi(0,t) = X(x_0,t)$  and $\Psi_{\zp} (0,t) > 0$.  Let $\Phi(\cdot,t):\Omega(t) \to \Dsp $ be the inverse of the map $\Psi(\cdot,t)$ and observe that $\Phi(\z(\al,t),t) \in \partial \Dsp$. Define $\h(\cdot,t):\Rsp \to \Rsp$ as
\begin{align}\label{def:hbddorig}
\h(\al,t) = -i\log\brac{\Phi(\z(\al,t),t)}.
\end{align}
where the ambiguity in the definition of $\log(z)$ is removed by ensuring that $\h$ is a continuous function on $\Rsp\times[0,T)$ and that $h(0,0) \in [0,2\pi)$. From this it is easy to see that $\h(\cdot,t)$ is an increasing function and is a homeomorphism on $\Rsp$ satisfying $h(\al + 2\pi,t) = h(\al,t) + 2\pi$. As we use both Lagrangian and conformal parameterizations, we will denote the Lagrangian parameter by $\al$ and the conformal parameter by $\ap$. Let $\hinv(\cdot,t)$ be its spacial inverse i.e.
\[
\h(\hinv(\ap,t),t) = \ap
\] 
From now on, we will fix our Lagrangian parametrization at $t=0$ by imposing
\begin{align*}
h(\al,0)= \al \quad  \tx { for all } \al \in \Rsp
\end{align*}
Hence the Lagrangian parametrization is the same as conformal parametrization at $t=0$. The functions $\Z, \Zt, \Ztt$ are defined as in \secref{sec:Notationunbounded}. Hence the functions $\Z(\cdot,t) : \Rsp \to \partial \Omega(t)$ and $\Zt(\cdot,t), \Ztt(\cdot,t) : \Rsp \to \Csp$ are $2\pi$ periodic functions and are the parameterizations of the boundary, the velocity and the acceleration in conformal coordinates. In particular we have $\Z(\ap,t) = \Psi(e^{i\ap},t)$. We define the material derivative $\Dt$ and the weighted derivatives $\Dap, \Dapbar, \Dapabs$ in the same way as in \secref{sec:Notationunbounded}. The variable $\w$ is also defined in the same way.

Define 
\begin{align*}
\Av(f) & = \frac{1}{2\pi}\int_0^{2\pi} f(\bp) \diff \bp \\
(\Hiltil f)(\ap) & = \frac{1}{2\pi i} \int_0^{2\pi} f(\bp) \cot\brac{\frac{\bp - \ap}{2}} \diff \bp 
\end{align*}
Observe that we have 
\begin{align*}
\Av(f) & = \hat{f}(0) \\
\widehat{(\Hiltil f)}(n)  & = \sgn(n) \hat{f}(n) \\
\papabs  & = -i\Hiltil \partial_{\alpha'} 
\end{align*}
where
\begin{align*}
\sgn(n) = 
\begin{cases}
1 \quad \tx{ if } n \geq 1 \\
0 \quad \tx{ if } n = 0 \\
-1 \quad \tx{ if } n \leq -1
\end{cases}
\end{align*}
We define the Hilbert transform as
\begin{align*}
\Hil = \Hiltil + \Av
\end{align*}
The projection operators $\Ph$ and $\Pa$ are defined in the same way as in \secref{sec:Notationunbounded}. Using the identity
\begin{align}\label{eq:trigidentity}
e^{i\ap} - e^{i\bp} = 2ie^{i\brac{\frac{\ap + \bp}{2}}}\sin\brac{\frac{\ap - \bp}{2}}
\end{align}
it is easily seen that
\begin{align*}
(\Hil f)(\ap) & =  \frac{1}{2\pi i} \int_0^{2\pi} f(\bp) \cot\brac{\frac{\bp - \ap}{2}} \diff \bp + \frac{1}{2\pi}\int_0^{2\pi} f(\bp) \diff \bp \\
& = \frac{e^{-i\frac{\ap}{2}}}{2\pi i} \int_{0}^{2\pi} \frac{f(\bp)}{\sin\brac{\frac{\bp - \ap}{2}}} e^{i\frac{\bp}{2}} \diff \bp \\
& = \frac{1}{i\pi} \int_{0}^{2\pi} \frac{f(\bp)}{e^{i\bp} - e^{i\ap}} ie^{i\bp} \diff \bp \\
\end{align*}
The Hilbert transform $\Hil$ defined above has the following property similar to \lemref{lem:Hilunbounded} and is proved in a similar manner. 
\begin{lemma}\label{lem:Hilbounded}
Let $1<p<\infty$ and let $F:\Dsp \to \Csp$ be a holomorphic function. Then the following are equivalent
\begin{enumerate}
\item $\sup_{0 < r < 1}\norm[L^p([0,2\pi], \diff \theta)]{F(r e^{i\theta})} < \infty$
\item $F(z)$ has a boundary value $f$, non-tangentially almost everywhere with $f \in L^p$ and $\Hil(f) = f$. 
\end{enumerate}
\end{lemma}
The functions $\U : \Dspbar\times[0,T) \to \Csp$ and $\Pfrak : \Dspbar\times[0,T) \to \Rsp $ are again defined as in \eqref{eq:UPfrak} and we see that $U(\cdot,t)$ is a holomorphic function on $\Dsp$.  Also note that its boundary value is given by $ \Ztbar(\ap,t) = \U(e^{i\ap},t)$ for all $\ap\in \Rsp$. From \lemref{lem:Hilbounded} we therefore see that $\Hil \Ztbar = \Ztbar$. From \secref{sec:derivofeqnbdd} we also have that $\Hil\brac{\frac{e^{i\ap}}{\Zap}} = \frac{e^{i\ap}}{\Zap}$ however $\Hil\brac{\frac{1}{\Zap}} \neq \frac{1}{\Zap}$. Similarly we have $\Hil\brac{\frac{\Ztbar - \Avg(\Ztbar)}{\Zap}} = \frac{\Ztbar - \Avg(\Ztbar)}{\Zap}$ however $\Hil\brac{\frac{\Ztbar}{\Zap}} \neq \frac{\Ztbar}{\Zap}$. Though we still have $\Hil(\Dap\Ztbar) = \Dap\Ztbar$. See \secref{sec:derivofeqnbdd} for more details. 

In a similar way to \secref{sec:Notationunbounded}, we can now write the Euler equations \eqref{eq:Euler2} as equations in $\Dsp$ as follows
\begin{align}\label{eq:EulerRiem2}
\begin{aligned}
& \Ut - \Psit \frac{\Uzp}{\Psizp} + \Ubar \frac{\Uzp}{\Psizp} = -\frac{1}{\Psizp}(\pxp - i\pyp)\Pfrak & \tx{ on } \Dsp \\
& \U(\cdot,t) \tx{ is holomorphic}  & \tx{ on } \Dsp \\
& \Pfrak = 0  &  \tx{ on } \partial \Dsp \\
& \tx{Trace of } \frac{-i}{\zp\Psizp}(\Ubar - \Psit) \tx{ is real valued}  & \tx{ on } \partial \Dsp
\end{aligned}
\end{align}
along with the normalizing conditions $\partial_t \Psi(0,t) = \Ubar(0,t)$ and $\Psizp(0,t) > 0$. The condition that the particles on the boundary stay on the boundary is equivalent to saying that the trace of $ \frac{-i}{\zp\Psizp}(\Ubar - \Psit)$ is real valued and in particular from \secref{sec:derivofeqnbdd} we see that $\brac{ \frac{-i}{\zp\Psizp}(\Ubar - \Psit)}\Big\vert_{\partial \Dsp} = b$ where $\Dt = \pt + \bvar\pap$ is the material derivative on the boundary $\partial \Dsp$. Again note that the process of obtaining \eqref{eq:EulerRiem2} from \eqref{eq:Euler2} is reversible so long as the interface $\partial\Omega(t)$ is non-self intersecting.

For $f_1, f_2, f_3, g : \Sone \to \Csp$, we define the following functions
\begin{align}
[f_1;g](\ap) & = ([f_1,\Hiltil]g)(\ap) = \frac{1}{2\pi i} \int_0^{2\pi} (f_1(\ap) - f_1(\bp)) \cot\brac{\frac{\bp - \ap}{2}} g(\bp) \diff \bp \label{eq:sqfgbounded}\\
\sqbrac{f_1,f_2; g}(\ap) & = -\frac{1}{4\pi i} \int_0^{2\pi} \brac{\frac{f_1(\ap) - \f_1(\bp)}{\sin\brac*[\big]{\frac{\bp - \ap}{2}}}} \brac{\frac{f_2(\ap) - \f_2(\bp)}{\sin\brac*[\big]{\frac{\bp - \ap}{2}}}} g(\bp) \diff \bp \label{eq:sqfoneftwogbounded}
\end{align}
and
\begin{align}\label{eq:foneftwofthreegbounded}
\begin{split}
& [f_1, f_2, f_3; g](\ap) \\
& = \frac{1}{8\pi i} \int_0^{2\pi} \brac{\frac{f_1(\ap) - \f_1(\bp)}{\sin\brac*[\big]{\frac{\bp - \ap}{2}}}} \brac{\frac{f_2(\ap) - \f_2(\bp)}{\sin\brac*[\big]{\frac{\bp - \ap}{2}}}} \brac{\frac{f_3(\ap) - \f_3(\bp)}{\sin\brac*[\big]{\frac{\bp - \ap}{2}}}} \cos\brac*[\Big]{\frac{\bp - \ap}{2}}g(\bp) \diff \bp
\end{split}
\end{align}

\subsection{The system of equations on the boundary for the bounded domain case}\label{sec:systembdrybdd}

Similar to \secref{sec:systembdryunbdd}, to solve \eqref{eq:EulerRiem2} we obtain a system of equations on the boundary of the domain in the variables $(\Zap,\Zt)$. We derive these equations in \secref{sec:derivofeqnbdd}. The equations are:
\begin{align}\label{eq:systemonebdd}
\begin{split}
\bvar & = \Real(\Id - \Hiltil)\brac{\frac{\Zt - \Av(\Zt)}{\Zap} } \\
\Azero & = \Imag\sqbrac*[\big]{\Zt,\Hiltil}\Ztapbar \\
(\pt + \bvar\pap)\Zap &= \Ztap - \bap\Zap \\
(\pt + \bvar\pap)\Ztbar & = i\frac{\Azero}{\Zap} 
\end{split}
\end{align}
along with the condition that the functions
\begin{align*}
\Psizp(re^{i\theta},t) & = \frac{1}{2\pi}\int_0^{2\pi} P_r(\theta - \bp)\cbrac{-ie^{-i\bp}\Zap(\bp,t)} \diff \bp \\
\U(re^{i\theta},t) & = \frac{1}{2\pi}\int_0^{2\pi} P_r(\theta - \bp)\Ztbar(\bp,t) \diff \bp
\end{align*}
are holomorphic functions on $\Dsp$ and satisfy for $t\geq 0$
\begin{align*}
\Psizp(0,t) > 0 \qquad \tx{ and } \quad \Psizp(\zp,t) \neq 0 \quad \tx{ for all } \zp \in \Dsp
\end{align*}

We note that one can obtain $\Z(\cdot,t)$ by the formula 
\begin{align*}
\Z(\ap,t) = \Z(\ap,0) +  \int_0^t \cbrac{\Zt(\ap,s) - \bvar(\ap,s)\Zap(\ap,s)} \diff s
\end{align*}
In particular instead of the variables $(\Zap,\Zt)$, one can view the system being in the variables $(\Z,\Zt)$. Similar to the unbounded case we note that once one has a solution to the above system, we can recover a solution to the system  \eqref{eq:EulerRiem2} by letting $\Psi$ and $\U$ be defined as above and defining $\Pfrak$ as the unique solution to the equation
\begin{align}\label{eq:DeltaPfrakbdd}
\Delta \Pfrak = -2\abs{\U_\zp}^2 \quad \tx{ on } \Dsp, \qq \Pfrak = 0 \quad \tx{ on } \partial \Dsp
\end{align}

Similar to the unbounded case, we again observe that the above system allows for the interface to self intersect. Also similarly, we can recover the function $\h(\al,t)$ as the solution to the ODE
\begin{align}\label{eq:hbdd}
\begin{aligned}
\frac{\diff h}{\diff t} &= \bvar(\h, t) \\
h(\al,0) &= \al
\end{aligned}
\end{align}
where $\bvar$ is given by \eqref{eq:systemonebdd}. From this we easily see that as long as $\sup_{[0,T]} \norm[\infty]{\bvarap}(t) <\infty$ we can solve this ODE and for any $t\in [0,T]$ we have that $h(\cdot,t)$ is a homeomorphism. Also using the fact that $b(\cdot,t)$ is a $2\pi$ periodic function, it is also easy to see that $h$ satisfies $h(\al + 2\pi,t) = h(\al,t) + 2\pi$. Hence the functions $z = \Z\compose \h, \zt = \Zt \compose \h$ are $2\pi$ periodic and are the Lagrangian parameterizations of the interface and the velocity on the boundary. We also note that the last  equation in \eqref{eq:systemonebdd} can be written as 
\begin{align}\label{form:Zttbarbdd}
\Zttbar  = i \frac{\Azero}{\Zap} 
\end{align}
From the calculations from \secref{sec:derivofeqnbdd}, we see that the gradient of the pressure on the boundary in conformal coordinates is (here $\hat{n}$ is the outward unit normal)
\begin{align}\label{eq:gradientpressurebdd}
-\frac{\partial P}{\partial \hat{n}}\compose \hinv = \frac{\Azero}{\Zapabs}
\end{align}

\section{Main results}\label{sec:mainresults}

\subsection{The unbounded domain case}\label{sec:resultunbdd}

In this subsection we collect our main results for the model \eqref{eq:EulerRiem} which is the same as \eqref{eq:Euler} but written in conformal coordinates. We define the energies
\begin{align}\label{eq:EoneEtwoEthree}
\begin{aligned}
\Eone(t) & = \brac{\norm[2]{\Ztapbar}^2 + g}\norm[2]{\pap\frac{1}{\Zap}}^2\\
\Etwo(t) & =  \brac{\norm[2]{\Ztapbar}^2 + g}\cbrac{\norm[2]{\Dt\nobrac{\pap\frac{1}{\Zap}}}^2 + \norm[\Hhalf]{\frac{\sqrt{\Ag}}{\Zap}\pap\frac{1}{\Zap}}^2} \\
\Ethree(t) & = \brac{\norm[2]{\Ztapbar}^2 + g}\cbrac{\norm[2]{\Dt\Dap^2\Ztbar }^2 + \norm[\Hhalf]{\frac{\sqrt{\Ag}}{\Zap}\Dap^2\Ztbar}^2} 
\end{aligned}
\end{align}
and also define
\begin{align}
\Ea(t) & = \brac{\Eone(t)^2 + \Etwo(t)}^\half \label{def:Ea}\\
\E(t) & = \brac{\Eone(t)^3 + \Etwo(t)^{\frac{3}{2}} + \Ethree(t)}^{\frac{1}{3}} \label{def:E}
\end{align}


Note that $\Ea(t) \lesssim \E(t)$. If $(\Ea)_\lamb(t)$ and $\E_\lamb(t)$ are the energies defined by \eqref{def:Ea} and \eqref{def:E} respectively for the solution $((\Z_\lamb)(\ap,t), (\Z_{\lamb})_t(\ap,t))$ with gravity $g_\lamb$, then it is easy to see that
\begin{align}\label{eq:scalingenergy}
(\Ea)_\lamb(t) = \lamb^{2s} \Ea(\lamb^s t) \qq \tx{and } \quad \E_\lamb (t) = \lamb^{2s} \E(\lamb^s t)
\end{align}
Hence both $\Ea(t)$ and $\E(t)$ scale like $\norm[\infty]{\grad u}^2(t)$ under all the scaling transformations \eqref{eq:scaling}. Similarly the energies $\Eone(t), \Etwo(t)$ and $\Ethree(t)$ scale like $\norm[\infty]{\grad u}^2(t), \norm[\infty]{\grad u}^4(t)$ and $\norm[\infty]{\grad u}^6(t)$ respectively under these scalings.
 
We now state our first main result which is an a priori estimate for the energies $\Ea(t)$ and $\E(t)$. 

\begin{theorem}\label{thm:aprioriE}
Let $T>0$ and let $(\Z, \Zt )(t)$ be a solution to the gravity water wave equation \eqref{eq:systemone} with gravity parameter $g \geq 0$ in the time interval $[0,T]$ with $(\Zap - 1, \frac{1}{\Zap} - 1, \Zt) \in \Linfty([0,T], H^s(\Rsp)\times H^s(\Rsp)\times H^{s + \half}(\Rsp))$ for some $s \geq 4$. Then  $\E(t) < \infty$ for all $t\in [0,T]$ and there exists a universal constant $c >0$ so that for all $t \in [0,T)$ we have
\begin{align}\label{eq:mainEatime}
\frac{d\Ea(t)}{dt} \leq c \brac{\norm[\Linfty\cap\Hhalf]{\Dap\Ztbar}  +  \brac{\norm[2]{\Ztapbar} + \sqrt{g}}\norm[2]{\pap\frac{1}{\Zap}}} \Ea(t) \leq c^2 \Ea(t)^{3/2}
\end{align}
and also 
\begin{align}\label{eq:mainEtime}
\frac{d\E(t)}{dt} \leq c\Ea(t)^\half\E(t) \leq c^2 \E(t)^{3/2}
\end{align}
\end{theorem}

We prove this theorem in \secref{sec:unboundedsection}, from \secref{sec:identities} to \secref{sec:closing}. Let us now highlight some salient features about the energies $\Ea(t)$ and $\E(t)$ and the above result:
\begin{enumerate}
\item The energies $\Ea(t)$ and $\E(t)$ are uniform in gravity. This is because the only place where the gravity parameter $g$ shows up in these energies is in the term $(\norm[2]{\Ztapbar}^2 + g)$ and in the variable $\Ag$. The variable $\Ag =  g + \Azero$ where $\Azero \geq 0$ and is independent of $g$ (see \eqref{eq:Ag} and \eqref{eq:Azero}). The above a priori estimate works for all $g \geq 0$, in particular even for $g = 0$, and the estimates are uniform in gravity. See \corref{cor:uniftimeSobolev} below. 

\item The energies $\Ea(t)$ and $\E(t)$ allow both smooth solutions and solutions with angled crests/cusps similar to \cite{Wu19}. More precisely they allow interfaces with corners of angle $\nu \pi$ for $0<\nu < \half$ and for cusps (which corresponds to $\nu = 0$). To see this, it is easier to look at the energy $\Ecal(t)$ defined below and observe from \lemref{lem:equivSobolevunbdd} that the energy $\Ecal(t)$ dominates the energy $\E(t)$. Then from \cite{Ag20} we easily see that for interfaces with angled crests/cusps, we have that the energy $\Ecal(t)$ is finite, which implies that the energy $\E(t)$ is also finite. See also \eqref{eq:Zasymptotics}.

\item From \eqref{eq:gradientpressure} we see that the gradient of the pressure on the boundary is given by 
\begin{align}\label{eq:Taylorsignterm}
-\frac{\partial P}{\partial \hat{n}}\compose \hinv = \frac{\Ag}{\Zapabs}
\end{align}
Now from \lemref{lem:decayatinfinity} and the formula \eqref{eq:Azero}, we see that if $\Ztap(\cdot, t) \in \Ltwo$ then $\Azero(\ap, t) \to 0$ as $\abs{\ap} \to \infty$. Also as $\Zap(\cdot, t) \to 1$ at infinity, we see that  $-\frac{\partial P}{\partial \hat{n}}\compose \hinv(\ap, t) \to g$ as $\abs{\ap} \to \infty$. If $g = 0$, the velocity $\Zt(\cdot, t)$ is not a constant function and satisfies $\Ztap(\cdot, t) \in \Ltwo$ and if the interface is $C^{1, \alpha}$ (to make sure that $\abs{\Zap} $ is bounded), then $-\frac{\partial P}{\partial \hat{n}}\compose \hinv(\ap, t) > 0$ for all $\ap \in \Rsp$ but $-\frac{\partial P}{\partial \hat{n}}\compose \hinv(\ap,t) \to 0$ as $\abs{\ap} \to \infty$. 

Hence if $g \to 0$, one does not have a uniform positive lower bound on the Taylor sign condition term \eqref{eq:Taylorsignterm}, even for smooth decaying initial data. On the other hand if the interface has a corner or cusp of angle $\nu \pi$ for $0\leq \nu <\half$, then one has $\frac{1}{\Zapabs} = 0 $ at that point and as $\Ag \in \Linfty$ (from the assumption $\Ztap \in \Ltwo$), we see that $-\frac{\partial P}{\partial \hat{n}}\compose \hinv = 0$ at the singularity, for any value of $g$. Hence in this case as well, we also don't have the Taylor sign condition satisfied. One of the features of the above a priori estimate is that it makes no assumptions on the Taylor sign condition term and works in both cases mentioned above.

\item The energy estimates \eqref{eq:mainEatime} and \eqref{eq:mainEtime} are invariant under all the scaling transformations \eqref{eq:scaling}. Scaling invariant estimates for $g = 1$ and $s = \half$ were first obtained in \cite{HuIfTa16} (see also \cite{AiIfTa24}). The estimates \eqref{eq:mainEatime} and \eqref{eq:mainEtime} are invariant not only for $s = \half$ but for all $s \in \Rsp$.

\end{enumerate}

An important consequence of the above a priori estimate is the following: 
\begin{cor}\label{cor:uniftimeSobolev}
Assume that the initial data $(\Z,\Zt)(0)$ satisfies $(\Zap - 1, \frac{1}{\Zap} - 1, \Zt)(0) \in H^{s}(\Rsp)\times H^s(\Rsp) \times H^{s + \half}(\Rsp)$ for some $s \geq 4$ and fix $g_0 > 0$. Then for $0<g \leq g_0$, there exists a time $T>0$ independent of $g$, such that on $[0,T]$ the initial value problem for \eqref{eq:systemone} with gravity $g$ has a unique solution $(\Z,\Zt)(t)$ satisfying $(\Zap - 1, \frac{1}{\Zap} - 1, \Zt) \in C^l([0,T], H^{s - l}(\Rsp)\times H^{s - l}(\Rsp) \times H^{s + \half - l}(\Rsp))$ for $l = 0,1$. 
\end{cor}

This corollary is proved in \secref{sec:mainexitencesection} right after \thmref{thm:existenceSobolevunbdd}. This result shows that the time of existence of the solutions is uniform even as $g \to 0$. In all previous results, for the type of initial data taken above, one would get that the time of existence $T \to 0$ as $g \to 0$. Note that for periodic boundary conditions, this result was already known as in that case, one does have a positive lower bound on $-\frac{\partial P}{\partial \hat{n}}\compose \hinv$ for $g = 0$ (assuming the initial velocity is not identically zero). This lower bound in the case of periodic solutions is due to compactness. However even in the periodic case, such an energy estimate will give a longer time of existence as compared to standard energy estimates (see the comments made after \thmref{thm:introbdd} in the introduction). 

More generally, using the a priori estimates of \thmref{thm:aprioriE}, one can prove a local wellposedness result for initial data which allows for interfaces with angled crests and cusps (and also allows smooth initial data). Now observe from the definition of the energy $\E(t)$, that it is not so clear on how to directly make sense of this energy if the interface has a corner or cusp (i.e. it is not clear whether the energy is well defined). To overcome this issue, we define an energy which is equivalent to it, but which makes sense for all such data. To do this, we use the fact that $Z(\ap,t)$ is the boundary value of the conformal map $\Psi(\zp,t)$ and that $\Ztbar(\ap,t)$ is the boundary value of the holomorphic function $U(\zp,t)$. We then define our main energy $\Ecal(t)$ as follows: 
\begingroup
\allowdisplaybreaks
\begin{align*}
\Ecalone(t) & =  \brac{\sup_{\yp<0} \norm[\Ltwo(\Rsp, \diff \ap)]{\pzp U(\cdot + i\yp,t)}^2 + g}\sup_{\yp<0} \norm[\Ltwo(\Rsp, \diff \ap)]{\pzp \frac{1}{\Psizp}(\cdot + i\yp,t)}^2 \\
\Ecaltwo(t) & =  \brac{\sup_{\yp<0} \norm[\Ltwo(\Rsp, \diff \ap)]{\pzp U(\cdot + i\yp,t)}^2 + g}\sup_{\yp<0}\norm[\Ltwo(\Rsp, \diff \ap)]{\frac{1}{\Psizp}\pzp\brac{\frac{1}{\Psizp}\pzp U}(\cdot + i\yp,t)}^2 \\*
& \qquad + \brac{\sup_{\yp<0} \norm[\Ltwo(\Rsp, \diff \ap)]{\pzp U(\cdot + i\yp,t)}^2 + g}^{\n 2}\sup_{\yp<0} \norm[\Hhalf(\Rsp, \diff \ap)]{\frac{1}{\Psizp}\pzp \frac{1}{\Psizp}(\cdot + i\yp,t)}^2 \\
\Ecalthree(t) & =  \brac{\sup_{\yp<0} \norm[\Ltwo(\Rsp, \diff \ap)]{\pzp U(\cdot + i\yp,t)}^2 + g}^{\n 3}\sup_{\yp<0}\norm[\Ltwo(\Rsp, \diff \ap)]{\frac{1}{\Psizp}\pzp\brac{\frac{1}{\Psizp}\pzp \frac{1}{\Psizp}}(\cdot + i\yp,t)}^2 \\*
& \quad +  \brac{\sup_{\yp<0} \norm[\Ltwo(\Rsp, \diff \ap)]{\pzp U(\cdot + i\yp,t)}^2 + g}^{\n 2}\sup_{\yp<0}\norm[\Hhalf(\Rsp, \diff \ap)]{\frac{1}{\Psizp^2}\pzp\brac{\frac{1}{\Psizp}\pzp \U}(\cdot + i\yp,t)}^2 \\
\Ecal(t) & = \brac{\Ecalone(t)^3 + \Ecaltwo(t)^{\frac{3}{2}} + \Ecalthree(t)}^{\frac{1}{3}}
\end{align*}
\endgroup
If $\Z$ and $\Zt$ are smooth enough then this energy is equivalent to 

\begin{align*}
\begin{aligned}
\Ecalone(t) & = \brac{\norm[2]{\Ztapbar}^2 + g}\norm[2]{\pap\frac{1}{\Zap}}^2\\
\Ecaltwo(t) & =  \brac{\norm[2]{\Ztapbar}^2 + g}\norm[2]{\Dap^2\Ztbar}^2 + \brac{\norm[2]{\Ztapbar}^2 + g}^2\norm[\Hhalf]{\Dap\frac{1}{\Zap}}^2 \\
\Ecalthree(t) & = \brac{\norm[2]{\Ztapbar}^2 + g}^3\norm[2]{\Dap^2\frac{1}{\Zap}}^2 +  \brac{\norm[2]{\Ztapbar}^2 + g}^2\norm[\Hhalf]{\frac{1}{\Zap}\Dap^2\Ztbar}^2
\end{aligned}
\end{align*}
For $g > 0$ these two energies $E(t)$ and $\Ecal(t)$ are equivalent and for $g = 0$, we have $E(t) \lesssim \Ecal(t)$. See \lemref{lem:equivSobolevunbdd} for the precise relation between these two energies.

We now write down the precise conditions on the initial data and the main existence result. The precise definition of the notion of solution is given in \defref{def:solution}. For $g > 0$ we prove uniqueness in a class of solutions called smoothly approximable solutions denoted by $\mathcal{SA}$, which are defined in \defref{def:solutionSA}. 

\medskip
\noindent\emph{Initial data:} Let $g \geq 0$ and let $(\U,\Psi)(0) : \Pminus \to \Csp$, $\Pfrak(0): \Pminus \to \Rsp$ be such that $\U(0)$ and $\Psi(0)$ are holomorphic functions with $\Psizp \neq 0$ for $\zp \in \Pminus$. We assume that $\Pfrak$ solves \eqref{eq:DeltaPfrak} and that $\lim_{\zp \to \infty} (\U,\Psizp, (\pxp - i\pyp) \Pfrak)(\zp,0) \to (0,1, ig)$. We also assume that 
\begin{align}\label{eq:czero}
c_0 := \nobrac{\sup_{\yp<0} \norm[H^1(\Rsp, \diff \ap)]{ U(\cdot + i\yp,0)} + \sup_{\yp<0} \norm[H^1(\Rsp, \diff \ap)]{ \frac{1}{\Psizp}(\cdot + i\yp,0) - 1} }  < \infty
\end{align}
and that $\Ecal(0) < \infty$. 

\begin{thm}\label{thm:existencemain}
Let $g\geq 0$ and let the initial data $(\U,\Psi,\Pfrak)(0)$ be as given above with $\Ecal(0) < \infty$. Then there exists a universal constant $c>0$ such that there exists $T \geq \frac{c}{\sqrt{\Ecal(0)}}$ so that there exists a solution $(\U,\Psi, \Pfrak)(t)$ to the water wave equation \eqref{eq:EulerRiem} in the time interval $[0,T]$ with the given initial data in the sense of \defref{def:solution}. 

If $g > 0$, then there exists a unique solution in the class $\mathcal{SA}$ (defined in \defref{def:solutionSA}) and which satisfies $\sup_{t \in [0,T]} \Ecal(t) < \infty$. Moreover in this case, if $T^*>0$ is the maximal time of existence, then either $T^* = \infty$ or $T^*<\infty$ and we have
\begin{align*}
\limsup_{t \to T^*} \brac{\norm[\Linfty\cap\Hhalf]{\Dap\Ztbar}(t)  +  \brac{\norm[2]{\Ztapbar}(t) + \sqrt{g}}\norm[2]{\pap\frac{1}{\Zap}}(t)} = \infty
\end{align*} 
\end{thm}
We prove this theorem in \secref{sec:mainexitencesection}. The above theorem is the rigorous version of \thmref{thm:intromain} mentioned in the introduction. See the comments made after \thmref{thm:intromain} for the important features of this theorem.

\begin{rmk}
The singular solutions established in \thmref{thm:existencemain} also have a rigidity property namely that the singularities (where singularities are defined as the set of all points where $\frac{1}{\Psizp}(\ap, t) = 0$) propagate via the Lagrangian flow, angled crests/cusps remain angled crested/cusped, the acceleration at the tip is $- ig$, and the angle of the crest does not change nor does it tilt. We refer the reader to \cite{Ag20} for the precise statement of the results (see also \thmref{thm:mainangle} in the next section). The results of \cite{Ag20} hold in exactly the same manner here as well, with the energy used in \cite{Ag20} replaced by $\Ecal(t)$. The proof of the rigidity results for $g > 0$ follow exactly in the same way and for $g = 0$, the proof goes through with some minor modifications.  
\end{rmk}


\subsection{The bounded domain case}\label{sec:resultbdd}

In this subsection we collect our main results for the model \eqref{eq:EulerRiem2} which is the same as \eqref{eq:Euler2} but written in conformal coordinates. The results of this subsection mimic the results established in the previous section \secref{sec:resultunbdd}. Similar to the a priori estimate \thmref{thm:aprioriE} we also prove an a priori estimate \thmref{thm:aprioriEbdd}. Let us now define the analogue of the energy $\Ecal(t)$ defined in the previous section. Define the energy
\begingroup
\allowdisplaybreaks
\begin{align*}
\Ecaltilone(t) & =  \sup_{0< r < 1} \norm[\Ltwo([0, 2\pi], \diff \theta)]{\pzp U(re^{i\theta},t)}^2\sup_{0< r < 1} \norm[\Ltwo([0, 2\pi], \diff \theta)]{\pzp \frac{1}{\Psizp}(re^{i\theta},t)}^2 \\*
& \quad +  \sup_{0< r < 1} \norm[\Ltwo([0, 2\pi], \diff \theta)]{\pzp U(re^{i\theta},t)}^2\sup_{0< r < 1} \norm[\Ltwo([0, 2\pi], \diff \theta)]{\frac{1}{\Psizp}(re^{i\theta},t)}^2 \\
\Ecaltiltwo(t) & =  \sup_{0< r < 1} \norm[\Ltwo([0, 2\pi], \diff \theta)]{\pzp U(re^{i\theta},t)}^2\sup_{0< r < 1}\norm[\Ltwo([0, 2\pi], \diff \theta)]{\frac{1}{\Psizp}\pzp\brac{\frac{1}{\Psizp}\pzp U}(re^{i\theta},t)}^2 \\*
& \qquad + \sup_{0< r < 1} \norm[\Ltwo([0, 2\pi], \diff \theta)]{\pzp U(re^{i\theta},t)}^4\sup_{0< r < 1} \norm[\Hhalf([0, 2\pi], \diff \theta)]{\frac{1}{\Psizp}\pzp \frac{1}{\Psizp}(re^{i\theta},t)}^2 \\
\Ecaltilthree(t) & =  \sup_{0< r < 1} \norm[\Ltwo([0, 2\pi], \diff \theta)]{\pzp U(re^{i\theta},t)}^6 \sup_{0< r < 1}\norm[\Ltwo([0, 2\pi], \diff \theta)]{\frac{1}{\Psizp}\pzp\brac{\frac{1}{\Psizp}\pzp \frac{1}{\Psizp}}(re^{i\theta},t)}^2 \\*
& \quad +  \sup_{0< r < 1} \norm[\Ltwo([0, 2\pi], \diff \theta)]{\pzp U(re^{i\theta},t)}^4 \sup_{0< r < 1}\norm[\Hhalf([0, 2\pi], \diff \theta)]{\frac{1}{\Psizp^2}\pzp\brac{\frac{1}{\Psizp}\pzp \U}(re^{i\theta},t)}^2 \\
\Ecaltil(t) & = \brac{\Ecaltilone(t)^3 + \Ecaltiltwo(t)^{\frac{3}{2}} + \Ecaltilthree(t)}^{\frac{1}{3}}
\end{align*}
\endgroup
If $\Z$ and $\Zt$ are smooth enough then this energy is equivalent to 

\begingroup
\allowdisplaybreaks
\begin{align*}
\Ecaltilone(t) & = \norm[2]{\Ztapbar}^2\brac{\norm*[\bigg][2]{\pap\brac*[\bigg]{\frac{e^{i\ap}}{\Zap}}}^2 + \norm[2]{\frac{1}{\Zap}}^2 }\\
\Ecaltiltwo(t) & =  \norm[2]{\Ztapbar}^2\norm[2]{\Dap^2\Ztbar}^2 + \norm[2]{\Ztapbar}^4\norm[\Hhalf]{\Dap\brac*[\bigg]{\frac{e^{i\ap}}{\Zap}}}^2 \\
\Ecaltilthree(t) & = \norm[2]{\Ztapbar}^6\norm[2]{\Dap^2\brac*[\bigg]{\frac{e^{i\ap}}{\Zap}}}^2 +  \norm[2]{\Ztapbar}^4\norm[\Hhalf]{\frac{e^{i\ap}}{\Zap}\Dap^2\Ztbar}^2
\end{align*}
\endgroup

We now write down the precise conditions on the initial data and the main existence result for the model \eqref{eq:EulerRiem2}. Similar to the unbounded case, we define a notion of solution in \defref{def:solutionbdd} and prove uniqueness in a class called smoothly approximable solutions denoted by $\mathcal{SA}$ which are defined in \defref{def:solutionSAbdd}. 

\medskip
\noindent\emph{Initial data:} Let $(\U,\Psi)(0) : \Dsp \to \Csp$, $\Pfrak(0): \Dsp \to \Rsp$ be such that $\U(0)$ and $\Psi(0)$ are holomorphic functions with $\Psizp(\zp, 0) \neq 0$ for $\zp \in \Dsp$ and $\Psizp(0,0) > 0$ and assume that $\Pfrak$ solves \eqref{eq:DeltaPfrakbdd}. We also assume that
\begin{align}\label{eq:czerobdd}
c_0 := \sup_{0< r < 1} \norm[H^1([0, 2\pi], \diff \theta)]{U(re^{i\theta},0)} + \sup_{0< r < 1} \norm[H^1([0, 2\pi], \diff \theta)]{\frac{1}{\Psizp}(re^{i\theta},0)}   < \infty
\end{align}
and that $\Ecaltil(0) < \infty$. 

\begin{thm}\label{thm:existencemainbdd}
Let the initial data $(\U,\Psi,\Pfrak)(0)$ be as given above with $\Ecaltil(0) < \infty$. Then there exists a universal constant $c>0$ such that there exists $T \geq \frac{c}{\sqrt{\Ecaltil(0)}}$ so that there exists a unique solution $(\U,\Psi, \Pfrak)(t)$ to the water wave equation \eqref{eq:EulerRiem2} with the given initial data in the time interval $[0,T]$ in the class $\mathcal{SA}$  and we have $\sup_{t \in [0,T]} \Ecaltil(t) < \infty$. Moreover if $T^*$ is the maximal time of existence then either $T^* = \infty$ or $T^* < \infty$ along with 
\begin{align*}
\limsup_{t \to T^*}\brac{\norm[\Linfty\cap\Hhalf]{\Dap\Ztbar}(t)  +  \norm[2]{\Ztapbar}(t)\brac{\norm[2]{\pap\frac{1}{\Zap}}(t) + \norm[2]{\frac{1}{\Zap}}(t) }} = \infty
\end{align*}
\end{thm}

We prove this result in \secref{sec:existenceanduniqbdd}. This theorem is the rigorous version of the first part of \thmref{thm:introbdd} in the introduction. See the comments made after \thmref{thm:introbdd} for the important features of this theorem.

The singular solutions from \thmref{thm:existencemainbdd} satisfy a rigidity property similar to the unbounded case and \cite{Ag20}. As we will need this rigidity result for our later blow up result, we now state the rigidity result precisely.  

We define the singular set of the interface at time $t$ as
\begin{align*}
    S(t) = \cbrac{\ap\in\mathbb{R} \ \Big| \  \Tr\brac{\frac{1}{\Psizp}}(\ap,t) = 0}
\end{align*}
The non-singular set is then defined as $NS(t) = \mathbb{R}\backslash S(t)$. Observe that both $S(t)$ and $NS(t)$ are $2\pi$ periodic sets and that corners and cusps are indeed included in the singular set. The rigidity result is now as follows:

\begin{thm}\label{thm:mainangle}
Let  $(\U,\Psi,\Pfrak)(t)$ be a solution to the water wave equation \eqref{eq:EulerRiem2} in the time interval $[0,T]$ in the class $\mathcal{SA}$ with $\sup_{t \in [0,T]} \Ecaltil(t) < \infty$. Then
\begin{enumerate}[leftmargin =*, align=left]
\item $S(t) =  \cbrac{h(\al,t) \in \Rsp \suchthat \al \in S(0) }$ for all $t\in[0,T]$. Moreover $S(t)$ is a closed set of measure zero for all $t\in[0,T]$. 
\item For every fixed $t\in [0,T]$, the  functions  $\brac{\frac{1}{\Psizp}\partial_{\zp}\U}(\cdot,t)$ and $\brac{\frac{1}{\Psizp}\overline{\partial_{\zp}\U}}(\cdot,t)$ extend to continuous functions on $\Dspbar$ with 
\begin{align*}
\Tr\brac{\frac{1}{\Psizp}\partial_{\zp}\U}(\ap,t) = \Tr\brac{\frac{1}{\Psizp}\overline{\partial_{\zp}\U}}(\ap,t) = 0 \quad \forall  \ap \in S(t)
\end{align*}

\item If $\al_{0} \in S(0)$, and if $\{\al_n\}_{n \geq 1}$ is any sequence such that $\al_n \in NS(0)$ for all $n\geq 1$ with $\al_n \to \al_0$, then for any $t \in [0,T]$
\begin{align}\label{eq:rigidlimit}
\lim_{\al_n \to \al_{0}} \frac{\frac{\Zap}{\Zapabs}(h(\al_n,t),t)) }{\frac{\Zap}{\Zapabs}(\al_n,0)} = 1
\end{align}
\item For all $\ap \in S(t)$, we have $\Ztt(\ap,t) = 0$. 
\end{enumerate}
\end{thm}
The first statement of the result says that the singularities are propagated by the Lagrangian flow. The second statement says that the gradient of the velocity extends continuous to the boundary and vanishes at the singularities. The third statement says that the nature of the singularity is preserved for later time i.e. a corner remains a corner and a cusp remains a cusp with the angle of the corner remaining the same with no tilting. The last statement says that the acceleration at the singularity is zero. The proof of the above rigidity result follows in exactly the same manner as the proof in \cite{Ag20} with virtually no changes and hence we skip the proof here. 

We can now state our result on blow up. 

%

\begin{theorem}\label{thm:blowup}
There exist initial data $(\U,\Psi, \Pfrak)(0)$ satisfying the assumptions of \thmref{thm:existencemainbdd} with $0 <\Ecaltil(0) < \infty$ and in addition satisfying the following:
\begin{enumerate}
\item The initial data is symmetric with respect to the x-axis, namely 
\begin{align*}
\Z(\ap,0) = \Zbar(-\ap,0) \tx{ and } \Zt(\ap,0) = \Ztbar(-\ap,0) \tx{ for } \ap \in [0,\pi]
\end{align*}
\item  We have $\Tr\brac{\frac{1}{\Psizp}}(0,0) = \Tr\brac{\frac{1}{\Psizp}}(\pi,0) = 0$. Let $d = \abs{\Z(0,0) - \Z(\pi,0)}$.
\item We have $\Zt(0,0) < 0$ and $\Zt(\pi,0) > 0$. Let $v = \abs{\Zt(0,0) - \Zt(\pi,0)}$. 
\end{enumerate}
For all such initial data, there exists $T^*>0$ satisfying $0< \frac{c}{\sqrt{\Ecaltil(0)}} \leq T^* \leq \frac{d}{v} < \infty$, where $c$ is a universal constant, so that the water wave equation \eqref{eq:EulerRiem2} has a unique solution $(\U,\Psi, \Pfrak)(t)$ in the class $\mathcal{SA}$ in the time interval $[0,T^*)$ with $\sup_{ t \in [0, T]} \Ecaltil(t) <\infty $ for all $T \in [0,T^*)$ and we have 
\begin{align*}
\limsup_{t \to T^*}\brac{\norm[\Linfty\cap\Hhalf]{\Dap\Ztbar}(t)  +  \norm[2]{\Ztapbar}(t)\brac{\norm[2]{\pap\frac{1}{\Zap}}(t) + \norm[2]{\frac{1}{\Zap}}(t) }} = \infty
\end{align*}
In particular $\limsup_{t \to T^*} \Ecaltil(t) = \infty$. 
\end{theorem}

This theorem is proved in \secref{sec:existenceanduniqbdd}. The above theorem is the rigorous version of the second part of \thmref{thm:introbdd} in the introduction. See \figref{fig:blowup} for a picture for this blow up scenario and the comments made after \thmref{thm:introbdd} for the important features of this theorem.

 The above result follows essentially directly from \thmref{thm:existencemainbdd} and \thmref{thm:mainangle}, and we give the details of the proof in \secref{sec:existenceanduniqbdd}.  This result shows that the local in time solutions of \thmref{thm:existencemainbdd} cannot in general be extended to global solutions. This also shows that the time of existence obtained in \thmref{thm:existencemainbdd} is optimal in the following sense: Consider an initial data $(\U,\Psi, \Pfrak)(0)$ satisfying the conditions of the above theorem and replace the initial velocity by $\U_\ep(0) = \ep \U(0)$ for some $\ep >0$ small. It is clear that this initial data also satisfies the conditions of the theorem and we have $\Ecaltil_\ep(0) = \ep^2\Ecaltil(0)$ and $v_\ep = \ep v$ (here $v$ is the variable used in the above theorem). Hence from the result above we see that as $\ep \to 0$, we have that $T^*_\ep \sim \frac{1}{\ep}$. This shows that the lower bound on the time of existence $T \geq \frac{c}{\sqrt{\Ecaltil(0)}}$ from \thmref{thm:existencemainbdd} cannot be replaced with $T \geq \frac{c}{\Ecaltil(0)^{\half + \delta}}$ for any $\delta > 0$ even for initial data with small $\Ecaltil(0)$.

\section{Proof of main results for unbounded domain case}\label{sec:unboundedsection}

\subsection{Some useful identities}\label{sec:identities}

Here we first collect the main identities commonly used in this paper. The identity \eqref{eq:maineqcomDap} is from \cite{KiWu18} whereas the rest are from Section 4 of \cite{Ag21}. 

\begin{enumerate}[leftmargin =*, align=left, label=\alph*)]
\item We have
\[
\frac{\Zap}{\Zapabs}\pap\frac{1}{\Zap}   = \pap \frac{1}{\Zapabs} + \w\Dapabs \wbar
\]
Observe that $\dis \pap \frac{1}{\Zapabs}$ is real valued and $\dis \w\Dapabs \wbar$ is purely imaginary. From this we obtain 
\begin{align}\label{form:RealImagTh}
\Real\brac{\frac{\Zap}{\Zapabs}\pap\frac{1}{\Zap} } = \pap \frac{1}{\Zapabs} \quad \qquad \Imag\brac{\frac{\Zap}{\Zapabs}\pap\frac{1}{\Zap}} = i\brac{ \wbar\Dapabs \w} 
\end{align}

\item For any complex valued function $f$, we have $\Hil (\Real f) =  i\Imag(\Hil f)$ and  $\Hil(i \Imag f) = \Real (\Hil f)$. Hence we get the following identities
\begin{align}
(\Id + \Hil)(\Real f) = f -i\Imag (\Id - \Hil) f  \label{form:Real}\\
(\Id + \Hil)(i \Imag f) = f - \Real (\Id - \Hil)f  \label{form:Imag}
\end{align}

\item We have
\begin{align}\label{form:Aonenew}
\Ag = g + i\Zt\Ztapbar -i(\Id + \Hil)\cbrac{\Real(\Zt\Ztapbar)}
\end{align}

\item We have 
\begin{align*}
\bvar = \frac{\Zt}{\Zap} -i(\Id + \Hil)\cbrac{\Imag\brac{\frac{\Zt}{\Zap}}}
\end{align*}
and hence 
\begin{align}\label{form:bvarapnew}
\bvarap = \Dap\Zt + \Zt\pap\frac{1}{\Zap}  -i\pap(\Id + \Hil)\cbrac{\Imag\brac{\frac{\Zt}{\Zap}}}
\end{align}

\item We now record some frequently used commutator identities. 
\begin{align}\label{eq:commutator}
\begin{aligned}
 [\pap, \Dt ] &= \bap \pap   \qquad &[\Dapabs, \Dt] &= \Real{(\Dap \Zt)} \Dapabs = \Real{(\Dapbar \Ztbar) \Dapabs} \\   \relax 
 [\Dap, \Dt]  &=  \brac{\Dap \Zt} \Dap   \qquad &  [\Dapbar, \Dt] &= \brac{\Dapbar \Ztbar } \Dapbar 
\end{aligned}
\end{align}
Using these we also obtain the following formulae
\begin{align}
\Dt \Zapabs & = \Dt e^{\Real\log \Zap} =  \Zapabs \cbrac{\Real(\Dap\Zt) - \bvarap} \label{form:DtZapabs} \\ 
 \Dt \frac{1}{\Zap} & = \frac{-1}{\Zap}(\Dap\Zt - \bvarap) = \frac{1}{\Zap}\cbrac{(\bvarap - \Dap\Zt - \Dapbar \Ztbar) + \Dapbar \Ztbar} \label{form:DtoneoverZap}
\end{align}
Observe that $(\bvarap - \Dap\Zt - \Dapbar \Ztbar)$ is real valued and this fact will be useful later on. Using the above commutator relations and \eqref{form:Zttbar} we also get 
\begin{align}\label{eq:maineqcomDap}
\sqbrac{\Dt^2 + i\frac{\Ag}{\Zapabs^2}\pap, \Dap} = -2(\Dap\Ztt)\Dap - 2(\Dap\Zt)\Dt\Dap
\end{align}

\medskip

\end{enumerate}

We now obtain some identities related to the function $\Ag$. These will be very important in proving the a priori estimate. From \eqref{eq:systemone} we see that (here $\Azero$ is $\Ag$ with $g = 0$)
\begin{align}\label{eq:Ag}
\Ag = g + \Azero 
\end{align}
where from a calculation from \cite{Wu97} and using the definition \eqref{eq:fonefn} we have
\begin{align}\label{eq:Azero}
\begin{split}
\Azero(\ap)  = -(\Imag\sqbrac{\Zt,\Hil}\Ztapbar)(\ap) & = -\Imag\cbrac{\frac{1}{i\pi} \int \frac{\Zt(\ap)-\Zt(\bp)}{\ap-\bp} \Ztbarbp(\bp) \diff\bp } \\ 
 & = \frac{1}{2\pi}\int \abs{\frac{\Zt(\ap) - \Zt(\bp)}{\ap-\bp}}^2 \diff\bp \\
 & = \frac{1}{2\pi}\norm[\Ltwo(\diff \bp)]{\frac{\Zt(\ap) - \Zt(\bp)}{\ap - \bp}}^2 \\
 & = \frac{i}{2} \sqbrac{\Zt,\Ztbar;1}(\ap)
 \end{split}
\end{align}
Hence $\Ag \geq g \geq 0$. Moreover if $\Zt$ is not a constant function, then $\Azero > 0$ everywhere and hence $\Ag > 0$ everywhere. On the other hand if $\Ztap \in \Ltwo$, then $\Azero(\ap) \to 0$ as $\abs{\ap} \to \infty$ from \lemref{lem:decayatinfinity}. Hence $\Azero$ is positive everywhere but does not have a uniform positive lower bound.  

Let us now compute the material derivative of $\Ag$. 

\begin{lemma}\label{lem:DtAgoverAg}
We have
\begin{align*}
 \frac{\Dt \Ag}{\Ag} & = \frac{\Imag\cbrac{-2\sqbrac{\Zt,\Zttbar;1} + \sqbrac{\Zt,\Zt;\Dap\Ztbar}}}{\Ag} + 2\Real(\Dap\Zt) - \bvarap \\
& = 2\Real\cbrac{\sqbrac{\Zt,\frac{1}{\Zap}; 1}(\ap)} + \frac{1}{\Ag}\Imag\cbrac{2i\sqbrac{\Zt,\Ag; \frac{1}{\Zap}}(\ap) + \sqbrac{\Zt,\Zt; \Dap\Ztbar}(\ap) } \\
& \quad + 2\Real(\Dap\Zt) - \bvarap
\end{align*}
\end{lemma}
\begin{proof}

Following  \cite{Wu97} (see also \cite{Wu16}), the variable $\avar$ is defined such that we have the identity 
\begin{align*}
\Ag = \brac*[\bigg]{\frac{\avar\zalabs^2}{\hal}}\compose \hinv
\end{align*}
Hence we see that 
\begin{align}\label{eq:DtAbyAgcomatbya}
 \frac{\Dt \Ag}{\Ag}  = \cbrac{\frac{\avar_t}{\avar} + 2\Real\brac{\frac{\ztal}{\zal}} - \frac{\h_{t \al}}{\hal}}\compose\hinv = \frac{\avar_t}{a}\compose\hinv + 2\Real(\Dap\Zt) - \bvarap
\end{align}
Now in \cite{Wu97} (see also \cite{Wu16}), the following formula was derived 
\begin{align*}
 \frac{\avar_t}{a}\compose\hinv =  \frac{-\Imag\brac{2\sqbrac{\Zt,\Hil}\Zttbarap + 2\sqbrac{\Ztt,\Hil}\pap\Ztbar - \sqbrac{\Zt,\Zt; \Dap\Ztbar} }}{\Ag}
\end{align*}
The above formula was derived for $g = 1$ but the same proof works for $g \geq 0$ as well. Now we simplify the first two terms of the numerator of $\frac{\avar_t}{\avar}\compose\hinv$
\begin{align*}
& \Imag\cbrac{ \sqbrac{\Zt,\Hil}\Zttapbar + \sqbrac{\Ztt,\Hil}\pap\Ztbar  } \\
& = \Imag\cbrac{ \frac{1}{i\pi} \int \frac{\Zt(\ap)-\Zt(\bp)}{\ap-\bp}\Zttbpbar(\bp) \diff\bp + \frac{1}{i\pi} \int \frac{\Ztt(\ap)-\Ztt(\bp)}{\ap-\bp}\Ztbpbar(\bp) \diff\bp }\\
& = -\frac{1}{\pi}\Real\cbrac{  \int \frac{\Zt(\ap)-\Zt(\bp)}{\ap-\bp}\Zttbpbar(\bp) \diff\bp +  \int \frac{\Ztt(\ap)-\Ztt(\bp)}{\ap-\bp}\Ztbpbar(\bp) \diff\bp } \\
& = -\frac{1}{\pi}\Real\cbrac{ \int \pbp\brac{\frac{\Zt(\ap)-\Zt(\bp)}{\ap-\bp}}\brac{\Zttbar(\ap) - \Zttbar(\bp)}\diff\bp +  \int \frac{\Ztt(\ap)-\Ztt(\bp)}{\ap-\bp}\Ztbpbar(\bp) \diff\bp } \\
& = -\frac{1}{\pi}\Real \int\brac{\frac{\Zt(\ap)-\Zt(\bp)}{\ap-\bp}}\brac{\frac{\Zttbar(\ap)-\Zttbar(\bp)}{\ap-\bp}} \diff\bp \\
& = -\Real\cbrac{i\sqbrac{\Zt,\Zttbar; 1}} \\
& = \Imag \sqbrac{\Zt,\Zttbar;1}
\end{align*}
Hence we have
\begin{align*}
 \frac{\Dt \Ag}{\Ag} = \frac{\Imag\cbrac{-2\sqbrac{\Zt,\Zttbar;1} + \sqbrac{\Zt,\Zt;\Dap\Ztbar}}}{\Ag} + 2\Real(\Dap\Zt) - \bvarap 
\end{align*}
Now from \eqref{form:Zttbar} we get
\begin{align}\label{eq:Zttbarfrac}
\begin{split}
& \frac{\Zttbar(\ap)-\Zttbar(\bp)}{\ap-\bp}  \\
& = -i\cbrac{\Ag(\ap)\brac{\frac{\frac{1}{\Zap}(\ap) - \frac{1}{\Zap}(\bp)}{\ap - \bp}} + \brac{\frac{\Ag(\ap) - \Ag(\bp)}{\ap - \bp}}\frac{1}{\Zap}(\bp)}
\end{split}
\end{align}
Hence
\begin{align*}
\sqbrac{\Zt,\Zttbar;1}(\ap) = -i\Ag(\ap)\sqbrac{\Zt, \frac{1}{\Zap}; 1}(\ap) -i \sqbrac{\Zt,\Ag;\frac{1}{\Zap}}(\ap)
\end{align*}
Therefore
\begin{align*}
 \frac{\Dt \Ag}{\Ag}  & = 2\Real\cbrac{\sqbrac{\Zt,\frac{1}{\Zap}; 1}(\ap)} + \frac{1}{\Ag}\Imag\cbrac{2i\sqbrac{\Zt,\Ag; \frac{1}{\Zap}}(\ap) + \sqbrac{\Zt,\Zt; \Dap\Ztbar}(\ap) } \\
& \quad + 2\Real(\Dap\Zt) - \bvarap
\end{align*}

\end{proof}

\subsection{The main equations}\label{sec:maineqn}

We will now derive our main equations used to obtain the energy estimates. Define 
\begin{align}\label{form:Jzero}
\Jzero = \Dt(\bvarap - \Dap\Zt - \Dapbar\Ztbar)
\end{align}
Observe that $\Jzero$ is real valued. Apply $\Dt$ to the formula for $\Dt\frac{1}{\Zap}$ in \eqref{form:DtoneoverZap} to get
\begin{align*}
\Dt^2\frac{1}{\Zap} = \frac{1}{\Zap}\cbrac{(\bvarap-\Dap\Zt)^2 + \Jzero + \Dt\Dapbar\Ztbar }
\end{align*}
Now from \eqref{form:Zttbar} we see that 
\begin{align*}
\Dt\Dapbar\Ztbar & = -(\Dapbar\Ztbar)^2 + \Dapbar\Zttbar \\
& = -(\Dapbar\Ztbar)^2 -\frac{i}{\Zapabs^2}\pap\Ag -i\Ag\Dapbar\frac{1}{\Zap}
\end{align*}
Define 
\begin{align}\label{form:Qzero}
\Qzero = (\bvarap-\Dap\Zt)^2 - (\Dapbar\Ztbar)^2 -\frac{i}{\Zapabs^2}\pap\Ag
\end{align}
Hence we see that 
\begin{align*}
\brac{\Dt^2 + i\frac{\Ag}{\Zapabs^2}\pap}\frac{1}{\Zap} = \frac{1}{\Zap}(\Jzero + \Qzero)
\end{align*}
Now we apply $\pap$ and commute. First using $[\pap,\Dt] = \bvarap\pap$ we see that
\begin{align*}
\sqbrac{\pap,\Dt^2}\frac{1}{\Zap} & = [\pap,\Dt]\Dt\frac{1}{\Zap} + \Dt[\pap,\Dt]\frac{1}{\Zap} \\
& = \bvarap\brac{\pap\Dt\frac{1}{\Zap}} + \Dt\brac{\bvarap\pap\frac{1}{\Zap}} \\
& =  \bvarap\brac{\pap\Dt\frac{1}{\Zap} + \Dt\pap\frac{1}{\Zap}} + (\Dt\bvarap)\brac{\pap\frac{1}{\Zap}}
\end{align*}
Hence we get our main equation as
\begin{align}\label{eq:paponeoverZap}
\brac{\Dt^2 + i\frac{\Ag}{\Zapabs^2}\pap}\pap\frac{1}{\Zap} = \Dap\Jzero + \Rzero
\end{align}
where
\begin{align}\label{form:Rzero}
\begin{split}
\Rzero & = \brac{\pap\frac{1}{\Zap}}(\Jzero+\Qzero) + \Dap\Qzero -  \bvarap\brac{\pap\Dt\frac{1}{\Zap} + \Dt\pap\frac{1}{\Zap}} \\
& \quad - (\Dt\bvarap)\brac{\pap\frac{1}{\Zap}} -2i\Ag\brac{\Dapabs\frac{1}{\Zapabs}}\brac{\pap\frac{1}{\Zap}} \\
& \quad - i\brac{\frac{1}{\Zapabs^2}\pap\Ag}\brac{\pap\frac{1}{\Zap}}
\end{split}
\end{align}

We will also need another equation for $\Dap^2\Ztbar$. To obtain it, we first apply $\Dt$ to \eqref{form:Zttbar} and then use \eqref{form:DtoneoverZap} to get
\begin{align*}
\Dt^2\Ztbar & = -i \frac{\Dt\Ag}{\Zap} -i\Ag \Dt\frac{1}{\Zap} \\
& = -i \frac{\Ag}{\Zap} \Dapbar\Ztbar -i \frac{1}{\Zap}\cbrac{\Dt\Ag + \Ag\brac{\bap - \Dap\Zt - \Dapbar\Ztbar}}
\end{align*}
Define 
\begin{align}\label{eq:Jone}
\Jone = \Dt\Ag + \Ag\brac{\bap - \Dap\Zt - \Dapbar\Ztbar}
\end{align}
Observe that $\Jone$ is real valued. The above equation can now be written as
\begin{align*}
\brac{\Dt^2 + i\frac{\Ag}{\Zapabs^2}\pap}\Ztbar = -i\frac{\Jone}{\Zap}
\end{align*}
Now applying $\Dap^2$ to both sides and using \eqref{eq:maineqcomDap} we get
\begin{align*}
& \brac{\Dt^2 + i\frac{\Ag}{\Zapabs^2}\pap}\Dap^2\Ztbar \\
& = \sqbrac{\Dt^2 + i\frac{\Ag}{\Zapabs^2}\pap, \Dap^2}\Ztbar -i \Dap^2\brac{\frac{\Jone}{\Zap}} \\
& = \Dap \sqbrac{\Dt^2 + i\frac{\Ag}{\Zapabs^2}\pap, \Dap} \Ztbar + \sqbrac{\Dt^2 + i\frac{\Ag}{\Zapabs^2}\pap, \Dap}\Dap\Ztbar \\
& \quad -i\Dap\brac{\Jone\Dap\frac{1}{\Zap} + \frac{ \wbar^2}{\Zapabs^2}\pap\Jone} \\
& = \Dap\cbrac{-2(\Dap\Ztt)(\Dap\Ztbar) - 2(\Dap\Zt)(\Dt\Dap\Ztbar)} -2(\Dap\Ztt)(\Dap^2\Ztbar) \\
& \quad - 2(\Dap\Zt)(\Dt\Dap^2\Ztbar)  -i(\Dap\Jone) \brac{\Dap\frac{1}{\Zap}} -i\Jone\Dap^2\frac{1}{\Zap} \\
& \quad -2i\wbar(\Dap\wbar)\brac{\frac{1}{\Zapabs^2}\pap\Jone} -i \frac{\wbar^3}{\Zapabs}\pap\brac{\frac{1}{\Zapabs^2}\pap\Jone} 
\end{align*}
Hence we get the equation
\begin{align}\label{eq:mainDapDapZtbar}
\brac{\Dt^2 + i\frac{\Ag}{\Zapabs^2}\pap}\Dap^2\Ztbar = \Rone  -i \frac{\wbar^3}{\Zapabs}\pap\brac{\frac{1}{\Zapabs^2}\pap\Jone} 
\end{align}
where
\begin{align}\label{eq:Rone}
\begin{split}
\Rone & = -2(\Dap^2\Ztt)(\Dap\Ztbar) -4(\Dap\Ztt)(\Dap^2\Ztbar) -2(\Dap^2\Zt)(\Dt\Dap\Ztbar) \\
& \quad -2(\Dap\Zt)(\Dap\Dt\Dap\Ztbar)  - 2(\Dap\Zt)(\Dt\Dap^2\Ztbar)  -i(\Dap\Jone) \brac{\Dap\frac{1}{\Zap}} \\
& \quad  -i\Jone\Dap^2\frac{1}{\Zap}  -2i\wbar(\Dap\wbar)\brac{\frac{1}{\Zapabs^2}\pap\Jone}
\end{split}
\end{align}

\subsection{Quantities controlled by the energy $\Ea(t)$}\label{sec:quantEa}

To prove the energy estimate for the energy $\Ea(t)$ namely \eqref{eq:mainEatime}, we need to prove several estimates first before we can start taking the time derivative of the energy. In this subsection, we collect some of these estimates which are then used in \secref{sec:closing} to prove \eqref{eq:mainEatime}.

Some of the estimates in this subsection are from \cite{Ag21} and for them we simply refer to \cite{Ag21} for a proof. The rest of the estimates are new and among the new estimates in this section, the most important are the ones for 
\begin{align*}
\norm[2]{\frac{1}{\sqrt{\Ag}}\Dapabs\Ag}, \norm[\infty]{\frac{\Dt\Ag}{\Ag}}  \tx{ and } \enspace \norm[2]{\frac{\Dapabs\Dt\Ag}{\sqrt{\Ag}}}
\end{align*}
which are controlled in the estimates \eqref{eq:controlhardestimatel2forAg}, \eqref{eq:controlhardDtAgLinfty} and \eqref{eq:controlhardDapabsDtAgL2} respectively. These three estimates are very delicate as in each of these, there is either a $\Ag$ or $\sqrt{\Ag}$ in the denominator, and if gravity $g = 0$ then $\Ag \to 0$ as $\abs{\ap} \to \infty$. We overcome this difficulty by using the special structure of the function $\Ag$.  

We will frequently use the quantity $B(t)$ defined as
\begin{align*}
B(t) = \norm[\Linfty\cap\Hhalf]{\Dap\Ztbar} + \brac{\norm[2]{\Ztapbar} + \sqrt{g}}\norm[2]{\pap\frac{1}{\Zap}} 
\end{align*}
This quantity is the one introduced in \eqref{eq:Bt} and is the quantity which appears in our main results namely \eqref{eq:mainEatime} in \thmref{thm:aprioriE} and also \thmref{thm:existencemain}. One important estimate that we prove in the following estimates is that of $B(t) \lesssim \Ea(t)^{1/2}$ which is proved in \eqref{eq:controllBt}. We now begin our estimates:

\begin{enumerate}[widest = 99, leftmargin =*, align=left, label=\arabic*)]
\item We have
\begin{align*}
\norm[\Hhalf]{\Ag} + \norm[\Linfty\cap\Hhalf]{\Azero} & \lesssim \norm[2]{\Ztapbar}^2 \\
\norm[\infty]{\Ag} & \lesssim g + \norm[2]{\Ztapbar}^2 \\
\norm[\infty]{\sqrt{\Ag}} & \lesssim \sqrt{g} + \norm[2]{\Ztapbar} 
\end{align*}
Proof: From \eqref{eq:Ag} and \eqref{eq:Azero} we see that $\Ag = g  + \Azero$ where $\Azero =  -\Imag\sqbrac{\Zt,\Hil}\Ztapbar $. The estimates now follow from \propref{prop:commutator}.
\medskip

\medskip
\item $\dis \norm*[\bigg][2]{\pap\frac{1}{\Zapabs}} + \norm[2]{\Dapabs \w} \lesssim \norm[2]{\pap\frac{1}{\Zap}}$ 
\medskip \\
Proof: This follows directly from \eqref{form:RealImagTh}. 
\medskip

\medskip

\item We have
\begin{align}\label{eq:papPaZtZap}
 \norm*[\Big][\infty]{\pap\Pa\brac*[\Big]{\frac{\Zt}{\Zap}}}  \lesssim \norm[\infty]{\Dap\Ztbar} + \norm[2]{\Ztapbar}\norm*[\Big][2]{\pap\frac{1}{\Zap}} \lesssim B(t)
\end{align}
Proof: See estimate 7 in section 5.1 in \cite{Ag21}. 
\medskip

\medskip

\item  We have 
\begin{align*}
\norm[\Linfty]{\bvarap}  & \lesssim  \norm[\Linfty]{\Dap\Ztbar} +  \norm[2]{\Ztapbar}\norm[2]{\pap\frac{1}{\Zap}} \lesssim B(t) \\
 \norm[\Linfty\cap\Hhalf]{\bvarap} &  \lesssim  \norm[\Linfty\cap\Hhalf]{\Dap\Ztbar} +  \norm[2]{\Ztapbar}\norm[2]{\pap\frac{1}{\Zap}} \lesssim B(t)
\end{align*}
Proof: See the proof of estimate 14 in section 5.1 in \cite{Ag21}. 
\medskip

\medskip

\item We have
\begin{align*}
\norm[2]{\Dapabs(\bvarap - \Dap\Zt - \Dapbar\Ztbar)} & \lesssim  \norm*[\Big][2]{\pap\frac{1}{\Zap}}^2\norm[2]{\Ztapbar} + \norm*[\Big][2]{\pap\frac{1}{\Zap}}\norm[\infty]{\Dap\Ztbar} \\
& \lesssim \norm[2]{\pap\frac{1}{\Zap}}B(t)
\end{align*}
Proof: Observe that as $(\bvarap -\Dap\Zt - \Dapbar\Ztbar)$ is real valued we have
\begin{align}\label{eq:DapabsbapDapZtDapbarZtbar}
\begin{aligned}
\Dapabs(\bvarap -\Dap\Zt - \Dapbar\Ztbar) & = \Real \cbrac{\frac{\w}{\Zap}(\Id - \Hil)\pap(\bvarap -\Dap\Zt - \Dapbar\Ztbar)}  \\
 & = \Real \cbrac{\w(\Id - \Hil)\Dap(\bvarap -\Dap\Zt - \Dapbar\Ztbar)} \\
 & \quad  -\Real\cbrac{\w\sqbrac{\frac{1}{\Zap},\Hil}\pap(\bvarap -\Dap\Zt - \Dapbar\Ztbar)} 
\end{aligned}
\end{align}
From \eqref{form:bvarapnew} we have $\dis \bvarap = \Dap\Zt + \Zt\pap\frac{1}{\Zap}  -i\pap(\Id + \Hil)\cbrac{\Imag\brac{\frac{\Zt}{\Zap}}}$. Hence we get
\begin{align}\label{eq:IdminusHDapbvarapDapZt}
\begin{aligned}
& (\Id -\Hil)\Dap(\bvarap -\Dap\Zt - \Dapbar\Ztbar) \\
& = (\Id - \Hil)\cbrac{-\Dap\Dapbar\Ztbar + (\Dap\Zt)\brac*[\Big]{\pap \frac{1}{\Zap}} + \frac{\Zt}{\Zap}\pap \brac*[\Big]{\pap \frac{1}{\Zap}} } \\
& = (\Id - \Hil)\cbrac{-\Dap\Dapbar\Ztbar + (\Dap\Zt)\brac*[\Big]{\pap \frac{1}{\Zap}}} + \sqbrac{\Pa\brac{\frac{\Zt}{\Zap}},\Hil}\pap \brac*[\Big]{\pap \frac{1}{\Zap}}
\end{aligned}
\end{align}
Now observe that as $\Dap\Dapbar\Ztbar = \Dap(\w^2 \Dap\Ztbar)$ we have
\begin{align*}
(\Id - \Hil)\Dap\Dapbar\Ztbar & = (\Id - \Hil)\cbrac{2w(\Dap\w)\Dap\Ztbar} + (\Id - \Hil)\cbrac{\w^2\Dap^2\Ztbar} \\
& = (\Id - \Hil)\cbrac{2w(\Dap\w)\Dap\Ztbar} + \sqbrac{\frac{\w^2}{\Zap}, \Hil}\pap\Dap\Ztbar
\end{align*}
Hence as $\norm[2]{\Dapabs\w} \lesssim \norm[2]{\pap\frac{1}{\Zap}}$ we have from \propref{prop:commutator}
\begin{align*}
& \norm[2]{\Dapabs(\bvarap - \Dap\Zt - \Dapbar\Ztbar)} \\
&  \lesssim  \norm*[\Big][2]{\pap\frac{1}{\Zap}}\cbrac{ \norm[\infty]{\Dap\Ztbar}  +  \norm*[\Big][\infty]{\pap\Pa\brac{\frac{\Zt}{\Zap}}} + \norm[\infty]{\bvarap} } \\
& \lesssim  \norm*[\Big][2]{\pap\frac{1}{\Zap}}^2\norm[2]{\Ztapbar} + \norm*[\Big][2]{\pap\frac{1}{\Zap}}\norm[\infty]{\Dap\Ztbar}
\end{align*}
\medskip

\medskip

\item We have
\begin{align}\label{eq:Dap2Ztbarandmore}
\begin{aligned}
&  \norm[2]{\Dap^2\Ztbar} + \norm[2]{\Dapabs^2\Ztbar} + \norm[2]{\Dapbar^2\Ztbar} + \norm[2]{\frac{1}{\Zapabs^2}\pap\Ztapbar} \\
&  \lesssim \norm*[\Big][2]{\pap\frac{1}{\Zap}}^2\norm[2]{\Ztapbar} + \norm*[\Big][2]{\pap\frac{1}{\Zap}}\norm[\infty]{\Dap\Ztbar} + \norm[2]{\Dt\brac{\pap\frac{1}{\Zap}}} \\
& \lesssim \norm[2]{\pap\frac{1}{\Zap}}B(t) + \norm[2]{\Dt\brac{\pap\frac{1}{\Zap}}}
\end{aligned}
\end{align}
Proof: Recall from  \eqref{form:DtoneoverZap} that $\Dt\frac{1}{\Zap} = \frac{1}{\Zap}(\bvarap - \Dap\Zt)$ and hence
\begin{align*}
\pap\Dt\frac{1}{\Zap} = \Dap(\bvarap - \Dap\Zt - \Dapbar\Ztbar) + \Dap\Dapbar\Ztbar + (\bvarap -\Dap\Zt)\brac{\pap\frac{1}{\Zap}}
\end{align*}
Now using $[\pap,\Dt] = \bvarap\pap$ we see that
\begin{align}\label{eq:DtpaponeoverZap}
\Dt\pap\frac{1}{\Zap} = \Dap(\bvarap - \Dap\Zt - \Dapbar\Ztbar) + \Dap\Dapbar\Ztbar  -\Dap\Zt\brac*[\Big]{\pap\frac{1}{\Zap}}
\end{align}
Hence we see that 
\begin{align*}
& \norm[2]{\Dapbar^2\Ztbar} \\
& =  \norm[2]{\Dap\Dapbar\Ztbar} \\
& \lesssim \norm[2]{\Dap(\bvarap - \Dap\Zt - \Dapbar\Ztbar)} + \norm[2]{\pap\frac{1}{\Zap}}\norm[\infty]{\Dap\Ztbar} + \norm[2]{\Dt\pap\frac{1}{\Zap}} \\
& \lesssim \norm*[\Big][2]{\pap\frac{1}{\Zap}}^2\norm[2]{\Ztapbar} + \norm*[\Big][2]{\pap\frac{1}{\Zap}}\norm[\infty]{\Dap\Ztbar} + \norm[2]{\Dt\pap\frac{1}{\Zap}}
\end{align*}
To get the estimate for $\norm[2]{\Dap^2\Ztbar}$ observe that 
\begin{align*}
\Dap^2\Ztbar = \Dap(\wbar^2\Dapbar\Ztbar) = 2\wbar(\Dap\wbar)\Dapbar\Ztbar + \wbar^2\Dap\Dapbar\Ztbar
\end{align*}
The estimate now follows from previous estimates. The other estimates are proven similarly. 
\medskip

\medskip

\medskip

\item We have
\begin{align}\label{eq:DapZtbarandothers}
\begin{aligned}
& \norm[\Linfty\cap\Hhalf]{\Dap\Ztbar} + \norm[\Linfty\cap\Hhalf]{\Dapabs\Ztbar} +  \norm[\Linfty\cap\Hhalf]{\Dapbar\Ztbar} \\
& \lesssim  \norm[\Linfty\cap\Hhalf]{\Dap\Ztbar} +  \norm*[\Big][2]{\pap\frac{1}{\Zap}}\norm[2]{\Ztapbar} \\
& \lesssim \norm*[\Big][2]{\pap\frac{1}{\Zap}}\norm[2]{\Ztapbar} + \norm[2]{\Dt\brac{\pap\frac{1}{\Zap}}}^\half\norm[2]{\Ztapbar}^\half \\
& \lesssim \Ea(t)^\half
\end{aligned}
\end{align}
Hence we can now bound the quantity $B(t)$ in terms of the energy 
\begin{align}\label{eq:controllBt}
B(t) \lesssim \Ea(t)^\half + \brac{\norm[2]{\Ztapbar} + \sqrt{g}}\norm[2]{\pap\frac{1}{\Zap}}  \lesssim \Ea(t)^\half
\end{align}
Proof: Using \propref{prop:Hhalfweight} with $f = \Ztapbar$, $w = \frac{1}{\Zap} $ and $h = \w$ we obtain
\begin{align*}
\norm[\Hhalf]{\Dapabs\Ztbar} & \lesssim \norm[\infty]{\w}\norm[\Hhalf]{\Dap\Ztbar} + \norm[2]{\pap\frac{1}{\Zapabs}}\norm[2]{\Ztapbar} + \norm[\infty]{\w}\norm[2]{\pap\frac{1}{\Zap}}\norm[2]{\Ztapbar} \\
& \lesssim \norm[\Hhalf]{\Dap\Ztbar} + \norm[2]{\pap\frac{1}{\Zap}}\norm[2]{\Ztapbar}
\end{align*}
We can control $\norm[\Hhalf]{\Dapbar\Ztbar}$ similarly and so the first inequality follows. To prove the second inequality, we use \propref{prop:LinftyHhalf} with $f = \Dap\Ztbar$ and $w = \frac{1}{\Zap}$
\begin{align*}
& \norm[\Linfty\cap\Hhalf]{\Dap\Ztbar}^2 \\
& \lesssim \norm[2]{\Ztapbar}\norm[2]{\Dap^2\Ztbar} + \norm[2]{\Ztapbar}^2\norm[2]{\pap\frac{1}{\Zap}}^2 \\
& \lesssim \norm*[\Big][2]{\pap\frac{1}{\Zap}}^2\norm[2]{\Ztapbar}^2 + \norm[2]{\Ztapbar}\norm*[\Big][2]{\pap\frac{1}{\Zap}}\norm[\infty]{\Dap\Ztbar} + \norm[2]{\Dt\pap\frac{1}{\Zap}}\norm[2]{\Ztapbar}
\end{align*}
Hence using $\dis ab\leq \frac{a^2}{2\ep} + \frac{\ep b^2}{2}$ for $\ep $ small on the second term we get the required estimate. 
\medskip

\medskip

\item $\dis \norm[2]{\Dapabs\bvarap} \lesssim  \norm*[\Big][2]{\pap\frac{1}{\Zap}} B(t) +  \norm[2]{\Dt\brac{\pap\frac{1}{\Zap}}}$
\medskip \\
Proof: Obvious from the previous estimate for $\norm[2]{\Dapabs(\bvarap - \Dap\Zt - \Dapbar\Ztbar)}$.

\medskip

\item We have
\begin{align}\label{eq:controlhardestimatel2forAg}
\begin{aligned}
\norm[2]{\frac{1}{\sqrt{\Ag}}\Dapabs\Ag} & \lesssim \norm[\Hhalf]{\Dapabs\Zt} + \norm[2]{\Ztapbar}\norm[2]{\pap\frac{1}{\Zapabs}} \\
 & \lesssim \norm[\Linfty\cap\Hhalf]{\Dap\Ztbar} + \norm[2]{\Ztapbar}\norm[2]{\pap\frac{1}{\Zap}} \\
 & \lesssim B(t)
\end{aligned}
\end{align}
Proof: Recall from \eqref{eq:Ag} and \eqref{eq:Azero}  that $\Ag = g + \Azero = g +  \frac{i}{2} \sqbrac{\Zt,\Ztbar;1}$. Hence from \propref{prop:tripleidentity} we obtain
\begin{align*}
& \Dapabs\Ag \\
& = \frac{i}{2}\cbrac{\sqbrac{\Dapabs\Zt,\Ztbar;1} + \sqbrac{\Zt,\Dapabs\Ztbar;1} + \sqbrac{\Zt,\Ztbar;\pap\frac{1}{\Zapabs}} -2 \sqbrac{\frac{1}{\Zapabs}, \Zt,\Ztbar;1}}
\end{align*}
\begin{enumerate}
\item We see that 
\begin{align*}
\abs{\sqbrac{\Dapabs\Zt,\Ztbar;1}(\ap)} \lesssim \norm[\Ltwo(\diff \bp)]{\frac{\Dapabs\Zt(\ap) - \Dapabs\Zt(\bp)}{\ap - \bp}}\norm[\Ltwo(\diff \bp)]{\frac{\Ztbar(\ap) - \Ztbar(\bp)}{\ap - \bp}}
\end{align*}
Therefore from \eqref{eq:Azero}, using $\Ag \geq \Azero$ and \propref{prop:Hardy}  we have
\begin{align*}
\norm[2]{\frac{1}{\sqrt{\Ag}}\sqbrac{\Dapabs\Zt,\Ztbar;1}}  \lesssim \norm[\Ltwo(\diff \bp \diff \ap)]{\frac{\Dapabs\Zt(\ap) - \Dapabs\Zt(\bp)}{\ap - \bp}} \lesssim \norm[\Hhalf]{\Dapabs\Zt}
\end{align*}
Similarly
\begin{align*}
\norm[2]{\frac{1}{\sqrt{\Ag}}\sqbrac{\Zt, \Dapabs\Ztbar;1}}  \lesssim \norm[\Hhalf]{\Dapabs\Zt}
\end{align*}
\item Observe that using \eqref{eq:Azero} and $\Ag \geq \Azero$ we obtain
\begin{align*}
& \abs{\sqbrac{\Zt,\Ztbar;\pap\frac{1}{\Zapabs}} (\ap)} \\ 
& \lesssim \norm[\Ltwo(\diff \bp)]{\frac{\Zt(\ap) - \Zt(\bp)}{\ap - \bp}} \norm[\Ltwo(\diff \bp)]{\brac{\frac{\Ztbar(\ap) - \Ztbar(\bp)}{\ap - \bp}}\brac{\pap\frac{1}{\Zapabs}}(\bp)} \\
& \lesssim \sqrt{\Ag(\ap)}\norm*[\bigg][2]{\pap\frac{1}{\Zapabs}}\norm[\Linfty(\diff \bp)]{\frac{\Ztbar(\ap) - \Ztbar(\bp)}{\ap - \bp}}
\end{align*}
Hence from \propref{prop:Hardy}
\begin{align*}
\norm[2]{\frac{1}{\sqrt{\Ag}}\sqbrac{\Zt,\Ztbar;\pap\frac{1}{\Zapabs}}} \lesssim \norm[2]{\Ztapbar}\norm[2]{\pap\frac{1}{\Zapabs}}
\end{align*}
\item We have from \eqref{eq:Azero}, $\Ag \geq \Azero$ and \propref{prop:Hardy}
\begin{align*}
& \abs{\sqbrac{\frac{1}{\Zapabs}, \Zt,\Ztbar;1}(\ap)} \\
& \lesssim  \norm[\Ltwo(\diff \bp)]{\frac{\Zt(\ap) - \Zt(\bp)}{\ap - \bp}} \norm[\Ltwo(\diff \bp)]{\brac{\frac{\frac{1}{\Zapabs}(\ap) - \frac{1}{\Zapabs}(\bp)}{\ap - \bp}} \brac{\frac{\Ztbar(\ap) - \Ztbar(\bp)}{\ap - \bp}}} \\
& \lesssim \sqrt{\Ag(\ap)} \norm[2]{\pap\frac{1}{\Zapabs}}\norm[\Linfty(\diff \bp)]{\frac{\Ztbar(\ap) - \Ztbar(\bp)}{\ap - \bp}}
\end{align*}
Hence from \propref{prop:Hardy} we have
\begin{align*}
\norm[2]{\frac{1}{\sqrt{\Ag}} \sqbrac{\frac{1}{\Zapabs}, \Zt,\Ztbar;1}}  \lesssim \norm[2]{\Ztapbar}\norm[2]{\pap\frac{1}{\Zapabs}}
\end{align*}
\end{enumerate}
Hence 
\begin{align*}
\norm[2]{\frac{1}{\sqrt{\Ag}}\Dapabs\Ag} & \lesssim \norm[\Hhalf]{\Dapabs\Zt} + \norm[2]{\Ztapbar}\norm[2]{\pap\frac{1}{\Zapabs}}
\end{align*}
The second estimate follows from this from estimates proved earlier for $\norm[2]{\pap\frac{1}{\Zapabs}}$ and $\norm[\Hhalf]{\Dapabs\Zt}$. 
\medskip

\medskip

\item We have
\begin{align*}
\norm[2]{\Dapabs\Ag} + \norm[2]{\Zttapbar}  \lesssim \brac{\norm[2]{\Ztapbar} + \sqrt{g}}B(t)
\end{align*}
Proof: The estimate for $\norm[2]{\Dapabs\Ag}$ follows directly from previous estimates. The estimate for $\norm[2]{\Zttapbar}$ then follows from \eqref{form:Zttbar}. 
\medskip

\medskip

\item For any $n \in \Zsp$ we have 
\begin{align*}
& \norm[\Hhalf]{\w^n\frac{\sqrt{\Ag}}{\Zapabs}\pap\frac{1}{\Zap}} + \norm[\Hhalf]{\w^n\frac{\sqrt{\Ag}}{\Zapabs}\pap\frac{1}{\Zapabs}} +  \norm[\Hhalf]{\w^n\frac{\sqrt{\Ag}}{\Zapabs^2}\pap \w}\\
& \lesssim_n  \norm[\Hhalf]{\frac{\sqrt{\Ag}}{\Zap}\pap\frac{1}{\Zap}} + \norm[2]{\pap\frac{1}{\Zap}}B(t)
\end{align*}
Proof: For the first estimate we use the weighted $\Hhalf$ estimate \propref{prop:Hhalfweight} with $\dis f = \pap\frac{1}{\Zap}$ weight $\dis w = \frac{\sqrt{\Ag}}{\Zap}$ and $\dis h = \w^{n + 1}$ to get
\begingroup
\allowdisplaybreaks
\begin{align*}
& \norm[\Hhalf]{\w^n\frac{\sqrt{\Ag}}{\Zapabs}\pap\frac{1}{\Zap}} \\
& \lesssim_n \norm[\Hhalf]{\frac{\sqrt{\Ag}}{\Zap}\pap\frac{1}{\Zap}} + \norm[2]{\pap\frac{1}{\Zap}}\norm[2]{\pap\brac{\frac{\sqrt{\Ag}}{\Zap}}} + \norm[2]{\pap\frac{1}{\Zap}}\norm[2]{\frac{\sqrt{\Ag}}{\Zapabs}\pap\w} \\
& \lesssim_n  \norm[\Hhalf]{\frac{\sqrt{\Ag}}{\Zap}\pap\frac{1}{\Zap}} + \norm[2]{\pap\frac{1}{\Zap}}^2\cbrac{\norm[2]{\Ztapbar} + \sqrt{g}} + \norm[2]{\pap\frac{1}{\Zap}}\norm[\Hhalf]{\Dap\Ztbar}
\end{align*}
\endgroup
Now from \eqref{form:RealImagTh} we have the relations
\begin{align*}
\Real\brac{\frac{\Zap}{\Zapabs}\pap\frac{1}{\Zap} } = \pap \frac{1}{\Zapabs} \quad \qquad \Imag\brac{\frac{\Zap}{\Zapabs}\pap\frac{1}{\Zap}} = i\brac{ \wbar\Dapabs \w} 
\end{align*}
Hence from the first estimate we see that 
\begin{align*}
& \norm[\Hhalf]{\frac{\sqrt{\Ag}}{\Zapabs}\pap\frac{1}{\Zapabs}} +  \norm[\Hhalf]{\wbar\frac{\sqrt{\Ag}}{\Zapabs^2}\pap \w} \\
& \lesssim  \norm[\Hhalf]{\frac{\sqrt{\Ag}}{\Zap}\pap\frac{1}{\Zap}} + \norm[2]{\pap\frac{1}{\Zap}}^2\cbrac{\norm[2]{\Ztapbar} + \sqrt{g}} + \norm[2]{\pap\frac{1}{\Zap}}\norm[\Hhalf]{\Dap\Ztbar}
\end{align*}
The other estimates follow similarly from \propref{prop:Hhalfweight}. 
\medskip

\medskip

\item For any $n \in \Zsp$ we have 
\begin{align*}
& \norm[\Hhalf]{\w^n\frac{\Ag}{\Zapabs}\pap\frac{1}{\Zap}} + \norm[\Hhalf]{\w^n\frac{\Ag}{\Zapabs}\pap\frac{1}{\Zapabs}} +  \norm[\Hhalf]{\w^n\frac{\Ag}{\Zapabs^2}\pap \w}  \\
& \lesssim_n   \cbrac{\norm[2]{\Ztapbar} + \sqrt{g}}\norm[\Hhalf]{\frac{\sqrt{\Ag}}{\Zap}\pap\frac{1}{\Zap}}  + B^2(t)
\end{align*}
Proof: Obvious from weighted $\Hhalf$ estimate \propref{prop:Hhalfweight} with $\dis f = \w^n \sqrt{\Ag}\pap\frac{1}{\Zap}$, $\dis f = \w^n \sqrt{\Ag}\pap\frac{1}{\Zapabs}$ or $\dis f = \w^n \frac{\sqrt{\Ag}}{\Zapabs}\pap\w$ and weight $\dis w = \frac{1}{\Zapabs}$ and $\dis h = \sqrt{\Ag}$.
\medskip

\medskip

\item We have
\begin{align*}
 \norm[\Linfty\cap\Hhalf]{\frac{1}{\Zapabs^2}\pap\Ag}  \lesssim B^2(t)  + \brac{\norm[2]{\Ztapbar} + \sqrt{g} }\norm[2]{\Dt\brac{\pap\frac{1}{\Zap}}}
\end{align*}
Proof: From estimate 12 of section 5.1 of \cite{Ag21}, we have the estimate
\begin{align*}
\norm[\Linfty\cap\Hhalf]{\frac{\wbar^2}{\Zapabs}(\Id - \Hil)\Dapabs\Ag} & \lesssim \norm[2]{\Ztapbar}\brac{ \norm[2]{\pap\frac{\Ztapbar}{\Zap^2}} + \norm[2]{\frac{1}{\Zap^2}\pap\Ztapbar} } \\
& \quad + \norm[2]{\pap\frac{1}{\Zap}}\norm[2]{\Dapabs\Ag}
\end{align*}
Now $\Real\cbrac{(\Id - \Hil)\Dapabs\Ag} = \Dapabs\Ag$ as $\Dapabs\Ag$ is real valued. Therefore the required $\Linfty$ estimate now follows easily. To get the $\Hhalf$ estimate, we use \propref{prop:Hhalfweight} with $f = (\Id - \Hil)\Dapabs\Ag$, $w = \frac{\wbar^2}{\Zapabs}$ and $h = \w^2$ to get
\begin{align*}
& \lesssim \norm[\Hhalf]{\frac{1}{\Zapabs}(\Id - \Hil)\Dapabs\Ag} \\
& \lesssim \norm[\Hhalf]{\frac{\wbar^2}{\Zapabs}(\Id - \Hil)\Dapabs\Ag} + \norm[2]{(\Id - \Hil)\Dapabs\Ag}\norm[2]{\pap\frac{1}{\Zap}}
\end{align*}
The required $\Hhalf$ estimate now follows from the fact that $\Dapabs\Ag$ is real valued and that the Hilbert transform is bounded on $\Ltwo$.

\medskip

\item Hence we have the estimate
\begin{align*}
\norm*[\bigg][2]{\Dapabs\brac*[\bigg]{\frac{1}{\Zapabs^2} \pap\Ag}}  \lesssim B^2(t)\norm[2]{\pap\frac{1}{\Zap}}  + B(t)\norm[2]{\Dt\brac{\pap\frac{1}{\Zap}}}
\end{align*}
Proof: From estimate 13 of section 5.1 of \cite{Ag21} we have
\begin{align*}
 \norm*[\bigg][2]{\Dapabs\brac*[\bigg]{\frac{1}{\Zapabs^2} \pap\Ag}} & \lesssim \norm[2]{\frac{1}{\Zap^2}\pap\Ztapbar}\brac{\norm[\infty]{\Dap\Zt} + \norm[\infty]{\pap\Pa\brac{\frac{\Zt}{\Zap}}}} \\
 & \quad + \norm[2]{\pap\frac{\Ztap}{\Zap^2}}\norm[\infty]{\Dap\Ztbar} + \norm[2]{\pap\frac{1}{\Zap}}\norm[\infty]{\frac{1}{\Zapabs^2}\pap\Ag}
\end{align*}
The required estimate now follows from previous estimates. 

\bigskip

\item We have
\begin{align}\label{eq:controlhardDtAgLinfty}
\norm[\infty]{\frac{\Dt\Ag}{\Ag}} \lesssim \norm[\infty]{\Dap\Ztbar} + \norm[2]{\Ztapbar}\norm[2]{\pap\frac{1}{\Zap}} \lesssim B(t)
\end{align}
Proof: From \lemref{lem:DtAgoverAg} we have that
\begin{align*}
 \frac{\Dt \Ag}{\Ag} & = 2\Real\cbrac{\sqbrac{\Zt,\frac{1}{\Zap}; 1}(\ap)} + \frac{1}{\Ag}\Imag\cbrac{2i\sqbrac{\Zt,\Ag; \frac{1}{\Zap}}(\ap) + \sqbrac{\Zt,\Zt; \Dap\Ztbar}(\ap) } \\
& \quad + 2\Real(\Dap\Zt) - \bvarap
\end{align*}
Observe from \eqref{eq:Ag}, \eqref{eq:Azero} that
\begin{align*}
\abs{\sqbrac{\Zt,\Zt; \Dap\Ztbar}(\ap)}  \lesssim \norm[\infty]{\Dap\Ztbar}\Ag(\ap)
\end{align*}
and also from \propref{prop:Hardy} we get
\begin{align*}
\norm[\infty]{\sqbrac{\Zt,\frac{1}{\Zap}; 1}} + \norm[\infty]{\bvarap} \lesssim \norm[\infty]{\Dap\Ztbar} + \norm[2]{\Ztapbar}\norm[2]{\pap\frac{1}{\Zap}}
\end{align*}
We now observe from \eqref{eq:Ag}, \eqref{eq:Azero} that
\begin{align*}
\abs{\sqbrac{\Zt,\Ag;\frac{1}{\Zap}}(\ap)} \lesssim \sqrt{\Ag(\ap)} \norm[\Ltwo(\diff \bp)]{\frac{\Ag(\ap) - \Ag(\bp)}{\ap - \bp}\frac{1}{\Zap}(\bp)}
\end{align*}
Hence it is enough to show that for all $\ap \in \Rsp$ we have
\begin{align}\label{eq:Agfrac}
\begin{split}
& \norm[\Ltwo(\diff \bp)]{\frac{\Ag(\ap) - \Ag(\bp)}{\ap - \bp}\frac{1}{\Zap}(\bp)} \\
& \lesssim \brac{\norm[\infty]{\Dap\Ztbar} + \norm[2]{\Ztapbar}\norm[2]{\pap\frac{1}{\Zap}}}\sqrt{\Ag(\ap)}
\end{split}
\end{align}
Now from \eqref{eq:Ag}, \eqref{eq:Azero} and the fact that $\Hil(\Ztapbar) = \Ztapbar$ we get
\begin{align}
\begin{aligned}\label{eq:Agfractemp1}
& \frac{\Ag(\ap) - \Ag(\bp)}{\ap-\bp} \\
& = -\frac{1}{\ap-\bp}\Imag\cbrac{\frac{1}{i\pi} \int \sqbrac{\frac{\Zt(\ap)-\Zt(s)}{\ap- s} - \frac{\Zt(\bp)-\Zt(s)}{\bp- s} } \Ztbarbp(s) \diff s  } \\ 
& = -\frac{1}{\ap-\bp}\Imag\cbrac{\frac{1}{i\pi} \int \brac{\Zt(\ap) - \Zt(s)}\brac{\frac{1}{\ap-s} - \frac{1}{\bp-s}} \Ztbarbp(s) \diff s  } \\
& \quad -\frac{1}{\ap-\bp}\Imag\cbrac{\frac{1}{i\pi} \int \brac{\frac{\Zt(\ap)-\Zt(\bp)}{\bp-s}} \Ztbarbp(s) \diff s  } \\
& = \Imag\cbrac{\frac{1}{i\pi} \int \frac{1}{\bp-s}\frac{\Zt(\ap)-\Zt(s)}{\ap-s}\Ztbpbar(s) \diff s - \frac{\Zt(\ap)-\Zt(\bp)}{\ap-\bp} \Ztbpbar(\bp)} 
\end{aligned}
\end{align}
Define 
\begin{align*}
\Zt(\ap,\bp) = \frac{\Zt(\ap)-\Zt(\bp)}{\ap-\bp}
\end{align*}
Hence for fixed $\ap$ we have
\begin{align*}
& \frac{\Ag(\ap) - \Ag(\bp)}{\ap-\bp} \\
& = \Imag\cbrac{\Hil\brac{\Zt(\ap,\bp)\Ztapbar(\bp)} - \Zt(\ap,\bp)\Ztapbar(\bp) } \\
& = -\Imag\cbrac{(\Id - \Hil)\brac{\Zt(\ap,\bp)\Ztapbar(\bp)}}
\end{align*}
where we consider the functions as functions of $\bp$ and $\ap$ is some fixed number.  Now observe that
\begin{align*}
& \abs{\frac{1}{\Zap}(\bp) \brac{\frac{\Ag(\ap) - \Ag(\bp)}{\ap - \bp}}} \\
& \lesssim \abs{\sqbrac{\frac{1}{\Zap}(\bp),\Hil}\Zt(\ap,\bp)\Ztapbar(\bp)} + \abs{(\Id - \Hil)\brac{\Zt(\ap,\bp)\Dap\Ztbar(\bp)}}
\end{align*}
Now we have from \propref{prop:commutator} and \eqref{eq:Ag}, \eqref{eq:Azero}
\begin{align*}
& \norm[\Ltwo(\diff \bp)]{\sqbrac{\frac{1}{\Zap}(\bp),\Hil}\Zt(\ap,\bp)\Ztapbar(\bp)} \\
& \lesssim \norm[2]{\pap\frac{1}{\Zap}}\norm[\Lone(\diff \bp)]{\Zt(\ap,\bp)\Ztapbar(\bp)} \\
& \lesssim  \norm[2]{\pap\frac{1}{\Zap}}\norm[\Ltwo(\diff \bp)]{\Zt(\ap,\bp)}\norm[2]{\Ztapbar} \\
& \lesssim \norm[2]{\Ztapbar}\norm[2]{\pap\frac{1}{\Zap}}\sqrt{\Ag(\ap)}
\end{align*}
and also from \eqref{eq:Ag}, \eqref{eq:Azero}
\begin{align*}
& \norm[\Ltwo(\diff \bp)]{(\Id - \Hil)\brac{\Zt(\ap,\bp)\Dap\Ztbar(\bp)}} \\
& \lesssim  \norm[\Ltwo(\diff \bp)]{\Zt(\ap,\bp)\Dap\Ztbar(\bp)} \\
& \lesssim  \norm[\infty]{\Dap\Ztbar}\norm[\Ltwo(\diff \bp)]{\Zt(\ap,\bp)} \\
& \lesssim \norm[\infty]{\Dap\Ztbar}\sqrt{\Ag(\ap)}
\end{align*}
Hence \eqref{eq:Agfrac} is shown thereby proving the main estimate.
\medskip

\medskip

\item We have
\begin{align}\label{eq:controlhardDapabsDtAgL2}
\begin{aligned}
& \norm[2]{\frac{\Dapabs\Dt\Ag}{\sqrt{\Ag}}} \\
& \lesssim B^2(t)  + \brac{\norm[2]{\Ztapbar} + \sqrt{g}}\brac{\norm[2]{\Dt\brac{\pap\frac{1}{\Zap}}} + \norm[\Hhalf]{\frac{\sqrt{\Ag}}{\Zap}\pap\frac{1}{\Zap}}}
\end{aligned}
\end{align}
Proof: From \lemref{lem:DtAgoverAg} we see that
\begin{align*}
\Dt\Ag = \Imag\cbrac{-2\sqbrac{\Zt,\Zttbar;1} + \sqbrac{\Zt,\Zt;\Dap\Ztbar}} + \Ag\cbrac{ 2\Real(\Dap\Zt) - \bvarap}
\end{align*}
Hence we have
\begin{align*}
\frac{\Dapabs\Dt\Ag}{\sqrt{\Ag}} & = \frac{1}{\sqrt{\Ag}}\Dapabs\Imag\cbrac{-2\sqbrac{\Zt,\Zttbar;1} + \sqbrac{\Zt,\Zt;\Dap\Ztbar}} \\
& \quad + \brac{\frac{1}{\sqrt{\Ag}}\Dapabs\Ag}\cbrac{ 2\Real(\Dap\Zt) - \bvarap} + \sqrt{\Ag}\Dapabs\cbrac{ 2\Real(\Dap\Zt) - \bvarap}
\end{align*}
Using previous estimates it is easy to see that
\begin{align*}
& \norm[2]{\brac{\frac{1}{\sqrt{\Ag}}\Dapabs\Ag}\cbrac{ 2\Real(\Dap\Zt) - \bvarap}} + \norm[2]{\sqrt{\Ag}\Dapabs\cbrac{ 2\Real(\Dap\Zt) - \bvarap}} \\
& \lesssim  \brac{ \norm[\Linfty\cap\Hhalf]{\Dap\Ztbar} + \brac{\norm[2]{\Ztapbar} + \sqrt{g}}\norm[2]{\pap\frac{1}{\Zap}}}^2 + \brac{\norm[2]{\Ztapbar} + \sqrt{g}}\norm[2]{\Dt\brac{\pap\frac{1}{\Zap}}}
\end{align*}
Now let us control the other terms. 
\begin{enumerate}
\item From \propref{prop:tripleidentity} we see that
\begin{align*}
\Dapabs\sqbrac{\Zt,\Zt;\Dap\Ztbar}  & = 2\sqbrac{\Dapabs\Zt, \Zt; \Dap\Ztbar} + \sqbrac{\Zt, \Zt; \pap\brac*[\bigg]{\frac{1}{\Zapabs}\Dap\Ztbar}} \\
& \quad - 2\sqbrac{\frac{1}{\Zapabs}, \Zt,\Zt; \Dap\Ztbar}
\end{align*}
Now using \eqref{eq:Ag} and \eqref{eq:Azero} we have
\begin{align*}
& \abs{\sqbrac{\Dapabs\Zt, \Zt; \Dap\Ztbar}(\ap)}\\
& \lesssim \sqrt{\Ag(\ap)}\norm[\Ltwo(\diff \bp)]{\frac{\Dapabs\Zt(\ap) - \Dapabs\Zt(\bp)}{\ap - \bp} \Dap\Ztbar(\bp)}
\end{align*}
Hence from \propref{prop:Hardy}
\begin{align*}
\norm[2]{\frac{1}{\sqrt{\Ag}}\sqbrac{\Dapabs\Zt, \Zt; \Dap\Ztbar}} \lesssim \norm[\infty]{\Dap\Ztbar}\norm[\Hhalf]{\Dapabs\Zt}
\end{align*}
We also see using \eqref{eq:Ag}, \eqref{eq:Azero} that
\begin{align*}
& \abs{\sqbrac{\Zt, \Zt; \pap\brac*[\bigg]{\frac{1}{\Zapabs}\Dap\Ztbar}} (\ap)} \\
& \lesssim \sqrt{\Ag(\ap)}\norm[\Ltwo(\diff \bp)]{\frac{\Zt(\ap) - \Zt(\bp)}{\ap - \bp}\pbp\brac{\frac{1}{\Zapabs}\Dap\Ztbar}(\bp)}
\end{align*}
Hence using \propref{prop:Hardy}
\begin{align*}
& \norm[2]{\frac{1}{\sqrt{\Ag}}\sqbrac{\Zt, \Zt; \pap\brac*[\bigg]{\frac{1}{\Zapabs}\Dap\Ztbar}}} \\
& \lesssim \norm[2]{\Ztapbar}\norm[2]{\pap\brac{\frac{1}{\Zapabs}\Dap\Ztbar}} \\
& \lesssim \brac{\norm[\Linfty\cap\Hhalf]{\Dap\Ztbar}  +  \brac{\norm[2]{\Ztapbar} + \sqrt{g}}\norm[2]{\pap\frac{1}{\Zap}}}^2 + \norm[2]{\Ztapbar}\norm[2]{\Dt\brac{\pap\frac{1}{\Zap}}}
\end{align*}
Finally we consider the last term and observe using \eqref{eq:Ag}, \eqref{eq:Azero}
\begin{align*}
& \abs{\sqbrac{\frac{1}{\Zapabs}, \Zt,\Zt; \Dap\Ztbar} (\ap)} \\
& \lesssim \sqrt{\Ag(\ap)}\norm[\Ltwo(\diff \bp)]{\frac{\frac{1}{\Zapabs}(\ap) - \frac{1}{\Zapabs}(\bp)}{\ap - \bp} \frac{\Zt(\ap) - \Zt(\bp)}{\ap - \bp} \Dap\Ztbar(\bp)}
\end{align*}
Hence using \propref{prop:Hardy} we obtain
\begin{align*}
\norm[2]{\frac{1}{\sqrt{\Ag}} \sqbrac{\frac{1}{\Zapabs}, \Zt,\Zt; \Dap\Ztbar} } \lesssim \norm[\infty]{\Dap\Ztbar}\norm[2]{\pap\frac{1}{\Zapabs}}\norm[2]{\Ztapbar}
\end{align*}

\item From \propref{prop:tripleidentity} we see that
\begin{align*}
\Dapabs\sqbrac{\Zt,\Zttbar;1} & = \sqbrac{\Dapabs\Zt, \Zttbar; 1} + \sqbrac{\Zt,\Dapabs\Zttbar; 1} + \sqbrac{\Zt,\Zttbar; \pap\frac{1}{\Zapabs}} \\
& \quad - 2\sqbrac{\frac{1}{\Zapabs}, \Zt,\Zttbar ;1 }
\end{align*}
We have from \eqref{eq:Zttbarfrac}, \propref{prop:Hardy} and \eqref{eq:Agfrac}
\begin{align*}
& \abs{\sqbrac{\Dapabs\Zt, \Zttbar; 1}(\ap)} \\
& \lesssim \norm[\Ltwo(\diff \bp)]{\frac{\Dapabs\Zt(\ap) - \Dapabs\Zt(\bp)}{\ap - \bp}}\norm[\Ltwo(\diff \bp)]{\frac{\Zttbar(\ap) - \Zttbar(\bp)}{\ap - \bp}} \\
& \lesssim \sqrt{\Ag(\ap)}\brac{\norm[\Linfty\cap\Hhalf]{\Dap\Ztbar} + \brac{\norm[2]{\Ztapbar} + \sqrt{g}}\norm[2]{\pap\frac{1}{\Zap}}} \norm[\Ltwo(\diff \bp)]{\frac{\Dapabs\Zt(\ap) - \Dapabs\Zt(\bp)}{\ap - \bp}}
\end{align*}
Therefore from \propref{prop:Hardy}
\begin{align*}
\norm[2]{\frac{1}{\sqrt{\Ag}}\sqbrac{\Dapabs\Zt, \Zttbar; 1} } \lesssim \brac{\norm[\Linfty\cap\Hhalf]{\Dap\Ztbar} + \norm[2]{\Ztapbar}\norm[2]{\pap\frac{1}{\Zap}}} \norm[\Hhalf]{\Dapabs\Zt}
\end{align*}
Now we see from \eqref{form:Zttbar}
\begin{align*}
\Dapabs\Zttbar = - i \Ag\Dapabs\frac{1}{\Zap} - \frac{i}{\Zap}\Dapabs\Ag
\end{align*}
Hence
\begin{align*}
 \sqbrac{\Zt,\Dapabs\Zttbar; 1} = -i\sqbrac{\Zt,\Ag\Dapabs\frac{1}{\Zap}; 1} - i\sqbrac{\Zt,\frac{1}{\Zap}\Dapabs\Ag; 1}
\end{align*}
Now from \eqref{eq:Ag}, \eqref{eq:Azero}
\begin{align*}
\abs{\sqbrac{\Zt,\Ag\Dapabs\frac{1}{\Zap}; 1}(\ap)} \lesssim \sqrt{\Ag(\ap)}\norm[\Ltwo(\diff \bp)]{\frac{\Ag\Dapabs\frac{1}{\Zap}(\ap) - \Ag\Dapabs\frac{1}{\Zap}(\bp)}{\ap - \bp}}
\end{align*}
Therefore from \propref{prop:Hardy}
\begin{align*}
& \norm[2]{\frac{1}{\sqrt{\Ag}}\sqbrac{\Zt,\Ag\Dapabs\frac{1}{\Zap}; 1}} \\
& \lesssim \norm[\Hhalf]{\Ag\Dapabs\frac{1}{\Zap}} \\
& \lesssim \brac{\norm[2]{\Ztapbar} + \sqrt{g}}\norm[\Hhalf]{\frac{\sqrt{\Ag}}{\Zap}\pap\frac{1}{\Zap}} + \brac{ \norm[\Linfty\cap\Hhalf]{\Dap\Ztbar} + \brac{\norm[2]{\Ztapbar} + \sqrt{g}}\norm[2]{\pap\frac{1}{\Zap}} }^2
\end{align*}

Now we also have from \eqref{eq:Ag}, \eqref{eq:Azero}
\begin{align*}
\abs{\sqbrac{\Zt,\frac{1}{\Zap}\Dapabs\Ag; 1}(\ap)} \lesssim \sqrt{\Ag(\ap)}\norm[\Ltwo(\diff \bp)]{\frac{\frac{1}{\Zap}\Dapabs\Ag(\ap) - \frac{1}{\Zap}\Dapabs\Ag(\bp) }{\ap - \bp}}
\end{align*}
Hence by \propref{prop:Hardy}
\begin{align*}
\norm[2]{\frac{1}{\sqrt{\Ag}} \sqbrac{\Zt,\frac{1}{\Zap}\Dapabs\Ag; 1}} & \lesssim \norm[\Hhalf]{\frac{\wbar}{\Zapabs^2}\pap\Ag} 
\end{align*}
Now using \propref{prop:Hhalfweight} with $\dis f = \frac{1}{\sqrt{\Ag}}\Dapabs\Ag$, $\dis w = \frac{\sqrt{\Ag}}{\Zapabs}$ and $\dis h = \wbar$, we get
\begin{align*}
& \norm[2]{\frac{1}{\sqrt{\Ag}} \sqbrac{\Zt,\frac{1}{\Zap}\Dapabs\Ag; 1}} \\
& \lesssim  \norm[\Hhalf]{\frac{1}{\Zapabs^2}\pap\Ag} + \brac{ \norm[\Linfty\cap\Hhalf]{\Dap\Ztbar} + \brac{\norm[2]{\Ztapbar} + \sqrt{g}}\norm[2]{\pap\frac{1}{\Zap}} }^2 \\
& \lesssim \brac{\norm[2]{\Ztapbar} + \sqrt{g} }\norm[2]{\Dt\brac{\pap\frac{1}{\Zap}}} + \brac{ \norm[\Linfty\cap\Hhalf]{\Dap\Ztbar} + \brac{\norm[2]{\Ztapbar} + \sqrt{g}}\norm[2]{\pap\frac{1}{\Zap}} }^2
\end{align*}
Let us now come back to the third term. We see from \eqref{eq:Ag}, \eqref{eq:Azero} that
\begin{align*}
\abs{\sqbrac{\Zt,\Zttbar; \pap\frac{1}{\Zapabs}}(\ap)} \lesssim \sqrt{\Ag(\ap)}\norm[\Ltwo(\diff \bp)]{\frac{\Zttbar(\ap) - \Zttbar(\bp)}{\ap - \bp}\brac{\pap\frac{1}{\Zapabs}}(\bp) }
\end{align*}
Hence from \propref{prop:Hardy}
\begin{align*}
\norm[2]{\frac{1}{\sqrt{\Ag}} \sqbrac{\Zt,\Zttbar; \pap\frac{1}{\Zapabs}} } & \lesssim \norm[2]{\pap\frac{1}{\Zapabs}}\norm[2]{\Zttapbar} \\
& \lesssim \brac{ \norm[\Linfty\cap\Hhalf]{\Dap\Ztbar} + \brac{\norm[2]{\Ztapbar} + \sqrt{g}}\norm[2]{\pap\frac{1}{\Zap}} }^2
\end{align*}
Finally let us control the last term. We see from \eqref{eq:Ag}, \eqref{eq:Azero} that
\begin{align*}
\abs{\sqbrac{\frac{1}{\Zapabs}, \Zt,\Zttbar ;1 }(\ap)} \lesssim \sqrt{\Ag(\ap)}\norm[2]{\Zttapbar}\norm[\Linfty(\diff \bp)]{\frac{\frac{1}{\Zapabs}(\ap) - \frac{1}{\Zapabs}(\bp)}{\ap - \bp}}
\end{align*}
Hence from \propref{prop:Hardy}
\begin{align*}
\norm[2]{\frac{1}{\sqrt{\Ag}} \sqbrac{\frac{1}{\Zapabs}, \Zt,\Zttbar ;1 }} & \lesssim \norm[2]{\Zttapbar}\norm[2]{\pap\frac{1}{\Zapabs}} \\
& \lesssim \brac{ \norm[\Linfty\cap\Hhalf]{\Dap\Ztbar} + \brac{\norm[2]{\Ztapbar} + \sqrt{g}}\norm[2]{\pap\frac{1}{\Zap}} }^2
\end{align*}
\end{enumerate}
Hence proved. 
\medskip

\medskip

\item We have
\begin{align*}
& \norm[\Linfty\cap\Hhalf]{\Jzero} \\
& = \norm[\Linfty\cap\Hhalf]{ \Dt(\bvarap - \Dap\Zt - \Dapbar\Ztbar)} \\
& \lesssim  B^2(t) +  \brac{\norm[2]{\Ztapbar} + \sqrt{g}}\norm[2]{\Dt\brac{\pap\frac{1}{\Zap}}}
\end{align*}
and
\begin{align*}
& \norm[\Hhalf]{\Dt\bvarap} + \norm[\Hhalf]{\pap\Dt\bvar} \\
& \lesssim  B^2(t)  +  \brac{\norm[2]{\Ztapbar} + \sqrt{g}}\brac{\norm[2]{\Dt\brac{\pap\frac{1}{\Zap}}} + \norm[\Hhalf]{\frac{\sqrt{\Ag}}{\Zap}\pap\frac{1}{\Zap}}}
\end{align*}
Proof: From estimate 20 of section 5.1 of \cite{Ag21} we have
\begin{align*}
\norm[\Linfty\cap\Hhalf]{ \Dt(\bvarap - \Dap\Zt - \Dapbar\Ztbar)} & \lesssim \norm[2]{\Zttapbar}\norm[2]{\pap\frac{1}{\Zap}} + \norm[2]{\Ztapbar}\norm[2]{\pap\Dt\frac{1}{\Zap}} \\
& \quad + \norm[\infty]{\bvarap}\norm[2]{\Ztapbar}\norm[2]{\pap\frac{1}{\Zap}}
\end{align*}
The first estimate now follows easily from this. Now using \propref{prop:Leibniz} we have
\begin{align*}
& \norm[\Hhalf]{\Dt\bvarap} + \norm[\Hhalf]{\pap\Dt\bvar}  \\
& \lesssim \norm[\Hhalf]{\Dt\bvarap} + \norm[\Linfty\cap\Hhalf]{\bvarap}^2 \\
& \lesssim \norm[\Hhalf]{\Dt(\bvarap - \Dap\Zt - \Dapbar\Ztbar)} + \norm[\Hhalf]{\Dt\Dapbar\Ztbar} + \norm[\Linfty\cap\Hhalf]{\bvarap}^2 \\
&  \lesssim \norm[\Hhalf]{\Dt(\bvarap - \Dap\Zt - \Dapbar\Ztbar)} + \norm[\Linfty\cap\Hhalf]{\Dapbar\Ztbar}^2  + \norm[\Hhalf]{\Dapbar\Zttbar} + \norm[\Linfty\cap\Hhalf]{\bvarap}^2
\end{align*}
Hence we only need to control $\norm[\Hhalf]{\Dapbar\Zttbar}$. Applying $\Dapbar$ to \eqref{form:Zttbar} we get
\begin{align*}
\norm[\Hhalf]{\Dapbar\Zttbar} \lesssim \norm[\Hhalf]{\frac{1}{\Zapabs^2}\pap\Ag} + \norm[\Hhalf]{\w\frac{\Ag}{\Zapabs}\pap\frac{1}{\Zap}}
\end{align*}
The required estimate now follows from previously proved estimates. 

\medskip

\medskip

\item We have
\begin{align*}
 \norm[\Linfty\cap\Hhalf]{\Qzero}  \lesssim B^2(t) + \brac{\norm[2]{\Ztapbar} + \sqrt{g}}\norm[2]{\Dt\brac{\pap\frac{1}{\Zap}}}
\end{align*}
and
\begin{align*}
 \norm[2]{\Dapabs\Qzero}  \lesssim B^2(t)\norm[2]{\pap\frac{1}{\Zap}} + B(t)\norm[2]{\Dt\brac{\pap\frac{1}{\Zap}}}
\end{align*}
Proof: Recall from \eqref{form:Qzero} that 
\begin{align*}
\Qzero = (\bvarap-\Dap\Zt)^2 - (\Dapbar\Ztbar)^2 -\frac{i}{\Zapabs^2}\pap\Ag
\end{align*}
The estimates now follow easily from previous estimates. 

\medskip

\item We have
\begin{align*}
\norm[2]{\Rzero}  \lesssim B^2(t)\norm[2]{\pap\frac{1}{\Zap}}  + B(t)\brac{\norm[2]{\Dt\brac{\pap\frac{1}{\Zap}}}  + \norm[\Hhalf]{\frac{\sqrt{\Ag}}{\Zap}\nobrac{\pap\frac{1}{\Zap}}}}
\end{align*}
\medskip\\
Proof: Recall from \eqref{form:Rzero} that 
\begin{align*}
\begin{split}
\Rzero & = \brac{\pap\frac{1}{\Zap}}(\Jzero+\Qzero) + \Dap\Qzero -  \bvarap\brac{\pap\Dt\frac{1}{\Zap} + \Dt\pap\frac{1}{\Zap}} \\
& \quad - (\Dt\bvarap)\brac{\pap\frac{1}{\Zap}} -2i\Ag\brac{\Dapabs\frac{1}{\Zapabs}}\brac{\pap\frac{1}{\Zap}} \\
& \quad - i\brac{\frac{1}{\Zapabs^2}\pap\Ag}\brac{\pap\frac{1}{\Zap}}
\end{split}
\end{align*}
\begin{enumerate}

\item We have
\begin{align*}
&   \norm[2]{\pap\frac{1}{\Zap}}\brac{\norm[\infty]{\Jzero} + \norm[\infty]{\Qzero}}  + \norm[2]{\Dapabs\Qzero} +  \norm[\infty]{\bvarap}\brac{\norm[2]{\pap\Dt\frac{1}{\Zap}} + \norm[2]{\Dt\pap\frac{1}{\Zap}}} \\
& \lesssim B^2(t)\norm[2]{\pap\frac{1}{\Zap}}  + B(t)\norm[2]{\Dt\brac{\pap\frac{1}{\Zap}}}
\end{align*}
\item We see from \eqref{form:Jzero}, \eqref{form:Zttbar} and the fact that $\Ag$ is real valued
\begin{align*}
\Dt\bvarap & = \Jzero + \Dt\Dap\Zt + \Dt\Dapbar\Ztbar \\
& = \Jzero - (\Dap\Zt)^2 - (\Dapbar\Ztbar)^2 + 2\Real\brac{\Dapbar\Zttbar} \\
& = \Jzero - (\Dap\Zt)^2 - (\Dapbar\Ztbar)^2 + 2\Real\brac{ \frac{-i}{\Zapabs^2}\pap\Ag -i \frac{\w\Ag}{\Zapabs}\pap\frac{1}{\Zap} } \\
& = \Jzero - (\Dap\Zt)^2 - (\Dapbar\Ztbar)^2 + 2\Real\brac{ -i \frac{\w\Ag}{\Zapabs}\pap\frac{1}{\Zap} }
\end{align*}
Now using the second estimate of \propref{prop:Hhalfweight} with $\dis f = g = \pap\frac{1}{\Zap}$ and $\dis w = \frac{\sqrt{\Ag}}{\Zapabs}$ we get
\begin{align*}
& \norm[2]{\frac{\sqrt{\Ag}}{\Zapabs}\brac{\pap\frac{1}{\Zap}}\brac{\pap\frac{1}{\Zap}}} \\
& \lesssim \norm[\Hhalf]{\frac{\sqrt{\Ag}}{\Zapabs}\brac{\pap\frac{1}{\Zap}}}\norm[2]{\pap\frac{1}{\Zap}} \\
& \quad + \norm[2]{\pap\frac{1}{\Zap}}^2\brac{\norm[\Linfty\cap\Hhalf]{\Dap\Ztbar} + \brac{\norm[2]{\Ztapbar} + \sqrt{g}}\norm[2]{\pap\frac{1}{\Zap}} }
\end{align*}

Hence 
\begin{align*}
& \norm[2]{(\Dt\bvarap)\brac{\pap\frac{1}{\Zap}}} \\
& \lesssim B^2(t)\norm[2]{\pap\frac{1}{\Zap}}  + B(t)\brac{\norm[2]{\Dt\brac{\pap\frac{1}{\Zap}}}  + \norm[\Hhalf]{\frac{\sqrt{\Ag}}{\Zap}\nobrac{\pap\frac{1}{\Zap}}}}
\end{align*}
\end{enumerate}
The last two terms of $\Rzero$ are controlled similarly and hence the estimate follows. 
\medskip

\medskip

\item We have
\begin{align*}
& \norm[2]{(\Id - \Hil)\Dt^2\brac{\pap\frac{1}{\Zap}}} \\
& \lesssim B^2(t)\norm[2]{\pap\frac{1}{\Zap}}  + B(t)\brac{\norm[2]{\Dt\brac{\pap\frac{1}{\Zap}}}  + \norm[\Hhalf]{\frac{\sqrt{\Ag}}{\Zap}\nobrac{\pap\frac{1}{\Zap}}}}
\end{align*}
Proof: For a function $f$ satisfying $\Pa f = 0$ we have from \propref{prop:tripleidentity}
\begin{align}\label{eq:IminusHDt2f}
\begin{split}
(\Id - \Hil)\Dt^2f & = \sqbrac{\Dt,\Hil}\Dt f + \Dt\sqbrac{\Dt,\Hil}f \\
& = \sqbrac{\bvar,\Hil}\pap\Dt f + \Dt\sqbrac{\bvar,\Hil}\pap f \\
& = 2\sqbrac{\bvar,\Hil}\pap\Dt f + \sqbrac{\Dt\bvar,\Hil}\pap f - \sqbrac{\bvar, \bvar ; \pap f}
\end{split}
\end{align}
Hence we have from \propref{prop:commutator} and \propref{prop:triple}
\begin{align*}
& \norm[2]{(\Id - \Hil)\Dt^2\brac{\pap\frac{1}{\Zap}}} \\
&  \lesssim \norm[\Hhalf]{\bvarap}\norm[2]{\Dt\pap\frac{1}{\Zap}} + \norm[\Hhalf]{\pap\Dt\bvar}\norm[2]{\pap\frac{1}{\Zap}} + \norm[\infty]{\bvarap}^2\norm[2]{\pap\frac{1}{\Zap}} \\
& \lesssim B^2(t)\norm[2]{\pap\frac{1}{\Zap}}  + B(t)\brac{\norm[2]{\Dt\brac{\pap\frac{1}{\Zap}}}  + \norm[\Hhalf]{\frac{\sqrt{\Ag}}{\Zap}\nobrac{\pap\frac{1}{\Zap}}}}
\end{align*}

\medskip

\item We have
\begin{align*}
& \norm[2]{(\Id - \Hil)\cbrac{i\frac{\Ag}{\Zapabs^2}\pap\brac{\pap\frac{1}{\Zap}}}} \\
& \lesssim B^2(t)\norm[2]{\pap\frac{1}{\Zap}} + B(t)\brac{\norm[2]{\Dt\brac{\pap\frac{1}{\Zap}}}  + \norm[\Hhalf]{\frac{\sqrt{\Ag}}{\Zap}\nobrac{\pap\frac{1}{\Zap}}}}
\end{align*}
Proof: Observe that 
\begin{align*}
(\Id - \Hil)\cbrac{i\frac{\Ag}{\Zapabs^2}\pap\brac{\pap\frac{1}{\Zap}}} = i\sqbrac{\frac{\Ag}{\Zapabs^2},\Hil}\pap\brac{\pap\frac{1}{\Zap}}
\end{align*}
Hence we have from \propref{prop:commutator}
\begin{align*}
& \norm[2]{(\Id - \Hil)\cbrac{i\frac{\Ag}{\Zapabs^2}\pap\brac{\pap\frac{1}{\Zap}}}} \\
& \lesssim \brac{\norm[\Hhalf]{\frac{1}{\Zapabs^2}\pap\Ag} + \norm[\Hhalf]{\frac{\Ag}{\Zapabs}\pap\frac{1}{\Zapabs}} }\norm[2]{\pap\frac{1}{\Zap}} \\
& \lesssim B^2(t)\norm[2]{\pap\frac{1}{\Zap}} + B(t)\brac{\norm[2]{\Dt\brac{\pap\frac{1}{\Zap}}}  + \norm[\Hhalf]{\frac{\sqrt{\Ag}}{\Zap}\nobrac{\pap\frac{1}{\Zap}}}}
\end{align*}

\medskip

\item We have
\begin{align*}
& \norm[2]{\Dapabs\Jzero} \\
& \lesssim B^2(t)\norm[2]{\pap\frac{1}{\Zap}}  + B(t)\brac{\norm[2]{\Dt\brac{\pap\frac{1}{\Zap}}}  + \norm[\Hhalf]{\frac{\sqrt{\Ag}}{\Zap}\nobrac{\pap\frac{1}{\Zap}}}}
\end{align*}
Proof: As $\Jzero$ given by \eqref{form:Jzero} is real valued, we see that
\begin{align*}
\Dapabs\Jzero & = \Real\cbrac{\frac{\w}{\Zap}(\Id - \Hil)\pap\Jzero} \\
& = \Real\cbrac{\w(\Id-\Hil)\Dap\Jzero} - \Real\cbrac{\w\sqbrac{\frac{1}{\Zap},\Hil}\pap\Jzero}
\end{align*}
From equation \eqref{eq:paponeoverZap} we have
\begin{align*}
\brac{\Dt^2 + i\frac{\Ag}{\Zapabs^2}\pap}\pap\frac{1}{\Zap} = \Dap\Jzero + \Rzero
\end{align*}
Applying $(\Id - \Hil)$ to the above equation and using \propref{prop:commutator} we obtain the estimate
\begin{align*}
& \norm[2]{\Dapabs\Jzero} \\
& \lesssim \norm[2]{(\Id - \Hil)\Dap\Jzero} + \norm[2]{\pap\frac{1}{\Zap}}\norm[\infty]{\Jzero} \\
& \lesssim \norm[2]{(\Id - \Hil)\cbrac{\Dt^2\brac{\pap\frac{1}{\Zap}}}} + \norm[2]{(\Id - \Hil)\cbrac{i\frac{\Ag}{\Zapabs^2}\pap\brac{\pap\frac{1}{\Zap}}}} \\
& \quad + \norm[2]{\Rzero} + \norm[2]{\pap\frac{1}{\Zap}}\norm[\infty]{\Jzero} \\
& \lesssim B^2(t)\norm[2]{\pap\frac{1}{\Zap}} + B(t)\brac{\norm[2]{\Dt\brac{\pap\frac{1}{\Zap}}}  + \norm[\Hhalf]{\frac{\sqrt{\Ag}}{\Zap}\nobrac{\pap\frac{1}{\Zap}}}}
\end{align*}

\end{enumerate}

\medskip

\subsection{Quantities controlled by the energy $\E(t)$}\label{sec:quantE}

This subsection is similar to \secref{sec:quantEa} but for the energy $\E(t)$ instead of $\Ea(t)$. Here we collect some of the estimates required to prove the energy estimate for $\E(t)$ namely \eqref{eq:mainEtime}. These estimates are then used in \secref{sec:closing} to complete the proof of \eqref{eq:mainEtime}. 

\begin{enumerate}[widest = 99, leftmargin =*, align=left, label=\arabic*)]
\item We have
\begin{align*}
& \norm[2]{\Dap\Dt\Dap\Ztbar} + \norm[2]{\Dap^2\Zttbar} + \norm[2]{\Ag\Dap^2\frac{1}{\Zap}}  + \norm[2]{\Dap^2\Ztt} \\
& \lesssim \norm[2]{\Dt\Dap^2\Ztbar} + B^2(t)\norm[2]{\pap\frac{1}{\Zap}} + B(t)\brac{\norm[2]{\Dt\brac{\pap\frac{1}{\Zap}}} + \norm[\Hhalf]{\frac{\sqrt{\Ag}}{\Zap}\pap\frac{1}{\Zap}}} \\
& \lesssim \norm[2]{\Dt\Dap^2\Ztbar} + \Ea(t)^\half\brac{\norm[2]{\Dt\brac{\pap\frac{1}{\Zap}}} + \norm[\Hhalf]{\frac{\sqrt{\Ag}}{\Zap}\pap\frac{1}{\Zap}}} + \Ea(t)\norm[2]{\pap\frac{1}{\Zap}}
\end{align*}
Proof: We see that
\begin{align}\label{eq:DapDtDapZtbar}
\Dap\Dt\Dap\Ztbar = (\Dap\Zt)\Dap^2\Ztbar + \Dt\Dap^2\Ztbar
\end{align}
and also
\begin{align}\label{eq:Dap2Zttbar1}
\begin{aligned}
\Dap^2\Zttbar & = \Dap\brac{(\Dap\Zt)\Dap\Ztbar + \Dt\Dap\Ztbar} \\
& = (\Dap^2\Zt)\Dap\Ztbar + (\Dap\Zt)\Dap^2\Ztbar + \Dap\Dt\Dap\Ztbar
\end{aligned}
\end{align}
From these identities we see that 
\begin{align*}
 \norm[2]{\Dap\Dt\Dap\Ztbar} + \norm[2]{\Dap^2\Zttbar} \lesssim \norm[2]{\Dt\Dap^2\Ztbar} + \norm[\infty]{\Dap\Ztbar}\brac{\norm[2]{\Dap^2\Ztbar} + \norm[2]{\Dap^2\Zt}}
\end{align*}
and hence the required estimates for these two quantities follow from previously proved estimates. Now we observe from \eqref{form:Zttbar}
\begin{align}\label{eq:Dap2Zttbar2}
\begin{aligned}
\Dap^2\Zttbar & = -i\Dap^2\brac{\frac{\Ag}{\Zap}} \\
& = -i\Dap\cbrac{\Ag\Dap\frac{1}{\Zap} + \frac{\wbar^2}{\Zapabs^2}\pap\Ag } \\
& = -i\brac{\frac{1}{\Zap^2}\pap\Ag}\pap\frac{1}{\Zap} -i\Ag\Dap^2\frac{1}{\Zap} - 2i(\wbar\Dap\wbar)\brac{\frac{1}{\Zapabs^2}\pap\Ag} \\
& \quad -i\wbar^3\Dapabs\brac{\frac{1}{\Zapabs^2}\pap\Ag}
\end{aligned}
\end{align}
Therefore we get
\begin{align*}
\norm[2]{\Ag\Dap^2\frac{1}{\Zap}} \lesssim \norm[2]{\Dap^2\Zttbar} + \norm[\infty]{\frac{1}{\Zapabs^2}\pap\Ag}\norm[2]{\pap\frac{1}{\Zap}} + \norm[2]{\Dapabs\brac{\frac{1}{\Zapabs^2}\pap\Ag}}
\end{align*}
The required estimate now follows from previously proved estimates. Finally using \eqref{form:Zttbar} we get that
\begin{align*}
\Dapbar^2\Zttbar & = \w^2\Dap(\w^2\Dap\Zttbar) \\
& = 2\w^3(\Dap\w)(\Dap\Zttbar) + \w^4\Dap^2\Zttbar \\
& = -2i\w^3(\Dap\w)\brac{\frac{1}{\Zap^2}\pap\Ag + \Ag\Dap\frac{1}{\Zap}} + \w^4\Dap^2\Zttbar \\
& = -2i\w^3(\Dap\w)\brac{\frac{1}{\Zap^2}\pap\Ag} -2i\w(\Dapabs\w)\brac{\frac{\Ag}{\Zapabs}\pap\frac{1}{\Zap}} + \w^4\Dap^2\Zttbar
\end{align*}
Hence we get 
\begin{align*}
\norm[2]{\Dap^2\Ztt} \lesssim \norm[2]{\pap\frac{1}{\Zap}}\norm[\infty]{\frac{1}{\Zapabs^2}\pap\Ag} + \norm[2]{(\Dapabs\w)\brac{\frac{\Ag}{\Zapabs}\pap\frac{1}{\Zap}}} + \norm[2]{\Dap^2\Zttbar}
\end{align*}
All terms other than the middle term have already been controlled. For that term we use \propref{prop:Hhalfweight} with $f = \Dapabs\w$, $g = \pap\frac{1}{\Zap}$ and $w = \frac{\Ag}{\Zapabs}$ to get
\begin{align*}
 & \norm[2]{(\Dapabs\w)\brac{\frac{\Ag}{\Zapabs}\pap\frac{1}{\Zap}}} \\
 & \lesssim \norm[\Hhalf]{\frac{\Ag}{\Zapabs^2}\pap\w}\norm[2]{\pap\frac{1}{\Zap}} + \norm[\Hhalf]{\frac{\Ag}{\Zapabs}\pap\frac{1}{\Zap}}\norm[2]{\pap\frac{1}{\Zap}} \\
 & \quad + \brac{\norm[\Linfty\cap\Hhalf]{\Dap\Ztbar}  +  \brac{\norm[2]{\Ztapbar} + \sqrt{g}}\norm[2]{\pap\frac{1}{\Zap}}}^2\norm[2]{\pap\frac{1}{\Zap}}
\end{align*}
Hence by combining these estimates, we get the required estimate for $\norm[2]{\Dap^2\Ztt}$.

\medskip

\item We have the estimate
\begin{align*}
\norm[\infty]{\Dap\Zttbar} + \norm[\infty]{\Ag\Dap\frac{1}{\Zap}} +  \norm[\infty]{\Dt\Dap\Ztbar} \lesssim \Ea(t)^\onebyfour\Ethree(t)^\onebyfour + \Ea(t) \lesssim \E(t)
\end{align*}
Proof: Letting $f = \Dap\Zttbar$ and $w = \frac{1}{\Zap}$ in \propref{prop:LinftyHhalf} and by using previously proved estimates we get
\begin{align*}
\norm[\Linfty\cap\Hhalf]{\Dap\Zttbar}^2 & \lesssim \norm[2]{\Zttapbar}\norm[2]{\Dap^2\Zttbar} + \norm[2]{\Zttapbar}^2\norm[2]{\pap\frac{1}{\Zap}}^2 \\
& \lesssim \Ea(t)^\onebytwo\Ethree(t)^\onebytwo + \Ea(t)^2
\end{align*}
This proves the estimate for $\norm[\infty]{\Dap\Zttbar}$. Using \eqref{form:Zttbar} we see that
\begin{align*}
\Dap\Zttbar = -i\Ag\Dap\frac{1}{\Zap} -i\frac{1}{\Zap^2}\pap\Ag
\end{align*}
Hence we get
\begin{align*}
\norm[\infty]{\Ag\Dap\frac{1}{\Zap}} \lesssim \norm[\infty]{\Dap\Zttbar} + \norm[\infty]{\frac{1}{\Zapabs^2}\pap\Ag}  \lesssim \Ea(t)^\onebyfour\Ethree(t)^\onebyfour + \Ea(t) \lesssim \E(t)
\end{align*}
Now as $\Dap\Zttbar = (\Dap\Zt)\Dap\Ztbar + \Dt\Dap\Ztbar$ we see that
\begin{align*}
\norm[\infty]{\Dt\Dap\Ztbar} \lesssim \norm[\infty]{\Dap\Zttbar} + \norm[\infty]{\Dap\Ztbar}^2 \lesssim \Ea(t)^\onebyfour\Ethree(t)^\onebyfour + \Ea(t)
\end{align*}

\item We have the estimates
\begin{align*}
\norm[\infty]{\frac{\Jone}{\Ag}} \lesssim B(t) \lesssim \Ea(t)^\half
\end{align*}
and
\begin{align*}
\norm[2]{\frac{1}{\sqrt{\Ag}}\Dapabs\Jone} & \lesssim  B^2(t)  + \brac{\norm[2]{\Ztapbar} + \sqrt{g}}\brac{\norm[2]{\Dt\brac{\pap\frac{1}{\Zap}}} + \norm[\Hhalf]{\frac{\sqrt{\Ag}}{\Zap}\pap\frac{1}{\Zap}}} \\
& \lesssim \Ea(t)
\end{align*}
Proof: From \eqref{eq:Jone} we see that
\begin{align*}
\norm[\infty]{\frac{\Jone}{\Ag}}  \lesssim \norm[\infty]{\frac{\Dt\Ag}{\Ag}} + \norm[\infty]{\bvarap} + \norm[\infty]{\Dap\Ztbar} \lesssim B(t)
\end{align*}
Now again from \eqref{eq:Jone} we obtain
\begin{align*}
\frac{1}{\sqrt{\Ag}}\Dapabs\Jone & = \frac{1}{\sqrt{\Ag}}\Dapabs\Dt\Ag + \brac{\frac{1}{\sqrt{\Ag}}\Dapabs\Ag }(\bvarap - \Dap\Zt - \Dapbar\Ztbar) \\
& \quad + \sqrt{\Ag}\Dapabs(\bvarap - \Dap\Zt - \Dapbar\Ztbar)
\end{align*}
Hence we get
\begin{align*}
\norm[2]{\frac{1}{\sqrt{\Ag}}\Dapabs\Jone} & \lesssim \norm[2]{\frac{1}{\sqrt{\Ag}}\Dapabs\Dt\Ag} + \norm[2]{\frac{1}{\sqrt{\Ag}}\Dapabs\Ag} \brac{\norm[\infty]{\bvarap} + \norm[\infty]{\Dap\Ztbar}} \\
& \quad + \norm[\infty]{\Ag}^\half\norm[2]{\Dapabs(\bvarap - \Dap\Zt - \Dapbar\Ztbar)}
\end{align*}
The required estimate now follows. 

\item We have the estimate
\begin{align*}
\norm[2]{(\Id - \Hil)\Dt^2\Dap^2\Ztbar} \lesssim \Ea(t)^\half\norm[2]{\Dt\Dap^2\Ztbar} + \Ea(t)\norm[2]{\Dt\brac{\pap\frac{1}{\Zap}}} + \Ea(t)^\threebytwo\norm[2]{\pap\frac{1}{\Zap}}
\end{align*}
Proof: As $\Pa \Dap^2\Ztbar = 0$, we see from \eqref{eq:IminusHDt2f} that
\begin{align*}
(\Id - \Hil)\Dt^2\Dap^2\Ztbar = 2\sqbrac{\bvar,\Hil}\pap\Dt \Dap^2\Ztbar + \sqbrac{\Dt\bvar,\Hil}\pap\Dap^2\Ztbar - \sqbrac{\bvar, \bvar ; \pap \Dap^2\Ztbar}
\end{align*}
Hence we have from \propref{prop:commutator} and \propref{prop:triple}
\begin{align*}
& \norm[2]{(\Id - \Hil)\Dt^2\Dap^2\Ztbar} \\
& \lesssim \norm[\Hhalf]{\bvarap}\norm[2]{\Dt\Dap^2\Ztbar} + \norm[\Hhalf]{\pap\Dt\bvar}\norm[2]{\Dap^2\Ztbar} + \norm[\infty]{\bvarap}^2\norm[2]{\Dap^2\Ztbar} 
\end{align*}
The required estimate now follows from previously proved estimates. 

\item We have the estimate
\begin{align*}
\norm[2]{(\Id - \Hil)\cbrac{i\frac{\Ag}{\Zapabs^2}\pap\Dap^2\Ztbar}} \lesssim \Ea(t)\norm[2]{\Dt\brac{\pap\frac{1}{\Zap}}} + \Ea(t)^\threebytwo\norm[2]{\pap\frac{1}{\Zap}}
\end{align*}
Proof: Observe that
\begin{align*}
(\Id - \Hil)\cbrac{i\frac{\Ag}{\Zapabs^2}\pap\Dap^2\Ztbar} = i\sqbrac{\frac{\Ag}{\Zapabs^2},\Hil}\pap\Dap^2\Ztbar
\end{align*}
Hence we have from \propref{prop:commutator}
\begin{align*}
& \norm[2]{(\Id - \Hil)\cbrac{i\frac{\Ag}{\Zapabs^2}\pap\Dap^2\Ztbar}} \\
& \lesssim \brac{\norm[\Hhalf]{\frac{1}{\Zapabs^2}\pap\Ag} + \norm[\Hhalf]{\frac{\Ag}{\Zapabs}\pap\frac{1}{\Zapabs}} }\norm[2]{\Dap^2\Ztbar}
\end{align*}
The required estimate now follows. 

\item We have the temporary estimate
\begin{align*}
\norm[2]{\Rone} & \lesssim \Ea(t)^\half\norm[2]{\Dt\Dap^2\Ztbar} + \E(t)\brac{\norm[2]{\Dt\brac{\pap\frac{1}{\Zap}}} + \norm[\Hhalf]{\frac{\sqrt{\Ag}}{\Zap}\pap\frac{1}{\Zap} }} \\
& \quad + \Ea(t)^\half\E(t)\norm[2]{\pap\frac{1}{\Zap}} + \norm[\infty]{\frac{1}{\Zapabs^2}\pap\Jone}\norm[2]{\pap\frac{1}{\Zap}}
\end{align*}
Proof: Recall from \eqref{eq:Rone} that
\begin{align*}
\Rone & = -2(\Dap^2\Ztt)(\Dap\Ztbar) -4(\Dap\Ztt)(\Dap^2\Ztbar) -2(\Dap^2\Zt)(\Dt\Dap\Ztbar) \\
& \quad -2(\Dap\Zt)(\Dap\Dt\Dap\Ztbar)  - 2(\Dap\Zt)(\Dt\Dap^2\Ztbar)  -i(\Dap\Jone) \brac{\Dap\frac{1}{\Zap}} \\
& \quad  -i\Jone\Dap^2\frac{1}{\Zap}  -2i\wbar(\Dap\wbar)\brac{\frac{1}{\Zapabs^2}\pap\Jone}
\end{align*}
Hence we see that
\begin{align*}
 \norm[2]{\Rone} & \lesssim \norm[\infty]{\Dap\Ztbar}\brac{\norm[2]{\Dap^2\Ztt} + \norm[2]{\Dap\Dt\Dap\Ztbar} + \norm[2]{\Dt\Dap^2\Ztbar} } \\
& \quad + \brac{\norm[\infty]{\Dap\Ztt} + \norm[\infty]{\Dt\Dap\Ztbar}}\brac{\norm[2]{\Dap^2\Ztbar} + \norm[2]{\Dap^2\Zt}} \\
& \quad + \norm[\infty]{\frac{1}{\Zapabs^2}\pap\Jone}\norm[2]{\pap\frac{1}{\Zap}} + \norm[\infty]{\frac{\Jone}{\Ag}}\norm[2]{\Ag\Dap^2\frac{1}{\Zap}}
\end{align*}
The estimate now follows using previous estimates. 

\item We have the estimate
\begin{align*}
& \norm[2]{\Dapabs\brac{\frac{1}{\Zapabs^2}\pap\Jone}} \\
&  \lesssim \Ea(t)^\half\norm[2]{\Dt\Dap^2\Ztbar} + \E(t)\brac{\norm[2]{\Dt\brac{\pap\frac{1}{\Zap}}} + \norm[\Hhalf]{\frac{\sqrt{\Ag}}{\Zap}\pap\frac{1}{\Zap} }} \\
& \quad + \Ea(t)^\half\E(t)\norm[2]{\pap\frac{1}{\Zap}} 
\end{align*}
Proof: Using the fact that $\Jone$ is real valued, we observe that
\begin{align*}
& \norm[2]{\Dapabs\brac{\frac{1}{\Zapabs^2}\pap\Jone}} \\
& \lesssim \norm[2]{\frac{\wbar^3}{\Zapabs}(\Id - \Hil)\pap\brac{\frac{1}{\Zapabs^2}\pap\Jone}} \\
& \lesssim \norm[2]{\sqbrac{\frac{\wbar^3}{\Zapabs},\Hil }\pap\brac{\frac{1}{\Zapabs^2}\pap\Jone}} + \norm[2]{(\Id - \Hil)\cbrac{\frac{\wbar^3}{\Zapabs}\pap\brac{\frac{1}{\Zapabs^2}\pap\Jone} }}
\end{align*}
Now using \propref{prop:commutator} and \eqref{eq:mainDapDapZtbar} we see that
\begingroup
\allowdisplaybreaks
\begin{align*}
& \norm[2]{\Dapabs\brac{\frac{1}{\Zapabs^2}\pap\Jone}} \\
& \lesssim \norm[2]{\pap\frac{1}{\Zap}}\norm[\infty]{\frac{1}{\Zapabs^2}\pap\Jone} + \norm[2]{\Rone} + \norm[2]{(\Id - \Hil)\Dt^2\Dap^2\Ztbar} \\*
& \quad + \norm[2]{(\Id - \Hil)\cbrac{i\frac{\Ag}{\Zapabs^2}\pap\Dap^2\Ztbar}} \\
& \lesssim \Ea(t)^\half\norm[2]{\Dt\Dap^2\Ztbar} + \E(t)\brac{\norm[2]{\Dt\brac{\pap\frac{1}{\Zap}}} + \norm[\Hhalf]{\frac{\sqrt{\Ag}}{\Zap}\pap\frac{1}{\Zap} }} \\*
& \quad + \Ea(t)^\half\E(t)\norm[2]{\pap\frac{1}{\Zap}} + \norm[\infty]{\frac{1}{\Zapabs^2}\pap\Jone}\norm[2]{\pap\frac{1}{\Zap}}
\end{align*}
\endgroup
Now in \propref{prop:LinftyHhalf} we put $f = \frac{1}{\Zapabs^2}\pap\Jone$ and $w = \frac{1}{\Zapabs}$ to get
\begin{align}\label{eq:oneZapabs2papJonetemp}
\begin{split}
& \norm[\infty]{\frac{1}{\Zapabs^2}\pap\Jone}^2 \\
& \lesssim \norm[2]{\Dapabs\Jone}\norm[2]{\Dapabs\brac{\frac{1}{\Zapabs^2}\pap\Jone}} + \norm[2]{\Dapabs\Jone}^2\norm[2]{\pap\frac{1}{\Zapabs}}^2 \\
& \lesssim \norm[\infty]{\Ag}^\half\norm[2]{\frac{1}{\sqrt{\Ag}}\Dapabs\Jone}\norm[2]{\Dapabs\brac{\frac{1}{\Zapabs^2}\pap\Jone}} + \norm[\infty]{\Ag}\norm[2]{\frac{1}{\sqrt{\Ag}}\Dapabs\Jone}^2\norm[2]{\pap\frac{1}{\Zap}}^2 \\
& \lesssim \Ea(t)\brac{\norm[2]{\Ztapbar} + \sqrt{g}}\norm[2]{\Dapabs\brac{\frac{1}{\Zapabs^2}\pap\Jone}} + \Ea(t)^3
\end{split}
\end{align}
Using this estimate in the previous estimate we get
\begin{align*}
& \norm[2]{\Dapabs\brac{\frac{1}{\Zapabs^2}\pap\Jone}} \\
& \lesssim \Ea(t)^\half\norm[2]{\Dt\Dap^2\Ztbar} + \E(t)\brac{\norm[2]{\Dt\brac{\pap\frac{1}{\Zap}}} + \norm[\Hhalf]{\frac{\sqrt{\Ag}}{\Zap}\pap\frac{1}{\Zap} }} \\
& \quad + \Ea(t)^\half\E(t)\norm[2]{\pap\frac{1}{\Zap}} + \Ea(t)^\half\brac{\norm[2]{\Ztapbar} + \sqrt{g}}^\half\norm[2]{\Dapabs\brac{\frac{1}{\Zapabs^2}\pap\Jone}}^\half\norm[2]{\pap\frac{1}{\Zap}}
\end{align*}
Now using the estimate $ab \leq \frac{a^2}{2\ep} + \frac{\ep b^2}{2}$ with $\ep$ small on the last term and absorbing it on the left hand side we get the required estimate. 

%
\item We have the estimate
\begin{align*}
\norm[\infty]{\frac{1}{\Zapabs^2}\pap\Jone} \lesssim \Ea(t)^\half E(t)
\end{align*}
This in particular implies that we have
\begin{align*}
\norm[2]{\Rone} & \lesssim \Ea(t)^\half\norm[2]{\Dt\Dap^2\Ztbar} + \E(t)\brac{\norm[2]{\Dt\brac{\pap\frac{1}{\Zap}}} + \norm[\Hhalf]{\frac{\sqrt{\Ag}}{\Zap}\pap\frac{1}{\Zap} }} \\
& \quad + \Ea(t)^\half\E(t)\norm[2]{\pap\frac{1}{\Zap}} 
\end{align*}
Proof: The first estimate follows directly from the estimate for $ \norm[2]{\Dapabs\brac{\frac{1}{\Zapabs^2}\pap\Jone}}$ and \eqref{eq:oneZapabs2papJonetemp}. The second estimate then follows immediately from the temporary estimate proved for $\norm[2]{\Rone}$. 

\end{enumerate}

\subsection{Closing the energy estimate}\label{sec:closing}

We first note the following lemma from Sec 5.2 of \cite{Ag21}.
\begin{lem} \label{lem:timederiv}
Let $T>0$ and let $f,\bvar \in C^2([0,T), H^2(\Rsp))$ with $\bvar$ being real valued. Let $\Dt = \pt + \bvar\pap$. Then 
\begin{enumerate}[leftmargin =*, align=left]
\item $\dis \frac{\diff }{\diff t} \int f \diff \ap = \int \Dt f \diff\ap + \int \bvarap f \diff\ap$
\item $\dis \abs*[\Big]{ \frac{d}{dt}\int \abs{f}^2 \diff \ap - 2\Real \int \bar{f} (\Dt f) \diff \ap} \lesssim \norm[2]{f}^2 \norm[\infty]{\bvarap}$
\item $\dis \abs{ \frac{d}{dt}\int (\papabs\bar{f})f \diff\ap- 2\Real \cbrac{\int (\papabs\bar{f})\Dt f \diff \ap}} \lesssim   \norm[\Hhalf]{f}^2 \norm[\infty]{\bvarap}$
\end{enumerate}
\end{lem}

Now using the above lemma we see that
\begin{align}\label{eq:timederivEone}
\begin{split}
& \frac{\diff }{\diff t}\Eone(t)  \\
& = \frac{\diff }{\diff t} \brac{\norm[2]{\Ztapbar}^2 + g}\norm[2]{\pap\frac{1}{\Zap}}^2 \\
& \lesssim \norm[\infty]{\bvarap}\brac{\norm[2]{\Ztapbar}^2 + g}\norm[2]{\pap\frac{1}{\Zap}}^2 + \norm[2]{\Ztapbar}\norm[2]{\Zttapbar}\norm[2]{\pap\frac{1}{\Zap}}^2 \\
& \quad + \brac{\norm[2]{\Ztapbar}^2 + g}\norm[2]{\pap\frac{1}{\Zap}}\norm[2]{\Dt\pap\frac{1}{\Zap}} \\
& \lesssim B(t) \Ea(t)
\end{split}
\end{align}
Now let us control $\Etwo(t)$. We first need to control
\begin{align*}
& \pt\cbrac{\norm[2]{\Dt\brac{\pap\frac{1}{\Zap}}}^2 + \norm[\Hhalf]{\frac{\sqrt{\Ag}}{\Zap}\pap\frac{1}{\Zap}}^2} 
\end{align*}
Let $\dis f = \pap\frac{1}{\Zap}$. Note that here $\Hil f = f$. Hence we need to control the time derivative of 
\begin{align*}
 \norm[2]{\Dt f}^2 + \norm[\Hhalf]{\frac{\sqrt{\Ag}}{\Zap} f}^2 
\end{align*}
for a function $f$ satisfying $\Hil f = f$. This will also help us in controlling $\Ethree$. Such a type of estimate was done in \cite{Ag21} but the estimate proved there strongly used the fact that $g = 1$ and does not work here. We will crucially use the new estimates proved in \secref{sec:quantEa} to obtain the estimate \eqref{eq:temptimederivfull}. The estimates from \secref{sec:quantEa} and \secref{sec:quantE} will then be further used to close the energy estimates. 

We see from \lemref{lem:timederiv} that 
\begin{align}\label{eq:ddttemp23}
 \abs{\frac{d}{dt} \int \abs{\Dt f}^2 \diff\ap - 2\Real \int (\Dt^2 f)(\Dt \bar{f}) \diff\ap}  \lesssim B(t)\norm[2]{\Dt f}^2
\end{align}
We also see from \lemref{lem:timederiv} that 
\begin{align*}
& \abs{\frac{d}{dt}  \int \abs*[\bigg]{ \papabs^\half \brac*[\bigg]{\frac{\sqrt{\Ag}}{\Zap} f}}^2\difff\ap -  2\Real \int  \cbrac{\papabs\brac*[\bigg]{\frac{\sqrt{\Ag}}{\Zap} f}} \Dt\brac*[\bigg]{\frac{\sqrt{\Ag}}{\Zapbar} \bar{f}} \difff\ap} \\
& \lesssim B(t)\norm[\Hhalf]{\frac{\sqrt{\Ag}}{\Zap} f}^2
\end{align*}
Now observe from \eqref{form:DtoneoverZap}
\begin{align*}
 \Dt\brac*[\bigg]{\frac{\sqrt{\Ag}}{\Zapbar} \bar{f}}  = \cbrac{\frac{\Dt\Ag}{2\Ag} + \bvarap - \Dapbar\Ztbar }\frac{\sqrt{\Ag}}{\Zapbar}\bar{f} + \frac{\sqrt{\Ag}}{\Zapbar}\Dt\bar{f}
\end{align*}
Hence
\begin{align*}
& \abs{2\Real \int  \cbrac{\papabs\brac*[\bigg]{\frac{\sqrt{\Ag}}{\Zap} f}} \Dt\brac*[\bigg]{\frac{\sqrt{\Ag}}{\Zapbar} \bar{f}} \difff\ap - 2\Real \int  \cbrac{\papabs\brac*[\bigg]{\frac{\sqrt{\Ag}}{\Zap} f}} \brac{\frac{\sqrt{\Ag}}{\Zapbar}\Dt\bar{f}} \difff\ap} \\
& \lesssim \norm[\Hhalf]{\frac{\sqrt{\Ag}}{\Zap} f}\norm[\Hhalf]{\cbrac{\frac{\Dt\Ag}{2\Ag} + \bvarap - \Dapbar\Ztbar }\frac{\sqrt{\Ag}}{\Zapbar}\bar{f}}
\end{align*}
From \propref{prop:Hhalfweight} with $f = \bar{f}$, $\dis w = \frac{\sqrt{\Ag}}{\Zapbar}$ and $\dis h = \cbrac{\frac{\Dt\Ag}{2\Ag} + \bvarap - \Dapbar\Ztbar }$ we get
\begingroup
\allowdisplaybreaks
\begin{align*}
& \norm[\Hhalf]{\cbrac{\frac{\Dt\Ag}{2\Ag} + \bvarap - \Dapbar\Ztbar }\frac{\sqrt{\Ag}}{\Zapbar}\bar{f}} \\
& \lesssim \norm[\infty]{\frac{\Dt\Ag}{2\Ag} + \bvarap - \Dapbar\Ztbar}\norm[\Hhalf]{\frac{\sqrt{\Ag}}{\Zapbar}\bar{f}} +  \norm[\infty]{\frac{\Dt\Ag}{2\Ag} + \bvarap - \Dapbar\Ztbar}\norm[2]{\pap\brac{\frac{\sqrt{\Ag}}{\Zapbar}}}\norm[2]{\bar{f}} \\*
& \quad + \norm[2]{\bar{f}}\norm[2]{\frac{\sqrt{\Ag}}{\Zapbar}\pap\cbrac{\frac{\Dt\Ag}{2\Ag} + \bvarap - \Dapbar\Ztbar }} \\
& \lesssim B(t)\norm[\Hhalf]{\frac{\sqrt{\Ag}}{\Zapbar}\bar{f}}  + B^2(t) \norm[2]{\bar{f}} \\*
& \quad +  \brac{\norm[2]{\Ztapbar} + \sqrt{g}}\brac{\norm[2]{\Dt\brac{\pap\frac{1}{\Zap}}} + \norm[\Hhalf]{\frac{\sqrt{\Ag}}{\Zap}\pap\frac{1}{\Zap}}}\norm[2]{\bar{f}}
\end{align*}
\endgroup
Therefore 
\begin{align}\label{eq:ddttemp1}
\begin{aligned}
& \abs{\frac{d}{dt}  \int \abs*[\bigg]{ \papabs^\half \brac*[\bigg]{\frac{\sqrt{\Ag}}{\Zap} f}}^2\difff\ap -  2\Real \int  \cbrac{\papabs\brac*[\bigg]{\frac{\sqrt{\Ag}}{\Zap} f}} \brac*[\bigg]{\frac{\sqrt{\Ag}}{\Zapbar} \Dt\bar{f}} \difff\ap} \\
& \lesssim B(t)\norm[\Hhalf]{\frac{\sqrt{\Ag}}{\Zap} f}^2  + B^2(t) \norm[2]{\bar{f}}\norm[\Hhalf]{\frac{\sqrt{\Ag}}{\Zap} f} \\
& \quad + \brac{\norm[2]{\Ztapbar} + \sqrt{g}}\brac{\norm[2]{\Dt\brac{\pap\frac{1}{\Zap}}} + \norm[\Hhalf]{\frac{\sqrt{\Ag}}{\Zap}\pap\frac{1}{\Zap}}}\norm[2]{\bar{f}}\norm[\Hhalf]{\frac{\sqrt{\Ag}}{\Zap} f}
\end{aligned}
\end{align}
We simplify further using $\papabs = i\Hil\pap$ and $\Hil f = f$
\begingroup
\allowdisplaybreaks
\begin{align*}
&\frac{\sqrt{\Ag}}{\Zapbar}\papabs\brac*[\bigg]{\frac{\sqrt{\Ag}}{\Zap} f} \\
& = i\sqbrac{\frac{\sqrt{\Ag}}{\Zapbar} ,\Hil}\pap\brac{\frac{\sqrt{\Ag}}{\Zap}f} + i\Hil\cbrac{\frac{\sqrt{\Ag}}{\Zapbar}\pap\brac{\frac{\sqrt{\Ag}}{\Zap}}f + \frac{\Ag}{\Zapabs^2}\pap f  } \\
& =  i\sqbrac{\frac{\sqrt{\Ag}}{\Zapbar} ,\Hil}\pap\brac{\frac{\sqrt{\Ag}}{\Zap}f} + i\Hil\cbrac*[\Bigg]{\frac{1}{2}\brac*[\Bigg]{\frac{1}{\Zapabs^2}\pap\Ag}f + \Ag\brac{\Dapbar\frac{1}{\Zap}}f  } \\*
& \quad - i\sqbrac{\frac{\Ag}{\Zapabs^2},\Hil}\pap f + i\frac{\Ag}{\Zapabs^2}\pap f
\end{align*}
\endgroup
Observe that in the second part of \propref{prop:Hhalfweight} by letting $f = f$, $g = \sqrt{\Ag}\w\pap\frac{1}{\Zap}$ and $w = \frac{\sqrt{\Ag}}{\Zapabs}$ we get
\begin{align*}
 \norm[2]{\Ag\brac{\Dapbar\frac{1}{\Zap}}f} \lesssim B(t)\norm[\Hhalf]{\frac{\sqrt{\Ag}}{\Zapabs}f } + \norm[\Hhalf]{\frac{\Ag}{\Zapbar}\pap\frac{1}{\Zap}}\norm[2]{f}  + B^2(t)\norm[2]{f}
\end{align*}

Hence we have the following estimate by using \propref{prop:commutator}
\begin{align*}
 & \norm*[\bigg][2]{\frac{\sqrt{\Ag}}{\Zapbar}\papabs\brac*[\bigg]{\frac{\sqrt{\Ag}}{\Zap} f} -  i\frac{\Ag}{\Zapabs^2}\pap f} \\
& \lesssim B(t)\norm[\Hhalf]{\frac{\sqrt{\Ag}}{\Zap}f}  + B^2(t) \norm[2]{f} \\
& \quad + \brac{\norm[2]{\Ztapbar} + \sqrt{g}}\brac{\norm[2]{\Dt\brac{\pap\frac{1}{\Zap}}} + \norm[\Hhalf]{\frac{\sqrt{\Ag}}{\Zap}\pap\frac{1}{\Zap}}}\norm[2]{f}
\end{align*}
Therefore combining the above estimate with \eqref{eq:ddttemp1} we get
\begin{align*}
& \abs{\frac{d}{dt}  \int \abs*[\bigg]{ \papabs^\half \brac*[\bigg]{\frac{\sqrt{\Ag}}{\Zap} f}}^2\difff\ap -  2\Real \int  \brac*[\bigg]{i\frac{\Ag}{\Zapabs^2}\pap f} (\Dt\bar{f}) \diff\ap} \\
& \lesssim B(t)\norm[\Hhalf]{\frac{\sqrt{\Ag}}{\Zap} f}\brac{\norm[\Hhalf]{\frac{\sqrt{\Ag}}{\Zap} f} + \norm[2]{\Dt\bar{f}}}  + B^2(t) \norm[2]{\bar{f}}\brac{\norm[\Hhalf]{\frac{\sqrt{\Ag}}{\Zap} f} + \norm[2]{\Dt\bar{f}}} \\
& \quad +  \brac{\norm[2]{\Ztapbar} + \sqrt{g}}\brac{\norm[2]{\Dt\brac{\pap\frac{1}{\Zap}}} + \norm[\Hhalf]{\frac{\sqrt{\Ag}}{\Zap}\pap\frac{1}{\Zap}}}\norm[2]{\bar{f}}\brac{\norm[\Hhalf]{\frac{\sqrt{\Ag}}{\Zap} f} + \norm[2]{\Dt\bar{f}}}
\end{align*}
Finally combining the above estimate with \eqref{eq:ddttemp23} we obtain
\begin{align}\label{eq:temptimederivfull}
\begin{split}
& \abs{\frac{d}{dt}\cbrac{\norm[2]{\Dt f}^2 + \norm[\Hhalf]{\frac{\sqrt{\Ag}}{\Zap} f}^2} - 2 \Real \int \brac{ \Dt^2 f +i\frac{\Ag}{\Zapabs^2}\pap f }(\Dt\bar{f}) \diff\ap} \\
& \lesssim  B(t)\brac{\norm[\Hhalf]{\frac{\sqrt{\Ag}}{\Zap} f}^2 + \norm[2]{\Dt\bar{f}}^2}  + B^2(t) \norm[2]{\bar{f}}\brac{\norm[\Hhalf]{\frac{\sqrt{\Ag}}{\Zap} f} + \norm[2]{\Dt\bar{f}}} \\
& \quad +  \brac{\norm[2]{\Ztapbar} + \sqrt{g}}\brac{\norm[2]{\Dt\brac{\pap\frac{1}{\Zap}}} + \norm[\Hhalf]{\frac{\sqrt{\Ag}}{\Zap}\pap\frac{1}{\Zap}}}\norm[2]{\bar{f}}\Bigg\{\norm[\Hhalf]{\frac{\sqrt{\Ag}}{\Zap} f}  \\
& \quad + \norm[2]{\Dt\bar{f}}\Bigg\}
\end{split}
\end{align}
We are now ready to control $\Etwo(t)$. Recall from \eqref{eq:paponeoverZap} that
\begin{align*}
\brac{\Dt^2 + i\frac{\Ag}{\Zapabs^2}\pap}\pap\frac{1}{\Zap} = \Dap\Jzero + \Rzero
\end{align*}
Hence for $\dis f = \pap\frac{1}{\Zap}$ we have
\begin{align*}
& \abs{ 2 \Real \int \brac{ \Dt^2 f +i\frac{\Ag}{\Zapabs^2}\pap f }(\Dt\bar{f}) \diff\ap} \\
& \lesssim \brac{\norm[2]{\Dapabs\Jzero} + \norm[2]{\Rzero}}\norm[2]{\Dt\brac{\pap\frac{1}{\Zap}}} \\
&  \lesssim B^2(t)\norm[2]{\pap\frac{1}{\Zap}}\norm[2]{\Dt\brac{\pap\frac{1}{\Zap}}}  + B(t)\cbrac{\norm[2]{\Dt\brac{\pap\frac{1}{\Zap}}}^2  + \norm[\Hhalf]{\frac{\sqrt{\Ag}}{\Zap}\brac{\pap\frac{1}{\Zap}}}^2}
\end{align*}
Therefore by using the above estimate and  putting $\dis f = \pap\frac{1}{\Zap}$ in \eqref{eq:temptimederivfull} we get
\begin{align}\label{eq:Elestimatemainproof}
\begin{aligned}
& \abs{ \frac{d}{dt}\cbrac{\norm[2]{\Dt \brac{\pap\frac{1}{\Zap}}}^2 + \norm[\Hhalf]{\frac{\sqrt{\Ag}}{\Zap} \pap\frac{1}{\Zap}}^2}} \\
&  \lesssim B^2(t)\norm[2]{\pap\frac{1}{\Zap}}\Bigg\{\norm[2]{\Dt\brac{\pap\frac{1}{\Zap}}}  + \norm[\Hhalf]{\frac{\sqrt{\Ag}}{\Zap}\brac{\pap\frac{1}{\Zap}}}\Bigg\} \\
& \quad + B(t) \cbrac{\norm[2]{\Dt\brac{\pap\frac{1}{\Zap}}}^2   + \norm[\Hhalf]{\frac{\sqrt{\Ag}}{\Zap}\brac{\pap\frac{1}{\Zap}}}^2}
\end{aligned}
\end{align}
Therefore we get from \lemref{lem:timederiv} the estimate
\begin{align}\label{eq:timederivEtwo}
\begin{split}
& \frac{\diff}{\diff t} \Etwo(t) \\
& = \frac{\diff}{\diff t}\cbrac{\brac{\norm[2]{\Ztapbar}^2 + g}\cbrac{\norm[2]{\Dt\brac{\pap\frac{1}{\Zap}}}^2 + \norm[\Hhalf]{\frac{\sqrt{\Ag}}{\Zap}\pap\frac{1}{\Zap}}^2}} \\
& \lesssim \brac{\norm[\infty]{\bvarap}\norm[2]{\Ztapbar}^2 + \norm[2]{\Zttapbar}\norm[2]{\Ztapbar}}\cbrac{\norm[2]{\Dt\brac{\pap\frac{1}{\Zap}}}^2 + \norm[\Hhalf]{\frac{\sqrt{\Ag}}{\Zap}\pap\frac{1}{\Zap}}^2} \\
& \quad + \brac{\norm[2]{\Ztapbar}^2 + g}\frac{\diff}{\diff t}\cbrac{\norm[2]{\Dt\brac{\pap\frac{1}{\Zap}}}^2 + \norm[\Hhalf]{\frac{\sqrt{\Ag}}{\Zap}\pap\frac{1}{\Zap}}^2} \\
& \lesssim B(t)\Etwo(t)^\half\Ea(t)
\end{split}
\end{align}
Combining this estimate with \eqref{eq:timederivEone} we see that
\begin{align*}
& \frac{\diff}{\diff t} \Ea(t) \\
& = \half \brac{\frac{2\Eone(t) \partial_t \Eone(t) + \partial_t \Etwo(t)}{(\Eone(t)^2 + \Etwo(t))^\half}} \\
& \lesssim B(t)\Ea(t) \\
& \lesssim \Ea(t)^\threebytwo
\end{align*}
thereby proving \eqref{eq:mainEatime}. 

Let us now control $\Ethree(t)$. Let $f = \Dap^2\Ztbar$ and observe that $\Hil f = f$. We see from \eqref{eq:mainDapDapZtbar} that 
\begin{align*}
& \abs{ 2 \Real \int \brac{ \Dt^2 f +i\frac{\Ag}{\Zapabs^2}\pap f }(\Dt\bar{f}) \diff\ap} \\
& \lesssim \brac{\norm[2]{\Rone} + \norm[2]{\Dapabs\brac{\frac{1}{\Zapabs^2}\pap\Jone}}}\norm[2]{\Dt \Dap^2\Ztbar} \\
& \lesssim \Ea(t)^\half\norm[2]{\Dt\Dap^2\Ztbar}^2 + \E(t)\brac{\norm[2]{\Dt\brac{\pap\frac{1}{\Zap}}} + \norm[\Hhalf]{\frac{\sqrt{\Ag}}{\Zap}\pap\frac{1}{\Zap} }}\norm[2]{\Dt\Dap^2\Ztbar} \\
& \quad + \Ea(t)^\half\E(t)\norm[2]{\pap\frac{1}{\Zap}}\norm[2]{\Dt\Dap^2\Ztbar}
\end{align*}
Hence from \eqref{eq:temptimederivfull} we see that
\begin{align*}
&  \abs{ \frac{d}{dt}\cbrac{\norm[2]{\Dt \Dap^2\Ztbar }^2 + \norm[\Hhalf]{\frac{\sqrt{\Ag}}{\Zap} \Dap^2\Ztbar }^2} } \\
& \lesssim \Ea(t)^\half\brac{\norm[2]{\Dt\Dap^2\Ztbar}^2 + \norm[\Hhalf]{\frac{\sqrt{\Ag}}{\Zap}\Dap^2\Ztbar}^2 } \\
& \quad  + \E(t)\brac{\norm[2]{\Dt\brac{\pap\frac{1}{\Zap}}} + \norm[\Hhalf]{\frac{\sqrt{\Ag}}{\Zap}\pap\frac{1}{\Zap} }}\brac{\norm[2]{\Dt\Dap^2\Ztbar} + \norm[\Hhalf]{\frac{\sqrt{\Ag}}{\Zap}\Dap^2\Ztbar}} \\
& \quad + \Ea(t)^\half\E(t)\norm[2]{\pap\frac{1}{\Zap}}\brac{ \norm[2]{\Dt\Dap^2\Ztbar} + \norm[\Hhalf]{\frac{\sqrt{\Ag}}{\Zap}\Dap^2\Ztbar}}
\end{align*}
Therefore we get from \lemref{lem:timederiv} the estimate
\begingroup
\allowdisplaybreaks
\begin{align*}
& \frac{\diff}{\diff t} \Ethree(t) \\
& = \frac{\diff}{\diff t}\cbrac{\brac{\norm[2]{\Ztapbar}^2 + g}\cbrac{\norm[2]{\Dt \Dap^2\Ztbar }^2 + \norm[\Hhalf]{\frac{\sqrt{\Ag}}{\Zap} \Dap^2\Ztbar }^2}} \\
& \lesssim \brac{\norm[\infty]{\bvarap}\norm[2]{\Ztapbar}^2 + \norm[2]{\Zttapbar}\norm[2]{\Ztapbar}}\cbrac{\norm[2]{\Dt \Dap^2\Ztbar }^2 + \norm[\Hhalf]{\frac{\sqrt{\Ag}}{\Zap} \Dap^2\Ztbar }^2} \\*
& \quad + \brac{\norm[2]{\Ztapbar}^2 + g}\frac{\diff}{\diff t}\cbrac{\norm[2]{\Dt \Dap^2\Ztbar }^2 + \norm[\Hhalf]{\frac{\sqrt{\Ag}}{\Zap} \Dap^2\Ztbar }^2} \\
& \lesssim \Ea(t)^\half\Ethree(t) + \Ea(t)\E(t)\Ethree(t)^\half \\
& \lesssim \Ea(t)^\half\E(t)^3
\end{align*}
Hence by combining this estimate with \eqref{eq:timederivEone} and \eqref{eq:timederivEtwo} we get
\begin{align*}
& \frac{\diff}{\diff t} \E(t) \\
& = \onebythree \brac{\frac{3\Eone(t)^2 \partial_t \Eone(t) + \threebytwo\Etwo(t)^\half\partial_t \Etwo(t) + \partial_t \Ethree(t)}{(\Eone(t)^3 + \Etwo(t)^\threebytwo + \Ethree(t))^\twobythree }} \\
& \lesssim \Ea(t)^\half \E(t) \\
& \lesssim \E(t)^\threebytwo
\end{align*}
\endgroup
thereby proving \eqref{eq:mainEtime}. This completes the proof of \thmref{thm:aprioriE}. 

\subsection{Proof of \thmref{thm:existencemain}}\label{sec:mainexitencesection}

In this section we complete the proof of \thmref{thm:existencemain}. Let us first define the notion of solution. The definition of the solution given below is very similar to the definition in \cite{Wu19} though we make some changes. These changes are important as we also prove a blow up criterion in this paper and so need to work with solutions which may exist for a very long time, or even exist globally (this is in contrast to \cite{Wu19}). 
\begin{definition}\label{def:solution}
Let $g \geq 0$ and let the initial data $(\U,\Psi,\Pfrak)(0)$ be given as in the paragraph above \thmref{thm:existencemain}. Let $ T>0$ and let $\U, \Psi: \Pminus \times \sqbrac{0,T} \to  \Csp$ and $\Pfrak : \Pminus \times \sqbrac{0,T} \to  \Rsp$. We say that $(\U,\Psi, \Pfrak)(t)$ solves the Cauchy problem for the system \eqref{eq:EulerRiem} with gravity $g$ in the time interval $[0,T]$ if the following holds:
\begin{enumerate}
\item $\U,\Psi,\Pfrak$ extend continuously to $\overline{\Pminus}\times \sqbrac{0,T}$ and $(\U,\Psi) \in C^1(\Pminus\times (0,T))$. Also for each $t \in [0,T]$ we have $\Pfrak(\cdot,t) \in C^1(\Pminus)$. 

\item $\U(\cdot,t), \Psi(\cdot,t)$ are holomorphic maps for each $t \in [0,T]$, $\Psi_{\zp}(\zp,t) \neq 0$ for $(\zp,t) \in \Pminus\times [0,T]$ and  $\lim_{\zp \to \infty}(\U, \Psi_t,\Psizp, (\pxp - i\pyp) \Pfrak)(\zp,t) \rightarrow (0, 0, 1, ig)$  for any $t\in\sqbrac{0,T}$. Moreover $\frac{1}{\Psizp}$ and $\frac{\Psi_t}{\Psizp}$ extend continuously to $\overline{\Pminus}\times \sqbrac{0,T}$ (and we will continue to denote these extensions as $\frac{1}{\Psizp}$ and $\frac{\Psi_t}{\Psizp}$ respectively). 

\item We have
\begin{align*}
\sup_{t \in [0,T]}\cbrac{\sup_{\yp<0} \norm[H^1(\Rsp, \diff \ap)]{ U(\cdot + i\yp,t)} + \sup_{\yp<0} \norm[H^1(\Rsp, \diff \ap)]{ \frac{1}{\Psizp}(\cdot + i\yp,t) - 1} } < \infty
\end{align*}
\item $(\U,\Psi,\Pfrak)$ solves \eqref{eq:EulerRiem} in $\Pminus\times [0,T]$ with the given initial data and $\Pfrak$ solves \eqref{eq:DeltaPfrak}. 
\end{enumerate}
\end{definition}

The above definition gives a good notion of solution for us to work with. Unfortunately we do not know whether solutions are unique in this very general class. The problem persists even if we assume that $\sup_{t\in [0,T]} \Ecal(t) < \infty$ for the solution. 

The fundamental issue is that even with the condition $\sup_{t\in [0,T]} \Ecal(t) < \infty$, we do not know whether there always exist smooth solutions which approximate the given solution in the above class appropriately. Hence we will now define a subclass of the above solutions which we call smoothly approximable solutions (denoted by $\mathcal{SA}$), for which there does exists smooth solutions which approximate the given solution in an appropriate manner. We only define this class of solutions for $g > 0$, as for $g = 0$ this strategy to prove uniqueness does not work properly. In particular we only prove uniqueness for $g > 0$.  

To give such a definition, as a priori the time of existence of the given solution could be very large (or even global), we only impose the condition of having smooth approximate solutions of the given solution locally in time.  So for any $\widetilde{T} \in [0,T]$, we impose that there exists a small time interval $[T_1, T_2] \subset [0,T]$ containing $\widetilde{T}$, such that there exists a sequence of smooth solutions converging to the given solution in the time interval $[T_1, T_2]$. Note that if $\widetilde{T} \in (0,T)$, then we impose that $T_1 < \widetilde{T} < T_2$. Obviously if $\widetilde{T} = 0$, then we should put $T_1 = 0$ and $T_2 >0$ and similarly when $\widetilde{T} = T$ we impose $T_1 < T$ and $T_2 = T$. This flexibility of being able to work only locally in time is used later on to prove the blow up criterion. Note that the definition used in \cite{Wu19} in contrast imposes global in time convergence of these approximate solutions.

To make precise on how the smooth solutions converge to the given solution, we also need to define the appropriate norm to take the difference. To do this, let $(\Z,\Zt)_a$ and $(\Z,\Zt)_b$ be two solutions of the water wave equation \eqref{eq:systemone}.  Let $h_a, h_b$ be the homeomorphisms from \eqref{eq:h} for the respective solutions and define
\begin{align}\label{eq:htilandUtilunbdd}
\htil = h_b \compose h_a^{-1} \quad \tx{ and } \quad \Util = U_{\htil} = U_{h_a}^{-1}U_{h_b}
\end{align}
While taking the difference of the two solutions, we will subtract in Lagrangian coordinates and then bring it to the conformal coordinate system of solution $A$. The operator $\Util$ takes a function in the conformal coordinate system of $B$ to the conformal coordinate system of $A$. We define 
\begin{align}\label{eq:Deltafunbdd}
\Delta (f) = f_a - \Util(f_b)
\end{align}
For example we have $\Delta(\Z) = \Z_a - \Util(\Z_b)$. We will usually write $\Util(f_b)$ as $\Util(f)_b$ for convenience. Define
\begin{align}\label{eq:Fcal}
\begin{aligned}
& \Fcal_{\Delta}((\Z,\Zt)_a, (\Z,\Zt)_b)(t) \\
& = \norm[\Hhalf]{\Delta(\Zt)} + \norm[\Hhalf]{\Delta(\Ztt)} + \norm[\Hhalf]{\Delta\brac{\frac{1}{\Zap}}} + \norm[2]{\htil_\ap - 1} \\
& \quad + \norm[2]{\Delta(\Dap\Zt)} + \norm[2]{\Delta(\Ag)} + \norm[2]{\Delta(\bvarap)}
\end{aligned}
\end{align}
This is the norm which is used to compute the difference between two solutions and was introduced in \cite{Wu19}. The usefulness of this norm comes from the following result:

\begin{thm}[\cite{Wu19}]\label{thm:Wudiffenergyestimate}
Let $g > 0$, $T>0$ and let $s \geq 4$. Let $(\Z,\Zt)_a$ and $(\Z,\Zt)_b$ be two solutions of the water wave equation \eqref{eq:systemone} in the time interval $[0,T]$ and assume that for both of these solutions we have $(\Zap - 1, \frac{1}{\Zap} - 1, \Zt) \in C^l([0,T], H^{s - l}(\Rsp)\times H^{s - l}(\Rsp) \times H^{s + \half - l}(\Rsp))$ for $l = 0,1$. Define the constant $L>0$ as
\begin{align*}
L & = \sup_{t \in [0,T]} \Ecal_a(t) + \sup_{t \in [0,T]} \Ecal_b(t) + T + g + \frac{1}{g} + \norm[H^1]{(\Zt)_a}(0) + \norm[H^1]{(\Zt)_b}(0) \\
& \quad + \norm[H^1]{\frac{1}{(\Zap)_a} - 1}(0) + \norm[H^1]{\frac{1}{(\Zap)_b} - 1}(0)
\end{align*}
where $\Ecal_a(t), \Ecal_b(t)$ is the energy $\Ecal(t)$ evaluated for the solutions $(\Z,\Zt)_a$ and $(\Z,\Zt)_b$ respectively. Then there exists a constant $C(L)>0$ depending only on $L$ so that
\begin{align*}
\sup_{t \in [0,T]}  \Fcal_{\Delta}((\Z,\Zt)_a, (\Z,\Zt)_b)(t)
\leq C(L) \Fcal_{\Delta}((\Z,\Zt)_a, (\Z,\Zt)_b)(0)
\end{align*}
\end{thm}
\begin{proof}
This is Theorem 3.7 of \cite{Wu19} and was proved there, though there is a slight difference in the formulation of the statement of the theorem. The quantity $\Fcal_{\Delta}(t)$ is exactly the same as the one used in Theorem 3.7 of \cite{Wu19}, however the energy $\Ecal(t)$ used in this paper is different from the one used in \cite{Wu19}. However this is not an issue as our energy $\Ecal(t)$ controls the energy used in \cite{Wu19}. To see this, first observe that from the first part of \lemref{lem:equivSobolevunbdd} we see that $\sup_{t \in [0,T]} E(t) \lesssim_L 1$ for both the solutions and hence by \eqref{eq:SobolevlowerfromEgzero} we get
\begin{align*}
\sup_{t \in [0,T]} \cbrac{ \norm[2]{(\Ztap)_a}(t) + \norm[2]{(\Ztap)_b}(t)+  \norm[\infty]{\frac{1}{(\Zap)_a}}(t) +  \norm[\infty]{\frac{1}{(\Zap)_a}}(t) } \lesssim_L 1
\end{align*}
Furthermore as $g>0$, from the definition of the energy $\Ecal(t)$ and \eqref{eq:DapZtbarandothers} we get
\begin{align*}
 \sup_{t \in [0,T]}\cbrac{\norm[2]{\pap\frac{1}{\Zap}}^2(t) +  \norm[\Hhalf]{\Dap\Ztbar}^2(t) + \norm[2]{\Dap^2\Ztbar}^2(t) + \norm[2]{\Dap^2\frac{1}{\Zap}}^2 + \norm[\Hhalf]{\frac{1}{\Zap}\Dap^2\Ztbar}^2 }  \lesssim_L 1
\end{align*}
for both the solutions. Hence the energy $\Ecal(t)$ as defined in equation (3.13) of \cite{Wu19} is controlled by $L$ and so the theorem now follows from Theorem 3.7 of \cite{Wu19}. 
\end{proof}

We are now ready to define the subclass of smoothly approximable solutions. 

\begin{definition}\label{def:solutionSA}
Let $(\U,\Psi, \Pfrak)(t)$ be a solution to \eqref{eq:EulerRiem} with gravity $g > 0$ in the time interval $[0,T]$ in the sense of \defref{def:solution}. We say that this solution lies in the smoothly approximable class $\mathcal{SA}$ if the following holds:

For any $\widetilde{T} \in [0,T]$, there exists $T_1, T_2$ such that $0\leq T_1 \leq \widetilde{T} \leq T_2 \leq T$ satisfying $T_2 - T_1>0$ and if $\widetilde{T} \in (0,T)$  then we have $T_1 < \widetilde{T} < T_2$,  and a sequence of smooth functions $(\Z^{(n)}, \Zt^{(n)}): \mathbb{R}\times \sqbrac{T_1,T_2} \rightarrow \mathbb{C}$ for $n \geq 1$ satisfying:
\begin{enumerate}[label=(\alph*)]
\item For each $n\in \mathbb{N}$ we have $\Zap^{(n)}(\ap,t) \neq 0$ for all $(\ap,t) \in \Rsp \times [T_1,T_2]$ and $\Bigl(\Zt^{(n)}, \Zap^{(n)} - 1, \frac{1}{\Zap^{(n)} - 1}\Bigr) \in C^l([T_1,T_2], H^{s + \half - l}(\Rsp), H^{s - l}(\Rsp), H^{s - l}(\Rsp))$ for all $s \geq 4$ and $l = 0,1$. 
\item We have 
\begin{align*}
\sup_{n \geq 1, t\in [T_1,T_2]} \cbrac*[\bigg]{\norm*[\big][H^1]{\Zt^{(n)}}(t) + \norm*[\bigg][H^1]{\frac{1}{\Zap^{(n)}} - 1}(t) } < \infty
\end{align*}
\item There exist holomorphic functions $(\U^{(n)}, \Psi^{(n)})(\cdot,t): \Pminus \to \Csp$ whose boundary values are $(\Ztbar^{(n)},\Z^{(n)})(\cdot,t)$ and which satisfy for each $n$ the property $\Psi_{\zp}^{(n)}(\zp,t) \neq 0$ for all $(\zp,t) \in \Pminus\times [T_1,T_2]$. Moreover for each $t \in [T_1,T_2]$ we have that $\lim_{\zp \to \infty}(\U^{(n)}, \Psi_t^{(n)}, \Psizp^{(n)})(\zp,t) = (0, 0, 1)$. Let $\Pfrak^{(n)}:\Pminus\times [T_1,T_2] \to \Rsp$ be the function satisfying 
\begin{align*}
\Delta \Pfrak^{(n)} = -2\abs*{\U_\zp^{(n)}}^2 \quad \tx{ on } \Pminus, \qq \Pfrak^{(n)} = 0 \quad \tx{ on } \partial \Pminus
\end{align*}
along with $(\pxp - i\pyp) \Pfrak^{(n)} \to ig$ as $\zp \to \infty$. Then $(\U^{(n)}, \Psi^{(n)}, \Pfrak^{(n)}, \frac{1}{\Psizp^{(n)}}) \to (\U, \Psi, \Pfrak, \frac{1}{\Psizp})$ uniformly on compact subsets of $\Pminusbar\times [T_1, T_2]$ as $n \to \infty$.

\item For each $n\in \mathbb{N}$, $(\Z^{(n)}, \Zt^{(n)})$ solves the system \eqref{eq:systemone} with gravity $g$ and we have $\sup_{n \in \Nsp, t \in [T_1, T_2]}\cbrac{ \Ecal_n(t) + \norm*[\big][2]{\Ztap^{(n)}}(t) } < \infty$. In addition if we fix the Lagrangian parameterizations for these solutions by imposing $h^{(n)}(\ap, \widetilde{T}) = h(\ap, \widetilde{T})$ for all $\ap \in \Rsp$ and $n \in \Nsp$, then $\sup_{t \in [T_1, T_2]}\Fcal_\Delta((\Z^{(n)}, \Zt^{(n)}), (\Z, \Zt))(t) \to 0$ as $n \to \infty$. 

\end{enumerate}
Furthermore we say that $(\U,\Psi, \Pfrak)(t)$ belongs in the class $\mathcal{SA}$ in the time interval $[0,T)$, if for any $0< T^* < T$ the solution belongs in the class $\mathcal{SA}$ in the time interval $[0, T^*]$. 
\end{definition}



We have the following relations between the different energies:
\begin{lemma}\label{lem:equivSobolevunbdd}
Let $T>0$ and let $(\Z, \Zt )(t)$ be a solution to the gravity water wave equation \eqref{eq:systemone} with gravity parameter $g \geq 0$ in the time interval $[0,T]$ with $(\Zap - 1, \frac{1}{\Zap} - 1, \Zt) \in \Linfty([0,T], H^s(\Rsp)\times H^s(\Rsp)\times H^{s + \half}(\Rsp))$ for some $s \geq 4$. Then we have the following:
\begin{enumerate}
\item There exists a universal constant $M_1>0$ such that $\E(t) \leq M_1\Ecal(t)$ for all $t \in [0,T]$. Moreover there exists a universal increasing function $C_1:[0,\infty) \to [0,\infty)$ such that for $g>0$ we have
\begin{align}\label{eq:controlEcalfromE}
\Ecal(t) \leq C_1\brac{\E(t) + g + \frac{1}{g} + \norm[2]{\Ztapbar}(t)} \quad \tx{for all } t\in [0,T]
\end{align}
and also
\begin{align}\label{eq:controlEcalfromEsup}
\sup_{t \in [0,T]}\Ecal(t) \leq C_1\brac{\sup_{t \in [0,T]}\E(t) + g +  \frac{1}{g} + T  + \norm[2]{\Ztapbar}(0)}
\end{align}


\item Let $c_1 = \norm[2]{\Ztap}^2(0) + g$ and assume that $c_1 > 0$. Then there exists a universal increasing function $C_2:[0,\infty) \to [0,\infty)$ so that for all $g \geq 0$ we have
\begin{align}\label{eq:SobolevlowerfromEgzero}
\begin{aligned}
&\sup_{t \in [0,T]} \cbrac{ \norm[H^1]{\Zt}(t) + \norm[H^1]{\frac{1}{\Zap} - 1} + \frac{1}{\norm[2]{\Ztap}^2(t) + g}} \\
&\leq C_2\brac{\sup_{t \in [0,T]}\E(t) + T + g + \frac{1}{c_1} + \norm[H^1]{\Zt}(0) + \norm[2]{\frac{1}{\Zap} - 1}(0)}
\end{aligned}
\end{align}

\item There exists a universal increasing function $C_3:[0,\infty) \to [0,\infty)$ such that for $g>0$ we have
\begin{align}\label{eq:SobolevfromEcal}
\begin{aligned}
& \sup_{t\in [0,T]}\cbrac{\norm[H^{2.5}]{\Zt}(t) + \norm[H^2]{\Zap - 1}(t) + \norm[H^2]{\frac{1}{\Zap} - 1}(t)} \\
& \leq C_3\brac{\sup_{t \in [0,T]}\E(t) + g + \frac{1}{g} + T + \norm[H^1]{\Zt}(0) + \norm[\infty]{\Zap}(0) + \norm[2]{\frac{1}{\Zap} - 1}(0) }
\end{aligned}
\end{align}

\end{enumerate}
\end{lemma}
\begin{proof}
\textbf{Step 1:} We first prove that $\E(t) \lesssim \Ecal(t)$. Note that $\Eone(t) \lesssim \Ecal(t)$ as it is the same as $\Ecalone(t)$. Now from \eqref{eq:DapZtbarandothers} we see that $\norm[\infty]{\Dap\Ztbar}^2 \lesssim \Ecal(t)$. Hence from \eqref{eq:DtpaponeoverZap} we see that 
\begin{align*}
\brac{\norm[2]{\Ztapbar} + \sqrt{g}}\norm[2]{\Dt\pap\frac{1}{\Zap}} \lesssim \Ecal(t)
\end{align*}
Now using \propref{prop:Hhalfweight} with $h = \sqrt{\Ag}$, $f = \pap\frac{1}{\Zap}$ and $w = \frac{1}{\Zap}$ we get
\begin{align*}
& \norm[\Hhalf]{\frac{\sqrt{\Ag}}{\Zap}\pap\frac{1}{\Zap}} \\
& \lesssim \norm[\infty]{\sqrt{\Ag}}\norm[\Hhalf]{\Dap\frac{1}{\Zap}} + \norm[2]{\pap\frac{1}{\Zap}}\norm[2]{\pap\brac{\frac{\sqrt{\Ag}}{\Zap}}} + \norm[\infty]{\sqrt{\Ag}}\norm[2]{\pap\frac{1}{\Zap}}^2 \\
& \lesssim \brac{\norm[2]{\Ztapbar} + \sqrt{g}}\norm[\Hhalf]{\Dap\frac{1}{\Zap}} + \norm[2]{\pap\frac{1}{\Zap}}B(t)
\end{align*}
Hence we see that $\Etwo(t)^\half \lesssim \Ecal(t)$. The same argument as above also shows that 
\begin{align*}
& \norm[\Hhalf]{\frac{\sqrt{\Ag}}{\Zap}\Dap^2\Ztbar} \\
& \lesssim \brac{\norm[2]{\Ztapbar} + \sqrt{g}}\norm[\Hhalf]{\frac{1}{\Zap}\Dap^2\Ztbar} + \norm[2]{\Dap^2\Ztbar}B(t)
\end{align*}
Now from \eqref{eq:DapDtDapZtbar}, \eqref{eq:Dap2Zttbar1} and \eqref{eq:Dap2Zttbar2} and from the argument there, it is easy to see that
\begin{align*}
& \norm[2]{\Dt\Dap^2\Ztbar} \\
& \lesssim \norm[2]{\Ag\Dap^2\frac{1}{\Zap}} + \Ea(t)^\half\brac{\norm[2]{\Dt\brac{\pap\frac{1}{\Zap}}} + \norm[\Hhalf]{\frac{\sqrt{\Ag}}{\Zap}\pap\frac{1}{\Zap}}} + \Ea(t)\norm[2]{\pap\frac{1}{\Zap}} \\
 & \lesssim \brac{\norm[2]{\Ztapbar}^2 + g}\norm[2]{\Dap^2\frac{1}{\Zap}} + \Ea(t)^\half\brac{\norm[2]{\Dt\brac{\pap\frac{1}{\Zap}}} + \norm[\Hhalf]{\frac{\sqrt{\Ag}}{\Zap}\pap\frac{1}{\Zap}}} \\
 & \quad + \Ea(t)\norm[2]{\pap\frac{1}{\Zap}} 
\end{align*}
Now from this and previously proved estimates we therefore get $\Ethree(t)^\half \lesssim \Ecal(t)^\threebytwo$. Hence proved. 

\textbf{Step 2:} We now prove \eqref{eq:controlEcalfromE} and \eqref{eq:controlEcalfromEsup}. Clearly $\Ecalone(t) \lesssim \E(t)$ as it is the same as $\Eone(t)$. Now from \eqref{eq:Dap2Ztbarandmore} we see that
\begin{align*}
\brac{\norm[2]{\Ztapbar} + \sqrt{g}}\norm[2]{\Dap^2\Ztbar} \lesssim \Ea(t)^\half \lesssim \E(t)^\half
\end{align*}
Using \propref{prop:Hhalfweight} with $h = \frac{1}{\sqrt{\Ag}}$, $w = \frac{1}{\Zap}$ and $f = \sqrt{\Ag}\pap\frac{1}{\Zap}$ we get
\begingroup
\allowdisplaybreaks
\begin{align*}
& \norm[\Hhalf]{\Dap\frac{1}{\Zap}} \\
& \lesssim \norm[\infty]{\frac{1}{\sqrt{\Ag}}}\norm[\Hhalf]{\frac{\sqrt{\Ag}}{\Zap}\pap\frac{1}{\Zap}} + \norm[2]{\pap\brac{\frac{1}{\Zap\sqrt{\Ag}}}}\norm[2]{\sqrt{\Ag}\pap\frac{1}{\Zap}} \\*
& \quad + \norm[\infty]{\frac{1}{\sqrt{\Ag}}}\norm[2]{\pap\frac{1}{\Zap}}^2\norm[\infty]{\sqrt{\Ag}} \\
& \lesssim \frac{1}{\sqrt{g}}\norm[\Hhalf]{\frac{\sqrt{\Ag}}{\Zap}\pap\frac{1}{\Zap}} + \frac{1}{\sqrt{g}}\brac{\norm[2]{\Ztapbar} + \sqrt{g}}\norm[2]{\pap\frac{1}{\Zap}}^2  + \frac{1}{g}B^2(t)
\end{align*}
\endgroup
Hence we see that $\Ecaltwo(t)^\half \lesssim C_1(\E(t) + g + \frac{1}{g} + \norm[2]{\Ztapbar}(t))$. Now by using the same proof as above we also get the estimate
\begin{align*}
& \norm[\Hhalf]{\frac{1}{\Zap}\Dap^2\Ztbar} \\
& \lesssim \frac{1}{\sqrt{g}}\norm[\Hhalf]{\frac{\sqrt{\Ag}}{\Zap}\Dap^2\Ztbar} + \frac{1}{\sqrt{g}}\brac{\norm[2]{\Ztapbar} + \sqrt{g}}\norm[2]{\pap\frac{1}{\Zap}}\norm[2]{\Dap^2\Ztbar} \\
& \quad + \frac{1}{g} B(t)\brac{\norm[2]{\Ztapbar} + \sqrt{g}}\norm[2]{\Dap^2\Ztbar}
\end{align*}
Hence $\brac{\norm[2]{\Ztapbar} + \sqrt{g}} \norm[\Hhalf]{\frac{1}{\Zap}\Dap^2\Ztbar} \lesssim C_1(\E(t) + g + \frac{1}{g} + \norm[2]{\Ztapbar}(t))$. Finally note that we have already proved the estimate
\begin{align*}
& \norm[2]{\Ag\Dap^2\frac{1}{\Zap}} \\
& \lesssim  \norm[2]{\Dt\Dap^2\Ztbar} + \Ea(t)^\half\brac{\norm[2]{\Dt\brac{\pap\frac{1}{\Zap}}} + \norm[\Hhalf]{\frac{\sqrt{\Ag}}{\Zap}\pap\frac{1}{\Zap}}} + \Ea(t)\norm[2]{\pap\frac{1}{\Zap}}
\end{align*}
From this we easily see that 
\begin{align*}
\brac{\norm[2]{\Ztapbar} + \sqrt{g}}\norm[2]{\Dap^2\frac{1}{\Zap}} \lesssim \frac{1}{g}\E(t)^\threebytwo
\end{align*}
Hence \eqref{eq:controlEcalfromE} now follows. To prove \eqref{eq:controlEcalfromEsup}, we first observe from \lemref{lem:timederiv} that
\begin{align}\label{eq:ddtZtapbarLtwo}
\begin{aligned}
 \frac{\diff}{\diff t} \brac{\norm[2]{\Ztapbar}^2 + g} 
& \lesssim \norm[\infty]{\bvarap}\norm[2]{\Ztapbar}^2 + \norm[2]{\Ztapbar}\norm[2]{\Zttapbar} \\
& \lesssim B(t)  \brac{\norm[2]{\Ztapbar}^2 + g} \\
& \lesssim \Ea(t)^\half  \brac{\norm[2]{\Ztapbar}^2 + g}
\end{aligned}
\end{align}
Hence there exists a universal constant $C_4>0$ such that
\begin{align}\label{eq:ZtapbarL2norm2plusgestimate}
\begin{aligned}
& \exp\cbrac{-C_4\int_0^t \Ea(s)^\half \diff s} \brac{\norm[2]{\Ztapbar}^2(0) + g} \\
& \leq \norm[2]{\Ztapbar}^2(t) + g \\
& \leq \exp\cbrac{C_4\int_0^t \Ea(s)^\half \diff s} \brac{\norm[2]{\Ztapbar}^2(0) + g}
\end{aligned}
\end{align}
Hence combining this with \eqref{eq:controlEcalfromE} proves \eqref{eq:controlEcalfromEsup}.

\textbf{Step 3:} We now prove \eqref{eq:SobolevlowerfromEgzero}. Let 
\begin{align*}
L = \sup_{t \in [0,T]}\E(t) + T + g + \frac{1}{c_1} + \norm[H^1]{\Zt}(0) + \norm[2]{\frac{1}{\Zap} - 1}(0)
\end{align*}
Now from \eqref{eq:ZtapbarL2norm2plusgestimate} we see that
\begin{align*}
\norm[2]{\Ztapbar}^2(t) + g + \frac{1}{\norm[2]{\Ztapbar}^2(t) + g} \lesssim_L 1
\end{align*}
Hence we see that 
\begin{align*}
\sup_{t \in [0,T]}\cbrac{\norm[2]{\Ztapbar}(t) + \norm[\infty]{\Dap\Ztbar}(t) + \norm[2]{\pap\frac{1}{\Zap}}(t)} \lesssim_L 1
\end{align*}
Now let 
\begin{align*}
f(t) = \norm[2]{\frac{1}{\Zap} - 1}^2 + \norm[2]{\Zt}^2 + 1
\end{align*}
Therefore we see that $f(0) \lesssim_L 1$. Now by following the proof of Lemma 6.2 in \cite{Ag21}, we see that $\pt f \lesssim_L f$. Hence \eqref{eq:SobolevlowerfromEgzero} therefore follows. 

\textbf{Step 4:} We now prove \eqref{eq:SobolevfromEcal}. Let 
\begin{align*}
M = \sup_{t \in [0,T]}\E(t) + g + \frac{1}{g} + T + \norm[H^1]{\Zt}(0) + \norm[\infty]{\Zap}(0) + \norm[2]{\frac{1}{\Zap} - 1}(0)
\end{align*}
As $g > 0$, by using the definition of $\Ecal(t)$ and \eqref{eq:controlEcalfromE} we therefore see that
\begin{align}\label{eq:tempestimateCM1}
 \sup_{t \in [0,T]}\cbrac{\norm[2]{\pap\frac{1}{\Zap}}^2(t) + \norm[2]{\Dap^2\Ztbar}^2(t) + \norm[2]{\Dap^2\frac{1}{\Zap}}^2 + \norm[\Hhalf]{\frac{1}{\Zap}\Dap^2\Ztbar}^2 }  \lesssim_M 1
\end{align}
Now we have 
\begin{align*}
\norm[\infty]{\Dap\frac{1}{\Zap}} \lesssim \norm[\infty]{\frac{1}{\Ag}}\norm[\infty]{\Ag\Dap\frac{1}{\Zap}} \lesssim \frac{1}{g}\E(t)
\end{align*}
Hence 
\begin{align*}
\sup_{t \in [0,T]} \norm[\infty]{\Dap\frac{1}{\Zap}} \lesssim_M 1
\end{align*}
Now observe that
\begin{align*}
(\pt + \bvar\pap)\Zap = \Dt\Zap = \Ztap - \bvarap\Zap = \Zap \brac{\Dap\Zt - \bvarap}
\end{align*}
As $\norm[\infty]{\Dap\Zt} + \norm[\infty]{\bvarap} \lesssim \Ea(t)^\half$, we see that for all $t \in [0,T]$ we have
\begin{align*}
\norm[\infty]{\Zap}(0) \lesssim_M \norm[\infty]{\Zap}(t) \lesssim_M \norm[\infty]{\Zap}(0)
\end{align*}
Combining this with \eqref{eq:tempestimateCM1} we easily get
\begin{align*}
\sup_{t \in [0,T]} \cbrac{\norm[H^1]{\pap\frac{1}{\Zap}}(t) + \norm[H^1]{\pap\Zap}(t) + \norm[H^1]{\Ztapbar}(t)} \lesssim_M 1
\end{align*}
Now observe that
\begin{align*}
\frac{1}{\Zap}\Dap^2\Ztbar = (\Dap\Ztbar)\Dap\frac{1}{\Zap} + \frac{1}{\Zap^3}\pap\Ztapbar
\end{align*}
From \propref{prop:Leibniz} we clearly have the estimate
\begin{align*}
\norm[\Hhalf]{ (\Dap\Ztbar)\Dap\frac{1}{\Zap}} \lesssim \norm[\Linfty\cap\Hhalf]{\Dap\Ztbar}\norm[\Linfty\cap\Hhalf]{\Dap\frac{1}{\Zap}} \lesssim_M 1  
\end{align*}
This implies that for all $t \in [0,T]$
\begin{align*}
\norm[\Hhalf]{\frac{1}{\Zap^3}\pap\Ztapbar}(t) \lesssim_M 1
\end{align*}
Hence again by using \propref{prop:Leibniz} we get for all $t \in [0,T]$
\begin{align*}
\norm[\Hhalf]{\pap\Ztapbar}(t) &\lesssim \norm[\infty]{\Zap}^3(t)\norm[\Hhalf]{\frac{1}{\Zap^3}\pap\Ztapbar}(t) +  \norm[\infty]{\Zap}^2(t)\norm[2]{\pap\Zap}(t)\norm[2]{\pap\Ztapbar}(t) \\
& \lesssim_M 1
\end{align*}
So we are left to prove only the lower order estimates, namely proving the estimate $\norm[2]{\Zt}(t) + \norm[2]{\frac{1}{\Zap} - 1}(t) + \norm[2]{\Zap - 1}(t) \lesssim_M 1$ for $t \in [0,T]$. This directly follows from \eqref{eq:SobolevlowerfromEgzero} as $L \leq M$ and from the fact that $\norm[\infty]{\Zap}(t) \lesssim_M 1$. Hence proved. 

\end{proof}

We now note down an existence result in regular enough Sobolev spaces.
\begin{thm}[\cite{Wu97}]\label{thm:existenceSobolevunbdd}
Let $g > 0$ and let $s \geq 4$. Assume that the initial data $(\Z,\Zt)(0)$ satisfies $(\Zap - 1, \frac{1}{\Zap} - 1, \Zt)(0) \in H^{s}(\Rsp)\times H^s(\Rsp) \times H^{s + \half}(\Rsp)$. Then there exists a $T > 0$ such that on $[0,T]$ the initial value problem for \eqref{eq:systemone} has a unique solution $(\Z,\Zt)(t)$ satisfying $(\Zap - 1, \frac{1}{\Zap} - 1, \Zt) \in C^l([0,T], H^{s - l}(\Rsp)\times H^{s - l}(\Rsp) \times H^{s + \half - l}(\Rsp))$ for $l = 0,1$. Moreover if $T^{*}$ is the maximal time of existence, then either $T^* = \infty$ or $T^{*} < \infty$ and 
\begin{align*}
\sup_{t \in [0,T^*)} \cbrac{\norm[H^2]{\Zap - 1}(t) + \norm[H^2]{\frac{1}{\Zap} - 1}(t) + \norm[H^{2 + \half}]{\Zt}(t) } = \infty
\end{align*}
\end{thm}
\begin{proof}
The existence part of the result is just a reformulation of Theorem 5.11 of \cite{Wu97}. The result written there is for $g = 1$ but the same result holds for $g>0$. Also the assumptions on the initial data are given there in terms of $\Zt$ and $\Ztt$, but from \eqref{form:Zttbar} we see that
\begin{align*}
\Zttbar = -i\brac{\frac{1}{\Zap} - 1}\Ag -i\Ag + ig
\end{align*} 
Now from \eqref{eq:systemone} we see that $\Ag = \g - \Imag \sqbrac{\Zt,\Hil}\Ztapbar \geq g >0$. From this it is easy to see that $(\Zap - 1, \frac{1}{\Zap} - 1, \Zt)(0) \in H^{s}(\Rsp)\times H^s(\Rsp) \times H^{s + \half}(\Rsp)$ is equivalent to $(\Zt,\Ztt)(0) \in H^{s + \half}(\Rsp) \times H^s(\Rsp)$ along with $-\frac{\partial P}{\partial \hat{n}} \compose \hinv(0) \geq  c >0$ for some $c>0$ (Recall from \eqref{eq:gradientpressure} that $-\frac{\partial P}{\partial \hat{n}} \compose \hinv  = \frac{\Ag}{\Zapabs}$). The blow up criterion written here is not directly implied by Theorem 5.11 of \cite{Wu97} as there the blow up criterion is in terms of one higher derivative. However by modifying the proof there one easily gets the above blow up criterion. Another way to see the blow up criterion is that it is implied by the blowup criterion of Theorem 3.6 in \cite{Wu19}. 
\end{proof}

We are now ready to prove \corref{cor:uniftimeSobolev} and \thmref{thm:existencemain}. 

\begin{proof}[Proof of \corref{cor:uniftimeSobolev}]

As the initial data $(\Z,\Zt)(0)$ satisfies $(\Zap - 1, \frac{1}{\Zap} - 1, \Zt)(0) \in H^{s}(\Rsp)\times H^s(\Rsp) \times H^{s + \half}(\Rsp)$ for some $s \geq 4$, by \thmref{thm:existenceSobolevunbdd} there exists a time $T_g>0$ depending on the gravity parameter $g$ such that the we have a unique solution $(\Z,\Zt)(t)$ in the time interval $[0,T_g]$ to the initial value problem for \eqref{eq:systemone} with the solution satisfying $(\Zap - 1, \frac{1}{\Zap} - 1, \Zt) \in C^l([0,T], H^{s - l}(\Rsp)\times H^{s - l}(\Rsp) \times H^{s + \half - l}(\Rsp))$ for $l = 0,1$. Moreover again by \thmref{thm:existenceSobolevunbdd}, if $T^*_g>0$ is the maximal time of existence, then either $T^*_g = \infty$ or $T^*_g < \infty$ and 
\begin{align*}
\sup_{t \in [0,T^*_g)} \cbrac{\norm[H^2]{\Zap - 1}(t) + \norm[H^2]{\frac{1}{\Zap} - 1}(t) + \norm[H^{2 + \half}]{\Zt}(t) } = \infty
\end{align*}
Now recall that the energy $\Ecal(t)$ defined in \secref{sec:mainresults} for smooth enough data was given by
\begin{align*}
\begin{aligned}
\Ecalone(t) & = \brac{\norm[2]{\Ztapbar}^2 + g}\norm[2]{\pap\frac{1}{\Zap}}^2\\
\Ecaltwo(t) & =  \brac{\norm[2]{\Ztapbar}^2 + g}\norm[2]{\Dap^2\Ztbar}^2 + \brac{\norm[2]{\Ztapbar}^2 + g}^2\norm[\Hhalf]{\Dap\frac{1}{\Zap}}^2 \\
\Ecalthree(t) & = \brac{\norm[2]{\Ztapbar}^2 + g}^3\norm[2]{\Dap^2\frac{1}{\Zap}}^2 +  \brac{\norm[2]{\Ztapbar}^2 + g}^2\norm[\Hhalf]{\frac{1}{\Zap}\Dap^2\Ztbar}^2
\end{aligned}
\end{align*}
and $\Ecal(t)  = \brac{\Ecalone(t)^3 + \Ecaltwo(t)^{\frac{3}{2}} + \Ecalthree(t)}^{\frac{1}{3}}$. From the assumptions of our initial data, we see that there exists a constant $L>0$ such that $\Ecal(0) \leq L$ for all $0< g\leq g_0$. Hence by the first part of \lemref{lem:equivSobolevunbdd} we see that $E(0) \leq M_1L$ for all $0< g\leq g_0$, where $M_1>0$ is a universal constant. Now by \eqref{eq:mainEtime} of \thmref{thm:aprioriE} we know that for all $t \in [0,T^*_g)$ we have
\begin{align*}
\frac{d\E(t)}{dt} \lesssim \E(t)^{3/2}
\end{align*}
As $E(0) \leq M_1L$ for all $0< g\leq g_0$, therefore there exists $M ,T >0$ both independent of $g$ so that for all $0< g \leq g_0$ we have
\begin{align*}
\sup_{t \in [0, \min\cbrac{T, T^*_g})} E(t) \leq M
\end{align*}
Hence by \eqref{eq:SobolevfromEcal} of \lemref{lem:equivSobolevunbdd}, we see that any fixed $0< g \leq g_0$ we have
\begin{align*}
\sup_{t \in [0,\min\cbrac{T, T^*_g})} \cbrac{\norm[H^2]{\Zap - 1}(t) + \norm[H^2]{\frac{1}{\Zap} - 1}(t) + \norm[H^{2 + \half}]{\Zt}(t) } < \infty
\end{align*}
This implies that $T^*_g \geq T$ thereby completing the proof of \corref{cor:uniftimeSobolev}. 

\end{proof}

\begin{proof}[Proof of \thmref{thm:existencemain}]
Let $0<\ep < 1$. Define 
\begin{align*}
\Psiep(\zp,0) = \Psi(\zp - i\ep, 0), \quad \Uep(\zp,0) = \U(\zp - i\ep,0) \quad \tx{ and }  \quad \hep(\al, 0) = \al 
\end{align*}
As $\U(\cdot,0)$ and $\Psi(\cdot,0)$ are holomorphic functions in $\Pminus$, we see that the functions $\Zep(\cdot,0)$ and $\Ztbarep(\cdot,0)$ defined by
\begin{align*}
\Zep(\ap,0) = \Psiep(\ap,0) \quad \tx{ and } \quad \Ztepbar(\ap,0) = \Uep(\ap,0)
\end{align*}
are smooth functions on $\Rsp$ for all $0<\ep < 1$. Moreover as $\Psizp \neq 0$ in $\Pminus$, we see that $\Zapep(\ap,0) \neq 0$ for all $\ap \in \Rsp$ and $0<\ep < 1$. Also from the initial data it is clear that for all $0<\ep < 1$ we have $\Zapep(\ap,0) \to 1$ and $\frac{1}{\Zapep}(\ap,0) \to 1$ as $\abs{\ap} \to \infty$. We also see from the initial data \eqref{eq:czero} that for all $0<\ep < 1$ we have
\begin{align}\label{eq:ZtZlowerepest}
\norm[H^1]{\Ztbarep}(0) + \norm[H^1]{\frac{1}{\Zapep} - 1}(0) \leq c_0
\end{align}

\medskip
\noindent\textbf{Step 1:} We first consider the case of $g>0$. We observe that $(\Zep, \Ztep)(0)$ satisfies the assumptions of \thmref{thm:existenceSobolevunbdd} and therefore there exists a time $T_\ep >0$ such that the initial value problem to \eqref{eq:systemone} has a unique smooth solution $(\Zep, \Ztep)(t)$ in the time interval $[0,T_\ep]$ so that for all $s \geq 4$ we have 
\begin{align}\label{eq:ZapZtSobolevlocalextest}
\sup_{t \in [0,T_\ep]} \cbrac{\norm[H^s]{\Zapep - 1}(t) + \norm[H^s]{\frac{1}{\Zapep} - 1}(t) + \norm[H^{s + \half}]{\Ztep}(t) } < \infty
\end{align}
Now from the definition of $\Ecal(t)$, it is clear that $\Ecal(\Zep, \Ztep)(0) \leq \Ecal(0)$. Also from \lemref{lem:equivSobolevunbdd} there exists a universal constant $M>0$ such that $\E(\Zep,\Ztep)(0) \leq M \Ecal(\Zep, \Ztep)(0) \leq M \Ecal(0)$. Hence from the same argument as in \corref{cor:uniftimeSobolev} we see that there exists $T, C_1>0$ depending only on $\Ecal(0)$ with $T\gtrsim \frac{1}{\sqrt{\Ecal(0)}}$ so that the smooth solution $(\Zep, \Ztep)(t)$ exists in $[0,T]$ and we have $\sup_{t \in [0,T]}\E(\Zep,\Ztep)(t) \leq C_1$. Now from \eqref{eq:controlEcalfromEsup} we see that there exists a constant $C_2$ depending only on $\Ecal(0), c_0$ and $g$ so that $\sup_{t \in [0,T]} \Ecal(\Zep,\Ztep)(t) \leq C_2$. Let $\Dapep = \frac{1}{\Zapep}\pap$ and $\abs{\Dapep} = \frac{1}{\abs*{\Zapep}}\pap$. Now from \eqref{eq:ZtapbarL2norm2plusgestimate} and the definition of $\Ecal(t)$ we see that there exists a constant $C_3>0$ depending only on $\Ecal(0), c_0$ and $g$ so that for all $t \in [0,T]$ we have
\begin{align*}
C_3 & \geq \norm[2]{\Ztapep}^2(t) + \norm[2]{\pap\frac{1}{\Zapep}}^2(t) + \norm[2]{(\Dapep)^2 \Ztbarep}^2(t) + \norm[2]{(\Dapep)^2\frac{1}{\Zapep}}^2(t) \\
& \quad + \norm[\Hhalf]{\frac{1}{\Zapep}(\Dapep)^2\Ztbarep}^2(t) 
\end{align*}
By using $\Hil(\Dapep \Ztbarep) = \Dapep \Ztbarep$ we also get
\begin{align*}
\norm[\Hhalf]{\Dapep\Ztbarep}^2 = i\int (\Dapep\Ztbarep)\brac{\pap\Dapep \Ztbarep} \diff\ap \leq \norm[2]{\Ztapep} \norm[2]{(\Dapep)^2\Ztbarep} \leq C_3
\end{align*}
Now as $g>0$, we see from \eqref{eq:SobolevlowerfromEgzero} that there exists a constant $C_4$ depending only on $\Ecal(0)$, $c_0$ and $g$ so that for all $t\in [0,T]$ we have
\begin{align*}
\norm[2]{\Ztep}(t) + \norm[2]{\frac{1}{\Zapep} - 1}(t) \leq C_4
\end{align*}
Hence we can now simply follow the proof of Theorem 3.9 of \cite{Wu19} to see that we have a solution $(\U,\Psi,\Pfrak)(t)$ to \eqref{eq:EulerRiem} in the time interval $[0,T]$ with $\sup_{t \in [0,T]}\Ecal(t) \leq C_2$. From the proof of Theorem 3.9 of \cite{Wu19}, we easily see that this solution satisfies essentially all the properties of the class $\mathcal{SA}$, except the final condition $(d)$ in \defref{def:solutionSA}. Now by construction we see that $\Fcal_\Delta((\Z^{(\ep)}, \Zt^{(\ep)}), (\Z, \Zt))(0) \to 0$ as $\ep \to 0$, and hence now by \thmref{thm:Wudiffenergyestimate}, we see that  $\sup_{t \in [T_1, T_2]}\Fcal_\Delta((\Z^{(\ep)}, \Zt^{(\ep)}), (\Z, \Zt))(t) \to 0$ as $\ep \to 0$. Hence the solution $(\U,\Psi,\Pfrak)(t)$ lies in the class $\mathcal{SA}$. 

For uniqueness, suppose we have two distinct solutions $(\U^a,\Psi^a, \Pfrak^a)(t)$ and $(\U^b,\Psi^b, \Pfrak^b)(t)$ in the class $\mathcal{SA}$ with the same initial data. Define the time
\begin{align*}
T_u = \inf\cbrac{t \in [0,T] \suchthat \exists \ap \in \Rsp, \tx{ such that } (\Z^a,\Zt^a)(\ap, t) \neq (\Z^b,\Zt^b)(\ap, t)}
\end{align*}
By the continuity of the functions $\Z, \Zt$ we see that $T_u \in [0, T)$ and that the two solutions are equal at time $T_u$.  Now from the definition of the class $\mathcal{SA}$, we see that there exists $T_u < T_2\leq T$ such that on the interval $[T_u, T_2]$ there exists smooth solutions $(\Z^{a, (n)}, \Zt^{a, (n)})$ and $(\Z^{b, (n)}, \Zt^{b, (n)})$ converging to $(\Z^a,\Zt^a)$ and $(\Z^b,\Zt^b)$ respectively in the norm of $\Fcal_\Delta$. Now observe that if a sequence of smooth solutions $(\Z^{(n)}, \Zt^{(n)})$ solves the system \eqref{eq:systemone} with gravity $g>0$ in the time interval $[T_1, T_2]$,  $\sup_{n \in \Nsp, t \in [T_1, T_2]}\cbrac{ \Ecal_n(t) + \norm*[\big][2]{\Ztap^{(n)}}(t) } \leq C$, then by using the definition of the energy $\Ecal(t)$ and the argument used above, we see that for all $n \in \Nsp$ and $t \in [T_1, T_2]$ we have
\begin{align*}
1 & \gtrsim_{C, g}  \norm[2]{\Ztap^{(n)}}^2(t) + \norm[2]{\pap\frac{1}{\Zap^{(n)}}}^2(t) + \norm[2]{(\Dap^{(n)})^2 \Ztbar^{(n)}}^2(t) + \norm[2]{(\Dap^{(n)})^2\frac{1}{\Zap^{(n)}}}^2(t) \\
& \quad + \norm[\Hhalf]{\Dap^{(n)}\Ztbar^{(n)}}^2(t) +  \norm[\Hhalf]{\frac{1}{\Zap^{(n)}}(\Dap^{(n)})^2\Ztbar^{(n)}}^2(t) 
\end{align*}
where $\Dap^{(n)} = \frac{1}{\Zap^{(n)}}\pap$. Hence we can now directly apply \thmref{thm:Wudiffenergyestimate} to see that we have $\sup_{t \in [T_u, T_2]}\Fcal_\Delta((\Z^{a, (n)}, \Zt^{a, (n)}), (\Z^{b, (n)}, \Zt^{b, (n)}))(t) \to 0$ as $n \to \infty$ thereby proving that $(\U^a,\Psi^a, \Pfrak^a)(t) = (\U^b,\Psi^b, \Pfrak^b)(t)$ for all $t \in [T_u , T_2]$. This proves uniqueness.

\medskip
\noindent\textbf{Step 2:} Let us now consider the case of $g = 0$. Observe that if $\norm[2]{\Ztapbar}(0) = 0$, then as $\Zt(0) \in H^1(\Rsp)$, we see that $\Zt(\ap ,0 ) = 0$ for all $\ap \in \Rsp$. As $g = 0$, this implies that the trivial solution namely $\Z(\ap, t) = \Z(\ap,0)$ and $\Zt(\ap,t) = 0$ for all $\ap \in \Rsp$, $t \in [0, \infty)$, is a global solution to the water wave equation. Observe that in this case we have $\Ecal(t) = 0$ for all $t \geq 0$. Hence assume that $\norm[2]{\Ztapbar}(0) = c_2 >0$. We mollify the initial data in the same way as in the $g>0$ case and so we still have \eqref{eq:ZtZlowerepest}. It is clear that $\norm*[\big][2]{\Ztapepbar}(0) \to c_2$ as $\ep \to 0$, and so let $0<\ep_0 \leq \min\cbrac{c_2^2, 1} $ be small enough so that for all $0<\ep\leq \ep_0$ we have $c_2/2 \leq \norm*[\big][2]{\Ztapepbar}(0) \leq c_2$. Let $g_\ep = \ep$. Hence for the initial data $(\Zep, \Ztep)(0)$ we have a unique smooth solution to  \eqref{eq:systemone} with gravity $g_\ep$ in a time interval $T_\ep>0$ satisfying \eqref{eq:ZapZtSobolevlocalextest}. By the choice of $\ep_0$ and $g_\ep$, it is clear that  for all $0<\ep \leq \ep_0$ we have
\begin{align*}
\frac{c_2^2}{4} \leq \norm[2]{\Ztapepbar}^2(0) + g_\ep \leq 2c_2^2
\end{align*}
and hence $\Ecal(\Zep, \Ztep)(0) \lesssim \Ecal(0)$. Hence from the same argument as above for $g > 0$, we see that there exists $T, C_5>0$ depending only on $\Ecal(0)$ with $T\gtrsim \frac{1}{\sqrt{\Ecal(0)}}$ so that the smooth solution $(\Zep, \Ztep)(t)$ exists in $[0,T]$ and we have $\sup_{t \in [0,T]}\E(\Zep,\Ztep)(t) \leq C_5$. 

Now from  \eqref{eq:SobolevlowerfromEgzero} we see that for all $t \in [0,T]$ and $0<\ep \leq \ep_0$ we have
\begin{align}\label{eq:Ztlowerepczeroztwo}
\norm[H^1]{\Ztep}(t) + g_\ep + \norm[H^1]{\frac{1}{\Zapep} - 1}(t) + \frac{1}{\norm*[\big][2]{\Ztapepbar}^2(t) + g_\ep}  \lesssim_{c_0, c_2, \Ecal(0)} 1
\end{align}
Let $\bvar^\ep$ and $\Ag^\ep$ be given by \eqref{eq:systemone} with $\Zt$, $\Zap$ and $g$ replaced by $\Ztep$, $\Zapep$ and $g_\ep$ respectively. Let $\Dt^\ep = \pt + \bvar^\ep\pap$.  From the quantities controlled in \secref{sec:quantEa} and we see that for all $t \in [0,T]$ and $0<\ep\leq \ep_0$ we have
\begin{align*}
\norm[\infty]{\Ag^\ep}(t) + \norm[\infty]{\Dapep\Ztep}(t) + \norm[\infty]{\bap^\ep}(t) + \norm[\infty]{\frac{\Dt^\ep\Ag^\ep}{\Ag^\ep}}(t) + \norm[2]{\frac{1}{\sqrt{\Ag^\ep}}\abs{\Dap^\ep}\Ag^\ep }(t) \lesssim_{c_2, \Ecal(0)} 1 
\end{align*}
Now from \eqref{form:DtoneoverZap} we see that 
\begin{align*}
\norm[\infty]{\Dt^\ep \frac{1}{\Zapep}} \lesssim \norm[\infty]{\frac{1}{\Zapep}}\brac{\norm[\infty]{\Dap^\ep\Ztep} + \norm[\infty]{\bvarap^\ep}}
\end{align*}
and from \eqref{form:Zttbar} we obtain
\begin{align*}
\norm[\infty]{\Dt^\ep \Zttep} \lesssim \norm[\infty]{\Ag^\ep}\norm[\infty]{\Dt^\ep \frac{1}{\Zapep}} + \norm[\infty]{\frac{1}{\Zapep}}\norm[\infty]{\Ag^\ep} \norm[\infty]{\frac{\Dt^\ep\Ag^\ep}{\Ag^\ep}}
\end{align*}
Hence we see that for all $t \in [0,T]$ and $0<\ep \leq \ep_0$ we have
\begin{align*}
1 & \gtrsim_{c_0, c_2, \Ecal(0)} \norm[\infty]{\Ztep}(t) + \norm[2]{\Ztapep}(t) + \norm[\infty]{\Zttep}(t) + \norm[2]{\Zttap^\ep}(t) + \norm[\infty]{\Dt^\ep\Zttep}(t)  \\
& \qquad + \norm[\infty]{\frac{1}{\Zapep}}(t) + \norm[2]{\pap\frac{1}{\Zapep}}(t) + \norm[\infty]{\Dt^\ep\frac{1}{\Zapep}}(t)
\end{align*}
Now from \eqref{eq:systemone} it is easy to see that 
\begin{align*}
\norm[H^1]{\bvar^\ep} \lesssim \norm[H^1]{\Ztep}\norm[H^1]{\frac{1}{\Zapep} - 1}
\end{align*}
From \eqref{eq:systemone}, the relation $\Hil \Zt = - \Zt$ and the fact that $\frac{1}{\Zapep} \to 1$ as $\abs{\ap} \to \infty$, we see that
\begin{align*}
\bvar^\ep = \Real\sqbrac{\Ztep,\Hil}\brac*[\bigg]{\frac{1}{\Zapep} - 1} + 2\Real(\Ztep)
\end{align*}
Now using \propref{prop:tripleidentity} we see that
\begin{align*}
\norm[\infty]{\Dt^\ep\bvar^\ep} & \lesssim \norm[\infty]{\Zttep} + \norm[\infty]{\sqbrac{\Zttep, \Hil}\brac*[\bigg]{\frac{1}{\Zapep} - 1} } + \norm[\infty]{\sqbrac{\Ztep, \Hil}\pap\cbrac{\bvar^\ep\brac*[\bigg]{\frac{1}{\Zapep} - 1} }} \\
& \quad + \norm[\infty]{\sqbrac{\bvar^\ep, \Ztep ; \brac*[\bigg]{\frac{1}{\Zapep} - 1}}}
\end{align*}
Therefore  from \propref{prop:commutator} and \propref{prop:triple} we obtain
\begin{align*}
\norm[\infty]{\Dt^\ep\bvar^\ep} & \lesssim \norm[\infty]{\Zttep} + \norm[2]{\Zttap^\ep}\norm[2]{\frac{1}{\Zapep} - 1} + \norm[2]{\Ztap^\ep}\norm[H^1]{\bvarap^\ep}\norm[H^1]{\frac{1}{\Zapep} - 1} 
\end{align*}
Hence we see that for all $t \in [0,T]$ and $0<\ep \leq \ep_0$ we have
\begin{align*}
 \norm[\infty]{\bvar^\ep}(t) + \norm[2]{\bvarap^\ep}(t) + \norm[\infty]{\Dt^\ep\bvar^\ep}(t) \lesssim_{c_0, c_2, \Ecal(0)} 1
\end{align*}
Now we can follow the proof of Theorem 3.9 of \cite{Wu19} to see that there exists a solution $(\U,\Psi,\Pfrak)(t)$ to \eqref{eq:EulerRiem} with zero gravity in the sense of \defref{def:solution}. 

\medskip
\noindent\textbf{Step 3:} In Step 3 and Step 4 we prove the blowup criterion for $g>0$. Let $(\U,\Psi, \Pfrak)(t)$ be a solution to \eqref{eq:EulerRiem} with gravity $g$ in $[0,T]$ in the class $\mathcal{SA}$ with $\sup_{t \in [0,T]}\Ecal(t) < \infty$. To prove the blowup criterion we need to prove a version of \eqref{eq:mainEatime}, so first we need to be able to make sense of the energy $\Ea(t)$ for this singular solution. Fix $t \in [0,T]$ and as usual let  $\Ztbar(\ap,t) = \U(\ap,t)$ and $\frac{1}{\Zap}(\ap,t) = \frac{1}{\Psizp}(\ap,t)$ be the boundary values of $\U$ and $\frac{1}{\Psizp}$. From the definition of the class  $\mathcal{SA}$, we see that $(\Ztbar, \frac{1}{\Zap} - 1)(\cdot,t) \in H^1(\Rsp)\times H^1(\Rsp)$. Hence all quantities in $\Ea(t)$ except for $\Dt\pap\frac{1}{\Zap}$ are well defined. We now give an equivalent definition of $\Dt\pap\frac{1}{\Zap}$ which makes sense for this singular solution and which agrees with the standard definition for smooth solutions. 

For smooth solutions, a straightforward calculation using \eqref{form:DtoneoverZap} gives us
\begin{align}\label{eq:DtpapfraconeoverZapnew}
\begin{aligned}
\Dt \pap\frac{1}{\Zap} & = -\wbar^2\brac{\pap\frac{1}{\Zap}}\Dapbar\Zt -2\w^2\brac{\pap\frac{1}{\Zap}}\Dap\Ztbar + \brac{\pap\frac{1}{\Zapbar}}\Dap\Ztbar \\
& \quad \wbar\Dapabs\brac{\bvarap - \Dap\Zt - \Dapbar\Ztbar} + \w^2\pap\brac{\frac{1}{\Zap}\Dap\Ztbar}
\end{aligned}
\end{align}
where from \eqref{eq:DapabsbapDapZtDapbarZtbar} and \eqref{eq:IdminusHDapbvarapDapZt} we see that
\begin{align}\label{eq:DapabsbvarDapZtnew}
\begin{aligned}
& \Dapabs\brac{\bvarap - \Dap\Zt - \Dapbar\Ztbar} \\
& = \Real \cbrac{\w(\Id - \Hil)\Dap(\bvarap -\Dap\Zt - \Dapbar\Ztbar)}  \\
& \quad  -\Real\cbrac{\w\sqbrac{\frac{1}{\Zap},\Hil}\pap(\bvarap -\Dap\Zt - \Dapbar\Ztbar)} \\
& = \Real\cbrac{\w(\Id - \Hil)\brac{ - \Dap\Dapbar\Ztbar + (\Dap\Zt)\brac{\pap\frac{1}{\Zap}}}} \\
& \quad + \Real\cbrac{\w\sqbrac{\Pa\brac{\frac{\Zt}{\Zap}}, \Hil}\pap\brac{\pap\frac{1}{\Zap}} } \\
& \quad  -\Real\cbrac{\w\sqbrac{\frac{1}{\Zap},\Hil}\pap(\bvarap -\Dap\Zt - \Dapbar\Ztbar)} 
\end{aligned}
\end{align}
Here from \eqref{form:bvarapnew} and the fact that $\bvarap$ is real valued we get
\begin{align}\label{eq:bvarapminusDapZtnew}
\bvarap - \Dap\Zt - \Dapbar\Ztbar = \Real\cbrac{\sqbrac{\frac{1}{\Zap},\Hil}\Ztap  + \sqbrac{\Zt,\Hil}\brac*[\Big]{\pap \frac{1}{\Zap}}}
\end{align}
and a straightforward computation gives us
\begin{align}\label{eq:DapDapbarZtbarnew}
\Dap\Dapbar\Ztbar = -2\w^2\brac{\pap\frac{1}{\Zap}}\Dap\Ztbar + \brac{\pap\frac{1}{\Zapbar}}\Dap\Ztbar + \w^2\pap\brac{\frac{1}{\Zap}\Dap\Ztbar}
\end{align}
With these equation we can now make sense of $\Dt\pap\frac{1}{\Zap}$ for a solution with $\Ecal(t) < \infty$. To see this, observe that $\frac{1}{\Zap}\Dap\Ztbar$ is the boundary value of $\frac{1}{\Psizp^2}\pzp \U$ and by the using the fact that $\Ecal(t) < \infty$ we see that $\frac{1}{\Zap}\Dap\Ztbar(\cdot,t ) \in H^1(\Rsp)$. Now define
\begin{align*}
S(t) = \cbrac{\ap \in \Rsp \suchthat \frac{1}{\Psizp}(\ap,t) = 0} \quad \tx{ and } \quad NS(t) = \Rsp\backslash S(t)
\end{align*}
Now as shown in \cite{Ag20}, we see that $S(t)$ is a closed bounded set of measure zero and hence $\w$ is a well defined function on $NS(t)$ with $\abs{\w} = 1$. It was also shown there that $\frac{1}{\Psizp} \U_z$ and $\frac{1}{\Psizp} \Ubar_z$ extend continuously to the boundary and hence $\Dap\Ztbar$ and $\Dap\Zt$ are well defined bounded continuous functions. From the definition of the class $\mathcal{SA}$, we clearly have $\Ztapbar(\cdot,t) , \pap\frac{1}{\Zap}(\cdot,t) \in \Ltwo(\Rsp)$ and it is also easy to see that the estimate \eqref{eq:papPaZtZap} holds for these solutions. Hence all terms are now justified and hence $\Dt\pap\frac{1}{\Zap}(\cdot,t)$ is a well defined function in $\Ltwo(\Rsp)$. 
\medskip

\noindent\textbf{Claim:}
Consider a solution in time $[0,T]$ in the class $\mathcal{SA}$ with $\sup_{t \in [0,T]} \Ecal(t) < \infty$ and let $M>0$ be such that
\begin{align*}
\sup_{t \in [0,T]} \brac{\norm[\Linfty\cap\Hhalf]{\Dap\Ztbar}(t)  +  \brac{\norm[2]{\Ztapbar}(t) + \sqrt{g}}\norm[2]{\pap\frac{1}{\Zap}}(t)} \leq M
\end{align*} 
Then there exists a universal constant $C>0$ such that for all $t \in [0,T]$ we have $\Ea(t) \leq e^{CMt} \Ea(0)$.

Let us now prove this claim. Assume that $L>0$ is such that $\sup_{t \in [0,T]} \Ecal(t) \leq L$. Assume that this claim is proved for all $0\leq t \leq T_1$ and we will now prove it for $t \in [T_1 , T_1 + \delta]$ where $\delta$ depends only on $L$ and $M$. Let $(\Z^\ep,\Zt^\ep)(\ap,T_1) = (\Psi, \Ubar)(\ap - i\ep, T_1) $ and consider the smooth solutions $(\Z^\ep, \Zt^\ep)(t)$ for $t \geq T_1$. From the definition of $\Ecal(t)$, it is clear that $\Ecal(\Z^\ep,\Zt^\ep)(T_1) \leq \Ecal(T_1) \leq L$ and hence there exists a $\delta_1 > 0$ depending only on $L$ such that these smooth solutions exist in $[T_1, T_1 + \delta_1]$ with $\sup_{t \in [T_1, T_1 + \delta_1]}\Ecal(\Z^\ep,\Zt^\ep)(t) \leq C(L)$, for some constant $C(L)$ depending only on $L$.  

We now show that $\lim_{\ep \to 0} \Ea(\Z^\ep,\Zt^\ep)(T_1) = \Ea(T_1)$. To see this, we see from the definition of $\Z^\ep, \Zt^\ep$ and the energy $\Ecal(t)$ that
\begingroup
\allowdisplaybreaks
\begin{align*}
\Ztapbar^\ep(\cdot,T_1) \to \Ztapbar, \quad \pap\frac{1}{\Zap^\ep}(\cdot, T_1) \to \pap\frac{1}{\Zap}(\cdot, T_1) \quad \tx{ in } \Ltwo(\Rsp) \\
\pap\brac{\frac{1}{\Zap^\ep}\Dap^\ep\Ztbar^\ep}(\cdot, T_1) \to \pap\brac{\frac{1}{\Zap}\Dap\Ztbar}(\cdot, T_1) \quad \tx{ in } \Ltwo(\Rsp) \\
 \frac{1}{\Zap^\ep}\pap\frac{1}{\Zap^\ep}(\cdot, T_1) \to \frac{1}{\Zap}\pap\frac{1}{\Zap}(\cdot, T_1) \quad \tx{ in } \Hhalf(\Rsp)
\end{align*}
\endgroup
This already shows that $\E_1(\Z^\ep, \Zt^\ep)(T_1) \to \E_1(T_1)$. Now as mentioned before, $\frac{1}{\Psizp} \U_z$ extends continuously to the boundary and so $\Dap\Ztbar$ is a continuous function. Also as $\Dap\Ztbar(\ap,t) \to 0$ as $\abs{\ap} \to \infty$, this implies that $\Dap\Ztbar(\cdot, T_1)$ is a uniformly continuous function and so $\Dap^\ep\Ztbar^\ep(\cdot, T_1) \to \Dap\Ztbar(\cdot, T_1)$  in $\Linfty(\Rsp)$. Now the same argument used in \cite{Ag20} used to prove the continuity of $\Dap\Ztbar$ also works exactly in the same was for the function $\frac{1}{\Zap}\pap\frac{1}{\Zap}$ and hence we also get that $\frac{1}{\Zap^\ep}\pap\frac{1}{\Zap^\ep}(\cdot, T_1) \to \frac{1}{\Zap}\pap\frac{1}{\Zap}(\cdot, T_1)$ in $\Linfty(\Rsp)$. Using the fact that $\Ztapbar^\ep(\cdot, T_1) \to \Ztapbar(\cdot, T_1)$ in $\Ltwo(\Rsp)$ and the fact that $\Ag^\ep \geq g>0$, it is easy to see from \propref{prop:commutator}, \propref{prop:Leibniz} and  \propref{prop:Hardy} that $\sqrt{\Ag^\ep}(\cdot,T_1) \to \sqrt{\Ag}(\cdot, T_1)$ in $\Linfty\cap\Hhalf(\Rsp)$.  Hence from \propref{prop:Leibniz} we get
\begin{align*}
\frac{\sqrt{\Ag^\ep}}{\Zap^\ep}\pap\frac{1}{\Zap^\ep}(\cdot, T_1) \to \frac{\sqrt{\Ag}}{\Zap}\pap\frac{1}{\Zap} \quad \tx{ in } \Hhalf(\Rsp)  
\end{align*}
As $\frac{1}{\Psizp}(\cdot, T_1)$ is continuous on $\Pminusbar$, it is clear that $\frac{1}{\Zap^\ep}(\cdot, T_1) \to \frac{1}{\Zap}(\cdot, T_1) $ in  $\Linfty(\Rsp)$, and hence $\w^\ep \to \w$ in $\Linfty(\Rsp)$ on the set $O_{\beta, T_1} = \cbrac{\ap \in \Rsp \suchthat \frac{1}{\abs{\Zap}}(\ap,T_1) > \beta}$ for any $\beta > 0$. As $\Dap\Ztbar(\ap,t) = 0$ for $\ap \in S(t)$ and as $S(t)$ is a compact set of measure zero, this then also implies that $\Dap^\ep\Zt^\ep(\cdot, T_1) \to \Dap\Zt(\cdot, T_1) $ in $\Linfty(\Rsp)$. Now using the fact that $\w^\ep \to \w$ pointwise on $NS(t)$ and by using dominated convergence theorem, it is easy to see that $\Dt^\ep\pap\frac{1}{\Zap^\ep}(\cdot, T_1) \to \Dt\pap\frac{1}{\Zap}(\cdot, T_1) $ in $\Ltwo(\Rsp)$. Therefore $\lim_{\ep \to 0} \Ea(\Z^\ep,\Zt^\ep)(T_1) = \Ea(T_1)$. 

Now observe that for all $\ep>0$, we have
\begin{align*}
\nobrac{\norm[\Linfty\cap\Hhalf]{\Dap^\ep\Ztbar^\ep}(T_1)  +  \brac{\norm[2]{\Ztapbar^\ep}(T_1) + \sqrt{g}}\norm[2]{\pap\frac{1}{\Zap^\ep}}(T_1)} \leq M
\end{align*}
Using \lemref{lem:timederiv} and \propref{prop:DtLinfty} we see that for all $t \in [T_1, T_1 + \delta_1)$ we have
\begin{align*}
& \frac{\diff}{\diff t} \brac{\norm[\Linfty\cap\Hhalf]{\Dap^\ep\Ztbar^\ep}  +  \brac{\norm[2]{\Ztapbar^\ep} + \sqrt{g}}\norm[2]{\pap\frac{1}{\Zap^\ep}}} \\
& \lesssim \norm[\infty]{\bvarap^\ep} \brac{\norm[\Linfty\cap\Hhalf]{\Dap^\ep\Ztbar^\ep}  +  \brac{\norm[2]{\Ztapbar^\ep} + \sqrt{g}}\norm[2]{\pap\frac{1}{\Zap^\ep}}} \\
& \quad + \norm[\Linfty\cap\Hhalf]{\Dt^\ep\Dap^\ep\Ztbar^\ep} + \norm[2]{\Dt^\ep\Ztapbar^\ep}\norm[2]{\pap\frac{1}{\Zap^\ep}} + \brac{\norm[2]{\Ztapbar^\ep} + \sqrt{g}}\norm[2]{\Dt^\ep\pap\frac{1}{\Zap^\ep}} \\
& \leq C(L)
\end{align*}
Hence there exists an $0<\delta < \delta_1$ depending only on $M $ and $L$ so that for all $t \in [T_1, T_1 + \delta]$ and $\ep>0$ we have
\begin{align*}
\nobrac{\norm[\Linfty\cap\Hhalf]{\Dap^\ep\Ztbar^\ep}(t)  +  \brac{\norm[2]{\Ztapbar^\ep}(t) + \sqrt{g}}\norm[2]{\pap\frac{1}{\Zap^\ep}}(t)} \leq 2M
\end{align*}
Hence using \eqref{eq:mainEatime} we see that there exists a universal constant $C>0$ so that for all $t \in [T_1, T_1 + \delta]$ and $\ep>0$ we have
\begin{align*}
\Ea(\Z^\ep, \Zt^\ep)(t) \leq e^{CM(t - T_1)}\Ea(\Z^\ep, \Zt^\ep)(T_1)
\end{align*}
Let $T_2 = T_1 + \delta$. As $\lim_{\ep \to 0} \Ea(\Z^\ep,\Zt^\ep)(T_1) = \Ea(T_1)$, to prove the claim it is therefore enough to show that we have $\limsup_{\ep \to 0} \Ea(\Z^\ep,\Zt^\ep)(T_2) \geq \Ea(T_2)$.

Now from the proof of Theorem 3.9 of \cite{Wu19}, we see that 
\begin{align}\label{eq:ZtpaponeZapconverHhalf}
\norm[\Hhalf]{\Zt^\ep - \Zt}(T_2) + \norm[\Hhalf]{\frac{1}{\Zap^\ep} - \frac{1}{\Zap}}(T_2)  \to 0
\end{align}
as $\ep \to 0$ and moreover there exists a subsequence $\ep_n \to 0$ such that
\begin{align*}
\frac{1}{\Zap^{\ep_n}}(\cdot, T_2) \to \frac{1}{\Zap}(\cdot, T_2) \quad \tx{ and }   \quad  \Zt^{\ep_n}(\cdot, T_2) \to \Zt(\cdot, T_2)
\end{align*}
uniformly on compact subsets of $\Rsp$. As all the quantities in the energy $\Ea(\Z^\ep,\Zt^\ep)(T_2)$ are uniformly bounded, we can assume by going to a subsequence if necessary that the sequences
\begin{align*}
\norm[2]{\Ztapbar^{\ep_n}}, \norm[2]{\pap\frac{1}{\Zap^{\ep_n}}}, \norm[2]{\Dt^{\ep_n}\pap\frac{1}{\Zap^{\ep_n}}} \quad \tx{ and } \norm[\Hhalf]{\frac{\sqrt{\Ag^{\ep_n}}}{\Zap^{\ep_n}}\pap\frac{1}{\Zap^{\ep_n}}} 
\end{align*}
are all convergent sequences. Now as $\Ztapbar^{\ep_n}(\cdot, T_2) \to \Ztapbar(\cdot, T_2)$ and also $\pap\frac{1}{\Zap^{\ep_n}}(\cdot, T_2) \to \pap\frac{1}{\Zap}(\cdot, T_2)$ in distributions, this then implies that they converge weakly in $\Ltwo(\Rsp)$ and hence we have $\limsup_{n \to \infty} \Eone(\Z^{\ep_n}, \Zt^{\ep_n})(T_2) \geq \Eone(T_2)$. Now by the same argument used in Step 2 of this proof, we see that 
\begin{align*}
\frac{\sqrt{\Ag^{\ep_n}}}{\Zap^{\ep_n}}\pap\frac{1}{\Zap^{\ep_n}}(\cdot, T_2) \to \frac{\sqrt{\Ag}}{\Zap}\pap\frac{1}{\Zap}(\cdot, T_2)
\end{align*}
weakly in $\Hhalf(\Rsp)$. Hence 
\begin{align*}
\limsup_{n \to \infty} \norm[\Hhalf]{\frac{\sqrt{\Ag^{\ep_n}}}{\Zap^{\ep_n}}\pap\frac{1}{\Zap^{\ep_n}}}(T_2) \geq \norm[\Hhalf]{\frac{\sqrt{\Ag}}{\Zap}\pap\frac{1}{\Zap}}(T_2)
\end{align*}
Now observe that the set $O_{\beta, T_2} = \cbrac{\ap \in \Rsp \suchthat \frac{1}{\Zapabs}(\ap,T_2) > \beta }$ is an open set and that the sequences 
\begin{align*}
\norm[\Ltwo(O_{\beta, T_2})]{\pap\Ztapbar^{\ep_n}}(T_2) \tx{ and } \norm[\Ltwo(O_{\beta, T_2})]{\pap^2\frac{1}{\Zap^{\ep_n}}}(T_2) 
\end{align*}
are uniformly bounded. Hence by the Arzela Ascoli theorem, there exists a subsequence (which by abuse of notation we still call $\ep_n$) so that
\begin{align*}
\Ztapbar^{\ep_n}(\cdot, T_2) \to \Ztapbar(\cdot, T_2) \qq \tx{ and } \pap\frac{1}{\Zap^{\ep_n}} \to \pap\frac{1}{\Zap}(\cdot, T_2)
\end{align*}
uniformly on compact subsets of $O_{\beta, T_2}$. Now as $NS(T_2) = \cup_{n \in \Nsp } O_{ \beta_n, T_2}$ for $\beta_n = \frac{1}{n}$, and by a diagonalization argument we get a subsequence (again calling it $\ep_n$) so that
\begin{align*}
\Ztapbar^{\ep_n}(\cdot, T_2) \to \Ztapbar(\cdot, T_2) \qq \tx{ and } \pap\frac{1}{\Zap^{\ep_n}} \to \pap\frac{1}{\Zap}(\cdot, T_2)
\end{align*}
uniformly on compact subsets of $NS(T_2)$. Also clearly $\w^{\ep_n} \to \w$ uniformly on compact subsets of $NS(T_2)$. From this it is clear that 
\begin{align*}
(\wbar^{\ep_n})^2\brac{\pap\frac{1}{\Zap^{\ep_n}}}\Dapbar^{\ep_n}\Zt^{\ep_n}(\cdot, T_2) \to \wbar^2\brac{\pap\frac{1}{\Zap}}\Dapbar\Zt (\cdot, T_2)
\end{align*}
weakly in $\Ltwo(NS(T_2))$. However as the measure of $S(T_2)$ is zero, this implies that this sequence is convergent in fact on $\Ltwo(\Rsp)$ (As all the terms are uniformly bounded on $\Ltwo(\Rsp)$). The same argument works for all terms of \eqref{eq:DapDapbarZtbarnew} and \eqref{eq:DtpapfraconeoverZapnew} except for the term $\wbar\Dapabs\brac{\bvarap - \Dap\Zt - \Dapbar\Ztbar}$. 

For this term, first observe that from \eqref{eq:DapabsbvarDapZtnew} and the above argument, it is easy to see that
\begin{align*}
\w^{\ep_n}(\Id - \Hil)\Dap^{\ep_n}(\bvarap^{\ep_n} -\Dap^{\ep_n}\Zt^{\ep_n} - \Dapbar^{\ep_n}\Ztbar^{\ep_n})(\cdot, T_2) \to \w(\Id - \Hil)\Dap(\bvarap -\Dap\Zt - \Dapbar\Ztbar)(\cdot, T_2)
\end{align*}
weakly in $\Ltwo(\Rsp)$. For the last term of \eqref{eq:DapabsbvarDapZtnew}, we first note that
\begin{align}\label{eq:bvarapDapZtLtwo}
\brac{\bvarap^{\ep_n} - \Dap^{\ep_n}\Zt^{\ep_n} - \Dapbar^{\ep_n}\Ztbar^{\ep}}(\cdot, T_2) \to \brac{\bvarap - \Dap\Zt - \Dapbar\Ztbar}(\cdot, T_2)
\end{align}
in $\Ltwo(\Rsp)$ from \eqref{eq:ZtpaponeZapconverHhalf}. Now using \eqref{eq:bvarapDapZtLtwo}, \eqref{eq:ZtpaponeZapconverHhalf} and the uniform bounds on 
\begin{align*}
\norm[\Hhalf\cap\Linfty]{\bvarap^{\ep_n} - \Dap^{\ep_n}\Zt^{\ep_n} - \Dapbar^{\ep_n}\Ztbar^{\ep}}(T_2) \qq \tx{ and  } \quad \norm[2]{\pap\frac{1}{\Zap^{\ep_n}}}(T_2)
\end{align*}
it is easy to see that 
\begin{align*}
\cbrac{\sqbrac{\frac{1}{\Zap^{\ep_n}},\Hil}\pap(\bvarap^{\ep_n} -\Dap^{\ep_n}\Zt^{\ep_n} - \Dapbar^{\ep_n}\Ztbar^{\ep_n})}(\cdot, T_2) \\
\to \cbrac{\sqbrac{\frac{1}{\Zap},\Hil}\pap(\bvarap -\Dap\Zt - \Dapbar\Ztbar)}(\cdot, T_2)
\end{align*}
in distributions. However as the sequence is uniformly bounded on $\Ltwo(\Rsp)$, we see that the convergence is in fact weakly in $\Ltwo(\Rsp)$. A similar argument works for the second term on the right hand side of \eqref{eq:DapabsbvarDapZtnew} as well and hence by combining this with the uniform convergence of $\w^{\ep_n} \to \w$ on compact subsets of $NS(T_2)$, we see that 
\begin{align*}
\Dt^{\ep_n}\pap\frac{1}{\Zap^{\ep_n}}(\cdot, T_2) \to \Dt\pap\frac{1}{\Zap}(\cdot, T_2)
\end{align*}
weakly in $\Ltwo(\Rsp)$. Hence this shows that $\limsup_{n \to \infty} \Etwo(\Z^{\ep_n}, \Zt^{\ep_n})(T_2) \geq \Etwo(T_2)$, thereby proving the claim. 

\medskip
\noindent\textbf{Step 4:} We are now ready to prove the blow up criterion for $g > 0$. We will prove this via contradiction. Assume that $0<T^* < \infty$ is the maximal time of existence and that
\begin{align*}
\sup_{t \in [0,T^*)} \brac{\norm[\Linfty\cap\Hhalf]{\Dap\Ztbar}(t)  +  \brac{\norm[2]{\Ztapbar}(t) + \sqrt{g}}\norm[2]{\pap\frac{1}{\Zap}}(t)} \leq M < \infty
\end{align*} 
From the definition of a solution in the class $\mathcal{SA}$ in time $[0, T^*)$, we see that for any $\delta>0$ small enough, $\sup_{t \in [0,T^* - \delta]} \Ecal(t) < \infty$, and hence by the claim proved in Step 3 we see that there exists a universal constant $C>0$ so that for all $t \in [0,T^*)$ we have $\Ea(t) \leq e^{CMt} \Ea(0)$. Define 
\begin{align*}
Q = e^{CMT^*} \Ea(0) < \infty
\end{align*}
Let $T \in [0, T^*)$ and as usual let $(\Z^\ep,\Zt^\ep)(\ap,T) = (\Psi, \Ubar)(\ap - i\ep, T) $ and consider the smooth solutions $(\Z^\ep, \Zt^\ep)(t)$ for $t \geq T$. By Step 3 we know that $\lim_{\ep \to 0} \Ea(\Z^\ep,\Zt^\ep)(T) = \Ea(T)$ and hence let $\ep_0 > 0$ be small enough so that for all $0<\ep \leq \ep_0$ we have
\begin{align*}
 \Ea(\Z^\ep,\Zt^\ep)(T) \leq 2\Ea(T) \leq 2Q
\end{align*}
Now by the estimates \eqref{eq:mainEatime}, \eqref{eq:mainEtime}, \eqref{eq:SobolevfromEcal} and the blow up criterion of \thmref{thm:existenceSobolevunbdd},  there exists $\delta_1 > 0$ depending only on $Q$ such that the smooth solutions $(\Z^\ep, \Zt^\ep)(t)$ exist in $[T, T + \delta_1]$ and 
\begin{align*}
 \Ea(\Z^\ep,\Zt^\ep)(t) \leq 4Q \qq \tx{ for all } t \in [T, T+\delta_1]
\end{align*}
Now from \eqref{eq:mainEtime} and part 1 of \lemref{lem:equivSobolevunbdd} we see that there exists a constant $Q_1>0$ depending only on $Q$ so that 
\begin{align*}
 \E(\Z^\ep,\Zt^\ep)(t) \leq Q_1  \Ecal(\Z^\ep,\Zt^\ep)(T) \leq Q_1  \Ecal(T) \qq \tx{ for all } t \in [T, T+\delta_1]
\end{align*}
Now from \eqref{eq:ZtapbarL2norm2plusgestimate} and an approximation argument, we see that there exists a constant $Q_2$ depending only on $Q, T^*, g$ and $c_0$ so that 
\begin{align*}
\norm[2]{\Ztapbar^\ep}(t)  \leq Q_2 \qq \tx{ for all } t \in [T, T+\delta_1]
\end{align*}
Hence from \eqref{eq:controlEcalfromE} we see that 
\begin{align*}
 \Ecal(\Z^\ep,\Zt^\ep)(t) \leq C_1\brac{Q_1 \Ecal(T) + g + \frac{1}{g} + Q_2} \qq \tx{ for all } t \in [T, T+\delta_1]
\end{align*}
Therefore by the uniqueness of solutions, we get that for all $ t \in [T, T+\delta_1]$ we have 
\begin{align}\label{eq:Ecaldeltaoneest}
\Ecal(t) \leq \liminf_{\ep \to 0}  \Ecal(\Z^\ep,\Zt^\ep)(t) \leq  C_1\brac{Q_1 \Ecal(T) + g + \frac{1}{g} + Q_2}
\end{align}
Define a function $J:[0,\infty) \to [0,\infty) $ given by
\begin{align*}
J(x) = C_1\brac{Q_1x + g + \frac{1}{g} + Q_2}
\end{align*}
Using \eqref{eq:Ecaldeltaoneest} repeatedly we see that for all $t \in [0,T^*)$ we have
\begin{align*}
\Ecal(t) \leq J^{(N)}(\Ecal(0))
\end{align*}
where $N$ is an integer such that $N > \frac{T^*}{\delta_1}$ and $J^{(N)}$ is the function $J$ composed with itself $N$ times. From this and the lower bound on the time of existence from Step 1, we clearly see that we can extend the solution beyond $T^*$. This contradicts the maximality of $T^*$. Hence the blow up criterion is proven.

\end{proof}

\section{Proof of the main results for bounded domain case}\label{sec:bounded}

\subsection{Derivation of equations}\label{sec:derivofeqnbdd}

In this subsection we derive the main equations and the various formulae used. We closely follow \cite{BiMiShWu17} though there are some important differences in the formulae we derive.

Note that from \lemref{lem:Hilbounded},  a function $f :\Sone \to \Csp$ satisfies $\Hil f = f$ if and only if $f = \Tr(F)$ where $F: \Dsp \to \Csp$ is holomorphic. Using this and the definition of $\Hiltil$ we see that a function $f :\Sone \to \Csp$ satisfies $\Hiltil f = f$ if and only if $f = \Tr(F)$ where $F: \Dsp \to \Csp$ is a holomorphic function with $F(0) = 0$. Let us now write down some common functions which satisfy these properties:
\begin{enumerate}
\item As $\Ztbar(\ap,t) = \U(e^{i\ap},t)$, we see that $\Ztbar(\ap,t) = \Tr(\U(z,t))$. Hence $\Hil \Ztbar = \Ztbar$. 
\item We first observe that since $\Phi(\Z(\ap,t),t) = e^{i\ap}$, by taking a derivative we obtain
\begin{align}\label{eq:PhizoverPhi}
\Phi_z(\Z(\ap,t),t)\Zap(\ap,t) = ie^{i\ap} 
\end{align}
Now as $\Phi(\Psi(z,t),t) = z$ for all $z \in \Dsp$, we see that
\begin{align}\label{eq:PhizPsi}
\Phi_z(\Psi(z,t),t) = \frac{1}{\Psi_z(z,t)}
\end{align}
Hence 
\begin{align}\label{eq:oneoverZapTr}
\frac{1}{\Zap(\ap,t)} = \frac{1}{ie^{i\ap}\Psi_z(e^{i\ap},t)} = \Tr\brac{\frac{1}{i z \Psi_z(z,t)}}  
\end{align}
Therefore we see that
\begin{align*}
\frac{e^{i\ap}}{\Zap(\ap,t)}  = \Tr\brac{\frac{1}{i \Psi_z(z,t)}} 
\end{align*}
and consequently $\Hil\brac{\frac{e^{i\ap}}{\Zap}} = \frac{e^{i\ap}}{\Zap}$. Using this we also see that $\Hil\brac{e^{-i\ap}\pap\brac{\frac{e^{i\ap}}{\Zap}}} = e^{-i\ap}\pap\brac{\frac{e^{i\ap}}{\Zap}}$. Now observe that
\begin{align*}
e^{-i\ap}\pap\brac{\frac{e^{i\ap}}{\Zap}} = \pap\frac{1}{\Zap} + \frac{i}{\Zap}
\end{align*}
Therefore by using the fact that $\Avg\brac{\pap\frac{1}{\Zap}} = 0$, we obtain
\begin{align}\label{eq:HiltilpaponeoverZap}
\Hiltil\cbrac{\pap\frac{1}{\Zap} + i\brac{\frac{1}{\Zap} - \Avg\brac{\frac{1}{\Zap}} } } = \pap\frac{1}{\Zap} + i\brac{\frac{1}{\Zap} - \Avg\brac{\frac{1}{\Zap}} } 
\end{align}

\item We see that 
\begin{align*}
\frac{\Ztbar - \Av(\Ztbar)}{\Zap}(\ap,t) = \Tr\brac{\frac{\U(z,t) - \U(0,t)}{i z \Psi_z(z,t)}}
\end{align*}
Hence we have $\Hil\brac{\frac{\Ztbar - \Av(\Ztbar)}{\Zap}} = \frac{\Ztbar - \Av(\Ztbar)}{\Zap}$. 
\item As $\Ztbar(\ap,t) = \U(e^{i\ap},t)$, we see that
\begin{align*}
\Ztapbar(\ap,t) = ie^{i\ap}\U_z(e^{i\ap},t) = \Tr\brac{i z \U_z(z,t)}
\end{align*}
Therefore $\Hiltil \Ztapbar = \Ztapbar$. 
\item From the above calculations we observe that 
\begin{align*}
(\Dap\Ztbar)(\ap,t) = \frac{\Ztapbar}{\Zap}(\ap,t) = \Tr\brac{\frac{\U_z(z,t)}{\Psi_z(z,t)}}
\end{align*}
Hence $\Hil (\Dap\Ztbar) = \Dap\Ztbar$. 
\end{enumerate}

Let us now derive some important equations. As $z(\cdot,t)$ is a counterclockwise parametrization of $\partial \Omega(t)$, the unit outward normal is $\hat{n} = -i \frac{\zal}{\abs{\zal}}$. Now as $P = 0$ on $\partial \Omega(t)$ we see that $\grad P (z,t) = i\avar \zal$, where
\begin{align*}
\avar = - \frac{1}{\abs{\zal}}\frac{\partial P}{\partial \hat{n}}
\end{align*}
Hence we see that the Euler equations on the boundary can be written as
\begin{align*}
\ztt = - i\avar\zal 
\end{align*}
Taking a time derivative and complex conjugate, we get
\begin{align*}
\ztttbar - i\avar\ztalbar = i\avar_t \zalbar
\end{align*}
Define $\A = (\avar\hal) \compose \hinv$ and $\Azero = \A\Zapabs^2 $.  Hence by precomposing the above equations with $\hinv$ we obtain
\begin{align}\label{eq:Zttbarbdd}
\Zttbar = i \frac{\Azero}{\Zap}
\end{align}
and 
\begin{align}\label{eq:mainbdd}
\brac{\Dt^2 -i\frac{\Azero}{\Zapabs^2}\pap} \Ztbar = i\brac{\frac{\avar_t}{\avar}\compose \hinv} \frac{\Azero}{\Zap}
\end{align}

\subsubsection{Formulae of $\Azero$ and $\frac{\avar_t}{\avar}\compose \hinv$}

We know that $\ztbar(\al, t) = \vboldbar(\z(\al,t),t)$ and hence we have the following identities:
\begin{enumerate}[leftmargin =*, align = left]
\item $\Dal\ztbar(\al,t) = \vboldbar_z(\z(\al,t),t)$ and $\Dal^2\ztbar(\al,t) = \vboldbar_{zz}(\z(\al,t),t)$
\item $\zttbar(\al,t) = \vboldbar_t(\z(\al,t),t) + \vboldbar_z(\z(\al,t),t)\zt(\al,t) $ and hence 
\begin{align*}
(\zttbar - (\Dal\ztbar)\zt)(\al,t) = \vboldbar_t(\z(\al,t),t)
\end{align*}
\item $\Dal(\zttbar - (\Dal\ztbar)\zt)(\al,t) = \vboldbar_{tz}(\z(\al,t),t)$
\item We also have 
\begin{align*}
\ztttbar(\al,t) & = \vboldbar_{tt}(\z(\al,t),t) + 2\vboldbar_{tz}(\z(\al,t),t)\zt(\al,t) + \vboldbar_{zz}(\z(\al,t),t)\zt^2(\al,t)  \\
& \quad + \vboldbar_z(\z(\al,t),t)\ztt(\al,t)
\end{align*}
\end{enumerate}
Precomposing the second identity above with $\hinv$ we get
\begin{align*}
(\Zttbar - (\Dap\Ztbar)\Zt)(\ap,t) = \Tr(\vboldbar_t(\Psi(z,t),t))
\end{align*}
Now multiplying both sides with $\Zap$ and using \eqref{eq:Zttbarbdd} and \eqref{eq:oneoverZapTr} we obtain
\begin{align*}
(i\Azero - \Zt\Ztapbar)(\ap,t) = \Tr(i z \Psi_z(z,t)\vboldbar_t(\Psi(z,t),t) )
\end{align*}
Note that the quantity inside the trace is a holomorphic function which vanishes at $z = 0$ and that $\Azero$ is real valued. Hence applying $\Id - \Hiltil$ to the above equation and taking imaginary part we get the formula for $\Azero$
\begin{align}\label{eq:Azerobdd}
\Azero = \Imag \sqbrac*[\big]{\Zt, \Hiltil} \Ztapbar
\end{align}

Now precomposing the last identity above for $\ztttbar$ with $\hinv$ and using the previous identities we get
\begin{align*}
\Ztttbar = \vboldbar_{tt}\compose \Z + 2\Dap(\Zttbar - (\Dap\Ztbar)\Zt)\Zt + (\Dap^2\Ztbar)\Zt^2 + (\Dap\Ztbar)\Ztt
\end{align*}
Hence using \eqref{eq:mainbdd} and \eqref{eq:Zttbarbdd}  we get
\begin{align}\label{eq:atoveraint}
\begin{aligned}
& i\brac{\frac{\avar_t}{\avar}\compose \hinv}\Azero \\
& =  \Zap \Ztttbar -i\frac{\Azero}{\Zapbar}\Ztapbar \\
& = \Zap \Ztttbar + \Ztt\Ztapbar \\
& = \Zap\vboldbar_{tt}\compose \Z + 2\Zt\pap(\Zttbar - (\Dap\Ztbar)\Zt) + \Zt^2(\pap\Dap\Ztbar) + 2\Ztt\Ztapbar
\end{aligned}
\end{align}
Now as $\Zap(\ap,t) = \Tr\cbrac{iz\Psi_z(z,t)}$ from \eqref{eq:oneoverZapTr} we have
\begin{enumerate}
\item $ (\Zap\vboldbar_{tt}\compose \Z)(\ap,t) = \Tr\cbrac{iz\Psi_z(z,t)\vboldbar_{tt}(\Psi(z,t),t)}$. Hence we see that
\begin{align*}
(\Id - \Hiltil)(\Zap\vboldbar_{tt}\compose \Z) = 0 
\end{align*}
\item $\pap(\Zttbar - (\Dap\Ztbar)\Zt)(\ap,t) = \Tr\cbrac{iz\Psi_z(z,t)\vboldbar_{tz}(\Psi(z,t),t)}$. Hence we see that
\begin{align*}
(\Id - \Hiltil)\pap(\Zttbar - (\Dap\Ztbar)\Zt) = 0 
\end{align*}
\item $\pap\Dap\Ztbar =  \Tr\cbrac{iz\Psi_z(z,t)\vboldbar_{zz}(\Psi(z,t),t)}$. Hence we see that
\begin{align*}
(\Id - \Hiltil)(\pap\Dap\Ztbar) = 0 
\end{align*}
\item $\Ztapbar =  \Tr\cbrac{iz\Psi_z(z,t)\vboldbar_{z}(\Psi(z,t),t)}$. Hence we see that
\begin{align*}
(\Id - \Hiltil)(\Ztapbar) = 0 
\end{align*}
\end{enumerate}
Hence applying $(\Id - \Hiltil)$ to \eqref{eq:atoveraint} and using \propref{prop:tripleidentity} we obtain
\begin{align*}
& (\Id - \Hiltil)\cbrac{i\brac{\frac{\avar_t}{\avar}\compose \hinv}\Azero} \\
& = 2\sqbrac*[\big]{\Zt,\Hiltil}\pap(\Zttbar - (\Dap\Ztbar)\Zt) + \sqbrac*[\big]{\Zt^2,\Hiltil}\pap\Dap\Ztbar + 2\sqbrac*[\big]{\Ztt,\Hiltil}\Ztapbar \\
& = 2\sqbrac*[\big]{\Zt,\Hiltil}\Zttapbar  + 2\sqbrac*[\big]{\Ztt,\Hiltil}\Ztapbar + \sqbrac*[\big]{\Zt^2,\Hiltil}\pap\Dap\Ztbar - 2\sqbrac*[\big]{\Zt,\Hiltil}\pap((\Dap\Ztbar)\Zt) \\
& = 2\sqbrac*[\big]{\Zt,\Hiltil}\Zttapbar  + 2\sqbrac*[\big]{\Ztt,\Hiltil}\Ztapbar - [\Zt,\Zt; \Dap\Ztbar]
\end{align*}
Taking imaginary part, we get
\begin{align}\label{eq:atovera}
\frac{\avar_t}{\avar}\compose \hinv = \frac{\Imag\cbrac*[\big]{2\sqbrac*[\big]{\Zt,\Hiltil}\Zttapbar  + 2\sqbrac*[\big]{\Ztt,\Hiltil}\Ztapbar - [\Zt,\Zt; \Dap\Ztbar]}}{\Azero}
\end{align}

\subsubsection{Formula of $\bvar$}

Let us first note some useful formulae
\begin{enumerate}
\item From the normalization of the conformal map, we know that $\Phi(X(x_0,t),t) = 0$ for all $t\geq 0$. Hence taking a time derivative we get
\begin{align*}
\Phi_t(X(x_0,t),t) + \Phi_z(X(x_0,t),t)u(X(x_0,t),t) = 0
\end{align*}
Now we see that $u(X(x_0,t),t) = \Ubar(0,t)$.  As $U$ is holomorphic in $\Dsp$ with boundary value $\Ztbar$, we see that 
\begin{align}\label{eq:AvZt}
u(X(x_0,t),t) = \Av(\Zt)
\end{align}
As $X(x_0,t) = \Psi(0,t)$ we obtain using \eqref{eq:PhizPsi}
\begin{align}\label{eq:Phitzero}
\Phi_t(\Psi(0,t),t) + \frac{\Av(\Zt)}{\Psi_z(0,t)} = 0
\end{align}
\item Again by the normalization of the conformal map, we know that $\Phi_z(\Psi(0,t),t) > 0$ for all $t\geq 0$. Hence taking a time derivative we get
\begin{align*}
\Imag\cbrac{\Phi_{zt}(\Psi(0,t),t) + \Phi_{zz}(\Psi(0,t),t)\Psi_t(0,t)} = 0
\end{align*}
Now as $\Psi(0,t) = X(x_0,t)$, we see that $\Psi_t(0,t) = u(X(x_0,t),t) =  \Av(\Zt)$. Hence we have
\begin{align*}
\Imag\cbrac{\Phi_{zt}(\Psi(0,t),t) + \Phi_{zz}(\Psi(0,t),t)\Av(\Zt)} = 0
\end{align*}
Now using \eqref{eq:PhizPsi} we get
\begin{align}\label{eq:PhiztAvZt}
\Imag\cbrac{\Phi_{zt}(\Psi(0,t),t) + \Av(\Zt)\brac{\frac{1}{\Psi_z}\pzp\frac{1}{\Psi_z} }(0,t) } = 0
\end{align}
\end{enumerate}

Now from \eqref{def:hbddorig} we see that
\begin{align*}
\h(\al,t) = -i\log\brac{\Phi(z(\al,t),t)}
\end{align*}
Taking a time derivative and precomposing with $\hinv$ we obtain
\begin{align*}
(\h_t\compose\hinv)(\ap,t) = \frac{-i}{\Phi(\Z(\ap,t),t)}\cbrac{\Phi_t(\Z(\ap,t),t) + \Phi_z(\Z(\ap,t),t)\Zt(\ap,t)}
\end{align*}
Hence using \eqref{eq:PhizoverPhi},  $\Phi(\Z(\ap,t),t) = e^{i\ap}$  and \eqref{eq:oneoverZapTr} we get
\begin{align*}
\bvar(\ap,t) & = \frac{\Phi_t(\Z(\ap,t),t)}{ie^{i\ap}} + \frac{\Zt(\ap,t)}{\Zap(\ap,t)} \\
& = \cbrac{\frac{\Phi_t(\Z(\ap,t),t)}{ie^{i\ap}} + \frac{\Av(\Zt)}{\Zap(\ap,t)} }+ \frac{\Zt(\ap,t) - \Av(\Zt)}{\Zap(\ap,t)} \\
& = \Tr\cbrac{\frac{\Phi_t(\Psi(z,t),t)}{iz} + \frac{\Av(\Zt)}{iz\Psi_z(z,t)}} + \frac{\Zt(\ap,t) - \Av(\Zt)}{\Zap(\ap,t)} 
\end{align*}
Now using \eqref{eq:Phitzero} we obtain
\begin{align}\label{eq:bvarint}
\begin{split}
 \bvar(\ap,t) & =  \Tr\cbrac{\frac{\Phi_t(\Psi(z,t),t) - \Phi_t(\Psi(0,t),t)}{iz} + \frac{\Av(\Zt)}{iz}\brac{\frac{1}{\Psi_z(z,t)} - \frac{1}{\Psi_z(0,t)}}} \\
 & \quad + \frac{\Zt(\ap,t) - \Av(\Zt)}{\Zap(\ap,t)} 
\end{split}
\end{align}
As the term inside $\Tr$ is a holomorphic function on $\Dsp$, applying $\Real(\Id - \Hiltil)$ to the above equation we obtain 
\begin{align*}
\bvar & = \Real\cbrac{\Av\Tr\cbrac{\frac{\Phi_t(\Psi(z,t),t) - \Phi_t(\Psi(0,t),t)}{iz} + \frac{\Av(\Zt)}{iz}\brac{\frac{1}{\Psi_z(z,t)} - \frac{1}{\Psi_z(0,t)}}}} \\
& \quad  + \Real(\Id - \Hiltil)\brac{\frac{\Zt - \Av(\Zt)}{\Zap} }
\end{align*}
Now as the term inside $\Tr$ is holomorphic, we see that the average value of the trace equals the value of the holomorphic function at $z = 0$. Hence by the definition of the complex derivative we get
\begin{align*}
\bvar & = \Real\cbrac{\frac{1}{i}\Phi_{tz}(\Psi(0,t))\Psi_z(0,t) + \frac{\Av(\Zt)}{i} \brac{\pz\frac{1}{\Psi_z}}(0,t) } \\
& \quad  + \Real(\Id - \Hiltil)\brac{\frac{\Zt - \Av(\Zt)}{\Zap} } \\
& = \Real\cbrac{\frac{\Psi_z(0,t)}{i} \brac{ \Phi_{tz}(\Psi(0,t)) + \Av(\Zt) \brac{\frac{1}{\Psi_z}\pz\frac{1}{\Psi_z}}(0,t)}} \\
& \quad  + \Real(\Id - \Hiltil)\brac{\frac{\Zt - \Av(\Zt)}{\Zap} }
\end{align*}
As $\Psi_z(0,t) > 0$, therefore from \eqref{eq:PhiztAvZt} we see that the first term vanishes. Hence we obtain the formula for $\bvar$
\begin{align}\label{eq:bbdd}
\bvar = \Real(\Id - \Hiltil)\brac{\frac{\Zt - \Av(\Zt)}{\Zap} }
\end{align}

Note that we have now derived the system \eqref{eq:systemonebdd}. 

\subsubsection{Some useful identities}

We now write down some useful identities similar to \secref{sec:identities}. First we observe that \eqref{form:RealImagTh}, \eqref{eq:commutator}, \eqref{form:DtZapabs} and \eqref{form:DtoneoverZap} hold here as well. Also it is easy to see that the analog of \eqref{form:Real} and \eqref{form:Imag} is 
\begin{align}
(\Id + \Hiltil)(\Real f) = f -i\Imag (\Id - \Hiltil) f  \label{form:Realbdd}\\
(\Id + \Hiltil)(i \Imag f) = f - \Real (\Id - \Hiltil)f  \label{form:Imagbdd}
\end{align}
for any complex valued function on $\Sone$. Now using \eqref{form:Realbdd} and the formula \eqref{eq:Azerobdd} we see that
\begin{align}\label{form:Aonenewbdd}
\Azero = \Imag \sqbrac*[\big]{\Zt, \Hiltil} \Ztapbar = \Imag\cbrac{(\Id - \Hiltil)(\Zt\Ztapbar)} = -i\Zt\Ztapbar + i(\Id + \Hiltil)\cbrac{\Real(\Zt\Ztapbar)}
\end{align}
Now using \eqref{form:Imagbdd} and \eqref{eq:bbdd} we obtain
\begin{align*}
\bvar =  \Real(\Id - \Hiltil)\brac{\frac{\Zt - \Av(\Zt)}{\Zap} } = \frac{\Zt - \Av(\Zt)}{\Zap} - i(\Id + \Hiltil)\Imag\cbrac{\frac{\Zt - \Av(\Zt)}{\Zap}}
\end{align*}
Hence
\begin{align*}
\bap & = \Dap\Zt + (\Zt - \Av(\Zt))\brac{\pap\frac{1}{\Zap}} - i(\Id + \Hiltil)\pap\Imag\cbrac{\frac{\Zt - \Av(\Zt)}{\Zap}} \\
& = \Dap\Zt + (\Zt - \Av(\Zt))\brac{\pap\frac{1}{\Zap} + i\brac{\frac{1}{\Zap} - \Av\brac{\frac{1}{\Zap}}}} \\
& \quad -i(\Zt - \Av(\Zt))\brac{\frac{1}{\Zap} - \Av\brac{\frac{1}{\Zap}}} - i(\Id + \Hiltil)\pap\Imag\cbrac{\frac{\Zt - \Av(\Zt)}{\Zap}}
\end{align*}
Therefore
\begin{align}\label{eq:bapbdd}
\begin{aligned}
& \bap - 2\Real(\Dap\Zt) \\
& = -\Dapbar\Ztbar + (\Zt - \Av(\Zt))\brac{\pap\frac{1}{\Zap} + i\brac{\frac{1}{\Zap} - \Av\brac{\frac{1}{\Zap}}}} \\
& \quad -i(\Zt - \Av(\Zt))\brac{\frac{1}{\Zap} - \Av\brac{\frac{1}{\Zap}}} - i(\Id + \Hiltil)\pap\Imag\cbrac{\frac{\Zt - \Av(\Zt)}{\Zap}}
\end{aligned}
\end{align}
Applying $\Real(\Id - \Hiltil)$ we get
\begin{align}\label{eq:bapbddcomm}
\begin{aligned}
& \bap - 2\Real(\Dap\Zt) \\
& = \Real\Bigg\{ -\sqbrac{\frac{1}{\Zapbar}, \Hiltil}\Ztapbar + \sqbrac{\Zt - \Av(\Zt), \Hiltil}\brac{\pap\frac{1}{\Zap} + i\brac{\frac{1}{\Zap} - \Av\brac{\frac{1}{\Zap}}}} \\
& \qquad \qquad -i(\Id - \Hiltil)\cbrac{(\Zt - \Av(\Zt))\brac{\frac{1}{\Zap} - \Av\brac{\frac{1}{\Zap}}}} \Bigg\}
\end{aligned}
\end{align}
Now similar to the calculation of \eqref{eq:Azero}, we see that
\begin{align}
\begin{aligned}\label{eq:Azerobddcalc}
\Azero(\ap) & = (\Imag \sqbrac*{\Zt,\Hiltil}\Ztapbar)(\ap) \\
& = \Imag\cbrac{\frac{1}{2\pi i} \int_0^{2\pi} (\Zt(\ap) - \Zt(\bp))\cot\brac{\frac{\bp - \ap}{2}}\Ztapbar(\bp) \diff \bp } \\
& = \frac{1}{8\pi} \int_0^{2\pi} \abs*[\Bigg]{\frac{\Zt(\ap) - \Zt(\bp)}{\sin\brac{\frac{\ap - \bp}{2}}}}^2 \diff \bp \\
& = \frac{1}{8\pi}\norm[\Ltwo([0,2\pi], \diff\bp)]{\frac{\Zt(\ap) - \Zt(\bp)}{\sin\brac{\frac{\ap - \bp}{2}}}}^2 \\
& = -\frac{i}{2}\sqbrac{\Zt,\Ztbar; 1}
\end{aligned}
\end{align}
From this it is clear that if $\Zt$ is continuous and not a constant function, then there exists $c>0$ such that $\Azero \geq c >0$ on $[0,2\pi]$. Here $c$ depends on the profile of $\Zt$ and changes with time in general. 

We now give an analogue of \lemref{lem:DtAgoverAg}. 
\begin{lemma}\label{lem:DtAgoverAgbdd}
We have
\begin{align*}
 \frac{\Dt \Azero}{\Azero} & = \frac{\Imag\cbrac{2\sqbrac{\Zt,\Zttbar;1} - \sqbrac{\Zt,\Zt;\Dap\Ztbar}}}{\Ag} + 2\Real(\Dap\Zt) - \bvarap \\
& = 2\Real\cbrac{\sqbrac{\Zt,\frac{1}{\Zap}; 1}(\ap)} + \frac{1}{\Azero}\Imag\cbrac{2i\sqbrac{\Zt,\Azero; \frac{1}{\Zap}}(\ap) - \sqbrac{\Zt,\Zt; \Dap\Ztbar}(\ap) } \\
& \quad + 2\Real(\Dap\Zt) - \bvarap
\end{align*}
\end{lemma}
Except for some sign changes due to the difference of parametrization of the interface, the formulae are exactly the same as \lemref{lem:DtAgoverAg} and the proof is also identical.

\subsubsection{Equations of motion}

Let us now derive the main equations. This will be very similar to \secref{sec:maineqn}. We define 
\begin{align}\label{form:Jzerobdd}
\Jzerobdd = \Dt(\bvarap - \Dap\Zt - \Dapbar\Ztbar)
\end{align}
Using \eqref{form:DtoneoverZap} we get
\begin{align*}
\Dt\brac{\frac{e^{i\ap}}{\Zap}} = \frac{e^{i\ap}}{\Zap} \cbrac{(\bap - \Dap\Zt - \Dapbar\Ztbar) + \Dapbar\Ztbar + i\bvar}
\end{align*}
Now following a similar calculation as in \secref{sec:maineqn} we get
\begin{align*}
\brac{\Dt^2 - i\frac{\Azero}{\Zapabs^2}\pap}\brac{\frac{e^{i\ap}}{\Zap}} = \frac{e^{i\ap}}{\Zap}(\Jzerobdd + \Qzerobdd)
\end{align*}
where
\begin{align}\label{eq:Qzerobdd}
\Qzerobdd = (\bvarap-\Dap\Zt + i\bvar)^2 - (\Dapbar\Ztbar)^2 + \frac{i}{\Zapabs^2}\pap\Azero + i\Dt\bvar + \frac{\Azero}{\Zapabs^2}
\end{align}
Now applying $\pap$ to the above equation and doing a similar calculation as in \secref{sec:maineqn} we get the main equation as
\begin{align}\label{eq:paponeoverZapbdd}
\brac{\Dt^2 - i\frac{\Azero}{\Zapabs^2}\pap}\pap\brac{\frac{e^{i\ap}}{\Zap}} = e^{i\ap}\Dap\Jzerobdd + \Rzerobdd
\end{align}
where
\begin{align}\label{form:Rzerobdd}
\begin{split}
\Rzerobdd & = \brac{\pap\brac{\frac{e^{i\ap}}{\Zap}}}(\Jzerobdd+\Qzerobdd) + e^{i\ap}\Dap\Qzerobdd -  \bvarap\brac{\pap\Dt\brac{\frac{e^{i\ap}}{\Zap}} + \Dt\pap\brac{\frac{e^{i\ap}}{\Zap}}} \\
& \quad - (\Dt\bvarap)\brac{\pap\brac{\frac{e^{i\ap}}{\Zap}}} + 2i\Azero\brac{\Dapabs\frac{1}{\Zapabs}}\brac{\pap\brac{\frac{e^{i\ap}}{\Zap}}} \\
& \quad + i\brac{\frac{1}{\Zapabs^2}\pap\Azero}\brac{\pap\brac{\frac{e^{i\ap}}{\Zap}}}
\end{split}
\end{align}
The other main equation is also derived in a very similar manner. We define
\begin{align}\label{eq:Jonebdd}
\Jonebdd = \Dt\Azero + \Azero\brac{\bap - \Dap\Zt - \Dapbar\Ztbar}
\end{align}
Following the same calculations as in \secref{sec:maineqn} we get
\begin{align*}
\brac{\Dt^2 - i\frac{\Azero}{\Zapabs^2}\pap}\Ztbar = i\frac{\Jonebdd}{\Zap}
\end{align*}
The formula \eqref{eq:maineqcomDap} also holds here and hence by following the same argument as in \secref{sec:maineqn} we obtain
\begin{align}\label{eq:mainDapDapZtbarbdd}
\brac{\Dt^2 - i\frac{\Azero}{\Zapabs^2}\pap}\Dap^2\Ztbar = \Ronebdd  + i \frac{\wbar^3}{\Zapabs}\pap\brac{\frac{1}{\Zapabs^2}\pap\Jonebdd} 
\end{align}
where
\begin{align}\label{eq:Ronebdd}
\begin{split}
\Ronebdd & = -2(\Dap^2\Ztt)(\Dap\Ztbar) -4(\Dap\Ztt)(\Dap^2\Ztbar) -2(\Dap^2\Zt)(\Dt\Dap\Ztbar) \\
& \quad -2(\Dap\Zt)(\Dap\Dt\Dap\Ztbar)  - 2(\Dap\Zt)(\Dt\Dap^2\Ztbar)  + i(\Dap\Jonebdd) \brac{\Dap\frac{1}{\Zap}} \\
& \quad  + i\Jonebdd\Dap^2\frac{1}{\Zap}  + 2i\wbar(\Dap\wbar)\brac{\frac{1}{\Zapabs^2}\pap\Jonebdd}
\end{split}
\end{align}
Note the equation is the same as \eqref{eq:mainDapDapZtbar} except for the minus sign in front of $\Azero$ and terms involving $\Jonebdd$. 

\subsection{The a priori estimate}

We define the energies
\begin{align*}
\Eonebdd(t) & = \norm[2]{\Ztapbar}^2\brac{\norm[2]{\pap\frac{1}{\Zap}}^2 + \norm[2]{\frac{1}{\Zap}}^2 }\\
\Etwobdd(t) & =  \norm[2]{\Ztapbar}^2\cbrac{\norm[2]{\Dt\nobrac{\pap\brac{\frac{e^{i\ap}}{\Zap}}}}^2 + \norm[\Hhalf]{\frac{\sqrt{\Azero}}{\Zap}\pap\brac{\frac{e^{i\ap}}{\Zap}}}^2} \\
\Ethreebdd(t) & = \norm[2]{\Ztapbar}^2 \cbrac{\norm[2]{\Dt\Dap^2\Ztbar }^2 + \norm[\Hhalf]{\frac{\sqrt{\Azero}}{\Zap}\Dap^2\Ztbar}^2} 
\end{align*}
and also define
\begin{align}
\Eabdd(t) & = \brac{\Eonebdd(t)^2 + \Etwobdd(t)}^\half \label{def:Eabdd}\\
\Ebdd(t) & = \brac{\Eonebdd(t)^3 + \Etwobdd(t)^{\frac{3}{2}} + \Ethreebdd(t)}^{\frac{1}{3}} \label{def:Ebdd}
\end{align}
Note that the energy is essentially the same as for  $\Rsp$ except for some small differences. The first difference is the addition of $\norm[2]{\Ztapbar}^2 \norm[2]{\frac{1}{\Zap}}^2$ in $\Eonebdd(t)$ which is a lower order term (note that this is very minor change as we are on a bounded domain). The other change is replacing $\pap\frac{1}{\Zap}$ with $\pap\brac{\frac{e^{i\ap}}{\Zap}}$ in $\Etwobdd(t)$. This is done as $\pap\frac{1}{\Zap}$ is no longer boundary value of a holomorphic function whereas $\pap\brac{\frac{e^{i\ap}}{\Zap}}$ is (see \eqref{eq:oneoverZapTr} and \eqref{eq:HiltilpaponeoverZap}). 

For this energy we have the following energy estimate.
\begin{theorem}\label{thm:aprioriEbdd}
Let $T>0$ and let $(\Z, \Zt )$ be a solution to the  water wave equation \eqref{eq:systemonebdd} in the time interval $[0,T]$ with $(\Zap - 1, \frac{1}{\Zap} - 1, \Zt) \in \Linfty([0,T], H^s(\Rsp)\times H^s(\Rsp)\times H^{s + \half}(\Rsp))$ for some $s \geq 4$. Then  $\E(t) < \infty$ for all $t\in [0,T]$ and there exists a universal constant $c >0$ so that for all $t \in [0,T)$ we have
\begin{align*}
\frac{d\Eabdd(t)}{dt} \leq c \brac{\norm[\Linfty\cap\Hhalf]{\Dap\Ztbar}  +  \nobrac{\norm[2]{\Ztapbar}}\brac{\norm[2]{\pap\frac{1}{\Zap}} + \norm[2]{\frac{1}{\Zap}} }} \Eabdd(t) \leq c^2 \Eabdd(t)^{3/2}
\end{align*}
and also 
\begin{align*}
\frac{d\Ebdd(t)}{dt} \leq c\Eabdd(t)^\half\Ebdd(t) \leq c^2 \Ebdd(t)^{3/2}
\end{align*}
\end{theorem}

The proof of this theorem is essentially the same as the one for \thmref{thm:aprioriE} with only a few changes with most of them related to some modifications to \secref{sec:quantEa}. Some of the common changes are:
\begin{enumerate}[ label=(\alph*)]
\item Replacing occurrences of $\Zt$ with $\Zt - \Av(\Zt)$.
\item Replacing $\norm[2]{\pap\frac{1}{\Zap}}$ with $\brac{\norm[2]{\pap\frac{1}{\Zap}} + \norm[2]{\frac{1}{\Zap}}}$ on the right hand side of the estimates. 
\item Replacing $\Dt\brac{\pap\frac{1}{\Zap}}$ and $\frac{\sqrt{\Ag}}{\Zap}\pap\frac{1}{\Zap}$ with $\Dt\brac{\pap\brac{\frac{e^{i\ap}}{\Zap}}}$ and $\frac{\sqrt{\Azero}}{\Zap}\pap\brac{\frac{e^{i\ap}}{\Zap}}$ respectively.
\end{enumerate}
We will give a few examples of these common changes and also explain any other non standard changes.

\begin{enumerate}[widest = 99, leftmargin =*, align=left, label=\arabic*)]

\item The estimates for $\norm[\Linfty]{\Azero}$ and $\norm[\Hhalf]{\Azero}$ remain the same. Similarly for $\norm[2]{\pap\frac{1}{\Zapabs}}$ and $\norm[2]{\Dapabs\w}$. Now observe that
\begin{align*}
& 2\pap\Pa\brac{\frac{\Zt - \Av(\Zt)}{\Zap}} \\
& = (\Id - \Hil)\cbrac{\Dap\Zt + (\Zt - \Av(\Zt))\pap\frac{1}{\Zap}} \\
& = 2\Dap\Zt - (\Id + \Hil)\Dap\Zt + (\Id - \Hil)\cbrac{(\Zt - \Av(\Zt))\pap\frac{1}{\Zap}} \\
& = 2\Dap\Zt + \sqbrac{\frac{1}{\Zap}, \Hil}\Ztap + \sqbrac{\Zt - \Av(\Zt), \Hil}\brac{\pap\frac{1}{\Zap} + i\brac{\frac{1}{\Zap} - \Av\brac{\frac{1}{\Zap}}}} \\
& \quad  - i(\Id - \Hil)\cbrac{(\Zt - \Av(\Zt))\brac{\frac{1}{\Zap} - \Av\brac{\frac{1}{\Zap}}}}
\end{align*}
Hence using \propref{prop:commutator} and \propref{prop:Hardy} we get
\begin{align*}
& \norm[\infty]{\pap\Pa\brac{\frac{\Zt - \Av(\Zt)}{\Zap}}} \\
& \lesssim \norm[\infty]{\Dap\Ztbar} + \norm[2]{\pap\frac{1}{\Zap}}\norm[2]{\Ztapbar}  + \norm[H^1]{(\Zt - \Av(\Zt))\brac{\frac{1}{\Zap} - \Av\brac{\frac{1}{\Zap}}}} \\
& \lesssim \norm[\infty]{\Dap\Ztbar} + \norm[2]{\pap\frac{1}{\Zap}}\norm[2]{\Ztapbar} \\
& \lesssim \norm[\infty]{\Dap\Ztbar} +  \brac{\norm[2]{\pap\frac{1}{\Zap}} + \norm[2]{\frac{1}{\Zap}}}\norm[2]{\Ztapbar} 
\end{align*}
Hence the estimate \eqref{eq:papPaZtZap} gets modified by the changes mentioned above. 

The estimate for $\norm[\infty]{\bap}$ now follows similar to the above calculation by using \eqref{eq:bapbddcomm}. We also note from \eqref{eq:systemonebdd} and \propref{prop:Hardy} that
\begin{align*}
\norm[\infty]{\bvar} \lesssim \norm[H^1]{\bvar} \lesssim \norm[H^1]{\frac{\Zt - \Av(\Zt)}{\Zap}} \lesssim \norm[2]{\Ztapbar}\brac{\norm[2]{\pap\frac{1}{\Zap}} + \norm[2]{\frac{1}{\Zap}}}
\end{align*}

\item To get the estimate for $\norm[2]{\Dapabs(\bap - \Dap\Zt - \Dapbar\Ztbar)}$ we follows as in the $\Rsp$ case and see that
\begin{align*}
\Dapabs(\bvarap -\Dap\Zt - \Dapbar\Ztbar)  = \Real \cbrac{\frac{\w}{\Zap}(\Id - \Hiltil)\pap(\bvarap -\Dap\Zt - \Dapbar\Ztbar)} 
\end{align*}
Note that to get this equality we had to use $\Id - \Hiltil$ instead of $\Id - \Hil$ (we are again using the fact that $\bap - \Dap\Zt - \Dapbar\Ztbar$ is real valued). However we see that $(\Id - \Hiltil)\pap = (\Id - \Hil)\pap$ and hence we get
\begin{align*}
\Dapabs(\bvarap -\Dap\Zt - \Dapbar\Ztbar)  = \Real \cbrac{\frac{\w}{\Zap}(\Id - \Hil)\pap(\bvarap -\Dap\Zt - \Dapbar\Ztbar)} 
\end{align*}
Now following the same proof as in the $\Rsp$ case and using \eqref{eq:bapbdd} instead of \eqref{form:bvarapnew} we get the estimate 
\begin{align*}
& \norm[2]{\Dapabs(\bvarap - \Dap\Zt - \Dapbar\Ztbar)} \\
& \lesssim \brac{ \norm*[\Big][2]{\pap\frac{1}{\Zap}} + \norm[2]{\frac{1}{\Zap}}}^2\norm[2]{\Ztapbar} + \brac{\norm*[\Big][2]{\pap\frac{1}{\Zap}} + \norm[2]{\frac{1}{\Zap}}}\norm[\infty]{\Dap\Ztbar}
\end{align*}
We follow this same approach for estimating the term $\norm[2]{\Dapabs\Jzerobdd}$ and also for the term $\norm[2]{\Dapabs\brac{\frac{1}{\Zapabs^2}\pap\Jonebdd}}$.

\item The proof of the estimate for $\norm[2]{\Dap\Dapbar\Ztbar}$ follows the same as for the  $\Rsp$ case and we get the estimate
\begin{align*}
\norm[2]{\Dap\Dapbar\Ztbar}  & \lesssim \brac{\norm*[\Big][2]{\pap\frac{1}{\Zap}} + \norm[2]{\frac{1}{\Zap}}}^2\norm[2]{\Ztapbar} + \brac{\norm*[\Big][2]{\pap\frac{1}{\Zap}} + \norm[2]{\frac{1}{\Zap}}}\norm[\infty]{\Dap\Ztbar} \\
& \quad + \norm[2]{\Dt\brac{\pap\brac{\frac{e^{i\ap}}{\Zap}}}}
\end{align*}
Similarly the estimates for $\norm[\Linfty\cap\Hhalf]{\Dapbar\Ztbar}$, $\norm[\Linfty\cap\Hhalf]{\Dap\Ztbar}$ and most other terms follow in the same manner as $\Rsp$.  From now on we will concentrate on the non-standard changes. 

\item Similar to the $\Rsp$ case, we easily obtain the estimates
\begin{align*}
& \norm[\Hhalf]{\w^n\frac{\sqrt{\Azero}}{\Zapabs}\pap\brac{\frac{e^{i\ap}}{\Zap}}} + \norm[\Hhalf]{e^{-i\ap}\w^n\frac{\sqrt{\Azero}}{\Zapabs}\pap\brac{\frac{e^{i\ap}}{\Zap}}} \\
& \lesssim_n  \norm[\Hhalf]{\frac{\sqrt{\Azero}}{\Zap}\pap\brac{\frac{e^{i\ap}}{\Zap}}} + \brac{\norm[2]{\pap\frac{1}{\Zap}} + \norm[2]{\frac{1}{\Zap}} }^2\norm[2]{\Ztapbar} \\
& \quad + \brac{\norm[2]{\pap\frac{1}{\Zap}}  + \norm[2]{\frac{1}{\Zap}}}\norm[\Linfty\cap\Hhalf]{\Dap\Ztbar}
\end{align*}
Now observe that 
\begin{align*}
e^{-i\ap}\w^n\frac{\sqrt{\Azero}}{\Zapabs}\pap\brac{\frac{e^{i\ap}}{\Zap}} = \w^n\frac{\sqrt{\Azero}}{\Zapabs}\pap\nobrac{\frac{1}{\Zap}} + i\w^n \frac{\sqrt{\Azero}}{\Zapabs\Zap}
\end{align*}
We can easily control the second term
\begin{align*}
& \norm[\Hhalf]{ i\w^n \frac{\sqrt{\Azero}}{\Zapabs\Zap}} \\
& \lesssim \norm[H^1]{ i\w^n \frac{\sqrt{\Azero}}{\Zapabs\Zap}} \\
& \lesssim  \brac{\norm[2]{\pap\frac{1}{\Zap}} + \norm[2]{\frac{1}{\Zap}} }^2\norm[2]{\Ztapbar} + \brac{\norm[2]{\pap\frac{1}{\Zap}}  + \norm[2]{\frac{1}{\Zap}}}\norm[\Linfty\cap\Hhalf]{\Dap\Ztbar}
\end{align*}
Therefore we get control for the first term. Now again by following the proof in the $\Rsp$ case, we get the estimates
\begin{align*}
& \norm[\Hhalf]{\w^n\frac{\sqrt{\Azero}}{\Zapabs}\pap\frac{1}{\Zap}} + \norm[\Hhalf]{\w^n\frac{\sqrt{\Azero}}{\Zapabs}\pap\frac{1}{\Zapabs}} +  \norm[\Hhalf]{\w^n\frac{\sqrt{\Azero}}{\Zapabs^2}\pap \w}\\
& \lesssim_n  \norm[\Hhalf]{\frac{\sqrt{\Azero}}{\Zap}\pap\brac{\frac{e^{i\ap}}{\Zap}}} + \brac{\norm[2]{\pap\frac{1}{\Zap}} + \norm[2]{\frac{1}{\Zap}} }^2\norm[2]{\Ztapbar} \\
& \quad + \brac{\norm[2]{\pap\frac{1}{\Zap}}  + \norm[2]{\frac{1}{\Zap}}}\norm[\Linfty\cap\Hhalf]{\Dap\Ztbar}
\end{align*}
Following this same logic, we also get the estimates
\begin{align*}
&  \norm[\Hhalf]{\w^n\frac{\Azero}{\Zapabs}\pap\brac{\frac{e^{i\ap}}{\Zap}}} +  \norm[\Hhalf]{e^{-i\ap}\w^n\frac{\Azero}{\Zapabs}\pap\brac{\frac{e^{i\ap}}{\Zap}}}  \\
& \quad + \norm[\Hhalf]{\w^n\frac{\Azero}{\Zapabs}\pap\frac{1}{\Zap}} + \norm[\Hhalf]{\w^n\frac{\Azero}{\Zapabs}\pap\frac{1}{\Zapabs}} +  \norm[\Hhalf]{\w^n\frac{\Azero}{\Zapabs^2}\pap \w}  \\
& \lesssim_n   \norm[2]{\Ztapbar}\norm[\Hhalf]{\frac{\sqrt{\Azero}}{\Zap}\pap\brac{\frac{e^{i\ap}}{\Zap}}} \\
& \quad + \brac{\norm[\Linfty\cap\Hhalf]{\Dap\Ztbar} + \norm[2]{\Ztapbar}\brac{\norm[2]{\pap\frac{1}{\Zap}} + \norm[2]{\frac{1}{\Zap}} }}^2
\end{align*}

\item We need a slight modification to the proof of the estimate of $\norm[2]{\Dapabs\brac{\frac{1}{\Zapabs^2}\pap\Azero}}$. This is because we no longer have the property that $\Hil f = f$ then $\Hil\brac{\frac{1}{\Zap^2}\pap f} = \frac{1}{\Zap^2}\pap f$, which is true in $\Rsp$. To remedy this, we write
\begin{align*}
 \Dapabs\brac*[\bigg]{\frac{1}{\Zapabs^2}\pap\Azero}  & = \Real \cbrac*[\bigg]{\frac{\w^3\wbar^3 e^{-i\ap} e^{i\ap}}{\Zapabs}(\Id - \Hiltil)\pap\brac*[\bigg]{\frac{1}{\Zapabs^2}\pap\Azero}} \\
 & = \Real \cbrac*[\bigg]{\frac{\w^3\wbar^3 e^{-i\ap} e^{i\ap}}{\Zapabs}(\Id - \Hil)\pap\brac*[\bigg]{\frac{1}{\Zapabs^2}\pap\Azero}}
\end{align*}
Now 
\begin{align*}
& \frac{\wbar^3 e^{i\ap}}{\Zapabs}(\Id - \Hil)\pap\brac*[\bigg]{\frac{1}{\Zapabs^2}\pap\Azero} \\
 & = (\Id - \Hil) \cbrac*[\bigg]{\Dap \brac*[\bigg]{\frac{e^{i\ap}}{\Zap^2}\pap\Azero} - 2\brac*[\bigg]{\frac{\wbar}{\Zapabs^2}\pap\Azero}\brac{\Dap\wbar}e^{i\ap} -i\frac{e^{i\ap}\wbar^2}{\Zap}\brac{\frac{1}{\Zapabs^2}\pap\Azero} } \\
 & \quad - \sqbrac{ \frac{\wbar^3e^{i\ap}}{\Zapabs} ,\Hil}\pap \brac*[\bigg]{\frac{1}{\Zapabs^2}\pap\Azero}
\end{align*}
Now we can follow the same proof as in the $\Rsp$ by using the property that if $\Hil f = f$,  then $\Hil(\Dap f) = \Dap f$ and $\Hil\brac{\frac{e^{i\ap}}{\Zap^2}\pap f } = \frac{e^{i\ap}}{\Zap^2}\pap f$. 

\item The estimate for  $\norm[\infty]{\frac{\Dt\Azero}{\Azero}} $ follows essentially the same was as in the $\Rsp$ with some minor modifications. First we obviously have to use \lemref{lem:DtAgoverAgbdd} instead of \lemref{lem:DtAgoverAg}. The main difference comes in the calculation of \eqref{eq:Agfractemp1}. To remedy this use the following identity
\begin{align*}
\cot\brac{\frac{s - \ap}{2}} - \cot\brac{\frac{s - \bp}{2}} = \frac{\sin\brac{\frac{\ap - \bp}{2}}}{\sin\brac{\frac{\ap- s}{2}}\sin\brac{\frac{\bp - s}{2}}}
\end{align*}
Then by noting the difference between the formulae of \eqref{eq:Azerobddcalc} and \eqref{eq:Azero}, instead of \eqref{eq:Agfractemp1} we get
\begin{align*}
& \frac{\Azero(\ap) - \Azero(\bp)}{\sin\brac{\frac{\ap - \bp}{2}} } \\
& = \Imag\cbrac{\frac{1}{2\pi i } \int_0^{2\pi} \frac{1}{\sin\brac{\frac{\bp - s}{2}}} \brac{\frac{\Zt(\ap) - \Zt(s)}{\sin\brac{\frac{\ap - s}{2}}}}\Ztapbar(s) \diff s + \frac{\Zt(\ap) - \Zt(\bp)}{\sin\brac{\frac{\ap - \bp}{2}}}\Ztapbar(\bp)  } 
\end{align*}
The rest of the proof is the same as the one in $\Rsp$. 

\item In the bounded domain case, we estimate the term $\norm[2]{(\Id - \Hil)\Dt^2\brac{\pap\brac{\frac{e^{i\ap}}{\Zap}}}}$ instead of $\norm[2]{(\Id - \Hil)\Dt^2\brac{\pap\frac{1}{\Zap}}}$ and similarly replace $\norm[2]{(\Id - \Hil)\cbrac{i\frac{\Ag}{\Zapabs^2}\pap\brac{\pap\frac{1}{\Zap}}}}$ with $ \norm[2]{(\Id - \Hil)\cbrac{i\frac{\Azero}{\Zapabs^2}\pap\brac{\pap\frac{e^{i\ap}}{\Zap}}}}$. The replacement of $\pap\frac{1}{\Zap}$ with $\pap\brac{\frac{e^{i\ap}}{\Zap}}$ is natural because the main equation \eqref{eq:paponeoverZapbdd} is in terms $\pap\brac{\frac{e^{i\ap}}{\Zap}}$. The proof to estimate the term $ \norm[2]{(\Id - \Hil)\cbrac{i\frac{\Azero}{\Zapabs^2}\pap\brac{\pap\frac{e^{i\ap}}{\Zap}}}}$ remains the same, however the proof for the term $\norm[2]{(\Id - \Hil)\Dt^2\brac{\pap\brac{\frac{e^{i\ap}}{\Zap}}}}$ gets slightly modified as the formula \eqref{eq:IminusHDt2f} does not hold exactly. It gets modified as follows: Let $f = \pap\brac{\frac{e^{i\ap}}{\Zap}}$ then we still have $\Hil f = f$ and using \eqref{eq:sqfgbounded} and \propref{prop:tripleidentity} we get
\begin{align*}
& (\Id - \Hil)\Dt^2f \\
& = \sqbrac{\Dt,\Hil}\Dt f + \Dt\sqbrac{\Dt,\Hil}f \\
& = \sqbrac{\bvar,\Hil}\pap\Dt f + \Dt\sqbrac{\bvar,\Hil}\pap f \\
& = \sqbrac{\bvar,\Hil}\pap\Dt f + \Dt\sqbrac*{\bvar,\Hiltil}\pap f + \Dt\sqbrac{\bvar,\Av}\pap f \\
& = \sqbrac{\bvar,\Hil}\pap\Dt f + \sqbrac*{\bvar,\Hiltil}\pap\Dt f + \sqbrac*{\Dt\bvar,\Hiltil}\pap f - \sqbrac{\bvar, \bvar ; \pap f} + \Dt\sqbrac{\bvar,\Av}\pap f
\end{align*}
Now all terms except for the last term are handled as in the $\Rsp$ case. For the last term we observe that
\begin{align*}
& \Dt\sqbrac{\bvar,\Av}\pap f \\
& = -\partial_t \cbrac{\frac{1}{2\pi}\int_0^{2\pi} \bvar(s) f_\ap(s) \diff s } \\
& =  - \nobrac{\frac{1}{2\pi}\int_0^{2\pi} (\partial_t \bvar)(s) f_\ap(s)  + \bvar(s)\pap(\partial_t f)(s)\diff s } \\
& =  - \nobrac{\frac{1}{2\pi}\int_0^{2\pi} (\Dt \bvar )(s) f_\ap(s)  + \bvar(s)\pap(\Dt f)(s)  } \\
& \quad + \nobrac{\frac{1}{2\pi}\int_0^{2\pi} (  \bvar\bvarap)(s) f_\ap(s)  + \bvar(s)\pap( \bvar f_\ap)(s)\diff s } \\
& =  \nobrac{\frac{1}{2\pi}\int_0^{2\pi} (\pap\Dt \bvar )(s) f(s)  + \bvarap(s)(\Dt f)(s)  } 
\end{align*}
where in the last step we integrated by parts. Hence from \propref{prop:Hardy} we see that
\begin{align*}
\norm[2]{\Dt\sqbrac{\bvar,\Av}\pap f} & \lesssim \norm[2]{\pap\Dt\bvarap}\norm[2]{f} + \norm[2]{\bvarap}\norm[2]{\Dt f} \\
& \lesssim \norm[BMO]{\pap\Dt\bvarap}\norm[2]{f} + \norm[\infty]{\bvarap}\norm[2]{\Dt f} \\
& \lesssim \norm[\Hhalf]{\pap\Dt\bvarap}\norm[2]{f} + \norm[\infty]{\bvarap}\norm[2]{\Dt f}
\end{align*}
Hence $\norm[2]{\Dt\sqbrac{\bvar,\Av}\pap f} $ can be controlled in the same way as is done in the $\Rsp$ case. Therefore the term $\norm[2]{(\Id - \Hil)\Dt^2\brac{\pap\brac{\frac{e^{i\ap}}{\Zap}}}}$ is controlled. By the same process, we also control the term $\norm[2]{(\Id - \Hil)\Dt^2\Dap^2\Ztbar }$. The proof for  $\norm[2]{(\Id - \Hil)\cbrac{i\frac{\Ag}{\Zapabs^2}\pap\Dap^2\Ztbar } }$ remains the same as in the $\Rsp$ case. 

\end{enumerate}
This concludes the changes needed with respect to \secref{sec:quantEa} and \secref{sec:quantE}. There are essentially no changes needed to \secref{sec:closing} and hence the proof of \thmref{thm:aprioriEbdd} is complete.

\subsection{Proof of \thmref{thm:existencemainbdd} and \thmref{thm:blowup}}\label{sec:existenceanduniqbdd}

The proof of \thmref{thm:existencemainbdd} is very similar to the proof of \thmref{thm:existencemain} with only minor differences. Note that in the unbounded domain case, the crucial place where gravity played a role in the uniqueness proof was the lower bound on $\Ag$ (see \eqref{eq:Ag}, \eqref{eq:Azero}) for $g>0$. In the compact domain case, even though we have no gravity, if the velocity is not constant, we still have a positive lower bound on $\Azero$ (see \eqref{eq:Azerolowerbound}, \eqref{eq:oneoversqrtAzerobdd}) due to compactness. Due to this, the uniqueness proof for the compact domain case goes through with only minor changes.  

As in the unbounded case, we need to define the notion of a solution and also define the class of smoothly approximable solutions $\mathcal{SA}$. The definitions are very similar to the unbounded case except for some minor changes. 

\begin{definition}\label{def:solutionbdd}
Let the initial data $(\U,\Psi,\Pfrak)(0)$ be given as in the paragraph above \thmref{thm:existencemainbdd}. Let $ T>0$ and let $\U, \Psi: \Dsp \times \sqbrac{0,T} \to  \Csp$ and $\Pfrak : \Dsp \times \sqbrac{0,T} \to  \Rsp$. We say $(\U,\Psi, \Pfrak)(t)$ solves the Cauchy problem for the system \eqref{eq:EulerRiem2} in the time interval $[0,T]$ if the following holds:
\begin{enumerate}
\item $\U,\Psi,\Pfrak$ extend continuously to $\Dspbar\times \sqbrac{0,T}$ and $(\U,\Psi) \in C^1(\Dsp\times (0,T))$. Also for each $t \in [0,T]$ we have $\Pfrak(\cdot,t) \in C^1(\Dsp)$. 

\item $\U(\cdot,t), \Psi(\cdot,t)$ are holomorphic maps for each $t \in [0,T]$, $\Psi_{\zp}(\zp,t) \neq 0$ for $(\zp,t) \in \Dsp\times [0,T]$ and $\Psizp(0,t) > 0$  for any $t\in\sqbrac{0,T}$. Moreover $\frac{1}{\Psizp}$ and $\frac{\Psi_t}{\Psizp}$ extend continuously to $\Dspbar\times \sqbrac{0,T}$ (and by abuse of notation we will continue to denote these extensions as $\frac{1}{\Psizp}$ and $\frac{\Psi_t}{\Psizp}$ respectively). 

\item We have
\begin{align*}
\sup_{t \in [0,T]}\cbrac{\sup_{0< r < 1} \norm[H^1([0, 2\pi], \diff \theta)]{U(re^{i\theta},t)} + \sup_{0< r < 1} \norm[H^1([0, 2\pi], \diff \theta)]{\frac{1}{\Psizp}(re^{i\theta},t)} }  < \infty
\end{align*}
\item $(\U,\Psi,\Pfrak)$ solves \eqref{eq:EulerRiem2} in $\Dsp\times [0,T]$ with the given initial data and $\Pfrak$ solves \eqref{eq:DeltaPfrakbdd}. 
\end{enumerate}
\end{definition}

Similar to the unbounded domain case, we need to define a notion of difference of two solutions.  Let $(\Z,\Zt)_a$ and $(\Z,\Zt)_b$ be two solutions of the water wave equation \eqref{eq:systemonebdd} and let $h_a, h_b$ be the homeomorphisms from \eqref{eq:hbdd} for the respective solutions. Define $\htil$ and $\Util$ in the same way as \eqref{eq:htilandUtilunbdd} i.e.
\begin{align*}
\htil = h_b \compose h_a^{-1} \quad \tx{ and } \quad \Util = U_{\htil} = U_{h_a}^{-1}U_{h_b}
\end{align*}
Similarly define $\Delta(f)$ as in \eqref{eq:Deltafunbdd} i.e.
\begin{align*}
\Delta (f) = f_a - \Util(f_b)
\end{align*}
We define 
\begin{align}\label{eq:Fcalbdd}
\begin{split}
& \Fcaltil_{\Delta}((\Z,\Zt)_a, (\Z,\Zt)_b)(t) \\
& = \norm[\Hhalf]{\Delta(\Zt)} + \norm[\Hhalf]{\Delta(\Ztt)} + \norm[\Hhalf]{\Delta\brac{\frac{e^{i\ap}}{\Zap}}} + \norm[2]{\htil_\ap - 1} + \norm[2]{\Delta(\Dap\Zt)}\\
& \quad  + \norm[2]{\Delta(\Azero)} + \norm[2]{\Delta(\bvarap)} + \abs{\Avg(\Delta(\Zt))} + \abs{\Avg\brac{\Delta\brac{\frac{e^{i\ap}}{\Zap}}}}
\end{split}
\end{align}
As compared to the definition \eqref{eq:Fcal}, here we have the addition of two terms which takes into account the difference of the averages of $\Zt$ and $\frac{e^{i\ap}}{\Zap}$. We can now define the class $\mathcal{SA}$ for the bounded case:

\begin{definition}\label{def:solutionSAbdd}
Let $(\U,\Psi, \Pfrak)(t)$ be a solution to \eqref{eq:EulerRiem2} in the time interval $[0,T]$ in the sense of \defref{def:solutionbdd}. We say that this solution lies in the smoothly approximable class $\mathcal{SA}$ if the following holds:

For any $\widetilde{T} \in [0,T]$, there exists $T_1, T_2$ such that $0\leq T_1 \leq \widetilde{T} \leq T_2 \leq T$ satisfying $T_2 - T_1>0$ and if $\widetilde{T} \in (0,T)$  then we have $T_1 < \widetilde{T} < T_2$,  and a sequence of smooth $2\pi$ periodic functions $(\Z^{(n)}, \Zt^{(n)}): \mathbb{R}\times \sqbrac{T_1,T_2} \rightarrow \mathbb{C}$ for $n \geq 1$ satisfying:
\begin{enumerate}[label=(\alph*)]
\item For each $n\in \mathbb{N}$ we have $\Zap^{(n)}(\ap,t) \neq 0$ for all $(\ap,t) \in \Rsp \times [T_1,T_2]$. 
\item We have 
\begin{align*}
\sup_{n \geq 1, t\in [T_1,T_2]} \cbrac*[\bigg]{\norm*[\big][H^1]{\Zt^{(n)}}(t) + \norm*[\bigg][H^1]{\frac{1}{\Zap^{(n)}}}(t) } < \infty
\end{align*}
\item There exist holomorphic functions $(\U^{(n)}, \Psi^{(n)})(\cdot,t): \Dsp \to \Csp$ whose boundary values are $(\Ztbar^{(n)},\Z^{(n)})(\cdot,t)$ and which satisfy for each $n$ the property $\Psi_{\zp}^{(n)}(\zp,t) \neq 0$ for all $(\zp,t) \in \Dsp\times [T_1,T_2]$ and $\Psi_\zp^{(n)}(0,t) > 0$ for $t \in [T_1,T_2]$. Let $\Pfrak^{(n)}$ be the function satisfying 
\begin{align*}
\Delta \Pfrak^{(n)} = -2\abs*{\U_\zp^{(n)}}^2 \quad \tx{ on } \Dsp, \qq \Pfrak^{(n)} = 0 \quad \tx{ on } \partial \Dsp
\end{align*}
Then $(\U^{(n)}, \Psi^{(n)}, \Pfrak^{(n)}, \frac{1}{\Psizp^{(n)}}) \to (\U, \Psi, \Pfrak, \frac{1}{\Psizp})$ uniformly on $\Dspbar\times [T_1, T_2]$ as $n \to \infty$.

\item For each $n\in \mathbb{N}$, $(\Z^{(n)}, \Zt^{(n)})$ solves the system \eqref{eq:systemonebdd}  and we have the uniform bound $\sup_{n \in \Nsp, t \in [T_1, T_2]}\cbrac{ \Ecaltil_n(t) + \norm[2]{\Ztap^{(n)}}(t)} < \infty$. In addition if we fix the Lagrangian parameterizations for these solutions by imposing $h^{(n)}(\ap, \widetilde{T}) = h(\ap, \widetilde{T})$ for all $\ap \in \Rsp$ and $n \in \Nsp$, then  $\sup_{t \in [T_1, T_2]}\Fcaltil_\Delta((\Z^{(n)}, \Zt^{(n)}), (\Z, \Zt))(t) \to 0$ as $n \to \infty$. 
\end{enumerate}
Furthermore we say that $(\U,\Psi, \Pfrak)(t)$ belongs in the class $\mathcal{SA}$ in the time interval $[0,T)$, if for any $0< T^* < T$ the solution belongs in the class $\mathcal{SA}$ in the time interval $[0, T^*]$. 
\end{definition}

Similar to \lemref{lem:equivSobolevunbdd} we have the following relations between the different energies. 

\begin{lemma}\label{lem:equivSobolevbdd}
Let $T>0$ and let $(\Z, \Zt )(t)$ be a solution to the water wave equation\eqref{eq:systemonebdd} in the time interval $[0,T]$ with $(\Zap, \frac{1}{\Zap}, \Zt) \in \Linfty([0,T], H^s(\Sone)\times H^s(\Sone)\times H^{s + \half}(\Sone))$ for some $s \geq 4$. Then we have the following:
\begin{enumerate}
\item There exists a universal constant $M_1>0$ such that $\Ebdd(t) \leq M_1\Ecaltil(t)$ for all $t \in [0,T]$. Moreover there exists a universal increasing function $C_1:[0,\infty) \to [0,\infty)$ such that if $\norm[2]{\Ztap}(0) >0$, then we have
\begin{align}\label{eq:controlEcalfromEbdd}
\Ecaltil(t) \leq C_1\brac{\Ebdd(t) + \norm[\infty]{\frac{1}{\sqrt{\Azero}}}(t) + \norm[2]{\Ztapbar}(t)} \quad \tx{ for all } t \in [0,T]
\end{align}
and also
\begin{align}\label{eq:controlEcalfromEsupbdd}
\sup_{t \in [0,T]}\Ecaltil(t) \leq C_1\brac{\sup_{t \in [0,T]}\Ebdd(t) + \norm[\infty]{\frac{1}{\sqrt{\Azero}}}(0) + T + \norm[2]{\Ztapbar}(0)}
\end{align}

\item Assume that  $\norm[2]{\Ztap}(0) > 0$. Then there exists a universal increasing function $C_2:[0,\infty) \to [0,\infty)$ so that we have
\begin{align}\label{eq:SobolevlowerfromEgzerobdd}
\begin{aligned}
&\sup_{t \in [0,T]} \cbrac{ \norm[H^1]{\Zt}(t) + \norm[H^1]{\frac{1}{\Zap}} + \frac{1}{\norm[2]{\Ztap}^2(t)}} \\
&\leq C_2\brac{\sup_{t \in [0,T]}\Ebdd(t) + T  + \norm[H^1]{\Zt}(0) + \frac{1}{\norm[2]{\Ztap}^2(0)} }
\end{aligned}
\end{align}
and also
\begin{align}\label{eq:SobolevfromEcalbdd}
\begin{aligned}
& \sup_{t\in [0,T]}\cbrac{\norm[H^{2.5}]{\Zt}(t) + \norm[H^2]{\Zap}(t) + \norm[H^2]{\frac{1}{\Zap}}(t)} \\
& \leq C_2\brac{\sup_{t \in [0,T]}\Ebdd(t)  +  \norm[\infty]{\frac{1}{\sqrt{\Azero}}}(0) + T + \norm[H^1]{\Zt}(0) + \frac{1}{\norm[2]{\Ztap}^2(0)}  + \norm[\infty]{\Zap}(0) }
\end{aligned}
\end{align}

\end{enumerate}
\end{lemma}
\begin{proof}
The proof of this lemma is essentially the same as the proof of \lemref{lem:equivSobolevunbdd}. The only main difference is that instead of the assumption of $g>0$ in \lemref{lem:equivSobolevunbdd}, we have the condition of having the term $\norm[\infty]{\frac{1}{\sqrt{\Azero}}}(0)$ on the right hand side. 
Now assume that $\norm[2]{\Ztap}(0) > 0$. Note that this implies that $\norm[2]{\Zt - \Avg(\Zt)}(0) > 0$. From \eqref{eq:Azerobddcalc} we have
\begin{align}\label{eq:Azerolowerbound}
\begin{aligned}
\Azero(\ap,t) & =  \frac{1}{8\pi} \int_0^{2\pi} \abs*[\Bigg]{\frac{\Zt(\ap, t) - \Zt(\bp, t)}{\sin\brac{\frac{\ap - \bp}{2}}}}^2 \diff \bp \\
& \gtrsim \int_0^{2\pi} \abs{\Z(\ap,t) - \Zt(\bp,t)}^2\diff \bp \\
& \gtrsim \norm[2]{\Zt - \Avg(\Zt)}^2(t)
\end{aligned}
\end{align}
Hence 
\begin{align*}
\norm[\infty]{\frac{1}{\sqrt{\Azero}}}(0) \lesssim \frac{1}{\norm[2]{\Zt - \Avg(\Zt)}(0) } < \infty
\end{align*}
Now we also have the same estimate as in \eqref{eq:ZtapbarL2norm2plusgestimate} here as well, namely that there exists a universal constant $C_4>0$ such that
\begin{align}\label{eq:ZtapbarL2norm2plusgestimatebdd}
\begin{aligned}
\exp\cbrac{-C_4\int_0^t \Eabdd(s)^\half \diff s} \norm[2]{\Ztapbar}^2(0)  \leq \norm[2]{\Ztapbar}^2(t)  \leq \exp\cbrac{C_4\int_0^t \Eabdd(s)^\half \diff s} \norm[2]{\Ztapbar}^2(0) 
\end{aligned}
\end{align}
Similarly by using the fact that $\norm[\infty]{\frac{\Dt\Azero}{\Azero}}$ is controlled by $\Eabdd$, we also easily get the estimate
\begin{align}\label{eq:oneoversqrtAzerobdd}
\begin{aligned}
\exp\cbrac{-C_4\int_0^t \Eabdd(s)^\half \diff s} \norm[\infty]{\frac{1}{\sqrt{\Azero}}}(0)  \leq \norm[\infty]{\frac{1}{\sqrt{\Azero}}}(t)  \leq \exp\cbrac{C_4\int_0^t \Eabdd(s)^\half \diff s} \norm[\infty]{\frac{1}{\sqrt{\Azero}}}(0) 
\end{aligned}
\end{align}
With these bounds, we can now follow the proof of \lemref{lem:equivSobolevunbdd} to complete the proof of this lemma and we leave the details to the reader. 
\end{proof}

Similar to \thmref{thm:existenceSobolevunbdd} we have the following existence result in Sobolev spaces.
\begin{thm}\label{thm:existenceSobolevbdd}
Let $s \geq 4$. Assume that the initial data $(\Z,\Zt)(0)$ satisfies the condition $(\Zap, \frac{1}{\Zap}, \Zt)(0) \in H^{s}(\Sone)\times H^s(\Sone) \times H^{s + \half}(\Sone)$ and there exists a $c>0$ such that $\Azero(\cdot,0) \geq c > 0$. Then there exists a $T > 0$ such that on $[0,T]$ the initial value problem for \eqref{eq:systemone} has a unique solution $(\Z,\Zt)(t)$ satisfying $(\Zap, \frac{1}{\Zap}, \Zt) \in C^l([0,T], H^{s - l}(\Sone)\times H^{s - l}(\Sone) \times H^{s + \half - l}(\Sone))$ for $l = 0,1$. Moreover if $T^{*}$ is the maximal time of existence, then either $T^* = \infty$ or $T^{*} < \infty$ and 
\begin{align*}
\sup_{t \in [0,T^*)} \cbrac{\norm[H^2]{\Zap - 1}(t) + \norm[H^2]{\frac{1}{\Zap} - 1}(t) + \norm[H^{2 + \half}]{\Zt}(t) + \norm[\infty]{\frac{1}{\Azero}}(t) } = \infty
\end{align*}
\end{thm}

The proof of this theorem is very similar to \thmref{thm:existenceSobolevunbdd} and so we skip it. There are some minor differences of this theorem with respect to \thmref{thm:existenceSobolevunbdd} which we now explain. First we have imposed the condition $\Azero(\cdot,0) \geq c > 0$ on the initial data. If $\norm[2]{\Zt - \Av(\Zt)}(0) = 0$, then the initial velocity is constant and hence the trivial solution namely the solution where the domain moves with constant speed for all time is clearly the unique global solution. If $\norm[2]{\Zt - \Av(\Zt)}(0) > 0$, then from \eqref{eq:Azerolowerbound} we do get $\Azero(\cdot,0) \geq c > 0$ and hence the condition is satisfied. Moreover this condition is equivalent to the condition that the Taylor sign condition is satisfied at $t = 0$ from \eqref{eq:gradientpressurebdd}. The main difference of this theorem with \thmref{thm:existenceSobolevunbdd} is in regards to the blow up criterion where there is an additional term of $ \norm[\infty]{\frac{1}{\Azero}}$. This term is necessary to ensure that the Taylor sign condition is satisfied as again can be seen from \eqref{eq:gradientpressurebdd}. In \thmref{thm:existenceSobolevunbdd} we do not such a term because as gravity $g > 0$, we already have $\Ag \geq g$ everywhere from \eqref{eq:Ag} and \eqref{eq:Azero} and therefore $\norm[\infty]{\frac{1}{\Ag}}  \leq \frac{1}{g} $. One of the important features of \thmref{thm:existencemainbdd} is that we remove this condition of  $ \norm[\infty]{\frac{1}{\Azero}}$ in the blow up criterion. 

We now write down a result which allows us to prove uniqueness of solutions in the class $\mathcal{SA}$.

\begin{thm}\label{thm:uniquenessbdd}
Let $s \geq 4$ and $T>0$. Let $(\Z,\Zt)_a$ and $(\Z,\Zt)_b$ be two solutions of the water wave equation \eqref{eq:systemonebdd} in $[0,T]$ satisfying $(\Zap, \frac{1}{\Zap}, \Zt)_a, (\Zap, \frac{1}{\Zap}, \Zt)_b \in C^l([0,T], H^{s - l}(\Sone)\times H^{s - l}(\Sone) \times H^{s + \half - l}(\Sone))$ for $l = 0,1$. Also assume that there exists a constant $c_1>0$ such that  
\begin{align*}
 c_1 & \geq \norm[2]{(\Ztap)_a}(0) + \norm[2]{(\Ztap)_b}(0) + \frac{1}{\norm[2]{(\Ztap)_a}(0)} + \frac{1}{\norm[2]{(\Ztap)_b}(0)}  \\
 & \quad  + \norm[\infty]{\frac{1}{(\Azero)_a}}(0) + \norm[\infty]{\frac{1}{(\Azero)_b}}(0)
\end{align*}
Then there exists a constant $L>0$ depending only on $c_1, T, \sup_{t \in [0,T]}\Ecaltil((\Z,\Zt)_a)(t)$ and $\sup_{t \in [0,T]}\Ecaltil((\Z,\Zt)_b)(t)$ such that
\begin{align*}
\sup_{t \in [0,T]} \Fcaltil_{\Delta}((\Z,\Zt)_a, (\Z,\Zt)_b)(t) \leq L \brac{\Fcaltil_{\Delta}((\Z,\Zt)_a, (\Z,\Zt)_b)(0) + \norm[\infty]{\Delta\brac{\frac{1}{\Zap}}}(0) }
\end{align*}
\end{thm}
\begin{proof}
The proof of this is very similar to the proof of Theorem 3.7 in \cite{Wu19} and so we will only describe the important differences. In the following $L>0$ will be a constant depending only on $c_1, T, \sup_{t \in [0,T]}\Ecaltil((\Z,\Zt)_a)(t)$ and $\sup_{t \in [0,T]}\Ecaltil((\Z,\Zt)_b)(t)$. Now note that from \eqref{eq:ZtapbarL2norm2plusgestimatebdd} and \eqref{eq:oneoversqrtAzerobdd} we see that for all $t \in [0,T]$
\begin{align*}
1 & \gtrsim_L \norm[2]{(\Ztap)_a}(t) + \norm[2]{(\Ztap)_b}(t) + \frac{1}{\norm[2]{(\Ztap)_a}(t)} + \frac{1}{\norm[2]{(\Ztap)_b}(t)}  \\
 & \quad  + \norm[\infty]{\frac{1}{(\Azero)_a}}(t) + \norm[\infty]{\frac{1}{(\Azero)_b}}(t)
\end{align*}
Hence from the definition of $\Ecaltil$ we see that
\begin{align*}
1 & \gtrsim_L  \norm[H^1]{\brac{\frac{1}{\Zap}}_{\n a}}(t)  +  \norm[H^1]{\brac{\frac{1}{\Zap}}_{\n b}}(t) + \norm[2]{\brac{\Dap^2\Ztbar}_a}(t) + \norm[2]{\brac{\Dap^2\Ztbar}_b}(t) \\
& \quad + \norm[2]{\brac{\Dap^2\brac{\frac{e^{i\ap}}{\Zap}}}_{\n a}}(t) + \norm[2]{\brac{\Dap^2\brac{\frac{e^{i\ap}}{\Zap}}}_{\n b}}(t) + \norm[\Hhalf]{\brac{\frac{e^{i\ap}}{\Zap}\Dap^2\Ztbar }_{\n a}}(t) \\
& \quad +  \norm[\Hhalf]{\brac{\frac{e^{i\ap}}{\Zap}\Dap^2\Ztbar }_{\n b}}(t)
\end{align*}
From the energy estimate we also get
\begin{align*}
1 & \gtrsim_L \norm[\Linfty\cap\Hhalf]{\brac{\Dap\Ztbar}_a}(t) + \norm[\Linfty\cap\Hhalf]{\brac{\Dap\Ztbar}_b}(t)
\end{align*}
With these estimate we now have control of all quantities which are required to follow the proof of Theorem 3.7 in \cite{Wu19}. There are only slight differences in the energy estimate which are as follows. Define
\begin{align*}
\kappa = \sqrt{\frac{(\Azero)_a}{\Util(\Azero)_b} \htil_\ap}
\end{align*}
and the energies
\begingroup
\allowdisplaybreaks
\begin{align*}
\mathfrak{F}_0(t) & =  \norm[2]{\htil_\ap - 1}^2 + \abs{\Avg(\Delta(\Zt))}^2 + \abs{\Avg\brac{\Delta\brac{\frac{e^{i\ap}}{\Zap}}}}^2 \\
\mathfrak{F}_1(t) & = \int_0^{2\pi} \frac{\kappa}{(\Azero)_a} \abs{\Delta(\Zap\Zttbar)}^2 \diff \ap - i \int_0^{2\pi}\pap(\Delta(\Ztbar))\overline{\Delta(\Ztbar)} \diff \ap \\
\mathfrak{F}_2(t) & = \int_0^{2\pi} \frac{\kappa}{(\Azero)_a} \abs{\Delta\cbrac{\Zap\Dt\brac{\frac{e^{i\ap}}{\Zap}}}}^2 \diff \ap - i \int_0^{2\pi}\pap\brac{\Delta\brac{\frac{e^{i\ap}}{\Zap}}}\overline{\Delta\brac{\frac{e^{i\ap}}{\Zap}}} \diff \ap \\
\mathfrak{F}_3(t) & = \int_0^{2\pi} \frac{\kappa}{(\Azero)_a} \abs{\Delta(\Zap\Ztttbar)}^2 \diff \ap - i \int_0^{2\pi}\pap(\Delta(\Zttbar))\overline{\Delta(\Zttbar)} \diff \ap
\end{align*}
\endgroup
The main difference from the energy used in Theorem 3.7 in \cite{Wu19} is that we have the addition of two terms in $\mathfrak{F}_0(t)$ and in $\mathfrak{F}_2(t)$ we have used the function$\brac{\frac{e^{i\ap}}{\Zap}}$ instead of $\brac{\frac{1}{\Zap}}$. The modification of $\mathfrak{F}_0(t)$ is necessary as we need to control the difference of average values of the solutions in a bounded domain as there is no corresponding decay condition at infinity as in the unbounded case.  In $\mathfrak{F}_2(t)$ we use $\brac{\frac{e^{i\ap}}{\Zap}}$  as $\brac{\frac{1}{\Zap}}$ is no longer the boundary value of a holomorphic function. 

We then define the energy 
\begin{align*}
\mathfrak{F}(t) = M \mathfrak{F}_0(t) + \mathfrak{F}_1(t) + \mathfrak{F}_2(t) + M^{-1}\mathfrak{F}_3(t)
\end{align*}
where $M>0$ is a constant taken large enough and depending only on values of $c_1$, $T, \sup_{t \in [0,T]}\Ecaltil((\Z,\Zt)_a)(t)$ and $\sup_{t \in [0,T]}\Ecaltil((\Z,\Zt)_b)(t)$. One can then show that $\mathfrak{F}(t)$ is equivalent to the energy $\Fcaltil_{\Delta}((\Z,\Zt)_a, (\Z,\Zt)_b)(t)$ and then prove the following estimate
\begin{align*}
\mathfrak{F}(t) + \int_0^t \mathfrak{F}(s) \diff s \leq L \brac{\mathfrak{F}(0) + \norm[\infty]{\Delta\brac{\frac{1}{\Zap}}}^2(0)} \qq \tx{ for all } t \in [0,T]
\end{align*}
The proof of this estimate follows in the same way as in Theorem 3.7 in \cite{Wu19} and so we skip it. 
\end{proof}

We are now ready to prove \thmref{thm:existencemainbdd}. 

\begin{proof}[Proof of \thmref{thm:existencemainbdd}]

The proof of this result is very similar to \thmref{thm:existencemain} and so we only highlight the main differences. We will first consider the existence part of the result. Let $\half<\ep \leq 1$ and define
\begin{align*}
\Psiep(\zp,0) = \Psi(\ep\zp, 0), \quad \Uep(\zp,0) = \U(\ep\zp,0) \quad \tx{ and }  \quad \hep(\al, 0) = \al 
\end{align*}
We again define their boundary values as 
\begin{align*}
\Zep(\ap,0) = \Psiep(\ap,0) \quad \tx{ and } \quad \Ztepbar(\ap,0) = \Uep(\ap,0)
\end{align*}
which are now smooth for $\half < \ep < 1$. It is clear that for all $\half < \ep < 1$, $\Psiep(\cdot,0), \Uep(\cdot, 0)$ satisfy the conditions on the initial data and we have
\begin{align*}
\sup_{0< r < 1} \norm[H^1([0, 2\pi], \diff \theta)]{\Uep(re^{i\theta},0)} + \sup_{0< r < 1} \norm[H^1([0, 2\pi], \diff \theta)]{\frac{1}{\Psizp^{\ep}}(re^{i\theta},0)} \leq c_0
\end{align*}
and also $\Ecaltil^{\ep}(0) \leq \Ecaltil(0)$. 

If $\norm[2]{\Ztap}(0) = 0$, then the initial velocity is a constant function i.e. there exists $c_1 \in \Csp$ such that $\Zt(\ap,0) = c_1$, for all $\ap \in [0,2\pi]$, and hence the trivial solution namely $\Z(\ap,t) = \Z(\ap, 0) + c_1t$ and $\Zt(\ap,t) = c_1$ for all $t \geq 0$ is a solution to the equation. Note that in this case the time of existence $T = \infty$ and $\Ecaltil(t) = 0$ for all $t \in [0,\infty)$. 

Now let $c_2 = \norm[2]{\Ztap}(0) > 0$. This in particular implies that $c_3 = \norm[2]{\Zt - \Av(\Zt)}(0) > 0$. Hence by the definition of $\Ztep$ we see that there exists $\half < \ep_0 < 1$ such that for all $\ep_0 \leq \ep < 1$ we have 
\begin{align*}
\frac{c_2}{2} \leq \norm[2]{\Ztap^\ep}(0) \leq c_2 \qq \tx{ and  } \quad \frac{c_3}{2} \leq \norm[2]{\Ztep - \Av(\Ztep)}(0) \leq c_3
\end{align*}
From \eqref{eq:Azerolowerbound} this means that for all $\ep_0 \leq \ep < 1$ we have
\begin{align*}
\norm[\infty]{\frac{1}{\sqrt{\Azero^\ep}}}(0) \lesssim \frac{1}{c_3}
\end{align*}
Now the proof of existence of solution in the class $\mathcal{SA}$ follows in the same was as \thmref{thm:existencemain}. 

Let us now prove the uniqueness of solutions in the class $\mathcal{SA}$. We first need to prove a basic property:

\smallskip
\noindent\textbf{Claim:} Consider a solution $(\U,\Psi,\Pfrak)(t)$  in the class $\mathcal{SA}$ in the time interval $[0,T]$. If there exists $t_0 \in [0,T]$ such that $\norm[2]{\Ztap}(t_0) >0$, then there exists $c_1, c_2>0$ such that for all $\ap \in \Rsp$ and $t \in [0,T]$ we have $\Azero(\ap,t) \geq c_1$ and $\norm[2]{\Ztap}(t) \geq c_2$.

To prove this claim, first note that by the definition of the class $\mathcal{SA}$ and the compactness of the interval $[0,T]$, there exists $N \geq 1$ and $0 = T_0 < T_1 < \cdots < T_N = T$ such that for each interval $[T_i, T_{i + 1}]$ for $0\leq i \leq N-1$, we have a sequence of smooth solutions $(\Z^{i, (n)}, \Zt^{i, (n)})$ converging to the solution $(\Z, \Zt)$ with  $\sup_{t \in [T_i, T_{i + 1}]}\Fcaltil_\Delta((\Z^{i, (n)}, \Zt^{i, (n)}), (\Z, \Zt))(t) \to 0$ as $n \to \infty$. Suppose $0\leq j \leq N-1$ is such that $t_0 \in [T_j, T_{j + 1}]$. Now as $\norm[2]{\Ztap}(t_0) >0$, this implies that $\norm[2]{\Zt - \Avg(\Zt)}(t_0) >0$. Now $\Fcaltil_\Delta((\Z^{j, (n)}, \Zt^{j, (n)}), (\Z, \Zt))(t_0) \to 0$ implies that $\norm*[\big][\Hhalf]{\Zt^{j, (n)} - \Zt}(t_0) \to 0$ and hence $\norm*[\big][2]{(\Zt^{j, (n)} - \Avg(\Zt^{j, (n)})) - (\Zt - \Avg(\Zt))}(t_0) \to 0$. Hence from \eqref{eq:Azerolowerbound} we see that there exists $c_j>0$ and $N_j \in \Nsp$ such that for $n \geq N_j$ we have $\Azero^{j, (n)}(\cdot, t_0) \geq c_j >0$.  Now from \eqref{eq:oneoversqrtAzerobdd} we see that there exists $c_j^*>0$ such that for all $t \in [T_j, T_{j + 1}]$ and $n \geq N_j$ we have $\Azero^{j, (n)}(\cdot, t) \geq c_j^* >0$. Now as $\Fcaltil_\Delta((\Z^{j, (n)}, \Zt^{j, (n)}), (\Z, \Zt))(t) \to 0$ as $n \to \infty$, this implies that $\norm[2]{\Azero^{j, (n)} - \Azero}(t) \to 0$. As $\Azero(\cdot, t)$ is a continuous function (by a similar argument as  \lemref{lem:decayatinfinity}) we see that $\Azero(\cdot, t) \geq c_j^*$. This directly implies that $c_j^* \leq \Azero(\cdot, t) \lesssim \norm[2]{\Ztap}^2(t)$ for all $t \in [T_j, T_{j + 1}]$. Now following this same proof for the other intervals, we prove the claim. 

Let us now complete the proof of uniqueness. If $\norm[2]{\Ztap}(0) = 0$, then by the above claim, we see that $\norm[2]{\Ztap}(t) = 0$ for all $t \in [0,T]$. Hence $\bvar(\ap,t) = \Azero(\ap,t) = 0$ for all $\ap\in \Rsp$ and $t \in [0,T]$ and therefore $\Ztt(\ap, t) = 0$ for all $\ap\in \Rsp$ and $t \in [0,T]$.  Hence there exists a constant $C \in \Csp$ such that $\Zt(\ap,t) = C$ for all $\ap\in \Rsp$ and $t \in [0,T]$. Therefore $\Z(\ap,t) = \Z(\ap, 0) + Ct$. This proves uniqueness for this case. 

If $\norm[2]{\Ztap}(0) > 0$, then by the above claim we see that  there exists $c_1, c_2>0$ such that for all $\ap \in \Rsp$ and $t \in [0,T]$ we have $\Azero(\ap,t) \geq c_1$ and $\norm[2]{\Ztap}(t) \geq c_2$. Now on each interval $[T_i, T_{i+1}]$ as in the proof of the claim, we can use \thmref{thm:uniquenessbdd} to prove uniqueness of the solution in the interval $[T_i, T_{i+1}]$. Hence we are done. 

\end{proof}

We now prove the blow up result \thmref{thm:blowup}. 

\begin{proof}[Proof of \thmref{thm:blowup}]
Consider an initial domain $\Omega(0)$ of the form \figref{fig:blowup} with corners of angle $\nu\pi$ with $0 \leq \nu < \half$ (here $\nu = 0$ corresponds to cusps) which is symmetric with respect to the x-axis and $0 \in \Omega(0)$. Let $\Psi(\cdot, 0) : \Dsp \to \Omega(0)$ be the conformal map with $\Psi(0,0) = 0$ and which is also symmetric with respect to the x-axis. Hence we see that for $\zp \in [-1,1]$, we have $\Psi(\zp,0) \in \Rsp$, $\Psizp(0,0) > 0$ and that $d = \abs{\Z(0,0) - \Z(\pi,0)} > 0$. Consider the initial velocity $\U(\zp, 0) = - \zp$. Note that $v = \abs{\Zt(0,0) - \Zt(\pi,0)} > 0$.  Now it is easy to check that this initial condition satisfies the conditions of the theorem and also has $0< \Ecaltil(0) < \infty$ (see \cite{Ag20} for details on computations regarding the conformal map when the domain has a corner or cusp).   

Now let $T^* > 0$ be the maximal time of existence in the class $\mathcal{SA}$ with this initial data. Clearly by \thmref{thm:existencemainbdd} we have that $T^* \geq \frac{c}{\sqrt{\Ecaltil(0)}}$. We argue by contradiction and assume that $T^* > \frac{d}{v}$. Let $T_0 = \frac{d}{v}$ and so we have a unique solution in the class $\mathcal{SA}$ in the time interval $[0,T_0]$ with $\sup_{t \in [0,T_0]} \Ecaltil(t) < \infty$. As the initial data is symmetric with respect to the x-axis, by using the approximations of the solution by smooth solutions given by the class $\mathcal{SA}$, we see that for each $t \in [0, T_0]$ the solution is also symmetric with respect to the x-axis. 

By the proof of \thmref{thm:existencemainbdd} we see that there exists a constant $C_1 >0$ such that $\norm[2]{\Ztap}(t) \geq C_1$ for all $t \in [0,T_0]$. As $\sup_{t \in [0, T_0]}\Ecaltil(t) < \infty$, this implies that there exists a constant $C_2>0$ such that for all $\zp \in \Dsp$ and $t \in [0, T_0]$ we have
\begin{align}\label{eq:Psizplowerblowup}
\abs{\frac{1}{\Psizp}(\zp,t) } \leq C_2 \quad \implies \abs{\Psizp(\zp, t)} \geq \frac{1}{C_2} > 0
\end{align}

Now by the assumptions we have $0, \pi \in S(0)$ and as the solution is symmetric with respect to the x-axis, we see from \thmref{thm:mainangle} that $0,\pi \in S(t)$ for all $t \in [0,T_0]$. Also from \thmref{thm:mainangle} we see that $\Ztt(0,t) = \Ztt(\pi, t) = 0$. Hence we have that $\Zt(0,t) = \Zt(0,0)$ and $\Zt(\pi,t) = \Zt(\pi,0)$ for all $t \in [0, T_0]$. Therefore $\Z(0,t) = \Z(0,0) + \Zt(0,0)t $ and $\Z(\pi,t) = \Z(\pi,0) + \Zt(\pi,0)t$ and so we have $\Z(0,T_0) = \Z(\pi, T_0)$. Consider the function $g: [-1,1] \to \Csp$ given by $g(\zp) = \Psi(\zp,T_0)$. By the fact that $\Psi(\cdot, T_0)$ is symmetric with respect to the x-axis implies that $g$ is real valued. By the fact that the solution lies in the class $\mathcal{SA}$, we see that $g$ is continuous on $[-1,1]$ and $C^1$ on $(-1,1)$.  As $\Z(0,T_0) = \Z(\pi, T_0)$ this implies that $g(-1) = g(1)$. Therefore by Rolle's theorem, there exists $z_0 \in (-1,1)$ such that $g'(z_0) = 0$. In particular we have $\Psizp(z_0, T_0) = 0$ which directly contradicts \eqref{eq:Psizplowerblowup}. 

Therefore $T^* \leq \frac{d}{v} < \infty$. The blow up condition follows from \thmref{thm:existencemainbdd} and hence we are done. 
\end{proof}

\section{Appendix}\label{sec:appendix}

Let $\Dt = \pt + \bvar\pap$ and let $\Dt^*$ be defined as $\Dt^* f = \pt f + \pap(\bvar f)$. 
We have
\begin{prop}\label{prop:tripleidentity}
Let $n \geq 1$ and let $f_1,\cdots, f_n, g, h \in \mathcal{S}(\Rsp)$. Consider the function $\sqbrac{f_1, \cdots, f_n;  g}:\Rsp \to \Csp$ defined by \eqref{eq:fonefn}. Then we have the following identities
\begin{align*}
h\pap\sqbrac{f_1, \cdots, f_n;  g} & = \sqbrac{h\pap f_1, \cdots, f_n;  g} + \cdots + \sqbrac{f_1, \cdots, h\pap f_n;  g}  \\
& \quad + \sqbrac{f_1, \cdots, f_n;  \pap(hg)} - n\sqbrac{h,f_1,\cdots, f_n; g}
\end{align*}
and 
\begin{align*}
\Dt \sqbrac{f_1, \cdots, f_n;  g} & = \sqbrac{\Dt f_1, \cdots, f_n;  g}  + \cdots + \sqbrac{f_1, \cdots, \Dt f_n;  g}  \\
& \quad + \sqbrac{f_1, \cdots, f_n;  \Dt^* g} - n\sqbrac{\bvar,f_1,\cdots, f_n; g}
\end{align*}  
If we have functions $f_1, f_2, f_3, g, h \in C^\infty(\Sone)$ and we consider the function $\sqbrac{f_1, \cdots, f_n;  g}: \Sone \to \Csp$ for $n = 1,2,3$ given by \eqref{eq:sqfgbounded}, \eqref{eq:sqfoneftwogbounded} and \eqref{eq:foneftwofthreegbounded}, then we have the same two identities as above for $n = 1, 2$. In addition we have the identity
\begin{align*}
[f^2, \Hiltil]\pap g - 2[f,\Hiltil]\pap(fg) = - [f,f;g]
\end{align*}
\end{prop}
\begin{proof}
First we consider the case of functions on $\Rsp$.  We see that
\begingroup
\allowdisplaybreaks
\begin{align*}
&h(\ap)\pap\sqbrac{f_1, \cdots, f_n;  g}  \\
& = h(\ap)\pap\brac{ \frac{1}{i\pi} \int \brac{\frac{f_1(\ap) - f_1(\bp)}{\ap - \bp}}\cdots \brac{\frac{f_n(\ap) - f_n(\bp)}{\ap-\bp}} g(\bp) \diff \bp } \\
& = h(\ap)f_1'(\ap) \brac{ \frac{1}{i\pi} \int \frac{1}{\ap-\bp}\brac{\frac{f_2(\ap) - f_2(\bp)}{\ap - \bp}}\cdots \brac{\frac{f_n(\ap) - f_n(\bp)}{\ap-\bp}} g(\bp) \diff \bp} \\
& \quad + \cdots + h(\ap)f_n'(\ap) \brac{ \frac{1}{i\pi} \int \frac{1}{\ap-\bp}\brac{\frac{f_1(\ap) - f_1(\bp)}{\ap - \bp}}\cdots \brac{\frac{f_{n-1}(\ap) - f_{n-1}(\bp)}{\ap-\bp}} g(\bp) \diff \bp} \\ 
& \quad - \frac{n}{i\pi} \int \brac{\frac{h(\ap) - h(\bp)}{\ap-\bp}}\brac{\frac{f_1(\ap) - f_1(\bp)}{\ap - \bp}}\cdots \brac{\frac{f_n(\ap) - f_n(\bp)}{\ap-\bp}} g(\bp) \diff \bp\\
& \quad - \frac{n}{i\pi} \int \frac{1}{\ap-\bp}\brac{\frac{f_1(\ap) - f_1(\bp)}{\ap - \bp}}\cdots \brac{\frac{f_n(\ap) - f_n(\bp)}{\ap-\bp}} h(\bp)g(\bp) \diff \bp\\
& = \sqbrac{h\pap f_1, \cdots, f_n;  g} + \cdots + \sqbrac{f_1, \cdots, h\pap f_n;  g} + \sqbrac{f_1, \cdots, f_n;  \pap(hg)} - n\sqbrac{h,f_1,\cdots, f_n; g}
\end{align*}
\endgroup
This proves the first identity. The second identity follows directly from this. The proof of these identities for the case of $\Sone$ is identical and so we skip it. Now let us prove the last identity. For $f,g \in C^\infty(\Sone)$ we observe from the first identity proved that
\begin{align}\label{eq:tempderivcomm}
f\pap\sqbrac{f ; g} = \sqbrac{ff';g} + \sqbrac{f;\pap(fg)} - \sqbrac{f,f;g}
\end{align}
On the other hand by using the definition of $\sqbrac{f;g}$ \eqref{eq:sqfgbounded} and by expanding the commutator we obtain
\begingroup
\allowdisplaybreaks
\begin{align*}
f\pap\sqbrac{f ; g} & = f\pap\sqbrac{f,\Hiltil}g \\
& = f\sqbrac{f',\Hiltil}g + f\sqbrac{f,\Hiltil}\pap g \\
& = \sqbrac{ff',\Hiltil}g - \sqbrac{f,\Hiltil}g\pap f + \sqbrac{f^2,\Hiltil}\pap g - \sqbrac{f,\Hiltil}f\pap g \\
& = \sqbrac{ff',\Hiltil}g - \sqbrac{f,\Hiltil}\pap(fg) + \sqbrac{f^2,\Hiltil}\pap g
\end{align*}
\endgroup
Combining this with \eqref{eq:tempderivcomm} gives us the required identity. 
\end{proof}

\begin{prop}\label{prop:Coifman} 
Fix $m \in \Nsp$. Let $H \in C^1(\Rsp),A_i \in C^1(\Rsp) $ for $i=1,\cdots m$ and $F\in C^\infty(\Rsp)$. Define 
\begin{align*}
C_1(H,A,f)(x) & = p.v. \int F\brac{\frac{H(x)-H(y)}{x-y}}\frac{\Pi_{i=1}^{m}(A_i(x) - A_i(y))}{(x-y)^{m+1}}f(y)\diff y \\
C_2(H,A,f)(x) & = p.v. \int F\brac{\frac{H(x)-H(y)}{x-y}}\frac{\Pi_{i=1}^{m}(A_i(x) - A_i(y))}{(x-y)^{m}} \partial_y f(y)\diff y
\end{align*}
Then there exists constants $c_1,c_2,c_3,c_4$ depending only on $m$, $F$ and $\norm[\infty]{H'}$ so that
\begin{enumerate}
\item $\norm[2]{C_1(H,A,f)} \leq c_1\norm[\infty]{A_1'}\cdots\norm[\infty]{A_m'}\norm[2]{f}$

\item $\norm[2]{C_1(H,A,f)} \leq c_2\norm[2]{A_1'}\norm[\infty]{A_2'}\cdots\norm[\infty]{A_m'}\norm[\infty]{f}$ 

\item $\norm[2]{C_2(H,A,f)} \leq c_3\norm[\infty]{A_1'}\cdots\norm[\infty]{A_m'}\norm[2]{f}$ 

\item $\norm[2]{C_2(H,A,f)} \leq c_4\norm[2]{A_1'}\norm[\infty]{A_2'}\cdots\norm[\infty]{A_m'}\norm[\infty]{f}$ 
\end{enumerate}
Now let $A_i \in C^1(\Sone)$ for $i = 1,\cdots m$. Define
\begin{align*}
B_1(A,f)(\ap) = p.v. \int_0^{2\pi} \frac{\Pi_{i=1}^m (A_i(\ap) - A_i(\bp))}{\cbrac{\sin\brac{\frac{\ap - \bp}{2}}}^{m+1}} f(\bp) \diff \bp \\
B_2(A,f)(\ap) = p.v. \int_0^{2\pi}  \frac{\Pi_{i=1}^m (A_i(\ap) - A_i(\bp))}{\cbrac{\sin\brac{\frac{\ap - \bp}{2}}}^m} \partial_{\bp}f(\bp) \diff \bp 
\end{align*} 
Then we have the following estimates
\begin{enumerate}
\item $\norm[2]{B_1(A,f)} \lesssim_m \norm[\infty]{A_1'}\cdots\norm[\infty]{A_m'}\norm[2]{f}$

\item $\norm[2]{B_1(A,f)} \lesssim_m \norm[2]{A_1'}\norm[\infty]{A_2'}\cdots\norm[\infty]{A_m'}\norm[\infty]{f}$ 

\item $\norm[2]{B_2(A,f)}\lesssim_m \norm[\infty]{A_1'}\cdots\norm[\infty]{A_m'}\norm[2]{f}$ 

\item $\norm[2]{B_2(A,f)} \lesssim_m \norm[2]{A_1'}\norm[\infty]{A_2'}\cdots\norm[\infty]{A_m'}\norm[\infty]{f}$ 
\end{enumerate}

\end{prop}
\begin{proof}
Let us first consider the real line case. The first estimate is a theorem by Coifman, McIntosh and Meyer \cite{CoMcMe82}. See also chapter 9 of \cite{MeCo97}. Proof of the second estimate can be found in \cite{Wu09}. The third and fourth estimates can be obtained from the first two by integration by parts. 

The estimates for $\Sone$ can be obtained from the real line case by some modifications. For the functions $B_1(A,f)$ and $B_2(A,f)$, we first replace $\sin\brac{\frac{\ap - \bp}{2}}$ by $e^{i\ap} - e^{i\bp}$ from the formula \eqref{eq:trigidentity}. Now using Proposition 2.2 and Proposition 2.3 of \cite{BiMiShWu17} the estimates follow. 
\end{proof}

\begin{prop}\label{prop:Hardy}
Let $f \in \Scalsp(\Rsp)$. Then we have
\begin{enumerate}

\item $\norm[\infty]{f} \lesssim_s \norm[H^s]{f}$ if $s>\frac{1}{2}$ and for $s=\half$ we have $\norm[BMO]{f} \lesssim \norm[\Hhalf]{f}$

\item $
\begin{aligned}[t]
\norm[\Ltwo(\Rsp, \diff \ap)]{\sup_{\bp}\abs{\frac{f(\ap) - f(\bp)}{\ap - \bp}}} \lesssim \norm[2]{f'}
\end{aligned}
$

\item $
\begin{aligned}[t]
\int \abs{\frac{f(\ap) - f(\bp)}{\ap - \bp}}^2 \diff \bp \lesssim \norm[2]{f'}^2 
\end{aligned}
$

\item $
\begin{aligned}[t]
\norm[\Hhalf]{f}^2 = \frac{1}{2\pi}\int\!\! \!\int \abs{\frac{f(\ap) - f(\bp)}{\ap - \bp}}^2 \diff \bp \diff\ap
\end{aligned}
$

\end{enumerate}
Now let $f \in C^{\infty}(\Sone)$. Then we have
\begin{enumerate}

\item $\norm[\infty]{f} \lesssim_s \norm[H^s]{f}$ if $s>\frac{1}{2}$ and for $s=\half$ we have $\norm[BMO]{f} \lesssim \norm[\Hhalf]{f}$. Moreover $\norm[p]{f - \Avg(f)} \lesssim_p \norm[BMO]{f}$ for any $1\leq p < \infty$. We also have
\begin{align*}
\norm[2]{f - \Avg(f)} \lesssim \norm[\infty]{f - \Avg(f)} \lesssim \norm[1]{f'} \lesssim \norm[2]{f'}
\end{align*}

\item $
\begin{aligned}[t]
\norm[\Ltwo([0,2\pi], \diff \ap)]{\sup_{\bp}\abs{\frac{f(\ap) - f(\bp)}{\sin\brac*[\big]{\frac{\bp - \ap}{2}}}}} \lesssim \norm[2]{f'}
\end{aligned}
$

\item $
\begin{aligned}[t]
\int_0^{2\pi} \abs{\frac{f(\ap) - f(\bp)}{\sin\brac*[\big]{\frac{\bp - \ap}{2}}}}^2 \diff \bp \lesssim \norm[2]{f'}^2 
\end{aligned}
$

\item $
\begin{aligned}[t]
\norm[\Hhalf]{f}^2 = \frac{1}{8\pi}\int_0^{2\pi}\int_0^{2\pi} \abs{\frac{f(\ap) - f(\bp)}{\sin\brac*[\big]{\frac{\bp - \ap}{2}}}}^2 \diff \bp \diff\ap
\end{aligned}
$

\end{enumerate}
\end{prop}
\begin{proof}
See \cite{Ag21} for a proof of the results on $\Rsp$. For functions on $\Sone$ we have:
\begin{enumerate}[leftmargin =*]
\item The first two estimates follow from standard Sobolev embedding results and properties of BMO functions. For the last estimate we only need to prove the middle inequality. Observe that
\begin{align*}
f(x) - f(y) = \int_y^x f'(s) \diff s
\end{align*}
Now averaging over $y$ gives
\begin{align*}
f(x) - \Avg(f) = \frac{1}{2\pi}\int_0^{2\pi} \brac{\int_y^{x} f'(s) \diff s} \diff y
\end{align*}
from which we easily get $\norm[\infty]{f(x) - \Avg(f)} \lesssim \norm[1]{f'}$. 
 
\item We identify $f$ with the $2\pi$ periodic function $\tilde{f} : \Rsp \to \Csp$ defined as $\tilde{f}(x) = f(e^{ix})$. Now we define $f^*:\Rsp \to \Csp$ as
\begin{align*}
f^* = 
\begin{cases}
f' \qq \tx{ on  } [-\pi, 3\pi) \\
0 \qq \tx{ otherwise}
\end{cases}
\end{align*}
Now for any fixed $\ap \in [0,2\pi)$ we see that
\begin{align*}
\sup_{\bp \in [\ap - \pi,\ap + \pi)}\abs{\frac{f(\ap) - f(\bp)}{\sin\brac*[\big]{\frac{\bp - \ap}{2}}}} \lesssim  \sup_{\bp \in [\ap - \pi, \ap + \pi)}\abs{\frac{f(\ap) - f(\bp)}{\ap - \bp}} \lesssim M(f^*)(\ap)
\end{align*}
where $M$ is the uncentered Hardy Littlewood maximal operator on $\Rsp$. As the maximal operator is bounded on $\Ltwo(\Rsp)$, the estimate follows.
\item This is a consequence of the second inequality.
\item The proof of this identity is the same as the one for the real line case. 
\end{enumerate}
\end{proof}

\begin{prop}\label{prop:commutator}
Let $f,g \in \mathcal{S}(\Rsp)$ with $s,a\in \Rsp$ and $m,n \in \Zsp$. Then we have the following estimates
\begin{enumerate}
\item $\norm*[\big][2]{\papabs^s\sqbrac{f,\Hil}(\papabs^{a} g )} \lesssim_{s,a}  \norm*[\big][BMO]{\papabs^{s+a}f}\norm[2]{g}$ \quad  for $s,a \geq 0$

\item $\norm*[\big][2]{\papabs^s\sqbrac{f,\Hil}(\papabs^{a} g )} \lesssim_{s,a}  \norm*[\big][2]{\papabs^{s+a}f}\norm[BMO]{g}$ \quad  for $s\geq 0$ and $a>0$

\item $\norm*[\big][2]{\sqbrac*[\big]{f,\papabs^\half}g } \lesssim \norm*[\big][BMO]{\papabs^\half f}\norm[2]{g}$

\item $\norm*[\big][2]{\sqbrac*[\big]{f,\papabs^\half}(\papabs^\half g) } \lesssim \norm*[\big][BMO]{\papabs f}\norm[2]{g}$

\item $\norm[\Linfty\cap\Hhalf]{\sqbrac{f,\Hil} g} \lesssim \norm*[\big][2]{f'}\norm[2]{g}$ 

\item $\norm[2]{\partial_{\ap}^m\sqbrac{f,\Hil}\partial_{\ap}^n g} \lesssim_{m,n} \norm*[\infty]{\partial_\ap^{(m+n)} f}\norm[2]{g}$ \quad  for $m,n \geq 0$

\item $\norm[2]{\partial_{\ap}^m\sqbrac{f,\Hil}\partial_{\ap}^n g} \lesssim_{m,n} \norm*[2]{\partial_\ap^{(m+n)} f}\norm[\infty]{g}$ \quad for $m\geq 0$ and $n\geq 1$

\item $\norm[2]{\sqbrac{f,\Hil}g} \lesssim \norm[2]{f'}\norm[1]{g}$
\end{enumerate}
The same estimates also hold for $f, g \in C^\infty(\Sone)$. In addition for $f, g \in C^\infty(\Sone)$ the above estimates except (3) and (4) also hold with $\Hil$ replaced with $\Hiltil$. 
\end{prop}
\begin{proof}
See \cite{Ag21} for a proof for the estimates on $\Rsp$. Now let $f, g \in C^\infty(\Sone)$. All of the required estimates except for the $\Linfty$ estimate of (5) and the estimate (8) follow from relatively straightforward modifications to the proof in the $\Rsp$ case. For (5) we see that
\begin{align*}
(\sqbrac*{f,\Hiltil}g)(\ap) = \frac{1}{2\pi i} \int_0^{2\pi} (f(\ap) - f(\bp))\cot\brac{\frac{\bp - \ap}{2}}g(\bp) \diff \bp
\end{align*}
Hence using \propref{prop:Hardy} and the Cauchy Schwarz inequality we get $\norm*[\infty]{\sqbrac*{f,\Hiltil}g} \lesssim \norm[2]{f'} \norm[2]{g}$. Now observe that
\begin{align*}
\sqbrac*{f,\Avg}g = f\Avg(g) - \Avg(fg) = (f - \Avg(f))\Avg(g) - \Avg((f - \Avg(f))g)
\end{align*}
Hence from \propref{prop:Hardy} we get
\begin{align*}
\norm[\infty]{\sqbrac*{f,\Avg}g} \lesssim \norm[\infty]{f - \Avg(f)}\norm[1]{g} \lesssim \norm[2]{f'}\norm[1]{g} \lesssim \norm[2]{f'}\norm[2]{g} 
\end{align*}
As $\Hil = \Hiltil + \Avg$, we therefore obtain $\norm*[\infty]{\sqbrac*{f,\Hil}g} \lesssim \norm[2]{f'} \norm[2]{g}$. The proof for (8) is similar to this computation. 
\end{proof}

\begin{prop}\label{prop:Leibniz}
Let $f,g,h \in \mathcal{S}(\Rsp)$ with $s\in \Rsp$. Then we have the following estimates
\begin{enumerate}
\item $\norm[2]{\papabs^s (fg)} \lesssim_s  \norm[2]{\papabs^s f}\norm[\infty]{g} + \norm[\infty]{f}\norm[2]{\papabs^s g}$ \quad for $s > 0$
\item $\norm[\Hhalf]{fg} \lesssim \norm[2]{f'}\norm[2]{g} + \norm[\infty]{f}\norm[\Hhalf]{g}$
\end{enumerate}
The same estimates also hold for $f,g,h \in C^\infty(\Sone)$.
\end{prop}
\begin{proof}
See \cite{Ag21} for a proof for the estimates on $\Rsp$ and the proof for functions on $\Sone$ is very similar to the $\Rsp$ case.
\end{proof}

\begin{prop}\label{prop:triple}
Let $f,g,h \in \mathcal{S}(\Rsp)$ . Then we have the following estimates
\begin{enumerate}

\item $\norm[2]{\sqbrac{f,g;h}} \lesssim \norm[2]{f'}\norm[2]{g'}\norm[2]{h}$

\item $\norm[2]{\sqbrac{f,g; h'}} \lesssim \norm[\infty]{f'}\norm[\infty]{g'}\norm[2]{h}$

\item $\norm[\Linfty\cap\Hhalf]{\sqbrac{f,g;h}} \lesssim \norm[\infty]{f'}\norm[2]{g'}\norm[2]{h}$

\end{enumerate}
The same estimates hold if instead we have $f,g,h \in C^\infty(\Sone)$.
\end{prop}
\begin{proof}
See \cite{Ag21} for a proof for the $\Rsp$ case. For functions on $\Sone$, the proofs are essentially identical to the real line case, with the only main difference being that for the second estimate we have to use the $\Sone$ estimates from \propref{prop:Coifman}. 
\end{proof}

\begin{prop} \label{prop:LinftyHhalf}
Let $f \in \mathcal{S}(\Rsp)$ and let $w$ be a smooth non-zero weight with $w,\frac{1}{w} \in \Linfty(\Rsp) $ and $w' \in \Ltwo(\Rsp)$. Then 
\begin{align*}
\norm[\Linfty\cap\Hhalf]{f}^2 \lesssim \norm[2]{\frac{f}{w}}\norm[2]{wf'} +  \norm[2]{\frac{f}{w}}^2\norm[2]{w'}^2
\end{align*}
If $f, w, \frac{1}{w} \in C^\infty(\Sone)$, then we have
\begin{align*}
\norm[\Linfty\cap\Hhalf]{f}^2 \lesssim \norm[2]{\frac{f}{w}}\norm[2]{wf'} +  \norm[2]{\frac{f}{w}}^2\norm[2]{w'}^2 + \norm[2]{f}^2
\end{align*}
\end{prop}
\begin{proof}
See \cite{Ag21} for a proof for the $\Rsp$ case (note that it does not matter whether we have $\norm[2]{wf'}$ or $\norm[2]{(wf)'}$ on the right hand side). Now for the $\Sone$ case, the proof of the $\Hhalf$ estimate is the same as for $\Rsp$ and we don't actually need the extra $\norm[2]{f}^2$ on the right hand side. For the $\Linfty$ estimate, observe that
\begin{align*}
f^2(\ap) - f^2(\bp) = 2\int_\bp^\ap \brac{\frac{f}{w}} (wf') \diff s
\end{align*}
Hence by averaging in $\bp$, we see that 
\begin{align*}
\abs{f^2(\ap) - \Avg(f^2)} \lesssim \norm[2]{\frac{f}{w}}\norm[2]{wf'} 
\end{align*}
Therefore we see that
\begin{align*}
\norm[\infty]{f}^2 \lesssim \norm[2]{\frac{f}{w}}\norm[2]{wf'}  + \abs{\Avg(f^2)} \lesssim \norm[2]{\frac{f}{w}}\norm[2]{wf'} + \norm[2]{f}^2
\end{align*} 
\end{proof}

\begin{prop}\label{prop:Hhalfweight}
Let $f,g \in \mathcal{S}(\Rsp)$ and let $w,h \in \Linfty(\Rsp)$ be smooth functions with $w',h' \in \Ltwo(\Rsp)$. Then 
\begin{align*}
\norm[\Hhalf]{fwh} \lesssim \norm[\Hhalf]{fw}\norm[\infty]{h} + \norm[2]{f}\norm[2]{(wh)'} + \norm[2]{f}\norm[2]{w'}\norm[\infty]{h}
\end{align*}
If in addition we assume that $w$ is real valued then
\begin{align*}
\norm[2]{fgw} \lesssim \norm[\Hhalf]{fw}\norm[2]{g} + \norm[\Hhalf]{gw}\norm[2]{f} + \norm[2]{f}\norm[2]{g}\norm[2]{w'} 
\end{align*}
If $f,g,h, w \in C^\infty(\Sone)$, then the first estimate also holds and the second estimate gets modified to
\begin{align*}
\norm[2]{fgw} \lesssim \norm[\Hhalf]{fw}\norm[2]{g} + \norm[\Hhalf]{gw}\norm[2]{f} + \norm[2]{f}\norm[2]{g}\brac{\norm[2]{w} + \norm[2]{w'} }
\end{align*}
\end{prop}
\begin{proof}
See \cite{Ag21} for a proof for the real line case. The proof of the $\Hhalf$ estimate on $\Sone$ case is identical to the real line case. Similarly for the $\Ltwo$ estimate, by following the proof we see that the same estimate holds if $f$ and $g$ have zero mean and hence
\begin{align*}
& \norm[2]{(f - \Av(f))(g - \Av(g))w} \\
& \lesssim \norm[\Hhalf]{(f - \Av(f)) w}\norm[2]{g - \Av(g)} + \norm[\Hhalf]{(g - \Av(g)) w}\norm[2]{f - \Av(f)} \\
& \quad + \norm[2]{f - \Av(f)}\norm[2]{g - \Av(g)}\norm[2]{w'} \\
& \lesssim \norm[\Hhalf]{fw}\norm[2]{g} + \norm[\Hhalf]{gw}\norm[2]{f} + \norm[2]{f}\norm[2]{g}\brac{\norm[2]{w} + \norm[2]{w'} }
\end{align*}
Now observe that
\begin{align*}
fgw =  (f - \Av(f))(g - \Av(g))w + \Av(f)gw + \Av(g)fw -\Av(f)\Av(g)w
\end{align*}
By taking the $\Ltwo$ norms on both side, we get the required estimate. 

\end{proof}

\begin{prop} \label{prop:DtLinfty}
Let $f  \in  C^3([0,T), H^3(\Rsp))$. Then for any $t\in [0,T)$ we have
\begin{align*}
\limsup_{s \to 0^+} \frac{\norm[\infty]{f(\cdot,t+s)} - \norm[\infty]{f(\cdot,t)}}{s} \leq \norm[\infty]{\pt f(\cdot,t)}
\end{align*}
The same estimate holds for $f  \in  C^3([0,T), H^3(\Tsp))$ as well. 
\end{prop}
\begin{proof}
See \cite{Ag21} for a proof for the $\Rsp$ case and the $\Tsp$ case is proved in the same manner.
\end{proof}

\begin{lem}\label{lem:decayatinfinity}
Let $f \in H^1(\Rsp)$ and consider the function $g : \Rsp \to \Rsp$
\begin{align*}
g(\ap) = \int_\Rsp \abs{\frac{f(\ap)-f(\bp)}{\ap-\bp}}^2\diff\bp
\end{align*}
Then $g$ is a continuous function and $g(\ap) \to 0$ as $\abs{\ap} \to \infty$. 
\end{lem}
\begin{proof}
From  \propref{prop:Hardy} it is clear that $\norm[\infty]{g} \lesssim \norm[2]{f'}^2$.  Let $f_n:\Rsp \to \Csp$ be a sequence of smooth functions with compact support such that $f_n \to f$ in $H^1(\Rsp)$. Consider the functions
\begin{align*}
g_n(\ap) = \int_\Rsp \abs{\frac{f_n(\ap)-f_n(\bp)}{\ap-\bp}}^2\diff\bp
\end{align*}
We again have the estimate $\norm[\infty]{g_n} \lesssim \norm[2]{f'_n}^2$ and it is clear that for each fixed $n$, we have $g_n$ is a continuous function and $g_n(\ap) \to 0  $ as $\abs{\ap} \to \infty$. Now by \propref{prop:Hardy} we see that for $n  \in \Nsp$
\begin{align*}
\norm[\infty]{g_n - g} \lesssim \norm[2]{f_n' - f'}\brac{\norm[2]{f_n'} + \norm[2]{f'}}
\end{align*}
Therefore $g_n \to g$ in $\Linfty(\Rsp)$ and hence we are done. 
\end{proof}

\medskip
\noindent \textbf{Acknowledgment}: The author was supported by the National Science Foundation under Grant No. DMS-1928930 while participating in a program hosted by MSRI during the Spring 2021 semester. The author also received funding from the European Research Council (ERC) under the European Union’s Horizon 2020 research and innovation programme through the grant agreement 862342. 

\medskip
\noindent \textbf{Data availability}: This article has no associated data.

\medskip
\noindent \textbf{Conflict of interest}: The author declares no conflict of interest.


\bibliographystyle{amsplain}
\providecommand{\bysame}{\leavevmode\hbox to3em{\hrulefill}\thinspace}
\providecommand{\MR}{\relax\ifhmode\unskip\space\fi MR }
\providecommand{\MRhref}[2]{%
  \href{http://www.ams.org/mathscinet-getitem?mr=#1}{#2}
}
\providecommand{\href}[2]{#2}

\end{document}